\documentclass[12pt,reqno]{amsart}

\usepackage{amsmath,amsthm,amsfonts,amssymb,times}
\usepackage{graphics, graphicx,fancyhdr}
\usepackage[dvipsnames]{xcolor}
\usepackage{tikz}
\usepackage[top=1in, left=1in, right=1in, bottom=1in]{geometry}
\usetikzlibrary{trees}
\usepackage{amssymb}
\usepackage{amscd}

\usepackage{amsfonts,dsfont}
\usepackage{setspace}
\usepackage{mathrsfs}
\usepackage{version}
\usepackage{bm}
\usepackage{color}
\usepackage{multirow}
\usetikzlibrary{positioning}
\usetikzlibrary{arrows}
\usepackage{rotating}

\newcommand{\rev}[1]{#1}
\makeatletter
\renewcommand\subsection{\@startsection{subsection}{2}%
	\z@{.5\linespacing\@plus0\linespacing}{.5\linespacing}%
	{\normalfont\bfseries}}
\makeatother

\makeatletter
\renewcommand\part{\@startsection{part}{2}%
	\z@{0.5\linespacing\@plus1\linespacing}{\linespacing}%
	{\normalfont\large\scshape\bfseries\centering}}
\makeatother

\definecolor{red}{rgb}{1,0,0}

\definecolor{orange}{rgb}{0.7,0.3,0}

\definecolor{blue}{rgb}{.2,.6,.75}

\definecolor{green}{rgb}{.4,.7,.4}

\usepackage[colorlinks=true, linkcolor=NavyBlue, citecolor=green, urlcolor=purple]{hyperref}   %% linkable cross-refs

\newtheorem{maintheorem}{Theorem}
\newtheorem{theorem}{Theorem}[section]
\newtheorem{lemma}[theorem]{Lemma}
\newtheorem{proposition}[theorem]{Proposition}
\newtheorem{maincorollary}{Corollary}
\newtheorem{corollary}[theorem]{Corollary}
\newtheorem{conjecture}{Conjecture}

\theoremstyle{definition}
\newtheorem{mainproblem}{Problem}
\newtheorem{problem}[theorem]{Problem}
\newtheorem{definition}[theorem]{Definition}

\theoremstyle{remark}
\newtheorem{remark}{Remark}[section]
\newtheorem*{remark*}{Remark}
\newtheorem*{remarks*}{Remarks}
\newtheorem{example}[remark]{Example}
\newtheorem*{example*}{Example}

\renewcommand{\leq}{\leqslant}
\renewcommand{\geq}{\geqslant}

\newcommand\id{\operatorname{id}}
\newcommand\Span{\operatorname{Span}}

\newcommand\Supp{\operatorname{Supp}}

\newcommand\res{\operatorname{res}}
\newcommand\nondegenerate{\operatorname{nondegenerate}}

\newcommand{\basic}{\operatorname{basic}}
\newcommand{\semibasic}{\operatorname{semibasic}}

\def\F{\mathbb{F}}
\def\R{\mathbb{R}}

\def\Z{\mathbb{Z}}
\def\E{\mathbb{E}}
\def\P{\mathbb{P}}
\def\Q{\mathbb{Q}}
\def\N{\mathbb{N}}
\newcommand{\HH}{\mathbb{H}}  %% conflicts with Hungarian \H o for Erd\H os

\def\A{\mathbf{A}}

\def\eps{\varepsilon}
\def\x{\mathbf{x}}

\newcommand{\md}[1]{\ensuremath{(\operatorname{mod}\, #1)}}

\numberwithin{equation}{section}

\newcommand{\cA}{\ensuremath{\mathcal{A}}}
\newcommand{\cB}{\ensuremath{\mathcal{B}}}

\newcommand{\cE}{\ensuremath{\mathcal{E}}}
\newcommand{\cF}{\ensuremath{\mathcal{F}}}

\newcommand{\cI}{\ensuremath{\mathcal{I}}}
\newcommand{\cJ}{\ensuremath{\mathcal{J}}}
\newcommand{\cK}{\ensuremath{\mathcal{K}}}

\newcommand{\cP}{\ensuremath{\mathcal{P}}}

\newcommand{\cU}{\ensuremath{\mathcal{U}}}

\newcommand{\sI}{\ensuremath{\mathscr{I}}}

\newcommand{\sL}{\ensuremath{\mathscr{L}}}

\newcommand{\sP}{\ensuremath{\mathscr{P}}}

\newcommand{\sS}{\ensuremath{\mathscr{S}}}

\newcommand{\sT}{\ensuremath{\mathscr{T}}}
\newcommand{\sV}{\ensuremath{\mathscr{V}}}
\newcommand{\sW}{\ensuremath{\mathscr{W}}}

\newcommand{\dalign}[1]{\[\begin{aligned} #1 \end{aligned}\]}

\newcommand{\be}{\begin{equation}}
\newcommand{\ee}{\end{equation}}

\newcommand{\ssum}[1]{\sum_{\substack{#1}}}  %%% stacked sum
\newcommand{\sprod}[1]{\prod_{\substack{#1}}}  %% stacked product
\newcommand{\fl}[1]{{\ensuremath{\left\lfloor {#1} \right\rfloor}}}  % floor
\newcommand{\cl}[1]{{\ensuremath{\left\lceil #1 \right\rceil}}} % ceiling
\renewcommand{\(}{\left(}
\renewcommand{\)}{\right)}

\newcommand{\pfrac}[2]{\left(\frac{#1}{#2}\right)}  %%% frac with paren

\newcommand{\er}{\mathrm{e}}   %% entropy sum
\newcommand{\om}{\omega}

\newcommand{\cc}{{\mathbf{c}}}
\newcommand{\bb}{{\mathbf{b}}}
\newcommand{\bn}{{\mathbf{n}}}
\newcommand{\btau}{\boldsymbol{\tau}}

\newcommand\length{\operatorname{length}}
\newcommand\yy{\mathbf{y}}

\newcommand{\bmu}{{\boldsymbol{\mu}}}
\newcommand{\brho}{\boldsymbol{\rho}}

\newcommand\un{1}

\newcommand{\bs}[1]{\boldsymbol{#1}}

\renewcommand{\le}{\leqslant}
\renewcommand{\leq}{\leqslant}
\renewcommand{\ge}{\geqslant}
\renewcommand{\geq}{\geqslant}

\newcommand{\one}{\mathbf{1}}  %%% special vectors

\newcommand{\spanone}{\langle \mathbf{1} \rangle}

\newcommand{\dee}{\,\mathrm{d}}

\AtBeginDocument{
  \addtocontents{toc}{\Small}
  \addtocontents{lof}{\Small}
}

\begin{document}

\title[Equal sums and the concentration of divisors]{Equal sums in random sets and the concentration of divisors}
%\date{\today}

\author{Kevin Ford}
\address{Department of Mathematics\\ University of Illinois at Urbana--Champaign\\ Urbana\\ Illinois 61801}
\email{ford126@illinois.edu}

\author{Ben Green}
\address{Mathematical Institute\\ Andrew Wiles Building\\ Radcliffe Observatory Quarter \\Woodstock Road\\ Oxford OX2 6GG, UK}
\email{ben.green@maths.ox.ac.uk}

\author{Dimitris Koukoulopoulos}
\address{D\'epartement de math\'ematiques et de statistique\\
	Universit\'e de Montr\'eal\\
	CP 6128 succ. Centre-Ville\\
	Montr\'eal, QC H3C 3J7\\
	Canada}
\email{dimitris.koukoulopoulos@umontreal.ca}


%    author one information

%\onehalfspace
%    \subjclass is required.
%    The 2010 edition of the Mathematics Subject Classification is
%    now available.  If you are citing a classification from the
%    new scheme, use the following input coding instead.
%\subjclass[2010]{Primary }

%    Abstract is required.

\begin{abstract}
	We study the extent to which divisors of a typical integer $n$ are concentrated. In particular, defining $\Delta(n) := \max_t \# \{d | n, \log d \in [t,t+1]\}$, we show that $\Delta(n) \geq (\log \log n)^{0.35332277\dots}$ for almost all $n$, a bound we believe to be sharp. This disproves a conjecture of Maier and Tenenbaum.  We also prove analogs for the concentration of
	divisors of a random permutation and of a random polynomial
	over a finite field.
	
	Most of the paper is devoted to a study of the following much more combinatorial problem of independent interest. Pick a random set $\mathbf{A} \subset \mathbb{N}$ by selecting $i$ to lie in $\mathbf{A}$ with probability $1/i$. What is the supremum of all exponents $\beta_k$ such that, almost surely as $D \rightarrow \infty$, some integer is the sum of elements of $\mathbf{A} \cap [D^{\beta_k}, D]$ in $k$ different ways?
	
	We characterise $\beta_k$ as the solution to a certain optimisation problem over measures on the discrete cube $\{0,1\}^k$, and obtain lower bounds for $\beta_k$ which we believe to be asymptotically sharp.
\end{abstract}

\maketitle

%\vspace{-0.5in}

\renewcommand{\baselinestretch}{0.0}\normalsize
\tableofcontents
\renewcommand{\baselinestretch}{1.0}\normalsize

%%%%%%%%%%%%%%%%%%%%%%%%%%%%%%%%%%%%%%%%%%%%%%%%%%%%%%%%
%%%%%%%%%%%%%%%%%%%%%%%%%%%%%%%%%%%%%%%%%%%%%%%%%%%%%%%%

\clearpage
\thispagestyle{fancy}
\fancyhf{} % sets both header and footer to nothing
\renewcommand{\headrulewidth}{0cm}
\lhead[{\scriptsize \thepage}]{}
\rhead[]{{\scriptsize\thepage}}
\part{Main results and overview of the paper}\label{intro-part}

\section{Introduction}
\label{intro}
\subsection{The concentration of divisors}
Given an integer $n$, we define the Delta function
\[  \Delta(n) := \max_t \# \{d | n, \log d \in [t,t+1]\},\] that is to say the maximum number of divisors $n$ has in any interval of logarithmic length $1$. Its normal order (almost sure behaviour) has proven quite mysterious, and
indeed it was a celebrated achievement of Maier and Tenenbaum \cite{MT84}, answering a question of Erd\H{o}s from 1948 \cite{Erdos48},
to show that $\Delta(n) > 1$ for \rev{almost all\footnote{A property of natural numbers is said to occur
	for \emph{almost all} $n$ if the number of exceptions below $x$ is $o(x)$ as $x\to\infty$.} $n$.}

Work on the  distribution of $\Delta$ began in the 1970s with Erd\H os and Nicolas \cite{ErdNic75, ErdNic76}. However, it was not until the work of Hooley \cite{Hooley} that the Delta function received proper attention. Among other things, Hooley showed how bounds on the average size of $\Delta$ can be used to count points on certain algebraic varieties. Further work on the normal and average behavior of $\Delta$ can be found in the papers of Tenenbaum \cite{Ten85, Ten86}, Hall and Tenenbaum \cite{HT82, HT84, HT86}, and of Maier and Tenenbaum  \cite{MT84, MT85, MT09}.  See also 
\cite[Ch. 5,6,7]{HT-Divisors}. Finally, Tenenbaum's survey paper \cite[p. 652--658]{Ten13} includes a history of the Delta function and description of  many applications in number theory.

The best bounds for $\Delta(n)$ for ``normal'' $n$
currently known were obtained in a more recent paper of Maier and Tenenbaum \cite{MT09}.

\noindent
\textbf{Theorem MT (Maier--Tenenbaum \cite{MT09})}
\textit{
\rev{Let $\eps>0$ be fixed. Then
\[ (\log \log n)^{c_1 - \eps} \leq  \Delta(n) \leq  (\log \log n)^{\log 2 + \eps},\]
for almost all $n$}, where
\[ c_1 = \frac{\log 2}{\log\big(\frac{1 - 1/\log 27}{1 - 1/\log 3}    \big)    } \approx 0.33827.\] 
}

It is conjectured in \cite{MT09} that the lower bound is optimal.

One of the main results of this paper is a disproof of this conjecture.

\begin{maintheorem}\label{main-thm-1}
\rev{Let $\eps>0$ be fixed. Then
\[ \Delta(n) \geq (\log \log n)^{\eta - \eps}\] 
for almost all $n$}, where $\eta = 0.35332277270132346711\dots$.
\end{maintheorem}
The constant $\eta$, which we believe to be sharp, is described in relation \eqref{eta-rho} below, just after the statement of Theorem \ref{beta-rho}.

\subsection{Packing divisors} 

Let us briefly attempt to explain, without details, why it was natural for Maier and Tenenbaum to make their conjecture, and what it is that allows us to find even more tightly packed divisors. 

We start with a simple observation. Let $n$ be an integer, and suppose we can find pairs of divisors $d_i, d'_i$ of $n$, $i = 1,\dots, k$, such that 
\begin{itemize}
\item $1 < d_i/d'_i \leq 2^{1/k}$;
\item The sets of primes dividing $d_id'_i$ are disjoint, as $i$ varies in $\{1,\dots, k\}$. 
\end{itemize}
Then we can find $2^k$ different divisors of $n$ in a dyadic interval, namely 
\rev{all products $a_1\cdots a_k$ where} $a_i$ is either $d_i$ or $d'_i$. 

In \cite{MT09}, Maier and Tenenbaum showed how to find many such pairs of divisors $d_i, d'_i$. To begin with, they look only at the large prime factors of $n$. They first find one pair $d_1, d'_1$ using the technique of \cite{MT84}. Then, using a modification of the argument, they locate a further pair $d_2$ and $d'_2$, but with these divisors not having any primes in common with $d_1, d'_1$. They continue in this fashion to find $d_3, d'_3$, $d_4,d_4'$, etc., until essentially all the large prime divisors of $n$ have been used. After this, they move on to a smaller range of prime factors of $n$, and so on. 

By contrast, we eschew an iterative approach and select $2^k$ close divisors from amongst the large prime divisors of $n$ in one go, in a manner that is combinatorially quite different to that of Maier and Tenenbaum.  We then apply a similar technique to a smaller range of prime factors of $n$, and so on. This turns out to be a more efficient way of locating proximal divisors. 

In fact, we provide a general framework that encapsulates all possible combinatorial constructions one might use to pack many divisors close to each other. To work in this generality it is necessary to use a probabilistic formalism. One effect of this is that, even though our work contains that of Maier and Tenenbaum as a special case, the arguments here will look totally different.

\subsection{Random sets and equal sums}\label{sec:equal sums} For most of the paper we do not talk about integers and divisors, but rather about the following model setting. Throughout the paper, $\A$ will denote a random set of positive integers in which $i$ is included in $\A$ with probability $1/i$, these choices being independent for different $i$s. 
We refer to $\A$ as a \emph{logarithmic random set}.

A large proportion of our paper will be devoted to understanding conditions under which there is an integer which can be represented as a sum of elements of $\A$ in (at least) $k$ different ways. In particular, we wish to obtain bounds on the quantities $\beta_k$ defined in the following problem.

\begin{mainproblem}\label{beta-k-def}
Let $k \geq 2$ be an integer. Determine $\beta_k$, the supremum of all exponents $c < 1$ for which the following is true: \rev{with probability tending to 1 as $D \rightarrow \infty$, there are distinct sets $A_1, \ldots, A_k \subset \A \cap [D^c, D]$ with equal sums, i.e., $\sum_{a \in A_1} a = \cdots = \sum_{a \in A_k} a$.}
\end{mainproblem}

The motivation for the random set $\A$ comes from our knowledge of
the anatomy of integers, permutations and polynomials.
For a random integer $m\le x$, with $x$ large, let $U_k$ be the event \rev{that
$m$ has  a} prime factor in the interval $(e^{k},e^{k+1}]$.
For a random permutation $\sigma\in S_n$, let 
$V_k$ be the \rev{event} that $\sigma$ has a cycle of size $k$,
and for a random monic polynomial $f$ of degree $n$ over $\mathbb{F}_q$, with $n$ large, let $W_k$ be the \rev{event} that
$f$ has an irreducible factor of degree $k$.
Then it is known \rev{(see e.g., \cite{ABT93, AT92, HT-Divisors})}
 that $U_k$, $V_k$ and $W_k$ 
each occur with probability close to $1/k$, and also that the
\rev{$U_k$} are close to independent for $k=o(\log x)$,
the $V_k$ are close to independent for $k=o(n)$,
\rev{and the $W_k$ are close to} independent for $k$ large
and $k=o(n)$.  Thus, the model set $\A$ captures the
factorization structure of random integers, 
random permutations and random polynomials over
a finite field.
 It is then relatively straightforward to
 transfer results about subset sums of $\A$ to 
 divisors of integers, permutations and polynomials.
Section \ref{transference} below contains details of
the transference principle.

The main result of this paper is an asymptotic lower bound on
$\beta_k$.

\begin{maintheorem}\label{beta-rho}
We have $\liminf_{r\to \infty} (\beta_{2^r})^{1/r} \ge \rho/2$,
where $\rho=0.28121134969637466015\dots$ is a specific constant defined	as 
\rev{the unique solution in $[0,1/3]$ of
\be\label{rho-eq-main}
\frac{1}{1 - \rho/2} = \lim_{j\to\infty} \frac{\log a_j}{2^{j-2}} ,
\ee}
where the sequence $a_j$ is defined by
\[ a_1 = 2, 
\quad
a_2 = 2 + 2^{\rho}, \quad
a_j = a_{j-1}^2 + a_{j-1}^{\rho} - a_{j-2}^{2\rho} \qquad  (j \geq 3).\]
\end{maintheorem}

\rev{The proof of Theorem \ref{beta-rho} will occupy the bulk of this paper, and has three basic parts:
	\begin{enumerate}
		\item[(a)] Showing that for every $r\ge 1$, $\beta_{2^r} \ge \theta_r$ for a certain explicitly defined constant $\theta_r$; 
		\item[(b)] Showing that $\lim_{r\to\infty} \theta_r^{1/r}$ exists;
		\rev{\item[(c)] Showing that \eqref{rho-eq-main} has a unique solution
			$\rho\in [0,1/3]$ and that $\rho=2\lim_{r\to\infty} \theta_r^{1/r}$.}
	\end{enumerate}
	
	In the sequel we shall refer to ``Theorem \ref{beta-rho} (a)'',
	``Theorem \ref{beta-rho} (b)'' and ``Theorem \ref{beta-rho} (c)''.
	\rev{Parts (a), (b) and (c) are quite} independent of one another, with the proof of (a) (given in subsection \ref{theta-r}) being by far the \rev{longest of the three.} The definition of $\theta_r$, while somewhat complicated, is fairly self-contained: see Definition \ref{theta-r-def}. \rev{Parts (b) and (c) are then problems of an} analytic and combinatorial flavour which can be addressed largely independently of the main arguments of the paper.  The formula \eqref{rho-eq-main} allows for a 
	quick computation of $\rho$ to many decimal places, as the limit on the
	right side converges extremely rapidly.  See section \ref{sec:analytic} for details.
	
	\medskip
	
	Let us now state an important corollary of Theorem \ref{beta-rho}.
	
\begin{maincorollary}\label{cor:lower bound on zeta}
Define
	\begin{equation}\label{zeta-def}
		\zeta_+ = \limsup_{k\to \infty} \frac{\log k}{\log (1/\beta_{k})} 
		\quad\text{and}\quad
		\zeta_- = \liminf_{k\to \infty} \frac{\log k}{\log (1/\beta_{k})} .
	\end{equation}
	Then
		\begin{equation}\label{eta-rho}
	\zeta_+\ge\zeta_-\ge \eta:= \frac{\log2}{\log(2/\rho)}  = 0.3533227\dots.
	\end{equation}
\end{maincorollary}

\begin{proof} 
Evidently, $\zeta_+\ge\zeta_-$. In addition, observe the trivial bound $\beta_k \le \beta_{k+1}$. Hence,
\begin{equation}\label{beta-dyadic}
\zeta_+ = \limsup_{r\to \infty} \frac{r\log 2}{\log (1/\beta_{2^r})} 
\quad\text{and}\quad
\zeta_- = \liminf_{r\to \infty} \frac{r\log 2}{\log (1/\beta_{2^r})} .
\end{equation}
We then use Theorem \ref{beta-rho} to find that $\zeta_-\ge\eta$.
\end{proof}

We conjecture that our lower bounds on $\beta_k$ are asymptotically sharp, so that the following holds:

\begin{conjecture}\label{main_conj}
We have $\zeta_+=\zeta_- = \eta$.
\end{conjecture}

We will address the exact values of $\beta_k$ in a future paper; in particular, we will show that
\[
\beta_3 = \frac{\log 3 -1}{\log 3 + \frac{1}{\xi}}
=0.02616218797316965133\dots
\]
and
\[
\beta_4 = \frac{\log 3-1}{\log 3 + \frac{1}{\xi} + \frac{1}{\xi \lambda}} = 0.01295186091360511918\dots
\]
where 
\[
\xi = \frac{\log 2 - \log(e-1)}{\log(3/2)}, \qquad
\lambda = \frac{\log 2 - \log(e-1)}{1 + \log 2 - \log(e-1)-\log(1+2^{1-\xi})}.
\]}

%%%%%%%%%%%%%%%%%%%%%%%%%%%%%%%%%%%%%%%%%%%%%%%%%%%%%%%%%%%%%
%
\subsection{Application to divisors of integers, permutations and polynomials}
%
%%%%%%%%%%%%%%%%%%%%%%%%%%%%%%%%%%%%%%%%%%%%%%%%%%%%%%%%%%%%%

The link between Problem \ref{beta-k-def} and 
the concentration of divisors  is given by the following 
Theorems.  The proofs are relatively straightforward and given 
in the next section. Recall from \eqref{zeta-def} \rev{the definition of $\zeta_+$.}

\begin{maintheorem}\label{delta-from-model}
\rev{For any $\eps>0$, we have
\[
\Delta(n) \ge (\log \log n)^{\zeta_+ - \eps}
\]
for almost every $n$.}
\end{maintheorem}

\rev{\begin{remark*}
		In principle, the proof of Theorem \ref{delta-from-model} yields an explicit bound on the size of the set of integers $n$ with $\Delta(n)\le (\log\log n)^{\zeta_+-\eps}$. However, incorporating such an improvement is a very complicated task. In addition, the obtained bound will presumably be rather weak without a better understanding of the theoretical tools we develop (cf.~Section \ref{overview-sec}).
	\end{remark*}}

The same probabilistic setup allows us to quickly make similar conclusions about
the distribution of divisors (product of cycles) of permutations and of polynomials over finite fields.

\begin{maintheorem}\label{delta-perm}
For a permutation $\sigma$ on $S_n$, denote by
\[
\Delta(\sigma) := \max_r \# \{ d| \sigma: \length(d)=r \},
\]
where $d$ denotes a generic divisor of $\sigma$; that is, $d$ is the product of
a subset of the cycles of $\sigma$.  

\rev{Let $\eps>0$ be fixed. If $n$ is sufficiently large in terms of $\eps$, then for at least $(1-\eps)(n!)$ of the permutations $\sigma\in S_n$, we have
\[
\Delta(\sigma) \ge  (\log n)^{\zeta_+ - \eps}.
\]}
\end{maintheorem}

\begin{maintheorem}\label{delta-polynomial}
Let $q$ be any prime power. For a polynomial $f\in \mathbb{F}_q[t]$,
let
\[
\Delta(f) = \max_r \# \{ g| f : \deg(g)=r \}.
\]

\rev{Let $\eps>0$ be fixed. If $n$ is sufficiently large in terms of $\eps$, then
at least $(1-\eps) q^n$ monic polynomials of degree $n$ satisfy
\[
\Delta(f) \ge  (\log n)^{\zeta_+ - \eps}.
\]}
\end{maintheorem}

\begin{conjecture}\label{conj:Delta}
The lower bounds given in Theorems \ref{delta-from-model}, \ref{delta-perm} and \ref{delta-polynomial} are sharp.  \rev{That is,
corresponding upper bounds with exponent $\zeta_+ + \eps$ hold.}
\end{conjecture}

If both Conjectures \ref{main_conj} and \ref{conj:Delta} hold,
then we deduce that the optimal exponent in the above theorems is equal to $\eta$.

\begin{remark*}
\rev{The exponent $\zeta_+-\eps$} in Theorems \ref{delta-from-model}, \ref{delta-perm} and \ref{delta-polynomial} depends only
on accurate asymptotics for $\beta_k$ as $k\to \infty$ \rev{or, even more weakly, for $\beta_{2^r}$ as  $r\to\infty$ (cf.~\eqref{beta-dyadic}).}
In this work, however, we develop a framework for determining $\beta_k$ exactly for each $k$. 
\end{remark*}

The quantity $\beta_k$ is also closely related to the densest
packing of $k$ divisors of a typical integer.  To be specific,
we define $\alpha_k$ be the supremum of all real numbers $\alpha$ such that
for almost every $n\in \N$, $n$ has $k$ divisors $d_1<\cdots <d_k$
with $d_k \le d_1 (1+ (\log n)^{-\alpha})$.
In 1964, Erd\H os \cite{Erdos64} conjectured that
$\alpha_2 = \log 3 -1$, and this was confirmed by
Erd\H os and Hall \cite{EH79} (upper bound) and 
Maier and Tenenbaum \cite{MT84} (lower bound).
The best bounds on $\alpha_k$ for $k\ge 3$ are given by
Maier and Tenenbaum \cite{MT09}, who showed that
\[
\alpha_k \le \frac{\log 2}{k+1} \qquad (k\ge 3)
\]
and (this is not stated explicitly in \cite{MT09})
\be\label{MT-alpha-lower}
\alpha_k \ge \frac{(\log 3-1)^m 3^{m-1}}{(3\log 3 -1 )^{m-1}}
\qquad (2^{m-1} < k \le 2^m, m\in \N).
\ee
See also \cite[p. 655--656]{Ten13}\footnote{The factor $3^{m-1}$ is missing in the stated  lower bounds for $\alpha_k$ in \cite{Ten13}.}.
In particular, it is not known if $\alpha_3 > \alpha_4$,
although Tenenbaum \cite{Ten13} conjectures that
the sequence 
$(\alpha_k)_{k\ge 2}$ is strictly decreasing.

We can quickly deduce a lower bound for $\alpha_k$ 
in terms of $\beta_k$.

\begin{maintheorem}\label{thm:alphak}
 \rev{For all $k\ge 2$ we have}
$\alpha_k \ge  \beta_k/(1-\beta_k)$. 
\end{maintheorem}

 In particular,
\[
\alpha_3 \ge  \frac{\beta_3}{1-\beta_3} = 0.0268650\dots,
\]
which is substantially larger than the bound from \eqref{MT-alpha-lower}, which is $\alpha_3 \ge 0.0127069\dots$.

Combining Theorem \ref{thm:alphak} with the bounds on
$\beta_k$ given in Theorem \ref{beta-rho}, we have
improved \rev{the lower bounds \eqref{MT-alpha-lower} 
for large $k$.} 
%(in fact, our bounds are better than \eqref{MT-alpha-lower} for all $k\ge 3$).

The upper bound on $\alpha_k$ is more delicate, and
a subject which we will return to in  a future paper.
For now, we record our belief that the lower bound in 
Theorem \ref{thm:alphak} is sharp.

\begin{conjecture}
For all $k\ge 2$ we have $\alpha_k = \beta_k/(1-\beta_k)$.
\end{conjecture}

\vspace{11pt}

\noindent {\bf Acknowledgements.} This collaboration began at the MSRI program on Analytic Number Theory, which took place in the first half of 2017 and which was supported by the National Science Foundation under Grant No.~DMS-1440140. All three authors are grateful to MSRI for allowing us the opportunity to work together.

The project was completed during a visit of KF and DK to Oxford in the first half of 2019. Both authors are grateful to the University of Oxford for its hospitality.

KF is supported by the National Science Foundation Grants
DMS-1501982 and DMS-1802139. In addition, his stay at Oxford in early 2019 was supported by a Visiting Fellowship at Magdalen College Oxford. BG is supported by a Simons Investigator Grant, which also funded DK's visit to Oxford. DK is also supported by the Courtois Chair II in fundamental research, by the Natural Sciences and Engineering Research Council of Canada (RGPIN-2018-05699) and by the Fonds de recherche du Qu\'ebec - Nature et technologies (2019-PR-256442 and 2022-PR-300951).

\newcommand\TV{\operatorname{TV}}
\section{Application to random integers, random permutations
and random polynomials}\label{transference}

In this section we \rev{assume the validity of Theorem \ref{beta-rho} and use it to} prove Theorems \ref{delta-from-model},
\ref{delta-perm}, \ref{delta-polynomial} and \ref{thm:alphak}. The two main ingredients in \rev{this deduction} are a simple combinatorial device (Lemma \ref{main-prop}), of a type often known as a ``tensor power trick'', used for building a large collection of equal subset sums,
and transference results (Lemmas \ref{prime model}, \ref{cycles} and \ref{polynom-factors}) giving a correspondence between the random set $\A$ and prime factors of a random integer, the cycle structure of a random permutation and the factorization of a random polynomial over a finite field. In the integer setting, this is a well-known principle 
following, e.g. from the Kubilius model of the integers
(Kubilius, Elliott \cite{ElliotTV1, ElliotTV2}, Tenenbaum
\cite{Ten99}).  We give a self-contained (modulo using the sieve) proof below.

Throughout this section, $\A$ denotes a logarithmic random set.

\subsection{A ``tensor power'' argument}

In this section we give a simple combinatorial argument,
first used in a related context in the work of Maier-Tenenbaum \cite{MT84}, which
shows how to use equal subsums in multiple intervals $((D')^c,D']$ to 
create many more common subsums in $\cA$.

\begin{lemma}\label{main-prop}
\rev{Let $k \in\Z_{\ge2}$ and $\eps>0$ be fixed.} Let $D_1,D_2$ be parameters depending on $D$ with \rev{$3 \le D_1 < D_2 \le D$}, $\log \log D_1 = o(\log \log D)$ and $\log \log D_2 = (1 - o(1)) \log \log D$ as $D\to\infty$. \rev{Then, with probability $\to 1$ as $D\to\infty$, there are distinct $A_1,\dots, A_M \subset \A \cap [D_1, D_2]$ with $\sum_{a \in A_1} a = \cdots = \sum_{a \in A_M} a$ and $M \geq (\log D)^{(\log k)/\log(1/\beta_k) - \eps}$.}
\end{lemma}

\begin{remark*}  In particular, the result applies when $D_1 = 3$ and $D_2 = D$, in which case it has independent combinatorial interest, giving a (probably tight) lower bound on the growth of the representation function for a random set.
\end{remark*}

\begin{proof}
Since increasing the value of $D_1$ only makes the proposition stronger, we may assume that $D_1 \rightarrow \infty$ \rev{as $D\to\infty$.
Let $0<\delta<\beta_k$, and set $\alpha := \beta_k - \delta$. Set 
\[
m: = \Big\lfloor \frac{\log \log D_2 - \log \log D_1}{-\log(\beta_k - \delta)}\Big\rfloor
\] }
and consider the intervals $[D_2^{\alpha^{i+1}}, D_2^{\alpha^i})$, $i = 0,1,\dots, m - 1$. Due to the choice of $m$, these all lie in $[D_1, D_2]$. 

Let $E_i$, $i = 0,1,2,\dots$ be the event that there are distinct $A^{(i)}_1,\dots, A^{(i)}_k \subset [D_2^{\alpha^{i+1}}, D_2^{\alpha^i})$ with $\sum_{a \in A^{(i)}_1} a = \cdots = \sum_{a \in A^{(i)}_k} a$. Then, by the definition of $\beta_k$ and the fact that $D_1 \rightarrow \infty$, we \rev{have $\P(E_i) = 1 - o(1)$, uniformly in $i=0,1,\ldots,m-1$.  Here and throughout the proof, $o(1)$ means a function tending to zero as $D\to\infty$, at a rate
which may depend on $k,\delta$.}
These events $E_i$ are all independent. The Law of Large Numbers then implies that, with \rev{probability $1 - o(1)$, at least $(1 - o(1))m$ of them occur, let us say for $i \in I$, $|I| = (1 - o(1))m$. }

From the above discussion, we have \rev{found $M := k^{|I|} = k^{(1 - o(1))m}$ distinct} sets $B_{\bs j} = \bigcup_{i \in I} A_{j_i}^{(i)}$, ${\bs j} \in [k]^{I}$, such that all of the sums $\sum_{a \in B_{\bs j}} a$ are the same. \rev{Note that 
\[ M = k^{(1 + O_k(\delta) + o(1)) \log \log D/\log(1/\beta_k)}.\]
Taking $\delta$ small enough and $D$ large enough, the result follows. }
\end{proof}

%%%%%%%%%%%%%%%%%%%%%%%%%%%%%%%%%%%%%%%%%%%%%%%%%%%%
%
\subsection{Modeling prime factors with a logarithmic random set}
%
%%%%%%%%%%%%%%%%%%%%%%%%%%%%%%%%%%%%%%%%%%%%%%%%%%%%%

Let $X$ be a large parameter, suppose that
\be\label{K-range}
1 \le K \le (\log X)^{1/2},
\ee
and let $I=[i_1,i_2] \cap \N$, where
\be\label{i1i2}
i_1 = \fl{K (\log\log X)^3}, \quad i_2=\fl{\frac{K\log X}{2\log\log\log X}}.
\ee
For a uniformly random positive integer $\mathbf{n}\le X$, let $\mathbf{n}=\prod_p p^{v_p}$ be the 
the prime factorization of $\mathbf{n}$,  where the product is over all primes. 
 Let $\sP_i$ be
the set of primes in $(e^{i/K}, e^{(i+1)/K}]$, 
and define the random set
\be\label{setBdef}
\rev{\mathbf{I} = \{ i\in I : \exists p\in\sP_i\ \text{such that}\ p|\bn\} .}
\ee
that is, the set of $i$ for which $\mathbf{n}$ has a prime factor in $\sP_i$.
By the sieve, it is known that the random variables $v_p$
are nearly independent for $p=X^{o(1)}$, and thus the
probability that $b_i\ge 1$ is roughly
\[
R_i := \sum_{p\in \sP_i} \frac{1}{p} \approx \frac{1}{i}.
\]
The next lemma makes this precise.

\rev{Recall the notion of \emph{total variation
distance} $d_{\TV}(X,Y)$ between two discrete real random vectors $X,Y$ defined on the
same probability space $(\Omega,\cF,\P)$:
\[
d_{\TV}(X,Y) = \max_{A\in\cF} | \P(X\in A) - \P(Y\in A)|.
\]
We have
\be\label{TV-joint}
d_{\TV}((X_1,\ldots,X_k),(Y_1,\ldots,Y_k)) \le \sum_{j=1}^k d_{\TV}(X_j,Y_j),
\ee
provided that the random variables $X_j,Y_j$ live on the same probability space for each $j$, that
$X_1,\ldots,X_k$ are independent, and $Y_1,\ldots,Y_k$ are also independent.
Although we believe this is a standard inequality, we could not find a 
good reference for it and give a proof of \eqref{TV-joint} in Lemma \ref{lem:dTV}.
In addition, recall the identity
\be\label{TV-ident}
d_{\TV}(X,Y) = \frac12 \sum_{t\in\Omega} | \P(X=t) - \P(Y=t) |
\ee
when $X$ and $Y$ take values in a probability space $(\Omega,\cF,\P)$ with $\Omega$ countable and $\cF$ being the power set of $\Omega$.
See, e.g. \cite[Proposition 4.2]{LPW}.
}

\begin{lemma}\label{prime model}
Uniformly for any collection $\sI$ of subsets of $I$, we have
\[
\P (\A \cap I \in  \sI ) = \rev{\P (\mathbf{I} \in \sI)}
+ O(1/\log\log X).
\]
\end{lemma}

\begin{proof}
\rev{For $i_1 \le i\le i_2$, let $\omega_i$ be the indicator function of the event that $\bn$ has a prime factor from $\sP_i$, let $Q_i$ be a Poisson
random variable with parameter $R_i$, with the different $Q_i$ independent,
and let $Z_i=1_{Q_i\ge 1}$\footnote{We use $1_E$ for the indicator function of a statement $E$; that is, $1_E=1$ if $E$ is true and $1_E=0$ if $E$ is false.}. Also, let $Y_i$ be a Bernoulli random variable with $\P(Y_i=1)=1/i$, again with the $Y_i$ independent.
Let $\boldsymbol{\omega}, \mathbf{Z}, \mathbf{Y}$ denote the vectors of
the variables $\omega_i,Z_i,Y_i$, respectively.
 By assumption, each $\sP_i \subset [\log X,
X^{1/3\log\log\log X}]$.  Hence, Theorem 1 of \cite{prime-poisson} implies that
\[
d_{\TV}(\boldsymbol{\omega},\mathbf{Z}) \ll \frac{1}{\log\log X}.
\]
In addition, note that $d_{\TV}(Z_i,Y_i)\ll 1/i^2$ for all $i$, something that can be easily proven using \eqref{TV-ident}. Combining this estimate with  \eqref{TV-joint}, we find that
\[
d_{\TV}(\mathbf{Z},\mathbf{Y}) \le \sum_{i=i_1}^{i_2} d_{\TV}(Z_i,Y_i) \ll \sum_{i=i_1}^{i_2} \frac{1}{i^2} \ll \frac{1}{\log\log X}.
\]
The triangle inequality then implies that $d_{\TV}(\boldsymbol{\omega},\mathbf{Y})\ll
1/\log\log X$, as desired.}
\end{proof}

\subsection{The concentration of divisors of integers}\label{sec23}

In this section we prove Theorems \ref{delta-from-model}
and \ref{thm:alphak}. Recall from \eqref{zeta-def} the definition of \rev{$\zeta_+$}.

\begin{proof}[Proof of Theorem \ref{delta-from-model}]
\rev{Fix $\eps>0$ and let $X$ be large enough in terms of $\eps$,} and let $\mathbf{n} \leq X$ be a uniformly sampled random integer. Generate a logarithmic random set $\A$. Set 
$K=10 \log\log X$, $D_1 = i_1$, $D=D_2 = i_2$, where $i_1$ and $i_2$ are defined by \eqref{i1i2}. With our choice of parameters, the hypotheses of Lemma \ref{main-prop} hold and hence, with probability $1 - o(1)$ \rev{as $X\to\infty$}, there are distinct sets $A_1,\dots, A_M \subset \A \cap [D_1, D_2]$ with $\sum_{a \in A_1} a = \cdots = \sum_{a \in A_M} a$ \rev{and $M:=\lceil (\log \log X)^{\zeta_+ - \eps}\rceil$. 
By Lemma \ref{A-normal}, with probability $1 - o(1)$, we have
\[
|\A \cap [D_1,D_2]| \leq  2 \log D_2 \leq 2 \log \log X+2\log K.
\]}
Write $F$ for the event that both of these happen.

\rev{Let $\bn$ be a random integer chosen uniformly in $[1,X]$, and let $\mathbf{I}$ be the random set associated to $\bn$ via \eqref{setBdef}.} By Lemma \ref{prime model}, the corresponding event $F'$
for \rev{$\mathbf{I}$} also holds with probability
$1-o(1)$; that is, $F'$ is the event that $|\mathbf{I} \cap [D_1,D_2]| \le 2\log D_2$ and that there are distinct subsets $I_1,\dots,I_M$
with equal sums. \rev{Assume we are in the event $F'$. For each $i\in\mathbf{I}$, $\mathbf{n}$ is divisible by some prime $p_i\in\sP_i$. In addition,  for each $r,s\in \{1,2,\ldots,M\}$, we have
\begin{align*}
\Big|\sum_{i \in I_r} \log p_i - \sum_{i \in I_s} \log p_i\Big| & \leq  \frac{|I_r| + |I_s|}{K} + \frac{1}{K} \Big|\sum_{i \in I_r} i - \sum_{i \in I_s} i\Big| \\ & \le  \frac{4 \log \log X+4\log K}{K} < \frac12.
\end{align*}
}
Writing \rev{$d_r := \prod_{i \in I_r} p_i$ for each $i$, we thus see that the $d_r$'s} are all divisors of $\mathbf{n}$ and their logarithms all lie in an interval of length $1$. \rev{It follows that $\P(\Delta(\mathbf{n}) \geq M) = 1 - o(1)$ when $\bn$ is a uniformly sampled random integer from $[1,X]$,} as required for Theorem \ref{delta-from-model}.
\end{proof}

\begin{proof}[Proof of Theorem \ref{thm:alphak}]
\rev{Fix $0<c < \beta_k/(1-\beta_k)$}, let $X$ be large and set
$K= (\log X)^{c}$.  Define $i_1,i_2$ by 
\eqref{i1i2}, let $D=i_2$ and define $c'$ by $D^{c'} = i_1$.
\rev{Let $\bn$ be a random} integer chosen uniformly in $[1,X]$.
\rev{We have
\[
c' = \frac{c}{c+1} + o(1) \qquad (X\to\infty),
\]
}and therefore $c' \le \beta_k-\delta$ for some $\delta>0$, which depends only on $c$.
By the definition of $\beta_k$ and Lemma \ref{prime model},
it follows that with probability $1-o(1)$, the set \rev{$\mathbf{I}$}
defined in \eqref{setBdef} has $k$ distinct subsets \rev{$I_1,\dots,I_k$}
with equal sums, \rev{and moreover (cf.~the proof of Theorem \ref{delta-from-model} above), $|\mathbf{I}| \le 2\log i_2$, so that $|I_j|\le2\log i_2$ for each $j$}.  Thus, with probability $1-o(1)$, there are 
primes $p_i\in \sP_i$ ($i\in \mathbf{I}$) such that
for any \rev{$r,s\in\{1,\dots,k\}$} we have
\[
\rev{\Big| \sum_{i\in I_r} \log p_i - \sum_{i\in I_s} \log p_i \Big| 
\le \frac{|I_r|+|I_s|}{K} \le \frac{4\log\log X}{(\log X)^c}}.
\]
Thus, setting \rev{$d_r = \prod_{i\in I_r} p_i$}, we see that
\rev{$d_r \le d_s \exp \big\{O\big(\frac{\log\log X}{(\log X)^c}\big) \big\}$ for any $r,s\in\{1,\dots,k\}$.}
Since $c$ is arbitrary subject to $c<\beta_k/(1-\beta_k)$, we conclude that
$\alpha_k \ge \beta_k/(1-\beta_k)$.
\end{proof}

\subsection{Permutations and polynomials over finite fields}

The connection between random logarithmic sets, random permutations and random 
polynomials is more straightforward, owing to the well-known approximations of
these objects by a vector of Poisson random variables.

For each $j$, let $Z_j$ be a Poisson random variable with parameter $1/j$,
and such that $Z_1,Z_2,\ldots,$ are independent. 
The next proposition states that, apart from the very longest cycles, the cycle lengths of a random permutation have a joint Poisson distribution.

\begin{lemma}\label{cycles}
For a random permutation $\sigma \in S_n$, let $C_j(\sigma)$ denote the number of cycles in $\sigma$ of length $j$.  Then for $r = o(n)$ as $n\to \infty$ we have
\[
d_{\TV} \Big( (C_1(\sigma),\ldots,C_r(\sigma)),(Z_1,\ldots,Z_r) \Big) = o(1).
\] 
\end{lemma}

\begin{proof}
In fact there is a bound $\ll e^{-n/r}$ uniformly in $n$ and $r$; see \cite{AT92}. 
\end{proof}

The next proposition states a similar phenomenon for the degrees of the irreducible factors of a random polynomial over $\F_q$, except that now one must also exclude the very smallest degrees as well.

\begin{lemma}\label{polynom-factors}
Let $q$ be a prime power. Let $f$ be a random, monic polynomial in $\mathbb{F}_q[t]$ of degree $n$. Let $Y_d(f)$
denote the number of monic, irreducible factors of $f$ which have degree $d$. Suppose that $10\log n \le r \le s\le \frac{n}{10\log n}$. Then
\[
d_{\TV} \Big( (Y_r(f),\ldots,Y_s(f)),(Z_r,\ldots,Z_s) \Big) = o(1)
\]
as $n \rightarrow \infty$.
\end{lemma}

\begin{proof}
For $r \leq i \leq s$, let $\hat{Z}_i$ be a negative binomial random variable\footnote{\rev{We say that the random variable $X$ has the distribution $\mathrm{NB}(r,p)$ with $r\in\N$ and $p\in(0,1]$ if $X$ takes values in $\Z_{\ge0}$ with the following frequency: $\P(X=k)=\binom{k+r-1}{r-1}(1-p)^kp^r$ for each $k\in\Z_{\ge0}$.}} $\mathrm{NB}(\frac{1}{i}\sum_{j|i} \mu(i/j) q^{j},q^{-i})$. Corollary 3.3 in 
\cite{ABT93} implies that
\begin{equation}\label{yz}
d_{\TV} \Big( (Y_r(f),\ldots,Y_s(f)),(\hat{Z}_r,\ldots,\hat{Z}_s) \Big) \ll 1/n
\end{equation}
\rev{uniformly in $q,n,r,s$ as in the statement of the lemma}. Note that 
\[
\rev{\frac{1}{i} \sum_{j|i} \mu(i/j) q^{j} = \frac{1}{i}q^i(1+ O(q^{-i/2}))=\frac{1}{i}q^i(1+O(1/n))}
\]
for $i\ge r\ge 10\log n$.  A routine if slightly lengthy calculation with \eqref{TV-ident} gives
\[
d_{\TV}(Z_i,\hat{Z}_i) \ll 1/n.
\]
Combining this with \eqref{TV-joint}, we arrive at
\[
d_{\TV}( (Z_r,\cdots,Z_s),(\hat{Z}_r,\ldots,\hat{Z}_s) ) \ll s/n = o(1).
\]
The conclusion follows from this, \eqref{yz} and the triangle inequality.
\end{proof}

\begin{proof}[Proof of Theorem \ref{delta-perm}]
\rev{Fix $\eps>0$, let $n$ be large enough in terms of $\eps$,} let $u=\log n$ and $v=n/\log n$.
For a random permutation $\sigma\in S_n$, let $\mathbf{C} = \{ j: C_j(\sigma)\ge 1 \}$,
\rev{and define the random set $\tilde{\A} = \{ j: Z_j \ge 1 \}$. 
As in the proof of Lemma \ref{prime model}, \eqref{TV-joint} and \eqref{TV-ident}
imply that
\[
d_{\TV} (\A \cap (u,v], \tilde{\A} \cap (u,v]) \ll \sum_{u<i\le v} \frac{1}{i^2}
\ll \frac{1}{u}.
\]
Lemma \ref{cycles} implies that
\[
 d_{\TV} (\tilde{\A} \cap (u,v], \mathbf{C} \cap (u,v]) =o(1) \qquad (n\to\infty).
\]
Hence,
\begin{align*}
d_{\TV} (\A \cap (u,v], \mathbf{C} \cap (u,v]) &\le 
d_{\TV} (\A \cap (u,v], \tilde{\A} \cap (u,v])+ d_{\TV} (\tilde{\A} \cap (u,v], \mathbf{C} \cap (u,v]) \\
& = o(1)
\end{align*}
}
as $n \rightarrow \infty$.
By Lemma \ref{main-prop}, with probability $\to 1$ as $n\to \infty$,
 $\A \cap (u,v]$ has $M$ distinct
subsets $A_1,\ldots,A_M$ with equal sums, \rev{where $M=\lceil (\log n)^{\zeta_+-\eps}\rceil$.}  Hence, $\mathbf{C}$ has distinct subsets $S_1,\ldots,S_M$ with equal sums with probability
$\to 1$ as $n\to\infty$.  Each subset $S_j$ corresponds to a distinct divisor of $\sigma$, the size of the divisor being
 the sum of elements of $S_j$.
\end{proof}

\begin{proof}[Proof of Theorem \ref{delta-polynomial}]
The proof is essentially the same as that of Theorem \ref{delta-perm}, except now we take
$u=10\log n$, $v=\frac{n}{10\log n}$, $\mathbf{C} = \{j: Y_j(f)\ge 1 \}$
 and use Lemma \ref{polynom-factors} in
place of Lemma \ref{cycles}.
\end{proof}

\section{Overview of the paper}\label{overview-sec}

The purpose of this section is to explain the main ideas that go into the proof of Theorem \ref{beta-rho} in broad strokes, as well as to outline the structure of the rest of the paper. The remainder of the paper splits into three parts, and we devote a subsection to each of these. Finally, in subsection \ref{previous-sec}, we make some brief comments about the relationship of our work to previous work of Maier and Tenenbaum \cite{MT84, MT09}. Further comments on this connection are made in Appendix \ref{previous-app}.

\subsection{Part \ref{part:reduction-to-optimizaton}: Equal sums and the optimization problem.}

Part \ref{part:reduction-to-optimizaton} provides a very close link between the key quantity $\beta_k$ (which is defined in Problem \ref{beta-k-def} and appears in all four of  Theorems \ref{beta-rho}, \ref{delta-from-model}, \ref{delta-perm} and \ref{delta-polynomial}) and a quantity $\gamma_k$, which on the face of it appears to be of a completely different nature, being the solution to a certain optimization problem (Problem \ref{opt-problem} below) involving the manner in which linear subspaces of $\Q^k$ intersect the cube $\{0,1\}^k$. 

At the heart of this connection is a fairly simple way of \rev{associating a \emph{flag} to $k$ distinct sets} $A_1,\dots, A_k \subset A$, where $A$ is a given set of integers (that we typically generate logarithmically).

\begin{definition}[Flags]\label{flag-pre-def} Let $k \in \N$.
	By an $r$-step \emph{flag} we mean  a nested sequence 
	\[ 
	\sV: \spanone = V_0 \leq  V_1 \leq  V_2 \leq  \cdots \leq  V_r \leq  \Q^k 
	\] 
	of vector spaces.\footnote{In the literature, the term ``flag'' means that the inclusions are proper, i.e., $\dim (V_{i+1}) > \dim V_i$ for all $i$.  In this paper, we will use the term more broadly to refer to an arbitrary nested sequence of subspaces.} Here $\mathbf{1} = (1,1,\dots, 1) \in \Q^k$.  A flag is \emph{complete} if $\dim V_{i+1} = \dim V_i + 1$ for $i = 0,1,\dots, r-1$.
\end{definition}

To each choice of distinct sets $A_1,\dots, A_k \subset A$, we associate a flag as follows. The Venn diagram of the subsets $A_1,\ldots,A_k$ produces a natural partition of $A$ into $2^k$ subsets,
which we denote by $B_\om$ for $\om\in \{0,1\}^k$.
Here $A_i = \sqcup_{\om:\om_i=1} B_\om$.
We iteratively select vectors $\om^1,\ldots,\om^r$
to maximize $\prod_{j=1}^r (\max B_{\om^j})$ subject
to the constraint that $\one,\om^1,\ldots,\om^r$ are linearly
independent over $\Q$. We then define\footnote{ \rev{Here and throughout the paper, $\Span(v_1,\ldots)$ denotes the $\Q$-span of vectors $v_1,\ldots$.}} $V_j  = \Span(\one,\om^1,\ldots,\om^j)$ for $j = 0,1,\dots, r$. 

The purpose of making this construction is difficult to describe precisely in a short paragraph. However, the basic idea is that the vectors $\omega^1,\dots, \omega^r$ and the flag $\sV$ provide a natural frame of reference for studying the equal sums equation

\begin{equation}\label{eq:k equal sums}
\sum_{a\in A_1}a=\cdots=\sum_{a\in A_k} a.
\end{equation}

Suppose now that $A_1,\dots, A_k \subset [D^c,D]$. Then the construction just described naturally leads, in addition to the flag $\sV$, to the following further data: thresholds $c_j$ defined by $\max B_{\om^j} \approx D^{c_j}$, and measures $\mu_j$ on $\{0,1\}^k$, which 
capture the relative sizes of the sets
$B_{\om} \cap (D^{c_{j+1}},D^{c_j}]$, $\omega \in \{0,1\}^k$. Full details of these constructions are \rev{given} in Section \ref{sec:upper}.  

The above discussion motivates the following definition, which will be an important one in our paper.

\begin{definition}[Systems]\label{adm-meas}
Let $(\sV,\cc,\bmu)$ be a triple such that:
\begin{enumerate}
	\item $\sV$ is an $r$-step flag whose members $V_j$ are distinct and spanned by elements of $\{0,1\}^k$;
	\item $\sV$ is \emph{nondegenerate}, which means that $V_r$ is not contained in any of the subspaces $\{ x \in \Q^k : x_i = x_j\}$, $i \neq j$; 
	\item $\cc=(c_1,\dots,c_r,c_{r+1})$ with $1\ge c_1 \ge \cdots \ge c_{r+1} \ge 0$;
	\item $\bmu=(\mu_1,\dots,\mu_r)$ is an $r$-tuple of probability measures;
	\item $\Supp(\mu_i)\subset V_i \cap \{0,1\}^k$ for all $i$.
\end{enumerate}
Then we say that $(\sV,\cc,\bmu)$ is a {\it system}. We say that a system is complete if its underlying flag is, in the sense of Definition \ref{flag-pre-def}.
\end{definition}

\begin{remark*} The nondegeneracy condition (b) arises naturally from the construction described previously, provided one assumes the sets $A_1,\dots, A_k$ are distinct.
\end{remark*}

We have sketched how a system $(\sV,\cc,\bmu)$ may be associated to any $k$ distinct sets $A_1,\dots, A_k \subset [D^c, D]$. Full details are given in subsection \ref{venn-linear}. There is certainly no canonical way to reverse this and associate sets $A_i$ to a system $(\sV,\cc,\bmu)$, \rev{even if the numbers $\mu_j(\omega)$ are all rational.}  However, given a set $\A  \subset [D^c,D]$ (which, in our paper, will be a logarithmic random set) and a system $(\sV,\cc,\bmu)$, there is a natural \emph{probabilistic} way to construct subsets $A_1,\dots, A_k \subset \A$ via their Venn diagram $(B_{\omega})_{\omega \in \{0,1\}^k}$: if $a \in \A \cap (D^{c_{j+1}}, D^{c_j}]$ then we put $a$ in $B_{\omega}$ with probability $\mu_j(\omega)$, these choices being independent for different $a$s.

This will be indeed be roughly our strategy for constructing, given a logarithmic random set $\A \subset [D^c, D]$, distinct subsets $A_1,\dots, A_k \subset \A \cap [D^c, D]$ satisfying the equal sums condition \eqref{eq:k equal sums}. Very broadly speaking, we will enact this plan in two stages, described in  Sections \ref{lower-bound-sec} and \ref{mt-argument} respectively. In Section \ref{lower-bound-sec}, which is by far the deeper part of the argument, we will show that (almost surely in $\A$) the distribution of tuples $(\sum_{a \in A_i} a)_{i = 1}^k$ is dense in a certain box adapted to the flag $\sV$, as the $A_i$ range over the random choices just described. Then, in Section \ref{mt-argument}, we will show that (almost surely) one of these tuples can be ``corrected'' to give the equal sums condition \eqref{eq:k equal sums}. This general mode of argument has its genesis in the paper \cite{MT84}
of Maier and Tenenbaum, but the details here will look very different. In addition to the fact that linear algebra and entropy play no role in Maier and Tenenbaum's work, they use a second moment argument which does not work in our setting. Instead we use an $\ell^p$ estimate with $p\approx 1$, building on ideas in \cite{DK10,DK14}.

In analysing the distribution of tuples $(\sum_{a \in A_i} a)_{i = 1}^k$, \rev{the notion of entropy} comes to the fore.

\begin{definition}[Entropy of a subspace] \label{entropy-def}Suppose that $\nu$ is a finitely supported probability measure on $\Q^k$ and that $W \leq  \Q^k$ is a vector subspace. Then we define
	\[ 
	\HH_{\nu}(W) := -\sum_x \nu(x) \log \nu(W + x).
	\]
\end{definition}

\begin{remark*} 
This the (Shannon) entropy of the distribution on cosets $W + x$ induced by $\nu$. Entropy will play a key role in our paper, and basic definitions and properties of it are collected in Appendix \ref{entropy-appendix}.
\end{remark*}

More important than the entropy itself will be a certain \rev{quantity $\er(\sV',\cc,\bmu)$,} assigned to \emph{subflags} of $\sV$. We give the relevant definitions now.

\begin{definition}[Subflags]\label{subflag-def}
Suppose that 
	\[ 
	\sV: \spanone = V_0 \leq  V_1 \leq  V_2 \leq  \cdots \leq  V_r \leq  \Q^k 
	\] is a flag. Then another flag 
	\[
	\sV': \spanone = V'_0 \leq  V'_1 \leq  V'_2 \leq  \cdots \leq  V'_r \leq  \Q^k
	\]
	is said to be a \emph{subflag} of $\sV$ if $V'_i \leq  V_i$ for all $i$. In this case we write $\sV' \leq \sV$. It is a \emph{proper subflag} if it is not equal to $\sV$.  
\end{definition}

\begin{definition}[$\er$-value]\label{e-value dfn}
Let $(\sV,\cc,\bmu)$ be a system, and let $\sV' \leq \sV$ be a subflag. Then we define the \emph{$\er$-value}
\be\label{entropy-dfn}
 \er(\sV', \cc,\bmu) := \sum_{j = 1}^r (c_j - c_{j+1}) \HH_{\mu_j}(V'_j) + \sum_{j = 1}^r c_j\dim (V'_j/V'_{j-1}) .
 \ee
\end{definition}

\begin{remark*} Note that 
\begin{equation}\label{full-e}   \er(\sV, \cc,\bmu) = \sum_{j = 1}^r c_j \dim (V_j/V_{j-1}),
\end{equation} 
since condition (e) of Definition \ref{adm-meas} implies that $\HH_{\mu_j}(V_j)=0$ for $1\le j\le r$.
\end{remark*}

\begin{definition}[Entropy condition]\label{ent-condition}
Let $(\sV,\cc,\bmu)$ be a system. We say that this system satisfies the \emph{entropy condition} if
\begin{equation}\label{entropy-cond}
	\er(\sV',\cc,\bmu)\ge \er(\sV,\cc,\bmu)
	\qquad\text{for all subflags}\ \sV'\ \text{of}\ \sV,
	\end{equation}
and the \emph{strict entropy condition} if \begin{equation}\label{strict-entropy-cond}
	\er(\sV',\cc,\bmu) >\er(\sV,\cc,\bmu)
	\qquad\text{for all proper subflags}\ \sV'\ \text{of}\ \sV.
	\end{equation}
\end{definition}

We cannot give a meaningful discussion of exactly why these definitions are the right ones to make in this overview. Indeed, it took the authors over a year of working on the problem to arrive at them. Let us merely say that

\begin{itemize}
\item If a random logarithmic set $\A \cap [D^c, D]$ almost surely admits distinct subsets $A_1,\dots, A_k$ satisfying the equal sums condition \eqref{eq:k equal sums}, then some associated system $(\sV,\cc,\bmu)$ satisfies the entropy condition \eqref{entropy-cond}. For detailed statements and proofs, see Section \ref{sec:upper}.
\item If a system $(\sV,\cc,\bmu)$ satisfies the strict entropy condition \eqref{strict-entropy-cond} then the details of the construction of sets $A_1,\dots, A_k$ satisfying the equal sums condition, outlined above, can be made to work. For detailed statements and proofs, see Sections \ref{lower-bound-sec} and \ref{mt-argument}.
\end{itemize}

With the above definitions and discussion in place, we are finally ready to introduce the key optimization problem, the study of which will occupy a large part of our paper. 

\begin{problem}[The optimisation problem]\label{opt-problem} Determine the value of $\gamma_k$, defined to be the supremum of all constants $c$ for which there is a system $(\sV,\cc,\bmu)$ such that $c_{r+1}=c$ and the entropy condition \eqref{entropy-cond} holds.
	
Similarly, determine $\tilde{\gamma}_k$, defined to be the supremum of all constants $c$ for which there is a system $(\sV,\cc,\bmu)$ such that $c_{r+1}=c$ and the strict entropy condition \eqref{strict-entropy-cond} holds.
\end{problem}

The precise content of the two bullet points above, and the main result of Part \ref{part:reduction-to-optimizaton} of the paper, is then the following theorem.

\begin{maintheorem}\label{thm:beta-gamma}
For every $k\ge 2$, we have 
\[
\tilde{\gamma}_k\le \beta_k\le \gamma_k.
\]
\end{maintheorem}
%\dcom{Pedantic question: should this be Theorem 3.8 or Theorem 6?}

\begin{remark}\label{rem: gamma}
(a) Presumably $\gamma_k = \beta_k = \tilde \gamma_k$.  Indeed, it is natural to think that any system satisfying \eqref{entropy-cond} can be perturbed an arbitrarily small amount to satisfy \eqref{strict-entropy-cond}. However, we have not been able to show that this is possible in general.

(b) \rev{It is not \emph{a priori} clear that $\gamma_k$ and $\tilde{\gamma}_k$ exist and are positive.  This will follow, e.g., from our work on ``binary systems'' in part IV of the paper, although there is an easier way to see this using the original Maier-Tenenbaum
argument, adapted to our setting; see Appendix \ref{previous-app} for a sketch
of the details.}
\end{remark}

\subsection{Part \ref{part:optimization}: The optimization problem} 

Part \ref{part:optimization} of the paper is devoted to the study of Problem \ref{opt-problem} in as much generality as we can manage. Unfortunately we have not yet been able to completely resolve this problem, and indeed numerical experiments suggest that a complete solution, for all $k$, could be very complicated. 

The main achievement of Part \ref{part:optimization} is to provide a solution of sorts when the flag $\sV$ is fixed, but one is free to choose \rev{$\mathbf{c}$ and $\bmu$}. Write $\gamma_k(\sV)$ (or $\tilde\gamma_k(\sV)$) for the solution to this problem, \rev{that is, the supremum of values $c=c_{r+1}\ge0$ for which
a system $(\sV,\cc,\bmu)$ exists satisfying \eqref{entropy-cond} (or \eqref{strict-entropy-cond}).}

Our solution applies only to rather special flags $\sV$, but this is unsurprising: for ``generic'' flags $\sV$, one would not expect there to be any choice of $\mathbf{c}$, $\bmu$,  for which $c_{r+1} > 0$, and so \rev{$\gamma_k(\sV)= 0$} in these cases. Such flags are of no interest in this paper.

We begin, in Section \ref{optimisation-sec}, by solving an even more specific problem in which the entropy condition  \eqref{entropy-cond} is only required to hold for certain very special subflags $\sV'$ of $\sV$, which we call \emph{basic flags}. These are flags of the form
\[ \sV'_{\basic(m)} :
 \langle \mathbf{1} \rangle = V_0 \leq V_1 \leq \cdots \leq V_{m-1} \leq V_m = V_m = \cdots = V_m.
\] We call this the \emph{restricted entropy condition}; to spell it out, this is the condition that 
 
\begin{equation}\label{basic-flag-equality} 
\er(\sV'_{\basic(m)}, \cc,\bmu) \geq \er(\sV, \cc,\bmu) 
\end{equation}
for $m = 0,1,\dots, r-1$ (the case $m = r$ being vacuous). 
 
 We write $\gamma_k^{\res}(\sV)$ for the maximum value of $c_{r+1}$ (over all choices of $\mathbf{c}$ and $\bmu$ such that $(\sV, \cc, \bmu)$ is a system) subject to this condition. Clearly 
\begin{equation}\label{res-unres} \gamma_k^{\res}(\sV) \geq \gamma_k(\sV).\end{equation} 

The main result of Section \ref{optimisation-sec} is Proposition \ref{main-optim}, which states that under certain conditions we have
\begin{equation}\label{opt-answer}\gamma_k^{\res}(\sV) = \frac{\log 3 - 1}{\log 3 + \sum_{i = 1}^{r-1} \frac{\dim (V_{i+1}/V_i)}{\rho_1 \cdots \rho_{r-1}}},\end{equation}
for \rev{certain parameters $\rho_1,\ldots, \rho_{r-1}$ depending} on the flag $\sV$. 

To define these, one considers the ``tree structure'' on $\{0,1\}^k \cap V_r$ induced by the flag $\sV$: the ``cells at level $j$'' are simply intersections with cosets of $V_j$, and we join a cell $C$ at level $j$ to a ``child'' cell $C'$ at level $j-1$ iff $C' \subset C$. The $\rho_i$ are then defined by setting up a certain recursively-defined function on this tree and then solving what we term the \emph{$\rho$-equations}. The details may be found in subsection \ref{rho-equations-sec}.  Proposition \ref{main-optim} also describes the measures $\bmu$ and the parameters $\cc$ for which this optimal value is attained. 

\medskip

In Section \ref{entropy-binary-alt}, we relate the restricted optimisation problem to the real one, giving fairly general conditions under which we in fact have equality in \eqref{res-unres}, that is to say $\gamma_k^{\res}(\sV) = \gamma_k(\sV)$. The basic strategy of this section is to show that for the $\mathbf{c}$ and $\bmu$ which are optimal for the restricted optimisation problem, the full entropy condition \eqref{entropy-cond} is in fact a consequence of the restricted condition \eqref{basic-flag-equality}. 

The arguments of this section make heavy use of the submodularity inequality for entropy, using this to drive a kind of ``symmetrisation'' argument. In this way one can show that an arbitrary $\er(\sV', \cc, \bmu)$ is greater than or equal to one in which $\sV'$ is \emph{almost} a basic flag; these ``semi-basic'' flags are then dealt with by hand.

To add an additional layer of complexity, we build a perturbative device into this argument so that our results also apply to $\tilde \gamma_k(\sV)$.

\subsection{Part \ref{part:binary-systems}: Binary systems}

The final part of the paper is devoted to a discussion of a particular type of flag $\sV$, the \emph{binary flags}, and the associated optimal systems $(\sV, \cc, \bmu)$, which we call \emph{binary systems}. 

\begin{definition}[Binary flag of order $r$]
Let $k = 2^r$ be a power of two. Identify $\Q^k$ with $\Q^{\mathcal{P}[r]}$ (where $\mathcal{P}[r]$ means the power set of $[r] = \{1,\dots, r\}$) and define an $r$-step flag $\sV$, $\langle \mathbf{1} \rangle = V_0 \leq V_1 \leq \cdots \leq V_r = \Q^{\mathcal{P}[r]}$, as follows: $V_i$ is the subspace of all $(x_S)_{S \subset [r]}$ for which $x_S = x_{S \cap [i]}$ for all $S \subset [r]$. 
\end{definition}

Whilst the definition is, in hindsight, rather simple and symmetric, it was motivated by extensive numerical experiment. We believe these flags to be asymptotically optimal for Problem \ref{opt-problem}, though we currently lack a proof.

There are two main tasks in Part \ref{part:binary-systems}. First, we must verify that the various conditions necessary for the results of Part \ref{part:optimization} hold for the binary flags. This is accomplished in Section \ref{binary-calcs}, the main statements being given in Section \ref{binary-system}. At the end of Section \ref{binary-system} we give the proof (and complete statement) of Theorem \ref{beta-rho}(a), conditional upon the results of Section \ref{binary-calcs}. This is the deepest result in the paper.

Following this we turn to Theorem \ref{beta-rho}(b). There are two tasks here. First, we prove that the parameters $\rho_i$ for the binary flags (which do not depend on $r$) tend to a limit $\rho$. This is not at all straightforward, and is accomplished in Section \ref{lim-rho-sec}. 

After that, in Section \ref{sec:analytic}, we describe this limit in terms of certain recurrence relations, which also provide a useful means of calculating it numerically. Theorem \ref{beta-rho}(b) is established at the very end of the paper.

Most of Part \ref{part:binary-systems} could, if desired, be read independently of the rest of the paper.

 \subsection{Relation to previous work} \label{previous-sec}
 
 Previous lower bounds for the a.s. behaviour of $\Delta$ are contained in two papers of Maier and Tenenbaum \cite{MT84, MT09}. Both of these bounds can be understood within the framework of our paper.

The main result of \cite{MT84} follows from the fact that 
\begin{equation}\label{gamm-2} \tilde\gamma_2 \geq 1 - \frac{1}{\log 3}.\end{equation}
Indeed by Theorem \ref{thm:beta-gamma} it then follows that $\beta_2 \geq 1 - \frac{1}{\log 3}$, and then from Theorem \ref{delta-from-model} it follows that for almost every $n$ we have
\begin{equation}\label{mt-1-bd} \Delta(n) \gg (\log \log n)^{-\log 2/\log(1 - \frac{1}{\log 3}) + o(1)}.\end{equation}
The exponent appearing here is $0.28754048957\dots$ and is exactly the one in \cite[Theorem 2]{MT84}.

The bound \eqref{gamm-2} is very easy to establish, and a useful exercise in clarifying the notation we have set up. Take $k = 2$, $r = 1$ and let $\sV$ be the flag $\langle\mathbf{1}\rangle = V_0 \leq V_1 = \Q^2$. Let $\cc = (c_1, c_2)$ with $c_1 = 1$ and 
\begin{equation}\label{c2-lower} c_2 < 1 - \frac{1}{\log 3}.\end{equation} Let $\mu_1$ be the measure which assigns weight $\frac{1}{3}$ to the points $\mathbf{0} = (0,0)$, $(0,1)$ and $(1,0)$ in $\{0,1\}^2$ (this being a pullback of the uniform measure on $\{0,1\}^2 / V_0$).

There are only two subflags $\sV'$ of $\sV$, namely $\sV$ itself and the basic flag $\sV'_{\basic(0)} : \langle\mathbf{1}\rangle = V'_0 \leq V'_1$ with $V'_0 = V'_1 = V_0 = \langle\mathbf{1}\rangle$. The entire content of the strict entropy condition \eqref{strict-entropy-cond} is therefore that
\[ \er(\sV'_{\basic(0)}, \cc,\bmu) > \er(\sV, \cc, \bmu),\] which translates to 
\[ (c_1 - c_2) \HH_{\mu_1}(V_0) > c_1.\]
We have $\HH_{\mu_1}(V_0) = \log 3$ and $c_1 = 1$, and so this translates to precisely condition \eqref{c2-lower}.

\begin{remark*}(a) With very little more effort (appealing to Lemma \ref{ent-trivial}) one can show that $\gamma_2 = \beta_2 = \tilde\gamma_2 = 1 - \frac{1}{\log 3}$.

(b) This certainly does not provide a shorter proof of Theorem \ref{mt-1-bd} than the one Maier and Tenenbaum gave, since our deductions are reliant on the material in Sections \ref{lower-bound-sec} and \ref{mt-argument}, which constitute a significant elaboration of the ideas from \cite{MT84}.
\end{remark*}

The main result of \cite{MT09} (Theorem 1.4 there) follows from the lower bound
\begin{equation}\label{gamma-2r}
\tilde\gamma_{2^r} \geq\Big(1 - \frac{1}{\log 3}\Big)
\Big(\frac{1 - 1/\log 3}{1 - 1/\log 27}\Big)^{r - 1},
\end{equation} 
which of course includes \eqref{gamm-2} as the special case $r = 1$. Applying Theorem \ref{thm:beta-gamma} and Theorem \ref{delta-from-model}, then letting $r \rightarrow \infty$, we recover \cite[Theorem 1.4]{MT09} (quoted as Theorem MT in Section \ref{intro}), namely the bound
\[ \Delta(n) \ge (\log \log n)^{\frac{\log 2}{\log\frac{1 - 1/\log 27}{1 - 1/\log 3}}  - o(1)  }\] for almost all $n$. The exponent here is $0.33827824168\dots$.

To explain how \eqref{gamma-2r} may be seen within our framework requires a little more setting up. Since it is not directly relevant to our main arguments, we defer this to Appendix \ref{previous-app}.

\clearpage
\thispagestyle{fancy}
\fancyhf{} % sets both header and footer to nothing
\renewcommand{\headrulewidth}{0cm}
\lhead[{\scriptsize \thepage}]{}
\rhead[]{{\scriptsize\thepage}}
\part{Equal sums and the optimisation problem}\label{part:reduction-to-optimizaton}

\section{The upper bound $\beta_k \leq  \gamma_k$}\label{sec:upper}

In this section we establish the bound in the title.  \rev{We recall the definitions of $\beta_k$ (Problem \ref{beta-k-def}) and $\gamma_k$ (Problem \ref{opt-problem}). We will in fact show a bit more, that if
$c>\gamma_k$ then}
\be\label{upper-bound:prob_to_zero}
\rev{\P\( \text{there are } \text{ distinct }A_1,\ldots,A_k\in [D^c,D]\text{ with equal sums} \) 
\to 0 \;\; \text{ as } D\to \infty.}
\ee

\subsection{Venn diagrams and linear algebra}\label{venn-linear}

Let $0 < c < 1$ be some fixed quantity, and let $D$ be a real number, large in terms of $c$. Suppose that $A_1,\dots, A_k \subset [D^c, D]$ are distinct sets. In this section we show that there is a rather natural way to associate a complete system $(\sV,\mathbf{c},{\bm \mu})$ (in the sense of Definition \ref{adm-meas}) to these sets. This system encodes the ``linear algebra of the Venn diagram of the $A_i$'' in a way that turns out to be extremely useful.

The Venn diagram of the $A_i$ has $2^k$ cells, indexed by $\{0,1\}^k$ in a natural way. Thus for each $\om=(\omega_1,\dots,\omega_k)\in \{0,1\}^k$, we define \begin{equation}\label{venn-def}
B_\om:= \bigcap_{i\,:\,\omega_i=1} A_i \bigcap_{i\,:\,\omega_i=0} 
(A_i)^c,
\end{equation}

\emph{The flag $\sV$.} Set $\Omega := \{ \om : B_{\om} \neq \emptyset\}$. We may put a total order $\prec$ on $\Omega$ by writing $\om' \prec \om$ if and only if $\max B_{\om'} < \max B_\om$.  We now select $r$ special vectors
$\om^1,\ldots,\om^r \in \Omega$, with $r\le k-1$, in the following manner. Let $\om^1 = \max_{\prec}(\Omega \setminus \{\bm 0, \bm 1 \} )$. Assuming we have chosen $\om^1,\dots,\om^j$ such that $\one,\om^1,\dots,\om^j$ are linearly independent over $\Q$, let $\om^{j+1} = \max ( \Omega \setminus \Span(\bm 1, \om^1,\dots, \om^j) )$, as long as such a vector exists.

Let \rev{$\mathbf{1}, \omega^1,\ldots, \omega^r$} be the set of vectors produced when this algorithm terminates. By construction, $\Omega \subset \Span(\one,\om^1,\dots,\om^r)$, or in other words $B_\om=\emptyset$ whenever \[ \om\in \{0,1\}^k\setminus \Span(\one,\om^1,\dots,\om^r).\]

Now define an $r$-step flag $\sV : \langle \mathbf{1}\rangle = V_0 < V_1 < \cdots < V_r$ by setting $V_j := \Span(\mathbf{1}, \omega^1,\cdots, \omega^j)$ for $1 \leq j \leq r$. 
\medskip

\emph{The parameters $\mathbf{c}$.} Now we construct the parameters $\mathbf{c} : 1 \geq c_1 \geq c_2 \geq \cdots \geq c_{r+1}$. For $j = 1,\dots, r$, we define \be\label{bj}
c_j = 1 + \frac{\lceil \log \max B_{\omega^j} - \log D\rceil}{\log D}.\ee
Thus \begin{equation}\label{sandwich} \max B_{\om^j} \in (\frac{1}{e}D^{c_j}, D^{c_j}]\end{equation} for $j = 1,\dots, r$.
Also set $c_{r+1}=c$. (\rev{The ceiling function $\lceil \cdot \rceil$ produces} a ``coarse'' or discretised set of possible thresholds $c_i$, suitable for use in a union bound later on; see Lemma \ref{num-data} below. The offset of $-\log D$ is to ensure that $c_1 \leq 1$.)

\medskip

\emph{The measures ${\bm \mu}$.} Set
\begin{equation}\label{bom-prime} B'_{\om} := \left \{ \begin{array}{ll} B_{\om} \setminus \rev{\{\max B_{\omega^j}\}} & \mbox{if $\om = \om^j$ for some $j$}, \\ B_{\om}  & \mbox{otherwise}. \end{array}\right.\end{equation}
Define
\begin{equation}\label{nu} \mu_j(\omega) := \frac{\# \big(B'_{\om} \cap (D^{c_{j+1}}, D^{c_{j}}]\big)}{\sum_{\om} \# \big(B'_{\omega} \cap (D^{c_{j+1}}, D^{c_{j}}]\big)} , \end{equation}
with the convention that if the denominator vanishes, then $\mu_j(\om) = 1_{\om = \mathbf{0}}$.

\begin{remark*}  It is important that we use the $B'_{\om}$ here, rather than the $B_{\om}$, for technical reasons that will become apparent in the proof of Proposition \ref{prop:ub} below.
\end{remark*}

\begin{lemma}
$(\sV,\mathbf{c},{\bm \mu})$ is a complete system (in the sense of Definition \ref{adm-meas}).
\end{lemma}
\begin{proof}
We need to check that $\Supp(\mu_j) \subset V_j$ for $j = 1,\dots, r$. By definition, if $\mu_j(\om) > 0$ then $B_{\om} \cap (D^{c_{j+1}}, D] \neq \emptyset$. This implies that $\max B_{\omega} > D^{c_{j+1}}$. \rev{On the other hand, \eqref{sandwich} implies that} $D^{c_{j+1}} \geq \max B_{\omega^{j+1}}$, and thus $\max B_{\om} > \max B_{\omega^{j+1}}$. \rev{By the construction of the vectors $\omega^i$, we must have $\omega \in \Span(\mathbf{1}, \omega^1,\cdots, \omega^j) = V_j$.}

We also need to check that $\sV$ is nondegenerate, also in the sense of Definition \ref{adm-meas}, that is to say $V_r$ is not contained in any hyperplane $\{\omega \in \Q^k : \omega_i = \omega_j\}$. This follows immediately from the fact that the $A_i$ are distinct. Since \[
A_i \triangle A_j =  \bigcup_{\substack{\om\in\{0,1\}^k\\ \omega_i\ne \omega_j}} B_\om,
\]
and so there is certainly some $\omega$ with $\omega_i \neq \omega_j$ and $B_\om\neq\emptyset$.
\end{proof}

Note that, in addition to the system $(\sV,\mathbf{c},{\bm \mu})$, the procedure described above outputs a sequence $\omega^1,\cdots, \omega^r$ of elements of $\{0,1\}^k$. We call the ensemble consisting of the system and the $\omega^i$ the \emph{linear data} associated to $A_1,\cdots, A_k$. 
\rev{We will only consider the event $\A \in \cE$, where
\be\label{a-norm}
\cE := \Big\{ A\subseteq [D^c,D]: \big| \# (A \cap (D^{\alpha}, D^{\beta}]) - (\beta - \alpha) \log D \big| \leq \log^{3/4} D \quad (c\le \alpha \le \beta \le 1) \Big\}.
\ee
By Lemma \ref{normalA}, $\P(\A \in \cE)=1-o(1)$ as $D\to \infty$.  In particular, if $A \in \cE$, we have  $|A \cap [D^c,D]| \le 2\log D$ for large enough $D$.}

\begin{lemma}\label{num-data}
\rev{Fix $k\in\Z_{\ge2}$ and suppose that $A \in \cE$.   The number of different ensembles of linear data arising from distinct sets $A_1,\cdots, A_k \subset A$  is $\ll (\log D)^{O(1)}$.}
\end{lemma}
\begin{proof}
The number of choices for $\omega^1,\cdots, \omega^r$ is $O(1)$, \rev{ and hence the
number of $\sV$ is also $O_k(1)$.} The thresholds $c_j$ are drawn from a fixed set of size $\log D$, and the numerators and denominators of the $\mu_j(\omega)$ are all \rev{integers $\leq 2\log D$.}
\end{proof}

\begin{remark} The $O(1)$ and the $\ll$ here both depend on $k$. However we regard $k$ as fixed here and do not indicate this dependence explicitly. If one is more careful then one can obtain results that are effective up to about $k \sim \log \log D$.
\end{remark}

\subsection{A local-to-global estimate}\label{upp:linalg}

Our next step towards establishing the bound $\beta_k \leq \gamma_k$ is to pass from the ``local'' event that a random logarithmic set $\A$ possesses a $k$-tuple of equal subsums $(\sum_{a\in A_1}a,\dots,\sum_{a\in A_k}a)$ to the ``global'' distribution of such subsums (with the subtlety that we must mod out by $\one$). The latter is controlled by the set $\sL_{\sV,\cc,\bmu}(\A)$ defined below.

\begin{definition}\label{def:L}
Given a set of integers $A$ and a system $(\sV,\mathbf{c},{\bm \mu})$, we write $
\sL_{\sV,\mathbf{c},{\bm \mu}}(A)$ for the set of vectors
\[
\sum_{\om\in\{0,1\}^k}\om\sum_{a\in B_\om} a\pmod\one,
\]
where $(B_\om)_{\om\in\{0,1\}^k}$ runs over all partitions of $A$ such that
\begin{equation}\label{eq:mu-adapted}
\mu_j(\omega) = \frac{\#\big(B_\omega\cap(D^{c_{j+1}},D^{c_j} ]\big)}{\#\big(A\cap(D^{c_{j+1}},D^{c_j}]\big)} \qquad \rev{(1\le j\le r, \; \omega\in\{0,1\}^k).}
\end{equation}
\end{definition}

\begin{proposition}\label{prop:ub}
Fix an integer \rev{$k\ge2$} and a parameter $0<c<1$. Let $D$ be large in terms of $c$ \rev{and $k$}, and let $\A \subset [D^c,D]$ be a logarithmic random set.  \rev{Let} 
\begin{equation}\label{etildedef}
\rev{\widetilde{\cE} = \Big\{ A\subseteq [D^c,D]: \big| \# (A \cap (D^{\alpha}, D^{\beta}]) - (\beta - \alpha) \log D \big| \leq 2\log^{3/4} D \quad (c\le \alpha \le \beta \le 1) \Big\}.}
\end{equation}
\rev{ Then we have
\begin{align}
    &\P \bigg(\exists\ \text{distinct}\ A_1,\dots,A_k\subseteq \A\ \text{such that}\ \sum_{a\in A_1}a=\cdots=\sum_{a\in A_k}a\bigg)  \nonumber \\
    &\quad\qquad  \le  (\log D)^{O(1)}
    \sup_{(\sV,\mathbf{c},{\bm \mu})} D^{-(c_1 + \cdots + c_r)}
    \E \mathbf{1}_{\A \in \widetilde{\cE}}
        |\sL_{\sV,\mathbf{c},{\bm \mu}}(\A)|+ \P (\cE^c).\label{prob:ub:statement}
\end{align}}
Here, the supremum is over all complete systems $(\sV,\mathbf{c},{\bm \mu})$ with $c_{r+1}=c$.
\end{proposition}
\begin{proof} 
\rev{Recall the definition of the set $\cE$, given in equation \eqref{a-norm}. We have
\begin{align*}
  &\P \bigg(\exists\ \text{distinct}\ A_1,\dots,A_k\subseteq \A\ \text{such that}\ \sum_{a\in A_1}a=\cdots=\sum_{a\in A_k}a\bigg)\\
&\quad \le \P (\cE^c) +  \sum_{\substack{\sV,\mathbf{c},{\bm \mu},(\omega^i)}} \; \sum_{\substack{A \in \mathscr{S}(\sV,\mathbf{c},{\bm \mu},(\omega^i))}} \P(\A =A),
\end{align*}
where, given linear data $\{(\sV,\mathbf{c},\mathbf{\mu}), \omega^1,\dots, \omega^r\}$, we write $\mathscr{S}(\sV,\mathbf{c},{\bs\mu},(\omega^i))$ to denote the set of all $A\in\cE$ that have $k$ distinct subsets  $(A_1,\dots,A_k)$ with equal sums-of-elements and associated linear data $\{(\sV,\mathbf{c},\mathbf{\mu}),\omega^1,\dots, \omega^r\}$. (The set $\A$ appearing in \eqref{prob:ub:statement} will be constructed below by removing certain elements from the logarithmic set $\A$ we started with; this new set belongs to $\widetilde{\cE}$, but not necessarily to $\cE$.)

Let us fix a choice of linear data  $\{(\sV,\mathbf{c},\mathbf{\mu}), \omega^1,\dots, \omega^r\}$ and let us abbreviate $\sS$ for the set $\mathscr{S}(\sV,\mathbf{c},{\bm \mu},(\omega^i))$.}
An elementary probability calculation gives
\begin{equation}\label{to-upper} 
E(\sS) := \sum_{A \in \mathscr{S}} \P(\A  = A) 
	= \sum_{A \in \mathscr{S}} \prod_{D^c < a \leq D} \Big(1 - \frac{1}{a}\Big) 
		\prod_{a \in A} \frac{1}{a - 1}.
\end{equation}
 
\rev{For each $A \in \sS$,} fix a choice of $(A_1,\ldots,A_k)$ \rev{with equal sums and such 
	that the linear data associated to $(A_1,\dots, A_k)$ is $\{ (\sV, \mathbf{c}, {\bm \mu}), \omega^1,\cdots, \omega^r\}$}. Let $B_\om$ be the cells of the Venn diagram corresponding to the $A_i$, as in \eqref{venn-def}, and then define the $B_\om'$ as in \eqref{bom-prime}.
Recall that \eqref{nu} holds, and define $K_j = \max B_{\om^j}$ for $1
\le j\le r$.
In particular, $K_1 > \cdots > K_r$.
Let $A' = A  \setminus \{K_1,\ldots,K_r\}$.
\rev{Note that $A'\in\tilde{\cE}$ if $D$ is large enough in terms of $k$. Moreover, we have}
\[
\sum_{a\in A_i} a=\sum_{\omega\in\{0,1\}^k}\omega_i\sum_{a\in B_\omega}a.
\]
Therefore, \rev{the equal sums condition} is equivalent to
\[
\sum_{\om\in\{0,1\}^k} \om \sum_{a\in B_\om} a
%=\Big( \sum_{a\in A_1} a, \cdots, \sum_{a\in A_k} a \Big)  
\rev{ \;= \;\mathbf{0} \; \md{\one} ,}
\]
 and hence
\be\label{special} 
\sum_{j = 1}^r  K_j \omega^j = - \sum_{\om} \omega \sum_{a' \in  B'_{\omega}}  a'\, \md{\mathbf{1}}.
\ee
 Since $\mathbf{1}, \omega^1,\cdots, \omega^r$ are linearly independent, the value of the right-hand side of \eqref{special} uniquely determines the numbers $K_j$, which themselves uniquely determine $A$ in terms of the sets $B_\om'$.
\rev{Therefore, given $A' \in \tilde{\cE}$, the number of possible sets $A$ is, by
Definition \ref{def:L}, at most $|\sL_{\sV,\mathbf{c},{\bm \mu}}(A')|$.
  Moreover by \eqref{sandwich} we have $K_j > \frac{1}{e} D^{c_j}$ for every $j$,
and therefore 
\begin{equation}\label{prob A -> A'} 
	\prod_{a \in A} \frac{1}{a-1} \ll D^{-(c_1 + \dots + c_r)} \prod_{a \in A'} \frac{1}{a - 1}.
\end{equation}}
   We sum over $A'$, and reinterpret the product on the \rev{right-hand side
   of \eqref{prob A -> A'}} in terms of $\P(\A=A')$.  This gives
\dalign{
E(\sS) &\ll D^{-(c_1 + \dots + c_r)}  \sum_{A' \in \tilde{\cE}} |\sL_{\sV,\cc,\bmu}(A')|  
	\prod_{D^c<a\le D} \Big(1-\frac{1}{a}\Big) 
	\prod_{a \in A'} \frac{1}{a - 1}\\
&= D^{-(c_1 + \dots + c_r)}  \sum_{A'\in \tilde{\cE}} |\sL_{\sV,\cc,\bmu}(A')| 
	\cdot \rev{\P(\A=A')}\\
&=  D^{-(c_1 + \dots + c_r)} \rev{ \E 1_{\A\in \tilde{\cE}} 
      \cdot   |\sL_{\sV,\mathbf{c},{\bm \mu}}(\A)|}.
}
By Lemma \ref{num-data} there are $(\log D)^{O(1)}$ possible choices for the \rev{linear} data
 $\{(\sV,\mathbf{c},\mathbf{\mu}), \omega^1,\dots, \omega^r\}$,
 and the proof is complete.
\end{proof}

\subsection{Upper bounds in terms of entropies}\label{section entropy}

Having established Proposition \ref{prop:ub}, we turn to the study of the sets $\sL_{\sV,\cc,\bmu}(A)$. We will bound their cardinality in terms of the quantities 
$\er(\sV',\mathbf{c},\mathbf{\mu})$ from Definition \ref{entropy-dfn}
with $\sV'$ a subflag of $\sV$.

\begin{lemma}\label{lem:bound for L} 
Let $(\sV,\cc,{\bm \mu})$ be a system and \rev{let $A\in \tilde{\cE}$,
where $\tilde{\cE}$ is defined in \eqref{etildedef}.}
 Then, for any subflag $\sV'$ of $\sV$,
	\begin{equation}\label{v-vp}
	|\sL_{\sV,\cc,{\bm \mu}}(A)|
	\ll_{\sV'} e^{O(\log^{3/4} D)}
	 D^{\er(\sV',\mathbf{c},\mathbf{\mu})}.
	\end{equation}
\end{lemma}

\begin{remark*} 
	\rev{The implied constant in} the $\ll_{\sV'}$  could be made explicit if desired (in terms of the quantitative rationality of a basis for the spaces in $\sV'$) but we have no need to do this. 
\end{remark*}

\begin{proof}[Proof of Lemma \ref{lem:bound for L}]
Given a set $X\subset[D^c,D]$, write $X^{(j)}:=X\cap(D^{c_{j+1}},D^{c_j}]$ for $j = 1,\dots, r$. Throughout the proof, we will assume that $A$ is a set of integers and that $(B_{\om})_{\om \in \{0,1\}^k}$ runs over all partitions of $A$ such that \eqref{eq:mu-adapted} is satisfied. In our new notation, this may be rewritten as
\begin{equation}\label{mu-adapted-new} 
\rev{|B_{\om}^{(j)}|= \mu_j(\omega)|A^{(j)}|,\quad j = 1,\dots, r, \; \; \omega \in \{0,1\}^k.}
\end{equation}

\rev{For each $j$, $1 \leq j \leq r$,} fix a linear projection $P_j : V_j \rightarrow V'_j$, and set $Q_j := \id_{V_j} - P_j$, \rev{so that $Q_j$ maps $V_j$ to itself.}
Set
\[ 
\mathscr{L}^P(A) := \Big\{ \sum_{j = 1}^r \sum_{\substack{\om \in \{0,1\}^k \\ \om\in V_j}} P_j(\om) \sum_{a \in B_{\om}^{(j)}} a \ \md{\mathbf{1}}  : \mbox{\eqref{mu-adapted-new} is satisfied} \Big\} 
\]
and 
\[ 
\mathscr{L}^Q(A) := \Big\{ \sum_{j = 1}^r \sum_{\substack{\om \in \{0,1\}^k \\ \om\in V_j}} Q_j(\om) \sum_{a \in B_{\om}^{(j)}} a \ \md{\mathbf{1}}  : \mbox{\eqref{mu-adapted-new} is satisfied} \Big\} .
\]
Since
\[
\sum_{\om\in\{0,1\}^k} {\om} \sum_{a\in B_\om} a	
	=\sum_{j=1}^r
	\sum_{\substack{\om \in \{0,1\}^k \\ \om\in V_j}} {P}_j(\om)\sum_{a\in B_\om^{(j)}} a	
	+\sum_{j=1}^r
	\sum_{\substack{\om \in \{0,1\}^k \\ \om\in V_j}}{Q}_j(\om) \sum_{a\in B_\om^{(j)}} a,
\] it follows immediately from the definition of $\mathscr{L}_{\sV,\mathbf{c},{\bm \mu}}(A)$ (Definition \ref{def:L}) that
\begin{equation}\label{prod-ells}
|\mathscr{L}_{\sV,\cc,\bmu}(A)|
	\leq |\mathscr{L}^P(A)|\cdot |\mathscr{L}^Q(A)|.
\end{equation}
We claim that 

\begin{equation}\label{lp-bdd} |\mathscr{L}^P(A)| \ll_{\mathscr{V'}} (\log D)^r  D^{\sum_{j = 1}^r c_j \dim (V'_j/V'_{j-1})} \end{equation} and that 

\begin{equation}\label{lq-bd}  |\mathscr{L}^Q(A)| \leq e^{O(\log^{3/4} D)}D^{\sum_{j = 1}^r (c_{j} - c_{j+1}) \HH_{\mu_j}(V'_j)}. \end{equation}
These bounds, substituted into \eqref{prod-ells}, immediately imply Lemma \ref{lem:bound for L}.

It remains to establish \eqref{lp-bdd} and \eqref{lq-bd}, which are proven in quite different ways.  We begin with \eqref{lq-bd}, which is a ``combinatorial'' bound, in that there cannot be too many choices for the data making up the sums in $\mathscr{L}^Q(A)$. For this, observe that $Q_j$ vanishes on $V'_j$ and hence is constant on cosets of $V'_j$.
Therefore the elements of $\mathscr{L}^Q(A)$ are determined by the sets $\bigcup_{\om \in v_j + V'_j} B^{(j)}_{\om}$, over all $v_j \in V_j/V'_j$ and $1 \leq j \leq r$. By \eqref{mu-adapted-new}, 
\[ 
\Big|\bigcup_{\om \in v_j + V'_j} B^{(j)}_{\om}\Big| 
	= \mu_j(v_j + V'_j) |A^{(j)}|,
\] 
and by Lemma \ref{entropy-multinomial} the number of ways of partitioning $A^{(j)}$ into sets of these sizes is bounded above by $e^{\HH(\mathbf{p}^{(j)}) |A^{(j)}|}$, where $\mathbf{p}^{(j)} = (\mu_j(v_j + V'_j))_{v_j \in V_j/V'_j}$. By Definition \ref{entropy-def}, $\HH(\mathbf{p}^{(j)}) = \HH_{\mu_j}(V'_j)$. Taking the product over $j = 1,\dots, r$ gives
\[ |\mathscr{L}^Q(A)| \leq e^{\sum_{j = 1}^r \HH_{\mu_j}(V'_j) |A^{(j)}|}.\]
From \rev{the assumption that $A\in \widetilde{\cE}$, where $\widetilde{\cE}$ is defined in \eqref{etildedef}}, we have
\[ |A^{(j)}| = (c_j - c_{j+1}) \log D + O(\log^{3/4} D).\]
Using this, and the trivial bound $\HH_{\mu_j}(V'_j) \leq \log |\Supp(\mu_j)| \leq \log (2^k)$, \eqref{lq-bd} follows.

Now we prove \eqref{lp-bdd}, which is a ``metric'' bound, the point being that none of the sums in $\mathscr{L}^P(A)$ can be too large in an appropriate sense. Pick a basis for $\Q^k$ adapted to $\sV'$: that is, a basis $e_1,\dots, e_k$ such that $V'_j = \rev{\Span(e_1,\dots, e_{\dim V'_j})}$ for each $j$, and $e_1 = \mathbf{1}$. There are positive integers $M,N = O_{\sV',\sV}(1)$ such that, in this basis, the $e_i$-coordinates of $P_j(\omega)$ are all rationals with denominator $M$ and absolute value at most $N$.

 Now for fixed $j$ and $\om$,  if $D$ is large then
\rev{$
 \sum_{a \in B_{\om}^{(j)}} a  \le D^{c_j} \log D,
 $}
 since $B_{\om}^{(j)} \subset (D^{c_{j+1}}, D^{c_j}]$ and by \rev{the assumption 
 that $A\in \widetilde{\cE}$.}
  Thus
 \[ 
 \sum_{\substack{\om \in \{0,1\}^k \\ \om \in V_j}} P_j(\om)  \sum_{a \in B_{\om}^{(j)}} a  
 	  \in  \Big\{ \sum_{1\le i\le \dim(V_j')} x_i e_i \in \Q^k : 
 Mx_i \in \Z,\ |x_i| \le r N D^{c_j} \log D \; (\text{for all }i) \Big\} ,
\]
and so
\[
 \sum_{j = 1}^r \sum_{\substack{\om \in \{0,1\}^k \\ \om\in V_j}} P_j(\om) \sum_{a \in B_{\om}^{(j)}} a  
 \in  \Big\{ \sum_{1\le i\le k} x_i e_i \in \Q^k : 
 \begin{array}{l}  Mx_i \in \Z\ \text{and}\ 	|x_i| \le r^2 N D^{c_j} \log D \\
 \text{for}\ \dim V'_{j-1} < i \leq \dim V_j'\ \text{and}\ 1\le j\le r
 \end{array}
 \Big\}.
\]

\rev{We must bound the} number of different values that the expression $\sum_{i=1}^k x_ie_i$ can take mod $\mathbf{1}$ when the coefficients $x_1,\dots,x_k$ are as above. 
\rev{Since $e_1 = \mathbf{1}$ and $x_1 M\in \Z$, given $x_2,\ldots,x_k$ there are at most
$M$ possibilities for $x_1$ mod $\one$.}  In addition, there are
\[ 
\ll \rev{(r^2MN)^{k-1}} (\log D)^r D^{\sum_{j = 1}^r c_j \dim (V'_j/V'_{j-1})}
\] 
possibilities for $x_2,\dots,x_k$, thereby concluding the proof of \eqref{lp-bdd} and hence of Lemma \ref{lem:bound for L}.
\end{proof}

\rev{A potential problem with applying Lemma \ref{lem:bound for L} is that
there may be infinitely many subflags $\sV'$ to consider, and the constant
implied by the $\ll$-symbol depends on $\sV'$.  As we shall see in the next
Lemma, however, we may reduce the problem to consideration of a finite number of
subflags, a tool which will be used in several parts of this paper.}

\rev{\begin{lemma}\label{subflags-equiv-classes}
 For a given $k$, the set of all flags 
\[
\sV': \spanone = V'_0 \leq  V'_1 \leq  V'_2 \leq  \cdots \leq  V'_r \leq  \Q^k
\]
may be partitioned into $O_k(1)$ equivalence classes such that any two flags
$\mathscr{V'},\mathscr{V''}$  in the
same equivalence class satisfy  $\dim V'_j = \dim V''_j$ for all $j$,
and for any threshholds $\cc$ satisfying $c_1\ge c_2 \ge \cdots \ge c_{r+1}$ and probability measures $\bmu$ supported on $\{0,1\}^k$, we have
$\HH_{\mu_j}(V'_j) = \HH_{\mu_j}(V''_j)$ for all $j$
and
$\er(\sV',\mathbf{c}, {\bm \mu}) = \er(\sV'',\mathbf{c}, {\bm \mu})$.
\end{lemma}
}

\rev{\begin{proof}
 We say that two subflags $\sV',\sV''$ are \emph{equivalent} if $V'_j, V''_j$ have the same intersection with $\{0,1\}^k$ and $\dim V'_j = \dim V''_j$, for all $j = 1,\ldots, r$. 
There are clearly only $O_k(1)$ equivalence classes, 
and the desired properties hold for members of the same equivalence class by the 
definition of $\HH_{\mu_j}(V'_j)$ and $\er(\mathscr{V'},\cc,\bmu)$.
\end{proof}
}

\rev{Armed with Lemma \ref{subflags-equiv-classes}, we immediately obtain from
Lemma \ref{lem:bound for L}, applied to one representative from each class,
 the following corollary.}

\begin{corollary}\label{min-v-e}
Let $(\sV,\mathbf{c},{\bm \mu})$ be a system and \rev{suppose that $A\in\tilde{\cE}$.} Then 
\[ 
	|\sL_{\sV,\cc,{\bm \mu}}(A)|
	\ll e^{O(\log^{3/4} D)}\min_{\sV' \leq \sV}  
	 D^{\er(\sV',\cc,\bmu)}.
	\]  
\end{corollary}

\subsection{The upper bound in Theorem \ref{thm:beta-gamma}}

We can now establish the upper bound in Theorem \ref{thm:beta-gamma}, that is to say the inequality $\beta_k \leq \gamma_k$.

We start by applying Proposition \ref{prop:ub}.
\rev{Together with Lemma \ref{normalA}, it implies that}
 \begin{align*}
    & \rev{\P\big(\exists\ \text{distinct}\ A_1,\dots,A_k\subseteq \A \cap (D^c, D]  \ \text{ with equal sums} \big)}\\
    &\quad\qquad  \le  (\log D)^{O(1)}
    \sup_{(\sV,\mathbf{c},{\bm \mu})} D^{-\er(\sV,\mathbf{c},{\bm \mu})}
    \rev{\E 1_{\A\in \tilde{\mathcal{E}}} 
        |\sL_{\sV,\mathbf{c},{\bm \mu}}(\A)|}+ O(e^{-\frac{1}{4} \log^{1/2} D}).
\end{align*}
Here, the \rev{supremum} is over complete systems $(\sV,\mathbf{c},{\bm \mu})$ with $c_{r+1} = c$, and we made the observation that for such systems we have 
\[ \er(\sV,\mathbf{c},{\bm \mu}) = c_1 + \dots + c_r,\] an immediate consequence of \rev{the definition of $\er(\sV,\cc,\bmu)$} and the fact that $\HH_{\mu_j}(V_j) = 0$ for all $j$ and that $\dim V_j = j+1$. Thus we may \rev{apply Corollary \ref{min-v-e}}, concluding that 
\begin{align*}
    \rev{\P\big(\exists\ \text{distinct}\ A_1,\dots,A_k\subseteq \A \cap (D^c, D]  \ \text{ with equal sums}\big)
     \le  D^{\theta+o(1)}}
   + O(e^{-\frac{1}{4} \log^{1/2} D}),
\end{align*}
where
\rev{\be\label{theta-maxmin}
\theta := \sup_{(\sV,\mathbf{c},{\bm \mu}) \,:\, c_{r+1} = c} 
		\min_{\sV' \leq \sV} 
		\big( \er(\sV', \mathbf{c}, {\bm \mu}) - \er(\sV,\mathbf{c},{\bm \mu})) \,;
\ee		
the supremum} is over all complete systems $(\sV,\mathbf{c},{\bm \mu})$ with $c_{r+1} = c$, and the minimum is over all subflags $\mathscr{V'} \leq \sV$.
\rev{Note that the minimum exists by Lemma \ref{subflags-equiv-classes}, since we may
restrict attention to a finite set of subflags $\mathscr{V'}$.
Moreover, the supremum is realised, meaning there is a system $(\sV,\cc,\bmu)$
for which the right side of \eqref{theta-maxmin} equals $\theta$.
Indeed, there are $O(1)$ choices for $\sV$, and with $\sV$ fixed 
the quantities $\mathbf{c},{\bm \mu}$ range over compact subsets of Euclidean space,
with the right side of \eqref{theta-maxmin} continuous in these variables.

Now, if we assume that $c>\gamma_k$, then the definition of $\gamma_k$ in Problem \ref{opt-problem} implies that there is no system $(\sV,\cc,\bmu)$ with $c_{r+1}=c$ and that satisfies the entropy condition \eqref{entropy-cond}. Equivalently, if $c_{r+1}=c$, then $\min_{\sV'\leq \sV} \big(\er(\sV', \mathbf{c}, {\bm \mu}) - \er(\sV,\mathbf{c},{\bm \mu}))\big)<0$. In particular, we have $\theta<0$. We have thus established \eqref{upper-bound:prob_to_zero}, as required.}

\begin{remark*}  \rev{In the above proof,} $(\sV,\mathbf{c},{\bm \mu})$ is a \emph{complete} system. However, for other aspects of our problem it is not natural to focus on the completeness condition, for which reason we omit it from the \rev{definition of $\gamma_k$.}
\end{remark*}

\section{The lower bound $\beta_k \geq \tilde \gamma_k$}\label{lower-bound-sec}

\subsection{Introduction and simple reductions}

The aim of this section and the next is to establish the lower bound $\beta_k \geq \tilde \gamma_k$.  We begin, in Lemma \ref{gammak-tilde:reduction} below,
by showing that we may restrict our attention to certain systems satisfying some additional regularity conditions.

We isolate a ``folklore'' lemma from the proof for which it is not easy to find a good reference. The authors thank Carla Groenland for a helpful conversation on this topic.

\begin{lemma}\label{cube-subspace}
Let $V$ be a subspace of $\Q^k$. Then $\# (V \cap \{0,1\}^k) \leq 2^{\dim V}$.
\end{lemma}
\begin{proof} We outline two quite different short proofs. Let $d := \dim V$. 

\emph{Proof 1.} We claim that there is a projection from $\Q^k$ onto some set of $d$ coordinates which is injective on $V$. From this, the result is obvious, since the image of $\{0,1\}^k$ under any such projection has size $2^d$. To prove the claim, let $e_1,\dots,e_n$ denote the standard basis on $\Q^n$. Note that if $W \leq \Q^n$ and if none of the quotient maps $\Q^n \mapsto \Q^n/\langle e_i\rangle$ is injective on $W$, then $W$ must contain a multiple of each $e_i$, and therefore $W = \Q^n$. Thus if $W$ is a proper subspace of $\Q^n$ then there is a projection onto some set of $(n-1)$ coordinates which is injective on $W$. Repeated use of this fact establishes the claim.

\emph{Proof 2.} Suppose that $\# (V \cap \{0,1\}^k)$ contains $2^d + 1$ points. These are all distinct under the natural ring homomorphism $\pi : \Z^k \rightarrow \F_2^k$, and so their images cannot lie in a subspace (over $\F_2$) of dimension $d$. Hence there are $v_1,\dots, v_{d+1} \in V$ such that $\pi(v_1),\dots, \pi(v_{d+1})$, are linearly independent over $\F_2$. The $(d +1) \times k$ matrix formed by these $\pi(v_i)$ therefore has a $(d+1) \times (d+1)$-subminor which is nonzero in $\F_2$. The corresponding subminor of the matrix formed by the $v_i$ is therefore an odd integer, and in particular not zero. This means that $v_1,\dots, v_{d+1}$ are linearly independent over $\Q$, contrary to the assumption that $\dim(V)=d$.
\end{proof}

\rev{We now record an immediate corollary of Lemma \ref{subflags-equiv-classes},
which provides a ``gap condition'' on the $\er$-quantities.

\begin{lemma}\label{lem:entropy-gap}
If the system $(\sV,\cc,\bmu)$ satisfies \eqref{strict-entropy-cond}
then there is an $\eps>0$ such that for all proper subflags $\sV'$,
\be\label{entropy-gap}
\er(\sV',\cc,\bmu) \ge \er(\sV,\cc,\bmu) + \eps.
\ee
\end{lemma}
}

\rev{For future reference, the next two lemmas  record more information about optimal systems for $\tilde{\gamma}_k$ and for $\gamma_k$, respectively.}

\rev{
\begin{lemma}\label{gammak-tilde:reduction} Let $k\in\Z_{\ge2}$. We have that $\tilde{\gamma}_k$ 
is the  supremum of all $c>0$ for which there is a system $(\sV,\cc,\bmu)$
such that $c_{r+1}=c$, \eqref{strict-entropy-cond} holds and we further have:
\begin{enumerate}
\item $1=c_1 > c_2 > \cdots >c_{r+1}=c$; 
\item $\HH_{\mu_j}(V_{j-1}) > \dim(V_j/V_{j-1})$ for $1\le j\le r-1$ and $\HH_{\mu_r}(V_{r-1}) > \frac{c_r}{c_r-c_{r+1}} \dim(V_r/V_{r-1})$ ; 
\item $\dim(V_1/V_0)=1$;
\item $\Supp(\mu_j) = V_j \cap \{0,1\}^k$ for $j=1,2,\ldots,r$;
\item for all $j$ and $\om$, $\mu_j(\om)=\mu_j(\one-\om)$.
\end{enumerate}
\end{lemma}
}

\begin{proof}
\rev{First of all, we show that we may assume that $c>0$ and that statement (d) holds. Indeed, if a system $(\sV,\bmu,\cc)$ satisfies \eqref{strict-entropy-cond}, then Lemma \ref{lem:entropy-gap} implies that \eqref{entropy-gap} holds for some $\eps>0$. As the difference between the left and right sides of \eqref{entropy-gap} is continuous in the quantities $c_j$ and $\mu_j(\omega)$, we may increase $c_{r+1}$ (and possibly some of the other $c_j$'s) a tiny bit and we may also adjust the measures $\mu_j$ by a small amount, so that $c_{r+1}>0$, statement (d) holds, and we also have that
	\[
	\er(\mathscr{V'},\cc,\bmu) \ge \er(\sV,\cc,\bmu) + \eps/2
	\]
for every proper subflag $\mathscr{V'}$. 
	
Next, we show that we may take $c_1=1$. Indeed, condition \eqref{strict-entropy-cond} implies that $\er(\sV',\cc,\bmu)\ge \er(\sV,\cc,\bmu)\ge0$ for all $\sV'\leq \sV$ by \eqref{full-e}. Now if $c_1<1$ and $\tilde{c}_j=c_j/c_1$ for each $j$,
then the perturbed system $(\sV,\tilde{\cc},\bmu)$ 
has a larger value of $c_{r+1}$, and moreover
also satisfies \eqref{strict-entropy-cond}, since for any subflag $\mathscr{V'}$ we have
\[
\er(\mathscr{V'},\tilde{\cc},\bmu) = (1/c_1) \er(\mathscr{V'},\cc,\bmu).
\]

Next, consider a system $(\sV,\cc,\bmu)$ satisfying $c_1=1$ and $c_{r+1}=c>0$, and consider the subflag $\sV' : \langle\one\rangle = V'_0 \le V_1'\le \cdots \le V_r'$, where $V_i' = V_i$ for $i\ne j$, and $V_{j}'=V_{j-1}$; that is, $\sV'$ has two consecutive copies of $V_{j-1}$.  By assumption (Definition \ref{adm-meas}), we have $V_{j-1}\neq V_j$, and thus $\sV'$ is a proper subflag of $\sV$. Thus
\[
\er(\sV',\cc,\bmu) - \er(\sV,\cc,\bmu) = 
	\begin{cases} 
		(c_j-c_{j+1}) \big( \HH_{\mu_j}(V_{j-1}) - \dim(V_j/V_{j-1}) \big)  &\text{if}\ j\le r-1,\\
		(c_r-c_{r+1})\HH_{\mu_r}(V_{r-1}) - c_r\dim(V_r/V_{r-1}) &\text{if}\ j=r.
	\end{cases}
\]
Since the left-hand side is positive, we conclude that (a) and (b) hold.

(c) Let $d=\dim(V_1/V_0)$. By Lemma \ref{cube-subspace}, we have $|V_1 \cap \{0,1\}^k| \le 2^{\dim V_1} = 2^{d+1}$ 
and hence $\mu_1$ is supported on at most $2^{d+1}-1$
cosets of $V_0$ (since $\mathbf{1} \in V_0$, the points $\mathbf{0}$ and $\mathbf{1}$ lie in the same coset). In particular, by Lemma \ref{ent-trivial}, $\HH_{\mu_1}(V_0) 
\le \log (2^{d+1}-1)$. On the other hand, $\HH_{\mu_1}(V_0) > d$ by statement (b). We must thus have $d = 1$, which is exactly statement (c).

(e) Assume the system $(\sV,\cc,\bmu)$ satisfies
 \eqref{strict-entropy-cond} and (a).} For every $j$ and $\om \in V_j$, we define
\[
\tilde{\mu}_j (\om) = \frac{\mu_j(\om)+\mu_j(\one-\om)}{2}.
\]
We then consider the system $(\sV,\cc,\tilde{\bmu})$, and must show that it also satisfies \eqref{strict-entropy-cond}. For this, it is enough to show that 
\begin{equation}\label{ent-mono} \HH_{\widetilde{\mu}_j}(V_j') \ge \HH_{\mu_j}(V_j')\end{equation} for all $j$. Indeed, we then have, \rev{ for every proper
subflag $\mathscr{V'}$,}
\[
\er(\sV',\cc,\widetilde{\boldsymbol{\mu}}) \ge 
\rev{\er(\sV',\cc,\bmu) > \er(\sV,\cc,\bmu)} = \er(\sV,\cc,
\widetilde{\boldsymbol{\mu}}).
\]
To prove \eqref{ent-mono}, write
\[ \HH_{\mu_j}(V'_j) = \sum_C L(\mu_j(C)), \quad \HH_{\tilde\mu_j}(V'_j) = \sum_C L(\tilde\mu_j(C)),\] where the sum is over all cosets $C$ of $V'_j$ and $L(t) = -t \log t$. 
Thus, since $-C$ runs over all cosets as $C$ does, \rev{we have
\[
\HH_{\mu_j}(V'_j) =  \sum_C \frac{ L(\mu_j(C)) + L(\mu_j(-C))}{2} .
\]
By the concavity of $L$, we have
\[
 \frac{ L(\mu_j(C)) + L(\mu_j(-C))}{2} \le L\Big(  \frac{\mu_j(C) + \mu_j(-C)}{2}\Big) = L(\tilde{\mu}_j(C)).
 \]
Claim \eqref{ent-mono} then readily follows.}
\end{proof}

\rev{
	\begin{lemma}\label{gammak:reduction} Let $k\in\Z_{\ge2}$ be such that $\gamma_k>0$. Then we have that $\gamma_k$ 
		is the  supremum of all $c>0$ for which there is a system $(\sV,\cc,\bmu)$
		such that $c_{r+1}=c$, \eqref{entropy-cond} holds and we further have:
		\begin{enumerate}
			\item $1=c_1 > c_2 > \cdots >c_{r+1}=c$; 
			\item $\HH_{\mu_j}(V_{j-1}) \ge \dim(V_j/V_{j-1})$ for $1\le j\le r-1$ and $\HH_{\mu_r}(V_{r-1}) \ge \frac{c_r}{c_r-c_{r+1}} \dim(V_r/V_{r-1})$ ; 
			\item $\dim(V_1/V_0)=1$;
			\item $\bigcup_{i=1}^j \Supp \mu_i$ spans $V_j$ for $j=1,2,\ldots,r$;
			\item for all $j$ and $\om$, $\mu_j(\om)=\mu_j(\one-\om)$.
		\end{enumerate}
	\end{lemma}

\begin{remark*}
As we will see in Part \ref{part:binary-systems}, we always have $\gamma_k>0$.
\end{remark*}
}

\begin{proof}
	\rev{The proof that we may take $c_1=1$ is the same as in Lemma \ref{gammak-tilde:reduction}. 
		
		Next, consider a system $(\sV,\cc,\bmu)$ satisfying $c_1=1$ and $c_{r+1}=c>0$, and 
		consider the subflag $\sV' : \langle\one\rangle = V'_0 \le V_1'\le \cdots \le V_r'$, where $V_i' = V_i$ for $i\le r-1$, and $V_r'=V_{r-1}$.  Thus
		\[
		\er(\sV',\cc,\bmu) - \er(\sV,\cc,\bmu) = 
			(c_r-c_{r+1})\HH_{\mu_r}(V_{r-1}) - c_r\dim(V_r/V_{r-1}) .
		\]
		Since the left-hand side is $\ge0$ and we have assumed that $c_{r+1}=c>0$ and that $V_{r-1}\neq V_r$, the latter being true from Definition \ref{adm-meas}, we conclude that
		\begin{equation}\label{gammak:reduction e1}
		c_r>c_{r+1} 
			\quad\text{and}\quad 
		\HH_{\mu_r}(V_{r-1}) \ge \frac{c_r}{c_r-c_{r+1}} \dim(V_r/V_{r-1}). 
	\end{equation}
	This proves part of statements (a) and (b). We shall now prove them fully. 
	
	(a) There are always indices $1=i_1<i_2<\cdots<i_s<i_{s+1}=r+1$ such that
	\[
	c_{i_j}=\cdots=c_{i_{j+1}-1}>c_{i_{j+1}} \quad\text{for}\ j=1,\dots,s.
	\]
	Crucially, note that $i_{s+1}=r+1$ because $c_r>c_{r+1}$ by \eqref{gammak:reduction e1}. Next, we define the system $(\mathscr{W},\boldsymbol{\nu},\mathbf{d})$, where $\mathscr{W}$ is an $s$-step flag and, for all $j\in\{1,\dots,s\}$, we have
\[
	W_j=V_{i_{j+1}-1},\quad \nu_j=\mu_{i_{j+1}-1},\quad\text{and}\quad 
		d_j=c_{i_{j+1}-1}.
	\]
	In particular, $W_s=V_{i_{s+1}-1}=V_r$ because $i_{s+1}=r$, and thus $\mathscr{W}$ is a non-degenerate flag system as per Definition \ref{adm-meas} (b). Clearly, $1=d_1>d_2>\cdots>d_s>d_{s+1}=c$, so in order to prove part (a), all that remains to show  is that the system  $(\mathscr{W},\boldsymbol{\nu},\mathbf{d})$ satisfies the entropy condition \eqref{entropy-cond}. This follows by a simple computation. Indeed, let $\mathscr{W}'$ be a subflag of $\mathscr{W}$. We then define $\sV'\le \sV$ by letting 
	$V_m'=W_j$ whenever $i_j\le m<i_{j+1}$. Hence,
	\begin{align*}
	\er(\sV',\bmu,\cc) 
			&= \sum_{m=1}^r (c_m-c_{m+1}) \HH_{\mu_m}(V_m') + \sum_{m=1}^r c_m\dim(V_m'/V_{m-1}') \\
			&= \sum_{j=1}^s (c_{i_{j+1}-1} -  c_{i_{j+1}}) \HH_{\mu_m}(V_m') + \sum_{j=1}^s c_{i_j} \dim(V_{i_j}'/V_{i_j-1}') \\
			&=\er(\mathscr{W}',\boldsymbol{\nu},\mathbf{d}).
	\end{align*}
	Consequently, since the system $(\sV,\bmu,\cc)$ satisfies condition \eqref{entropy-cond}, so does $(\mathscr{W},\boldsymbol{\nu},\mathbf{d})$. This proves that we may always assume condition (a).
	
	(b) Consider a system $(\sV,\cc,\bmu)$ satisfying (a). We then argue as in Lemma \ref{gammak-tilde:reduction}, by considering the subflag $\sV'$ with $V_i' = V_i$ for $i\ne j$, and $V_{j}'=V_{j-1}$. We then have 
		\[
		\er(\sV',\cc,\bmu) - \er(\sV,\cc,\bmu) = 
		\begin{cases} 
			(c_j-c_{j+1}) \big( \HH_{\mu_j}(V_{j-1}) - \dim(V_j/V_{j-1}) \big)  &\text{if}\ j\le r-1,\\
			(c_r-c_{r+1})\HH_{\mu_r}(V_{r-1}) - c_r\dim(V_r/V_{r-1}) &\text{if}\ j=r.
		\end{cases}
		\]
		Since the left-hand side is $\ge0$ and $c_j-c_{j+1}>0$ for all $j=1,\dots,r$, statement (b) follows. 
		
		(c) Assuming statement (b), we may prove statement (c) by arguing as in Lemma \ref{gammak-tilde:reduction}.

		(d) Suppose that (a) holds. Consider the flag
		$\sV': \langle\one\rangle \le V_1'\le  \cdots \le V_r'$, where 
		\[
		V_j' = \Span \bigg(\bigcup_{i=1}^j \Supp(\mu_j)\bigg)
		\qquad (1\le j\le r).
		\]
		It is easy to see from the definition of a system (Definition \ref{adm-meas}) that $\sV'$ is a subflag of $\sV$. We have $\HH_{\mu_j}(V'_j)=0$ for all $j$, and hence 
		\dalign{
			\er(\sV',\cc,\bmu) & = \sum_{i=1}^r c_i \dim(V_i'/V_{i-1}') \\
			& = -c_1  + c_r \dim(V_r') + \sum_{i=1}^{r-1} (c_i-c_{i+1})
			\dim(V_i') \\
			& \ge -c_1  + c_r \dim(V_r) + \sum_{i=1}^{r-1} (c_i-c_{i+1}) \dim(V_i) = \er(\sV,\cc,\bmu),
		}
		by \eqref{strict-entropy-cond}. Since $c_i-c_{j+1}>0$ for all $i\le r-1$, and $c_r>c_{r+1}\ge 0$, we must have that $V_i'=V_i$ for all $i$, which is precisely statement (d).
		
		(e) This statement is proven as in Lemma \ref{gammak-tilde:reduction}. }
\end{proof}

\rev{The bound $\beta_k \ge \tilde{\gamma}_k$ will now follow from the following proposition, as long as we can show that the quantity $\tilde{\gamma}_k$ is well-defined and positive. The latter will be accomplished in Section \ref{binary-system}, where we construct a system satisfying the strict entropy condition \ref{strict-entropy-cond}. An alternative construction is given in Appendix \ref{previous-app}.}

As usual, $\A$ is a logarithmic random set. 

\rev{
\begin{proposition}\label{bgamk}
Let $c>0$ and suppose that there is a system $(\sV,\cc,\bmu)$ such that:
\begin{enumerate}
	\item[(i)] $1=c_1 > c_2 > \cdots >c_{r+1}=c$;
	\item[(ii)] There is some $\eps>0$ such that $\er(\sV',\cc,\bmu) \ge \er(\sV,\cc,\bmu) + \eps$ for all proper subflags $\sV'$ of $\sV$. 
	\item[(iii)]  $\Supp(\mu_j) = V_j \cap \{0,1\}^k$ for $j=1,2,\ldots,r$.
\end{enumerate} 
Let $\delta>0$, and assume that $D$ is large enough in terms of $\delta,\eps$ and $(\sV,\cc,\bmu)$. Then the probability that $\A\cap[D^c,D]$ has $k$ distinct subsets with equal sums is $\ge1-\delta$.
\end{proposition}
}

The proof of Proposition \ref{bgamk} is perhaps the most
difficult part of this paper, and will occupy this and the next section. 
\rev{Throughout the remainder of this section and throughout the next section, we will fix a system $(\sV,\cc,\bmu)$ with $c_{r+1}=c$ satisfying conditions (i)--(iii) of Proposition \ref{bgamk}. Constants implied by $O-$ and $\ll-$symbols may depend on this system.}

The main result, which we will prove in this section and the next, is Proposition \ref{main-b-restate} below. 

\begin{definition}[Nondegenerate maps]\label{nondegenerate-def}
A map $\psi : X \rightarrow \{0,1\}^k$ is said to be \emph{nondegenerate} if the image of $\psi$ is not contained in
any of the subspaces $\{x\in \Q^k : x_i=x_j\}$.
\end{definition}

The map $\psi$ is a ``Venn diagram selection function'', that is, the value of $\psi(b)$ specifies
which piece of the Venn diagram of $k$ subsets $X_1,\ldots,X_k$
 of $X$ that $b$ belongs to.  In the notation \eqref{nu} of the previous section, $\psi(a)=\om$ means that $a\in B_\om$.
The condition that $\psi$ is nondegenerate is equivalent to 
$X_1,\ldots,X_k$ being  distinct, and is similar to the
property of a flag $\sV$ being nondegenerate.

\begin{proposition}\label{main-b-restate}
\rev{With probability tending to 1 as $D\to\infty$,}
  there exists a nondegenerate map $\psi : \A\cap (D^c,D] \to \{0,1\}^k$ such that  $\sum_{a \in \A} a \psi(a) \in \spanone$.
\end{proposition}

The map $\psi$ will be constructed using the data from the
system $(\sV,\cc,\bmu)$.
Before we embark on the proof of this result, we show \rev{how} to deduce Proposition \ref{bgamk} from it.

\begin{proof}[Proof of Proposition \ref{bgamk}, assuming Proposition \ref{main-b-restate}] 
By Proposition \ref{main-b-restate}, we know that with probability $1 - o_{D \rightarrow \infty}(1)$ there is \rev{a nondegenerate map} $\psi : \A \cap (D^c,D] \rightarrow \{0,1\}^k$ such that 
 $\sum_{a \in \A} a \psi(a)$ lies in $\spanone$, that is to say, it is a constant vector. We will show that this map induces $k$ distinct subsets of $\A$ 
 \rev{with equal sums.}

Let $\psi_i:\A \cap (D^c,D]\to\Q$, $i=1,\dots,k$, denote the projection of $\psi$ onto the $i$-th coordinate of $\Q^k$, so that $\psi=(\psi_1,\dots,\psi_k)$. 
Define $A_i := \{ a \in \A : \psi_i(a) = 1\}$.
These sets are distinct because if $A_i = A_j$, then the image
of $\psi$ would take values in the hyperplane $\{\rev{x} \in \Q^k : x_i = x_j\}$, contrary to the fact that $\psi$ is nondegenerate.
Moreover, for all $i,j$ we have
\[ 
\sum_{a \in A_i}a - \sum_{a \in A_j}a = \sum_{a \in \A} a  \psi_i(a) - \sum_{a \in \A} a \psi_j(a)  = 0,
\]
and \rev{so $A_1,\ldots,A_k$ do indeed have equal sums.}
\end{proof}

%
%\kcom{Is this remark needed?}
%\begin{remark}
%Our proof of Proposition \ref{main-b-restate} applies to
%a more general situation, where we allow each measure
%$\mu_i$ to have support in an arbitrary finite subset of
%$V_i$.
%  That is, the proof of Proposition \ref{main-b-restate} makes no use of the specific
%requirement that $\Supp(\mu_i) \subset \{0,1\}^k$.
%\end{remark}

\subsection{Many values of $\sum_{a \in \A} a \psi(a)$, and a moment bound}\label{42}

We turn now to the task of proving Proposition \ref{main-b-restate}.
We will divide the proof of Proposition \ref{main-b-restate} into two parts. The first and more difficult part, which we prove in this section, states that (with high probability) $\sum_{a \in \A} a \psi(a)$ takes many different values modulo $\spanone$ as $\psi$ ranges over all nondegenerate maps $\psi:\A \cap (D^c,D]\to \{0,1\}^k$. The precise statement is Proposition \ref{ellpS-repeat} below.  The deduction of Proposition 
\ref{main-b-restate} from Proposition \ref{ellpS-repeat}
will occupy Section \ref{mt-argument}.

\rev{Let $0<\kappa\le\min_{1\le j\le r} (c_j-c_{j+1})-2/\log D$} be a small quantity, which may depend on $D$.
Let
 \be\label{A-prime}
 \A^j = \{ a\in \A : D^{c_{j+1}+\kappa} < a\leq D^{c_j}/e \} \qquad  (1\le j\le r), \quad \A' := \bigcup_{j = 1}^r \A^j.
 \ee
The purpose of working with $\A'$ rather than $\A$ is to ensure that some gaps are left for the subsequent argument in the next section (based on ideas of Maier and Tenenbaum \cite{MT84}), in which we show that one of the many sums $\sum_{a \in \A'} a \psi(a)$ guaranteed by Proposition \ref{ellpS-repeat} may be modified, using the \rev{elements of $\A \cap (D^c,D] \setminus \A'$}, to 
be in $\spanone$. 

\begin{definition}[Compatible functions]\label{compatible} 
\rev{We say that} a map $\psi : \A' \rightarrow \{0,1\}^k$ is \emph{compatible} if, for all $j$, \rev{$a\in \A^j$ implies $\psi(a)\in V_j$.}
\end{definition}

\rev{\begin{remark*}
Recall that $\Supp(\mu_j)=V_j\cap\{0,1\}^k$ for all $j$ by condition (iii) of Proposition \ref{bgamk}. Setting $B_\om^{(j)}=\{a\in \A^j:\psi(a)=\om\}$, we see that 
$\psi$ being compatible is equivalent to $B_\om^{(j)} \ne \emptyset$ only if $\mu_j(\om)>0$, and is consistent with earlier notation \eqref{nu}.
\end{remark*}}

\begin{proposition}\label{ellpS-repeat}
There exist real numbers $\kappa^*>0$, $p>1$ and $t>0$
\rev{(which depend on the system $(\sV,\cc,\bmu)$)}
 so that the following is true.  Let $\delta > 0$
 and suppose that $D$ is sufficiently large as a function of $\delta$.
 Uniformly for $0\le \kappa \le \kappa^*$,
 we have with probability at least $1 - \delta$, 
that $\sum_{a \in \A'}a \psi(a)$ takes at least  
\[ (t\delta)^{\frac{1}{p-1}} D^{\sum_j c_j \dim (V_j/V_{j-1})}  
\] different values modulo $\spanone$,
as $\psi$ ranges over all nondegenerate, compatible maps $\psi$.
\end{proposition}

\begin{remark*}
By \eqref{A-prime}, it  clearly suffices to prove \rev{Proposition \ref{ellpS-repeat}} for  $\kappa=\kappa^*$.
\end{remark*}

We will deduce Proposition \ref{ellpS-repeat} from a moment bound. Firstly, define the representation function
$\rev{r_{\A'}}: \Q^k/\spanone \to \R$ by 
\[ \rev{r_{\A'}(x)} := \sum_{\substack{\psi : \A' \rightarrow \{0,1\}^k\\ \sum_{a \in \A'} a \psi(a)- x\in \spanone}} \rev{w_{\A'}(\psi)},\] 
where the summation is over all maps $\psi: \A'\to \{0,1\}^k$, and where
\[ 
\rev{w_{\A'}(\psi)} := \prod_{j=1}^r \prod_{a \in \A^j} \mu_j(\psi(a)).
\]
This weight function \rev{$w_{\A'}$} is chosen so that it is large only when
$\psi$ is \emph{balanced}, that is, when for all
$j$ and $\om$, the set $\A^j$ has about $\mu_j(\om)|\A_j|$ elements $a$
with $\psi(a)=\om$.   Observe that if $\psi(a)\not\in 
\Supp(\mu_j)$ for some $j$ and some $a\in \A^j$, then
\rev{$w_{\A'}(\psi)=0$}, and thus only compatible $\psi$ contribute to the sum $r_{\A}(x)$. 
\rev{However, $w_\A(\psi)$ might be non-zero for some degenerate maps $\psi$, and these 
will be removed by a separate argument below.}

The crucial moment bound for the deduction of Proposition \ref{ellpS-repeat} is given below. 

\begin{proposition}\label{ell-p-bound}
	\rev{Let 
			\[
			\cE^* = \Big\{ A\subseteq [D^c,D]: \# (A \cap (y/e,y]) \le \sqrt{y}/100 \quad (D^c\le y\le D) \Big\}.
			\]
There is a $p > 1$ and $\kappa^*>0$} so that uniformly for $0\le \kappa \le \kappa^*$ \rev{and for all $D\ge e^{100/c}$}
we have the moment bound
\[ 
\E \Big[\rev{\un_{\A'\in\cE^*}} \sum_x r_{\rev{\A'}}(x)^p \Big]
\rev{\ll  D^{-(p - 1) \sum_j c_j \dim(V_j/V_{j-1})}.}
\] 
\end{proposition}

\begin{proof}[Proof of Proposition \ref{ellpS-repeat}, assuming Proposition \ref{ell-p-bound}] 
Define also
\[
\rev{\tilde{r}}_{\rev{\A'}}(x) := \sum_{\substack{\psi : \A' \to  \{0,1\}^k
\\ \psi \text{ is compatible and nondegenerate} \\ \sum_{a \in \A'} a \psi(a) -x \in \spanone}} w_{\rev{\A'}}(\psi).
\]
We have
\[
\sum_{x} r_{\rev{\A'}}(x) = \prod_{j=1}^r \Bigg( \sum_{\om} \mu_j(\om) \Bigg)^{|\A^j|} =  \prod_{j=1}^r 1 = 1
\]
for any $\rev{\A'}$.  On the other hand, when $\psi$ is non-compatible, then $w_{\rev{\A'}}(\psi)=0$ because we know that $\Supp(\mu_j)=V_j\cap\{0,1\}^k$ for all $j$ by our assumption of \rev{condition (iii) of Proposition \ref{bgamk}. In addition, if $\psi$ is degenerate, then its image is contained in $\{x\in\Q^k:x_i=x_j\}\cap\{0,1\}^k$ for some $i\neq j$. Since $V_r\not\subset\{x\in\Q^k:x_i=x_j\}$, there must exist some $\om\in V_r\cap\{0,1\}^k=\Supp(\mu_r)$ that is not in the support of $\psi$.} Therefore,
\dalign{
\sum_{x} (\rev{r_{\A'}(x) - \tilde{r}_{\A'}(x)}) \leq  \rev{\sum_{\om \in \Supp(\mu_r)} (1-\mu_r(\om))^{|\A^r|}.}
}
\rev{Since $c_r>c_{r+1}$ by our assumption of condition (i) of Proposition \ref{bgamk}, 
Lemma \ref{normalA} implies} \rev{$|\A^r| \ge \frac{1}{2}(c_{r}-c_{r+1}) \log D$} with 
probability $> 1 - O(e^{-(1/4)\log^{1/2} D})$, 
and thus the right side above is $o(1)$ with this same 
probability.  \rev{The same lemma also implies that $\A'\in \cE^*$ with probability $> 1 - O(e^{-(1/4)\log^{1/2} D})$.}

Now fix a small $\delta>0$.  The above discussion implies that, with probability at least $1 - \delta/2$ (for $D$ sufficiently large), we have
\begin{equation}\label{sur-lower} 
\sum_{x} \rev{\tilde{r}_{\A'}}(x) \geq \frac{1}{2}\rev{\quad\text{and}\quad \A'\in\cE^*} .\end{equation}
\rev{On the other hand, Markov's inequality and Proposition \ref{ell-p-bound} imply that,} with probability at least $1 - \delta/2$, we have
\begin{equation}\label{ell-p-markov} 
	\rev{\un_{\A'\in\cE^*}}\sum_x \rev{\tilde{r}_{\A'}}(x)^p \leq
\rev{\un_{\A'\in\cE^*}}\sum_x r_{\rev{\A'}}(x)^p \ll  \delta^{-1} D^{-(p - 1) \sum_j c_j \dim (V_j/V_{j-1})}.
\end{equation}
By H\"older's inequality, 
\begin{equation}\label{holder}
	 \rev{\un_{\A\in\cE^*}}\sum_x \rev{\tilde{r}_{\A'}}(x) 
	 	\leq |\Supp(\rev{\tilde{r}_{\A'}})|^{1 - 1/p} \big(\rev{\un_{\A'\in\cE^*}} \sum_x \rev{\tilde{r}_{\A'}}(x)^p \big)^{1/p}.
 \end{equation}
With probability at least $1 - \delta$, both \eqref{sur-lower} and \eqref{ell-p-markov} hold, and in this case \eqref{holder} gives 
\[ |\Supp(\rev{\tilde{r}_{\A'}})| \gg_p \delta^{\frac{1}{p-1}} D^{\sum_j c_j \dim (V_j/V_{j-1})}.
\]
This completes the proof of Proposition \ref{ellpS-repeat}.
\end{proof}

The rest of the section is devoted to the proof of Proposition \ref{ell-p-bound}.

\subsection{An entropy condition for adapted systems}\label{convexity}

\rev{For reasons that} will become apparent, in the proof of Proposition \ref{ell-p-bound} we will need to apply the entropy gap
condition not only with subflags $\sV'$ of $\sV$, but with
a more general type of system.

\begin{definition}[Adapted system]\label{adapted}
Given a system $(\sV,\cc,\bmu)$, 
the pair $(\sW,\mathbf{b})$ is
 \emph{adapted} to $(\sV,\cc,\bmu)$ if
$\sW: \spanone = W_0 \le W_1 \le \cdots \le W_s$ is a complete
flag with $W_s\le V_r$, 
and $\bb=(b_1,\ldots,b_s)$ satisfies \rev{$1\ge b_1 \ge \cdots \ge b_s\ge 0$ and}
the condition
\[
W_i \leq  V_j
\qquad\text{whenever}\qquad 
b_i>c_{j+1} .
\]
We say that  $(\sW,\mathbf{b})$ is \emph{saturated} if  $s=\dim(V_r)-1$ and if for all $j\le r$, there are exactly $\dim V_j-1$ values of $i$ with $b_i>c_{j+1}$. Otherwise, we call $(\sW,\bb)$ \emph{unsaturated}.
\end{definition}

\begin{remark*} For the definition of complete flag, see Definition \ref{flag-pre-def}. We make a few comments to motivate the term {\it saturated}. Let 
\begin{equation}\label{eq:Wb mj}
m_j=\#\{i:b_i>c_{j+1}\}\qquad(0\le j\le r),
\end{equation}
so that the $b_i$'s belonging to the interval $(c_{j+1},c_j]$ are precisely $b_{m_{j-1}+1},\dots,b_{m_j}$. 
Since $W_i\le V_j$ whenever $b_i>c_{j+1}$, we infer that
\begin{equation}\label{eq:Wb-flag adapted}
W_{m_j}\le V_j  \qquad(1\le j\le r).
\end{equation}
Since $\sW$ is complete, we know that $\dim(W_i)=i+1$, and thus  $m_j\le \dim(V_j)-1$. In particular,  $(\sW,\bb)$ is saturated if, and only if, we have equality in \eqref{eq:Wb-flag adapted} for all $j$.
\qed \end{remark*}

We need some further notation, which reflects that
$\A'$ is supported on intervals with gaps.
 For $1\le j\le r$, let
\be\label{Ij}
\rev{I_j = (c_{j+1}+\kappa,c_j].}
\ee
\rev{Recall that we take $\kappa$ small enough so that each $I_j$ has length $\ge2/\log D$, 
that is, $\kappa \le \min_j (c_j-c_{j+1})-2/\log D$.}
%Also, if  $(\sW,\mathbf{b})$ is an adapted system for
%$(\sV,\cc,\bmu)$, we set
%\be\label{Gj}
%G_i = G_i(\bb) = (b_{i+1},b_i],
%\ee

There is a natural analogue of \rev{the $\er$-value (cf.~Definition \ref{e-value dfn})} for adapted systems.

\begin{definition}\label{adapted-sys}
Given an adapted system $(\sW, \mathbf{b})$, we define
\[ 
\er(\sW,\mathbf{b}) = \er(\sW,\mathbf{b};\sV,\cc,\bmu) := \sum_{i,j} \rev{\lambda([b_{i+1},b_i] \cap I_j)} \HH_{\mu_j}(W_i) + \sum_i b_i,
\]
where $\lambda$ denotes \rev{the} Lebesgue measure on $\R$.

\rev{Finally, we define} 
\be\label{deltab}
\delta(\bb) = \max_{i,j} \{ c_j-b_i : b_i \in I_j \},
\ee
that is to say $\delta(\bb)$ is the smallest non-negative real number with the property that
\[
c_j-\delta(\bb)\le b_i\le c_j\qquad(1\le j\le r,\ i\in I_j).
\]
\end{definition}

Adapted systems $(\sW,\bb)$ can, in a certain sense, be interpreted in terms of convex superpositions of pairs $(\sV',\cc)$, $\sV' \leq \sV$ a subflag.  The next lemma \rev{gives us a strict inequality analogous to condition (ii) of Proposition \ref{bgamk}}, unless $\sW$ is saturated and has a small value of $\delta(\bb)$, which corresponds to the convex superposition which gives rise to $(\sW,\mathbf{b})$ having weight $\approx 1$ on the trivial subflag $(\sV,\cc)$. 

\begin{lemma}\label{adapted-gap} \rev{Let $(\sV,\bmu,\cc)$ be a system satisfying conditions (i)--(ii) of Proposition \ref{bgamk}. Let $\eps$ be as in condition (ii).} Suppose that $(\sW, \mathbf{b})$ is an adapted system to $(\sV,\bmu,\cc)$ such that $b_i$ lies in some set $I_j$ for each $i$. Suppose, further, that  $\kappa$ is small enough in terms of $\eps$, \rev{ and that $\kappa \le \frac12 \min_j (c_j-c_{j+1}).$}
\begin{enumerate}
\item 
If $(\sW,\bb)$ is unsaturated, then $\er(\sW, \mathbf{b}) \geq \er(\sV,\cc,\bmu) + \eps/2$.
\item If $(\sW,\mathbf{b})$ is a saturated, then
$\er(\sW,\bb) \geq \er(\sV,\cc,\bmu) + \eps \delta(\bb)/2$.
\end{enumerate}
\end{lemma}

\begin{proof} We treat both parts together for most of the proof. Let $m_j$ be defined by \eqref{eq:Wb mj}. \rev{In particular, $m_0=0$ because $c_1=1$.} Note that $\max_{i\in I_j}(c_j-b_i)=c_j-b_{m_j}$, and \rev{let $h$ be such that 
\[
\delta(\bb)=c_{h}-b_{m_{h}} .
\]}
Without loss of generality, we may assume that $b_{m_{h}}<c_{h}$; the case $b_{m_{h}}=c_{h}$ will then follow by continuity. 

Set $b=b_{m_{h}}$ and note that
\[
\er(\sW,\bb)\ge \min \left\{ \er(\sW,\bb'): 
\begin{array}{l}
b'_i \in [c_{j+1}+\kappa,c_j]\ 
\mbox{when $i\in(m_{j-1},m_j]$ and $j\neq h$},\\ 
b_i'\in[b,c_{h}]\ \mbox{when $ i\in(m_{h-1},m_{h})$},\ b_{m_{h}}'=b,\\
b_1'\ge b_2'\ge\cdots\ge b_s'
\end{array}\right\}.
\]
The quantity $\er(\sW,\bb')$
is linear in each variable $b_i'$ and the region over which we consider the above minimum is a polytope. As a consequence, the minimum of $\er(\sW,\bb')$ must occur at one of the vertices of the polytope. In particular, there are indices $\ell_j\in(m_{j-1},m_j]$ for $j=1,\dots,r$ such that
\begin{equation}\label{eq:eWbmin}
\er(\sW,\bb)\ge \er(\sW,\bb^*),
\qquad\text{where}\qquad
b_i^*=\begin{cases}
c_j&\text{if}\ m_{j-1}<i\le \ell_j,\\
c_{j+1}+\kappa &\text{if}\ \ell_j<i\le m_j,\ j\neq h,\\
b&\text{if}\ \ell_{h}<i\le m_{h} .
\end{cases}
\end{equation}
In fact, note that we must have $\ell_{h} <m_{h}$ because $b_{m_{h}}^*=b$ and we have assumed that $b<c_{h}$.

Using the linearity of $\er(\sW,\cdot)$ once again, we find that
\be\label{Mbell}
\er(\sW,\bb^*) = \frac{c_{h}-b}{c_{h}-c_{h+1}-\kappa}
\er(\sW,\bb^{(1)}) + \frac{b-c_{h+1}\rev{-}\kappa}{c_{h}-c_{h+1}-\kappa }\er(\sW,\bb^{(2)}),
\ee
where $b_i^{(1)} = b_i^{(2)} = b_i^*$ for $i\in\{1,\dots,s\}\setminus (\ell_{h},m_{h}]$, $b_i^{(1)} = c_{h+1}+\kappa$ for $i\in(\ell_{h},m_{h}]$ and $b_i^{(2)} = c_{h}$ for $i\in(\ell_{h},m_{h}]$.

Fix $\bb'\in\{ \bb^{(1)}, \bb^{(2)} \}$. In addition, define the indices $i_1,\dots,i_r$ by letting $i_j=\ell_j$ when $j\neq h$ or $\bb'=\bb^{(1)}$, while letting $i_{h}=m_{h}$ when $\bb'=\bb^{(2)}$. We then have
\[
b_i'=
	\begin{cases}
		c_j&\text{if}\ m_{j-1}<i\le i_j,\\
		c_{j+1}+\kappa&\text{if}\ i_j<i\le m_j.
	\end{cases}
\]
A straightforward calculation implies that
\begin{align}
\er(\sW,\bb')
\rev{=\er(\sV',\cc,\bmu)+ S\kappa + (m_r-i_r)c_{r+1}} ,
\label{eq:from Wb to V'}
\end{align}
where $\sV'$ is the subflag of $\sV$ with $V_j'=W_{i_j}$ and 
\[
S=\sum_{j=1}^r \big(m_j-i_j-\HH_{\mu_j}(W_{i_j})\big).
\]
(Note that $\sV'$ is indeed a subflag since $W_{i_j}\le W_{m_j}\le V_j$ by \eqref{eq:Wb-flag adapted}.)

If $\sV'=\sV$, we must have that $W_{i_j}=V_j$ for all $j$. Since $W_{i_j}\le W_{m_j}\le V_j$, we infer that $W_{m_j}=V_j$, as well as that $i_j=m_j$ for all $j$. In particular, the flag $(\sW,\bb)$ we started with must be saturated and $S=0$ (since $i_j=m_j$ and $\HH_{\mu_j}(W_{i_j})=\HH_{\mu_j}(V_j)=0$ for all $j$). 

We are now ready to complete the proof of both parts of the lemma.

\medskip

(a) By the above discussion, if $(\sW,\bb)$ is unsaturated, then $\sV'\neq\sV$. Therefore, \rev{by assumption of condition (ii) of Proposition \ref{bgamk}, we have $\er(\sW,\bb')\ge\er(\sV,\cc,\bmu)+\eps$} for $\bb'\in\{\bb^{(1)},\bb^{(2)}\}$. Inserting this inequality into \eqref{Mbell} implies that\rev{ $\er(\sW,\bb^*)\ge \er(\sV,\cc,\bmu)+\eps+O(\kappa)$.} Since $\er(\sW,\bb)\ge \er(\sW,\bb^*)$, the proof of part (a) is complete by assuming that $\kappa$ is small enough in terms of $\eps$.

\medskip

(b) Assume that $(\sW,\bb)$ is saturated. We can only have that $\sV'=\sV$ if $i_{h}=m_{h}$. Since $\ell_{h}<m_{h}$, this can only happen when $\bb'=\bb^{(2)}$. As a consequence, assuming again that $\kappa$ is small enough in terms of $\eps$, we have that
\[
\rev{\er(\sW,\bb')\ge 
	 \er(\sV,\cc,\bmu)+1_{\bb'=\bb^{(1)}} \cdot \eps/2.}
\]
Inserting this into \eqref{Mbell} yields the inequality
\[
\rev{\er(\sW,\bb^*) 
\ge \er(\sV,\cc,\bmu)}
+\frac{c_{h}-b}{c_{h}-c_{h+1}-\kappa}
\cdot \frac{\eps}{2}.
\]
Since $b=c_{h}-\delta(\bb)$, \rev{ $0<c_{h}-c_{h+1}-\kappa\le1$,} and $\er(\sW,\bb)\ge \er(\sW,\bb^*)$, we find that \rev{$\er(\sW,\bb)\ge\er(\sV,\cc,\bmu)+\eps\delta(\bb)/2$.} This completes the proof of part (b) of the lemma. 
\end{proof}

\subsection{Proof of the moment bound}\label{ell-p-pr}
In this subsection we prove Proposition \ref{ell-p-bound}.
For a vector $\mathbf{n}= (n_0,n_1,n_2,\ldots,n_r)$
with 
\[
0=n_0 \le n_1 \le \cdots \le n_r,
\]
define the event 
 \[
 S(\mathbf{n})=\{\A' : \rev{\# \A^j = n_{j}-n_{j-1}\quad(1\le j\le r)}\}.
 \]
When $\A'$ lies in $S(\mathbf{n})$, we write
\[
\A' = \{a_1, a_2, \ldots, a_{n_r} \}, 
	\qquad a_1 > a_2 > \ldots > a_{n_r},
\]
so that
\begin{equation}\label{eq:Aj-nj}
a_t \in \A^j
\qquad\mbox{if and only if}\qquad n_{j-1}< t \le n_j.
\end{equation}
We may define, for any compatible $\psi$, the auxilliary
function 
\begin{equation}\label{eq:theta-psi}
\rev{\theta: [n_r] \to V_r \cap \{0,1\}^k}
\qquad\text{such that}\qquad 
\theta(t)=\psi(a_t).
\end{equation}
The salient property of $\theta$ is that it is determined by 
the ordering of the elements in $\A^j$ and not by the elements themselves.
We denote by $\boldsymbol{\Theta}_{\bn}$ the set of compatible functions $\theta$, that is, those functions satisfying
\begin{equation}
\theta(t)\in \Supp(\mu_j) \qquad\text{whenever}\qquad t\le n_j,\quad 1\le j\le r\ . \label{theta-adm}
\end{equation}

In the event $S(\mathbf{n})$, if $\psi$ is an compatible function and $\theta$ is defined by \eqref{eq:theta-psi}, we have 
\begin{equation}\label{w-def-abuse} 
w_{\rev{\A'}}(\psi) = w_{\mathbf{n}}(\theta) := \prod_{j=1}^r \prod_{n_{j-1}<t \le n_j} \mu_j(\theta(t)),
\end{equation}
where the notation $w_{\mathbf{n}}$ (in place of $w_{\A}$) reflects the fact that $w$ only depends on $\theta$, and not otherwise on $\A$. In this notation,
\[
r_{\rev{\A'}}(x) = \ssum{\theta\in\boldsymbol{\Theta}_{\bn} \\
 \sum_t \theta(t) a_t -x \in \spanone}
w_{\mathbf{n}}(\theta).
\]

Writing $r_{\rev{\A'}}^p = r_{\rev{\A'}}^{p-1} r_{\rev{\A'}}$ and interchanging the order of summation, it follows that if \rev{$\rev{\A'}$ lies in $S(\mathbf{n})$}, then
\begin{align} \nonumber 
\sum_x r_{\rev{\A'}}(x)^p 
& = \sum_{\theta \in\boldsymbol{\Theta}_{\bn}}
\bigg( r_{\rev{\A'}}\Big(\sum_t a_t \theta(t)\Big)\bigg)^{p-1} w_{\mathbf{n}}(\theta) \\ 
& =  \sum_{\theta \in \boldsymbol{\Theta}_{\bn}} 
\bigg( \ssum{\theta' \in \boldsymbol{\Theta}_{\bn} \\ \eqref{thetaprime}} w_{\mathbf{n}}(\theta') \bigg)^{p-1} w_{\mathbf{n}}(\theta), 
\label{decouple} 
\end{align}
where the inner summation is over all compatible
functions $\theta'$ satisfying
\be
\label{thetaprime}
 \sum_t a_t 
 (\theta'(t) - \theta(t)) \in \spanone.
\ee

Similar to the argument in subsection \ref{upp:linalg}, we find a
flag $\sW$ and special values of $i$ which have the
effect of isolating terms in the relation
\eqref{thetaprime}.
With $\theta, \theta',\mathbf{n}$ fixed, let 
\[
\Omega = \Omega(\theta, \theta') = \{ \theta'(t)-\theta(t) : 1\le t\le n_r \} 
\]
and 
\[
\rev{s = \dim\big( \Span (\one, \Omega) \big)- 1.}
\]
We now choose a special basis of $\Span(\one, \Omega)$. For each $\om\in\Omega$, let
\[
K_\om = \min \{t : \theta'(t)-\theta(t)=\om \},
\]
and place a total ordering
on $\Omega$
by saying that $\om \prec \om'$ if $K_{\om} < K_{\om'}$.
Let $\om^1$ be the minimum element in  $\Omega\setminus \spanone$,
$\om^2 = \min (\Omega \setminus \Span(\one,\om^1)),\dots, \om^s = \min (\Omega \setminus \Span(\one,\om^1,\ldots,\om^{s-1}))$, where $s$ is such that $\Omega\subset \Span(\one,\om^1,\ldots,\om^s)$.  Finally, let
\[
W_j = \Span(\one,\om^1,\ldots,\om^j), \qquad
\tau_j = K_{\om^j} \qquad (1\le j\le s),
\]
\[
\btau(\theta,\theta',\bn)=(\tau_1,\ldots,\tau_s),
\]
and form the flag 
\[
\sW=\sW(\theta,\theta',\bn) \ : \ W_0 \le W_1 \le \cdots \le W_s.
\] 
We note that in the special case $\theta=\theta'$, we have $s=0$ and $\sW$ is a 
trivial flag with only one space $W_0$. 

Now we divide up the sample space of $\A'$ into events describing the rough size of the critical elements $a_{\tau_j}$. 
By construction, 
\[
a_{\tau_j} = \max \{a_t\in \A' : \theta'(t)-\theta(t) = \om^j \}.
\]
Similarly to Section \ref{sec:upper}, for $1\le i\le s$ let
\begin{equation}\label{eq:bi-dfn lb}
b_i 
= \rev{1+\frac{\cl{\log a_{\tau_i}-\log D}}{\log D} }
\qquad \text{so that}\quad
a_{\tau_i}\in(D^{b_i}/e,D^{b_i}] .
\end{equation}

The definition of $\A'$ implies that for each $i$, there is some $j$ with $b_i\in I_j=(c_{j+1}+\kappa,c_j]$. 
Moreover, we have the implications
\[
b_i>c_{j+1}
	\quad\implies\quad 
\tau_i \le n_j
\quad\implies\quad 
\om^i = \theta(\tau_i) -\theta'(\tau_i) \in V_j,
\]
where we used \eqref{theta-adm} to obtain the second implication. Since $b_1\ge b_2\ge\cdots \ge b_i$, we infer the stronger relation
\begin{equation}\label{eq:W-V inclusion}
b_i>c_{j+1}\qquad\implies\qquad 
W_i \le V_j.
\end{equation}
Therefore, the pair $(\sW,\bb)$ is adapted to $(\sV,\cc,\bmu)$.

Using the inequality $(x+ y)^{p-1} \leq  x^{p-1} + y^{p-1}$ repeatedly, we may partition \eqref{decouple} according to the values of $\sW(\theta,\theta')$ and $\btau(\theta,\theta')$, obtaining (still assuming $S(\mathbf{n})$)
\[ \sum_x r_{\rev{\A'}}(x)^p 
 \leq  \sum_{\sW, \btau,\theta} \bigg( \sum_{\substack{\theta' \in \boldsymbol{\Theta}_{\bn} ,\  \eqref{thetaprime} \\ \sW(\theta,\theta',\bn) = \sW, \ \btau(\theta,\theta',\bn) = \btau } } w_{\mathbf{n}}(\theta')\bigg)^{p-1} w_{\mathbf{n}}(\theta).
 \]
 
We need to separately consider other elements of $\A'$ 
that lie in the intervals $(D^{b_i}/e,D^{b_i}]$, and so we
define 
\[
\cB = \{b_i : 1\le i\le s\}\qquad\text{and}\qquad
\bs\ell=(\ell_b)_{b\in \cB}, 
\qquad \text{where}\qquad 
\ell_b = \#\big(\A' \cap (D^b/e,D^b]\big) .
\]
By assumption, $\sum_b \ell_b \ge s$. It may happen that
$b_i=b_{i+1}$ for some $i$, in which case $|\cB| < s$.
With $\bn, \btau, \bb, \bs\ell$ all fixed,
consider the event 
\[
E(\bb,\btau,\bn,\bs\ell)
\]
defined as the intersection of 
\begin{itemize}
\item $S(\bn)$;
\item $\rev{a_{\tau_i}} \in (D^{b_i}/e,D^{b_i}]$ for all $i$;
\item $|\A' \cap (D^b/e,D^b]| = \ell_b$ for all $b\in \cB$.
\end{itemize}
%Under $S(\bn)$, the event 
%$E(\bb,\btau,\bn, \bs\ell)$ must occur
%for some $\bb,\btau, \bs\ell$.
%The event $E(\bb,\btau,\bn, \bs\ell)$  fixes the \emph{number}
%of elements of $\A$ in various sets, but otherwise does not
%specify any of the elements themselves. 

Taking expectations over $\rev{\A'}$, we get
\begin{align*} 
\E &\Big[\un_{\rev{\A'\in S(\mathbf{n})\cap\cE^*}}\sum_x  r_{\rev{\A'}}(x)^p \Big] \\ 
	& \leq  \E\bigg[\sum_{\substack{\sW,\btau,\bb,\theta, \bs\ell \\ \rev{\ell_b\le D^{b/2}/100\ \forall b\in\cB}}} w_{\mathbf{n}}(\theta)\bigg( \sum_{\substack{\theta'\in \boldsymbol{\Theta}_{\bn},\  \eqref{thetaprime} \\ \sW(\theta,\theta',\bn) = \sW,\ \btau(\theta,\theta',\bn) = \btau } } 
	w_{\mathbf{n}}(\theta')\bigg)^{p-1} 
	 \un_{E(\bb,\btau, \bn, \bs\ell)}\bigg],
\end{align*}
\rev{where the condition that $\ell_b\le D^{b/2}/100$ comes from the fact that we taking expectations over $\A'\in\cE^*$.} By H\"older's inequality with exponents $\frac{1}{p-1}$, $\frac{1}{2-p}$, this implies that 
\begin{align}
\nonumber \E\Big[\un_{\rev{\A'\in S(\mathbf{n})\cap\cE^*}}\sum_x  r_{\rev{\A'}}(x)^p \Big]
	 &\leq  \sum_{\substack{\sW,\btau,\bb,\theta,\bs\ell \\  \rev{\ell_b\le D^{b/2}/100\ \forall b\in\cB}}} w_{\mathbf{n}}(\theta)\P( E(\bb,\btau,\mathbf{n}, \bs\ell))^{2 - p} \times \\
	&\qquad\quad \times \bigg\{\sum_{\substack{\theta'\in \boldsymbol{\Theta}_{\bn}  \\ \sW(\theta,\theta',\bn) = \sW \\ \btau(\theta,\theta',\bn) = \btau}} w_{\mathbf{n}}(\theta')\P \big[ E(\bb,\btau,\bn, \bs\ell)
 \wedge \eqref{thetaprime} \big] \bigg\}^{p-1}.\label{eq1}
\end{align}

\noindent\emph{Claim.} \rev{Let $\ell_b\le D^{b/2}/100$ for all $b\in\cB$. Then we have}
\begin{equation}\label{cond-prop-E}
\P\big( \eqref{thetaprime}\ \big| \ E(\bb,\btau,\bn, \bs\ell)  \big) 
\ll D^{-(b_1+\cdots+b_s)}e^{\sum_b\ell_b}.
\end{equation}

\begin{proof}[Proof of Claim] \rev{Let us begin by analyzing the event $E(\bb,\btau,\bn, \bs\ell)$ we are conditioning on. Consider the set $\bigcup_j(D^{c_{j+1}+\kappa},D^{c_j}]\setminus \bigcup_{b\in\cB} (D^b/e,D^b]$. There is a unique way to write it as $\bigcup_{m=1}^M I_m$, where the sets $I_m$ are intervals of the form $(A,B]$ with their closures $\bar{I}_m$ mutually disjoint. Now, the event  $E(\bb,\btau,\bn, \bs\ell)$ is equivalent to there being mutually disjoint sets of consecutive integers $\cI_m$ ($1\le m\le M$) and $\cJ_b$ ($b\in\cB$) such that:
		\begin{itemize}
			\item The sets $\cI_m$ $(1\le m\le M)$ and $\cJ_b$ $(b\in\cB)$ together form a partition of the set $[n_r]$;
			\item For all $m\in\{1,\dots,M\}$, we have $a_n\in I_m$ if and only if $n\in \cI_m$;
			\item For all $b\in\cB$, we have $a_n\in (D^b/e,D^b]$ if and only if $n\in \cJ_b$;
			\item $\tau_i\in\cJ_{b_i}$ for all $i$;
			\item $|\cJ_b|=\ell_b$ for all $b\in\cB$.
		\end{itemize}
The above discussion allows us to describe the distribution law of $\A'$ under the event $E(\bb,\btau,\bn, \bs\ell)$: given a choice of the intervals $\cI_m$ and $\cJ_b$, we construct independent logarithmic random sets $\A^*_m$ on $I_m$ and $\tilde{A}_b$ on $(D^b/e,D^b]$ such that $\#\A'\cap I_m=\#\cI_m$ for all $m$ and $\#\tilde{A}_b=\ell_b$ for all $b$. Then $\A'$ is the union of all $\A^*_m$'s and all $\tilde{\A}_b$'s.
}

\rev{Having explained how the distribution of $\A'$ looks like under the event $E(\bb,\btau,\bn,\bs\ell)$, let us now prove our claim. We argue as in the proof of Proposition
 \eqref{prop:ub}.} Relation \eqref{thetaprime} implies
\[
\sum_{i=1}^s \om^i a_{\tau_i} + \sum_{t\not\in \{\tau_1,\ldots,\tau_s\}} a_t
(\theta'(t)-\theta(t)) = a_0 \one
\]
for some $a_0\in \Z$. Since $\one, \om^1,\ldots,\om^s$ are linearly independent,
this uniquely determines their coefficients 
 $a_0, a_{\tau_1},\dots, a_{\tau_s}$ in terms of the other $a_i$'s.
For each $b\in\cB$, let
\[
m_b = \# \{i: b_i = b\}
\quad\text{and}\quad 
N_b = \#\big(\Z\cap (D^{b}/e,D^{b}]\big) = (1-1/e)D^b + O(1).
\] 
Then, \rev{given $\A_m^*$ for all $m$} and $b\in\cB$, there are at most
\[
\binom{N_b}{\ell_b-m_b} \le \frac{N_b^{\ell_b-m_b}}{(\ell_b-m_b)!}
\ll \ell_b^{m_b}\cdot \frac{((1-1/e)D)^{b(\ell_b-m_b)}}{\ell_b!} 
\ll \frac{D^{b(\ell_b-m_b)}}{\ell_b!} 
\]
choices for $\tilde{\A}_b$ \rev{(since $m_b$ of each elements are determined by the remaining $\ell_b-m_b$ elements and by the elements of the $\A_m^*$ that we have fixed)}, where we used that $\ell_b^{m_b}\le \ell_b^k\ll (1-1/e)^{-\ell_b}$. In addition, Lemma \ref{A-cond} implies that the probability of occurrence of a \rev{given} set $X_b\subset\Z\cap \rev{(D^b/e,D^b]}$ as the set $\tilde{\A}_b$, conditionally to the event that $\#\tilde{\A}_b=\ell_b$, is
\[
\ll \frac{\ell_b!}{(\sum_{D^b/e<m\le D^b}1/(m-1))^{\ell_b}} 
	\prod_{x\in X_b}\frac{1}{x}\prod_{D^b/e<m\le D^b}\bigg(1-\frac{1}{m}\bigg)
	\ll \frac{\ell_b!}{(D^b/e)^{\ell_b}}.
\]
Putting the above estimates together, we conclude that
\dalign{
\P\big( \eqref{thetaprime}\ \big| \ E(\bb,\btau,\bn, \bs\ell)  \big) 
 &\ll\prod_{b\in \cB}\frac{e^{\ell_b}}{D^{bm_b}}
 =D^{-(b_1+\cdots+b_s)}e^{\sum_b\ell_b},
}
upon noticing that $\sum_{b\in \cB} m_b b = \sum_i b_i$. This proves our claim that \eqref{cond-prop-E} holds.
\end{proof}

In the light of \eqref{cond-prop-E}, relation \eqref{eq1} becomes
\begin{align}\nonumber 
\E & \Big[\un_{\rev{\A'\in S(\mathbf{n})\cap\cE^*}}\sum_x  r_{\rev{\A'}}(x)^p \Big]\\ 
& \ll  \sum_{\sW,\btau,\bb,\bs\ell} 
D^{-(p-1) \sum_j b_j}e^{\sum_b\ell_b}
\E\bigg[
\sum_{\theta\in \boldsymbol{\Theta}_{\bn} } w_{\mathbf{n}}(\theta)\bigg(\sum_{\substack{\theta'\in \boldsymbol{\Theta}_{\bn}  \\ \sW(\theta,\theta',\bn) = \sW \\ \btau(\theta,\theta',\bn) = \btau}} w_{\mathbf{n}}(\theta')\bigg)^{p-1}
\un_{E(\bb,\boldsymbol{\tau},\mathbf{n},\bs\ell)} \bigg].\label{eq2}
\end{align}

To evaluate the bracketed expression, first recall the definition \eqref{w-def-abuse} of $w_{\mathbf{n}}(\theta')$, and note that the conditions $\sW(\theta,\theta',\bn) = \sW$, $\btau(\theta,\theta',\bn) = \btau$ together imply that 
\[
\theta'(t) - \theta(t) \in W_i
\qquad(1\le t < \tau_{i+1},\ 0\le i\le s),
\]
where we have defined \rev{$\tau_0:=0$ and} $\tau_{s+1} := n_r+1$.
For brevity, write
\[
T_{i,j} = (n_{j-1},n_j] \cap [\tau_i,\tau_{i+1}) \cap \N, \qquad
(\rev{0}\le i\le s,\ 1\le j\le r).
\]
Some of these sets are empty. In any case, we have
\begin{equation}\label{weight-bd} \sum_{\substack{\theta'\in \boldsymbol{\Theta}_{\bn}  \\ \sW(\theta,\theta',\bn) = \sW \\ \btau(\theta,\theta',\bn) = \btau}} w_{\mathbf{n}}(\theta') \leq   \sprod{0\le i\le s \\ 1\le j\le r} \;\;\; \prod_{t \in T_{i,j}} \mu_j (\theta(t) + W_i).
\end{equation}

From \eqref{w-def-abuse}, and the fact that the discrete intervals $T_{i,j}$ are disjoint and cover $[n_r]$, we have 
\[ 
w_{\mathbf{n}}(\theta) = \prod_{i,j} 
\prod_{t \in T_{i,j}} \mu_{j}(\theta(t)).
\]
With these observations, we conclude that
\begin{align}\label{eq3} \nonumber\sum_{\theta\in \boldsymbol{\Theta}_{\bn} } w_{\mathbf{n}}(\theta)\bigg( \sum_{\substack{\theta'\in \boldsymbol{\Theta}_{\bn} \\ \sW(\theta,\theta',\bn) = \sW \\ \btau(\theta,\theta',\bn) = \btau}} w_{\mathbf{n}}(\theta')\bigg)^{p-1} 
&  \leq   \sum_{\theta\in \boldsymbol{\Theta}_{\bn} } \prod_{i,j} \prod_{t \in T_{i,j}} \mu_{j} (\theta(t)) \mu_{j}(W_{i} + \theta(t))^{p-1} \\
 & = \prod_{i,j} \eta(i,j,p,\sW)^{|T_{i,j}|},
 \end{align}
where
\begin{equation}\label{cross-cor} \rev{\eta(i,j,p,\sW) :=  \sum_{\om \in \Supp(\mu_j)} \mu_j(\om) \mu_j(W_i + \om)^{p-1}.} 
\end{equation}
Substituting into \eqref{eq2}, and summing over $\mathbf{n}$, 
we get
\begin{equation}\label{eq55} 
\E \Big[\rev{\un_{\A'\in\cE^*}}\sum_{x} r_{\rev{\A'}}(x)^p \Big]
\ll \sum_{\sW, \bb} 
D^{-(p-1) \sum_j b_j} 
 \sum_{\btau,\bn,\bs\ell} e^{\sum_b\ell_b}
 \E \bigg[
\un_{E(\bb,\btau,\bn,\bs\ell)}
\prod_{i,j} \eta(i,j,p,\sW)^{|T_{i,j}|}
\bigg].
\end{equation}

If $V_j \le W_i$, then $\mu_j(W_i+\om)=1$ for all $\om$ and
thus $\eta(i,j,p,\sW)=1$.  For all $i,j,p,\sW$
we have $\eta(i,j,p,\sW)\le 1$. 
Thus, we require lower bounds on $|T_{i,j}|$
in the case $V_j \not\le W_i$.

\medskip

\noindent\emph{Claim.} Assume that $E(\bb,\btau,\bn,\bs\ell)$ holds. Given $i$ such that $b_{i+1}<b_i$ and $j\in\{1,\dots,r\}$, define 
\[
M_{i,j} 
:= (D^{c_{j+1}+\kappa},\rev{D^{c_j}/e}] \cap
(D^{b_{i+1}}, D^{b_i}/e]
\] 
Then
\begin{equation}
\label{Mij}
\{t: a_t \in M_{i,j}\}\subset T_{i,j} .
\end{equation}

\begin{proof}[Proof of Claim]
Let $t$ be such that $a_t\in M_{i,j}$. In particular, $D^{b_{i+1}}<\rev{a_t}\le D^{b_i}/e$. This relation and the definition of $b_i$ in \eqref{eq:bi-dfn lb} imply that $a_{\tau_{i+1}}<a_t<a_{\tau_i}$ and hence $\tau_i<t<\tau_{i+1}$, where we used that $a_1>a_2>\cdots>a_{n_r}$. In addition, since $D^{c_{j+1}+\kappa}<a_t\le D^{c_j}$, we have that $a_t\in \A^j$. Thus, $n_{j-1}<t\le n_j$ by \eqref{eq:Aj-nj}. This completes the proof of the claim.
\end{proof}

A direct consequence of \eqref{Mij} is that
\[
|T_{i,j}| \ge \big| \rev{\A'} \cap M_{i,j} \big|.
\]
Combining this inequality with \eqref{eq55}, we get
\begin{align*} 
\E \Big[\rev{\un_{\A'\in\cE^*}}\sum_x r_{\rev{\A'}}(x)^p\Big] 
\ll  \sum_{\sW,  \bb} 
	D^{-(p-1) \sum_j b_j}  
	\sum_{\bn,\btau,\bs\ell}  e^{\sum_b\ell_b} 
	\E \bigg[
	\un_{E(\bb,\boldsymbol{\tau},\mathbf{n},\bs\ell)}
	\prod_{i,j} \eta(i,j,p,\sW)^{|\A \cap M_{i,j}|}\bigg].
\end{align*}

\rev{Fix $\bb$ and $\sW$, and let $E'(\bb,\bs\ell)$ be the event that $|\A'\cap(D^b/e,D^b]|=\ell_b$ for all $b\in\cB$. Given $\A'\in E'(\bb,\bs\ell)$, we have at most $\prod_b \ell_b\le e^{\sum_b\ell_b}$ choices for $\tau_1,\dots,\tau_s$.  Hence,}
\dalign{
\sum_{\bn,\btau,\bs\ell} 
&e^{\sum_b\ell_b}\E\bigg[
\un_{E(\bb,\boldsymbol{\tau},\mathbf{n},\bs\ell)} \eta(i,j,p,\sW)^{|\rev{\A'} \cap M_{i,j}|} \bigg] \\
&\le \sum_{\bn,\bs\ell}e^{2\sum_b\ell_b} \E \bigg[
 \un_{S(\bn)}\un_{E'(\rev{\bb,}\bs\ell)} 
  \prod_{i,j} \eta(i,j,p,\sW)^{|\rev{\A'} \cap M_{i,j}|}\bigg].
}
Since the events $S(\bn)$ are mutually disjoint, we arrive at the inequality
\begin{equation}
\label{eq57}
 \E\Big[\rev{\un_{\A'\in\cE^*}}  \sum_x r_{\rev{\A'}}(x)^p \Big]
 \le  \sum_{\sW,\bb} D^{-(p-1)\sum_j b_j} \, \E\Big[\prod_{b\in\cB}e^{2|\tilde{\A}_b|} \prod_{i,j}\eta(i,j,p,\sW)^{|\rev{\A'} \cap M_{i,j}|} \Big].
\end{equation}
 
Next, we estimate the right hand side of \eqref{eq57}. The intervals $M_{i,j}$ and $(D^b/e,D^b]$ are
mutually  disjoint by \eqref{Mij},  hence the quantities
$|\rev{\A'} \cap M_{i,j}|$ and $|\tilde{\A}_b|$ are independent.  Using Lemma \ref{A-moments}, we obtain
\begin{align*}
\E &\Big[\prod_{b\in\cB}e^{2|\tilde{\A}_b|}
\prod_{i,j}\eta(i,j,p,\sW)^{|\rev{\A'} \cap M_{i,j}|}\Big] \\
&\le
\exp\Big\{ \sum_{b\in\cB}\sum_{D^b/e<m\le D^b}\frac{2e-1}{m}+\sum_{i,j} \big(\eta(i,j,p,\sW)-1\big)\sum_{m\in M_{i,j}} \frac{1}{m} \Big\} \\
&\ll \exp\Big\{ \sum_{i,j} \big(\eta(i,j,p,\sW)-1\big)\sum_{m\in M_{i,j}} \frac{1}{m} \Big\}.
\end{align*}

Recall that \rev{$I_j=(c_{j+1}+\kappa,c_j]$, define
\[
G_i=G_i(\bb) = (b_{i+1},b_i],
\] 
 and recall that} $\lambda$ denotes the Lebesgue measure on $\R$. Then, by the definition of $M_{i,j}$, we have
\[
\sum_{m\in M_{i,j}} \frac{1}{m}
 = \lambda(I_j \cap G_i)\log D 
+ O(1).
\]
Substituting into \rev{the definition of $\er()$ (Definition \ref{adapted-sys}),} this gives
\begin{equation}\label{eq59}
 \E \Big[\rev{\un_{\A'\in\cE^*}} \sum_x r_{\rev{\A'}}(x)^p \Big]
  \ll \sum_{\sW, \bb} D^{-E(p,\sW,\bb)},
\end{equation}
where
\begin{align*}
E(p,\sW,\bb)
&:= 
(p-1) \sum_j b_j - \mathop{\sum\sum}_{i,j} \big(\eta(i,j,p,\sW) - 1\big) \lambda(I_j \cap G_i)\\ 
& = (p-1) \er(\sW,\bb) - \mathop{\sum\sum}_{i,j} \big[\eta(i,j,p,\sW) - 1+(p-1)\HH_{\mu_j}(W_i)\big] 
\lambda(I_j \cap G_i)
.\end{align*}

Recall the definition \eqref{cross-cor} of $\eta(i,j,p,\sW)$. If $W_i \ge V_j$, then $\mu_j(W_i + x) = 1$ whenever $x \in \Supp(\mu_j)$, and so in this case $\eta(i,j,p,\sW) = 1.$  Since $\HH_{\mu_j}(W_i)=0$ in this case,
we have
\be\label{trivial}
\eta(i,j,p,\sW) - 1 + (p-1)\HH_{\mu_j}(W_i) = 0
\qquad (V_j \le W_i).
\ee
For any fixed $i,j,\sW$, we have
\[
\frac{\textrm{d}}{\textrm{d}p}
\eta(i,j,p,\sW)\Big|_{p=1} = -\HH_{\mu_j}(W_i),
\]
and so
\begin{equation}\label{eq60} 
\eta(i,j,p,\sW) - 1 + (p-1) \HH_{\mu_j}(W_l) \ll (p-1)^2 \qquad (V_j \not\le W_i).
\end{equation}
We deduce from \eqref{eq59}, \eqref{trivial} and \eqref{eq60}
that
\begin{equation}\label{sumij}
E(p,\sW,\bb) = (p-1)\er(\sW,\bb) 
-\mathop{\sum\sum}_{\substack{i,j :\ V_j \not\le W_i}} 
\lambda(I_j \cap G_i)O((p-1)^2).
\end{equation}
To continue, we separate two cases.

\medskip

\noindent {\it Case 1.} $(\sW,\bb)$ is unsaturated. 

\smallskip In the above case, Lemma \ref{adapted-gap}(a) implies that \rev{$\er(\sW,\bb) \geq \er(\mathscr{V,\cc,\bmu}) + \eps/2$.}  Consequently, 
\rev{\begin{align*}
E(p,\sW,\bb) &\ge (p-1)\er(\sV,\cc,\bmu)+\frac{(p-1)\eps}{2}+O((p-1)^2)
\\ &\ge (p-1)\er(\sV,\cc,\bmu)+\frac{(p-1)\eps}{4},
\end{align*}}
provided that $p-1$ is small enough in terms of $\eps$ (and $k$). 

Since there are $O(1)$ choices for $\sW$ and $\log^{O(1)} D$ choices for $\bb$, the contribution of such flags to the right hand side of \eqref{eq59} is 
\begin{equation}\label{eq:unsat-flag}
 \sum_{(\sW, \bb)\ \text{unsaturated}} D^{-E(p,\sW,\bb)}\ll \rev{D^{-(p-1)\er(\sV,\cc,\bmu)}.}
\end{equation}

\medskip

\noindent {\it Case 2.} $(\sW,\bb)$ is saturated. (Recall from Definition \ref{adapted} that $(\sW,\bb)$ is called saturated when $s=\dim(V_r)-1$ and for all $j\le r$, there are exactly $\dim V_j-1$ values of $i$ with $b_i>c_{j+1}$.)

\smallskip

Fix for the moment a pair $(i,j)$ such that 
\begin{equation}\label{eq:sat-flat i,j}
V_j\not\le W_i 
\qquad\text{and}\qquad
\lambda(I_j \cap G_i)>0 .
\end{equation}
The second condition is equivalent to knowing that
\[
b_i>c_{j+1}
\qquad\text{and}\qquad 
b_{i+1}<c_j.
\]
In particular, we have $W_i\le V_j$ by \eqref{eq:W-V inclusion}. Note though that we have assumed $V_j\not\le W_i$. Therefore, $W_i<V_j$. Since $\dim(W_i)=i+1$, we infer that
\[
i\le\dim(V_j)-2 .
\]
Since we have assumed that $(\sW,\bb)$ is saturated, the above inequality implies that $b_{i+1}>c_{j+1}$. Recalling the definition \eqref{deltab} of $\delta(\bb)$, we conclude that 
\[
b_{i+1}\ge c_j-\delta(\bb).
\]
This implies that $G_i\cap I_j\subset[c_j-\delta(\bb),c_j]$ for any pair $(i,j)$ satisfying \eqref{eq:sat-flat i,j}. As a consequence,
\[
\sum_{i:\ V_j\not\le W_i} \lambda(I_j \cap G_i) \le \delta(\bb)\qquad(1\le j\le r).
\]
Since we  also have that \rev{$\er(\sW, \bb) \geq \er(\sV,\cc,\bmu) + \eps  \delta(\bb)/2$} by Lemma \eqref{adapted-gap}(b), it follows that
\rev{\begin{equation}\label{eq:kappa sat-flag}
\begin{split}
E(p,\sW,\bb) &\ge (p-1)\er(\sV,\cc,\bmu)+\eps\delta(\bb)/2 +O((p-1)^2\delta(\bb))\\ &\ge 
(p-1)\er(\sV,\cc,\bmu)+\eps\delta(\bb)/4,
\end{split}
\end{equation}}
provided that $p-1$ is small enough compared to $\eps$. 

\medskip

Using \eqref{eq:kappa sat-flag}, we see that the contribution of saturated flags to the right hand side of \eqref{eq59} is 
\[ 
 \sum_{(\sW, \bb)\ \text{saturated}} D^{-E(p,\sW,\bb)}\ll 
 D^{-(p-1) \er(\sV\rev{,\cc,\bmu})} 
\sum_{s=0}^r \sum_{b_1,\ldots,b_s} D^{-(p-1)\eps \delta(\bb)/4},
\]
where we used that there are $O(1)$ choices for $\sW$. \rev{Recall \eqref{eq:bi-dfn lb}, which implies that the numbers $b_i$  are restricted to the set $\{m/\log D: m\in \N\}$.} Thus the number of $\bb$ with $\delta(\bb)=m/\log D$ is
  at most $(m+1)^s$ and 
  \[
  \sum_{s=0}^r \sum_{b_1,\ldots,b_s} D^{-(p-1)\eps \delta(\bb)/4}
  \le \sum_{s=0}^r \sum_{m\ge 0} (m+1)^s e^{-(p-1)(\eps/4) m}
  \ll_{\eps,p} 1.
  \]
We thus conclude that
\[ 
\sum_{(\sW, \bb)\ \text{saturated}}
D^{-E(p,\sW,\bb)}
\ll  D^{-(p-1) \er(\sV\rev{,\cc,\bmu})}.
\]
If we combine the above inequality with \eqref{eq:unsat-flag} and \eqref{eq59}, we \rev{establish Proposition \ref{ell-p-bound}.}
\qed

\section{An argument of Maier and Tenenbaum} \label{mt-argument}

The aim of this section is to prove Proposition \ref{main-b-restate}. The reader may care to recall the statement of that proposition now, as well as the definition of a compatible map (Definition \ref{compatible}).
\rev{As in the previous section, the system $(\sV,\cc,\bmu)$ is
fixed, and satisfies conditions (i)--(iii) of Proposition \ref{bgamk}.} We also fix a basis  $\{ \one,\om^1,\ldots,\om^d \}$ of $V_r$ such that $V_j=\Span(\one,\om^1,\ldots,\om^{\dim(V_j)-1})$ for each $j$
\rev{ and such that $\omega^i\in \{0,1\}^k$ for each $i$.
Denote $\Omega = \Supp(\mu_r) =  V_r \cap \{0,1\}^k$.}

We begin with an observation related to the solvability of \eqref{special}, \rev{which we recall here for the convenience of the reader:
\be\label{special-again} 
\sum_{j = 1}^r  K_j \omega^j = - \sum_{\om} \omega \sum_{a' \in  B'_{\omega}}  a'\, \md{\mathbf{1}}.\ee}Let $\Lambda$ denote the $\Z$-span of $\one,\om^1,\ldots,\om^d$ (that is, the lattice generated by $\one,\om^1,\ldots,\om^d$).  Every vector $\om\in\Omega$ is  a rational combination of the basis elements $\one,\om^1,\ldots,\om^d$. Hence, there is some $M\in\N$ such that $M\om\in\Lambda$ for each $\om\in\Omega$. In particular, note that the right-hand side of \rev{\eqref{special-again}} lies generically in the lattice $\Lambda/M=\{x/M: x\in \Lambda\}$. \rev{However, we must ensure that \eqref{special-again} is solvable with $K_1,\dots,K_r\in\Z$. Equivalently, the right-hand side of \eqref{special-again} must lie in $\Lambda$, which can be guaranteed when the coefficients of all vectors $\omega$ in it lie in $M\Z$.}

In this section, implied constants in $O()$ and $\ll$ notations may depend on the system $(\sV,\cc,\bmu)$ and basis $\om^1,\ldots,\om^d$; in particular, on $k$, $d$ and $M$.

\subsection{The sets $\mathscr{L}_i(\A)$ and lower bounds for their size}

The main statement of this subsection, Proposition \ref{ellpS-new}, is a variant of Proposition \ref{ellpS-repeat},
where we stipulate that all elements lie in $\Lambda$. \rev{This will later ensure that \eqref{special-again} is solvable with $K_1,\dots,K_r\in\Z$.}

Fix $\kappa>0$ satisfying $\kappa \le \frac{\kappa^*}{2}$, where $\kappa^*$ is the constant from Proposition \ref{ellpS-repeat}. In particular, $\kappa\le 1/2$.  We introduce the sets
\begin{equation}\label{intervals-def}  I_i(D) :=  \bigcup_{j = 1}^r (D^{c_{j+1}}, D^{c_j(1-\kappa/i)}] , \qquad i = 1,2,\cdots .\end{equation}
Thus each $I_i(D)$ is simply a union of $r$ intervals in $\Lambda$, and we have the nesting
\[ I_1(D) \subset I_2(D) \subset \cdots \subset (D^{c}, D].\]

\newcommand{\ovr}{\overline{V}_r}
\newcommand{\opsi}{\overline{\psi}}
For any $\om\in V_r$ we denote by $\overline{\om}$ 
the projection onto $\ovr \rev{:=} V_r/\spanone=\Span\{ \om^1,\ldots,\om^d \}$.
In addition
let $\opsi(a)=\overline{\psi(a)}$ for $a\in \A$.

The reader may wish to recall the definition of nondegenerate (Definition \ref{nondegenerate-def}) and compatible (Definition \ref{compatible}) maps.

\begin{definition}\label{ell-def} Write $\mathscr{L}_{i}(\A)$ for the set of all $\sum_{a \in \A} a\opsi(a)$ that lie in $\Lambda$, where $\psi$ ranges over all nondegenerate, compatible maps supported on $I_i(D)$.\end{definition}

\begin{proposition}\label{ellpS-new}
Let $\delta > 0$ and $i\in\N$, and let $D$ be sufficiently large in terms of $\delta$. Then with probability at least $1 - \delta$ in the choice of $\A \cap I_i(D)$,
\begin{equation}\label{ell-lower} |\mathscr{L}_{i}(\A)| \gg \delta^\alpha D^{(1 - \kappa/i)\sum_j c_j \dim (V_j/V_{j-1})}, \end{equation}
where $\alpha$ is a positive constant depending at most on $(\sV,\cc,\bmu)$.
\end{proposition}

\begin{proof}
Let
%We divide $I_i$ into two sets
\[
 I_i'(D)=\bigcup_{j = 1}^r (D^{(c_{j+1}+\kappa^*)(1-\kappa/i)},  D^{c_j(1-\kappa/i)} ] 
 	\subset \bigcup_{j = 1}^r (D^{c_{j+1}(1+\kappa/2)},  D^{c_j(1-\kappa/i)}] \subset I_i(D),
%\cJ=(D^{c_2(1-\kappa/i)},D^{(c_2+\kappa^*)(1-\kappa/i)}],\qquad  I'_i(D) := I_i(D) \setminus \cJ. 
\]
where the first inclusion follows by noticing that $(c_{j+1}+\kappa^*)(1-\kappa/i)\ge c_{j+1}(1+\kappa/2)$ for $c_{j+1}\in[0,1]$,  $0\le \kappa\le \kappa^*/2\le 1/2$ \rev{and $i\ge1$}.
Write $\mathscr{L}'_i(\A)$ for the set of all $\sum_{a \in \A} a \overline{\psi}(a)$, where $\psi$ ranges over all nondegenerate, compatible maps supported on $I'_i(D)$, but without the stipulation that the sum is in $\Lambda$.
We now apply Proposition \ref{ellpS-repeat}  with $D$ replaced by $D^{1-\kappa/i}$ and $\delta$ replaced by $\delta/2$ to conclude that
\[
|\sL_i'(\A)| \gg \delta^{\alpha}  D^{(1 - \kappa/i)\sum_j c_j \dim (V_j/V_{j-1})}
\]
with probability at least $1-\delta/2$, where $\alpha=\rev{1/(p-1)}$ with $p$ as in Proposition \ref{ellpS-repeat}.

We now use the elements of $\A \cap (I_i(D)\setminus I_i'(D))$ to 
create many sums $\sum_{a \in \A} \overline{\psi}(a)$
which do lie in $\Lambda$. 
\rev{Let $G:= (D^{c_{r+1}(1-\kappa/i)},\delta^{-1}D^{c_{r+1}(1-\kappa/i)}]$, which is a subset of $I_i(D)\setminus I'_i(D)$.
  Let $\cE$ be the event that \rev{$\A \cap G$ contains at least} $2^k$ elements
that are $\equiv m\pmod{M}$ for each $m\in \{1,\ldots,M\}$. Lemma \ref{A-normal} (applied with $B=\{b\in\Z\cap G : b\equiv m\pmod M\}$ and $\eps=1/3$) implies that
if $\delta$ is sufficiently small then
$\P(\cE) \ge 1 - \delta/2$.  

Assume now that we are in the event $\cE$. \rev{Let us fix a set $\cK\subset \A\cap G$ that contains exactly $2^k$ elements
that are $\equiv m\pmod{M}$ for each $m\in \{1,\ldots,M\}$.} Take any nondegenerate, compatible function $\psi : \A \rightarrow \{0,1\}^k$ supported on $I'_i(D)$, and write
\[
\sum_{a\in I'_i(D)} a {\psi}(a) = \sum_{\om \in \Omega} \om N_\om. 
\]
Recall that $\Supp(\mu_r)=V_r\cap\{0,1\}^k$ by condition (iii) of Proposition \ref{bgamk}. Hence, for each $\om \in \Omega$, we may find an element $a_\omega \in \cK$ 
satisfying $a_\om \equiv -N_\om \pmod{M}$.
Setting $\psi_0(a_{\om})=\om$ for each $\om$, 
and $\psi_0(a)=\psi(a)$ for $a\in I'_i(D)$, and $\psi_0(a)=\mathbf{0}$ for all other $a\in I_i(D)$.} We have 
\[
\sum_{a\in I_i(D)} a \psi_0(a) = \sum_{\om\in \Omega}
(a_{\om}+N_\om) \om \in \Lambda ,
\]
since $M|(a_\om+N_\om)$ for all $\om$.
Moreover, $\psi_0$ is nondegenerate and compatible by construction.
Consequently, $\sum_a a\overline{\psi}_0(a) \in \Lambda$ (by removing
the coefficient of $\one$).
Since there are at most \rev{$2^{|\cK|} \le  2^{M2^k}$} choices
for $\{a_\om : \om\in \Omega\}$, the map from $\sum_{a\in I'_i(D)} a\overline{\psi}(a)$
to $\sum_{a\in I_i(D)} a\overline{\psi}_0(a)$ is
at most \rev{$2^{M2^k}$}-to-1.We conclude that with probability
$\ge 1-\delta$, 
\[
|\sL_i(\A)| \ge 2^{-\rev{M2^k}} |\sL_i'(\A)| \gg \delta^{\alpha}  D^{(1 - \kappa/i)\sum_j c_j \dim (V_j/V_{j-1})},
\]
the implied constant only depending on $k,M$ and $\alpha$, which
are all fixed.
\end{proof}

\subsection{Putting $\mathscr{L}_i(\A)$ in a box}\label{upper-bd-subsec}

In the last section, we showed that (with high probability) $\mathscr{L}_i(\A)$ is large. In this section we show that with high probability it is contained in a box (in coordinates
$\om^1,\ldots,\om^d$); putting these results together one then sees that $\mathscr{L}_i(\A)$ occupies a positive proportion
of lattice points in the box, the bound being independent of $D$.

\rev{For $t \in \{1,\dots, d\}$, write $j(t)$ for the unique $j$ such that $\dim V_{j-1} < t \leq \dim V_j$. In addition, let $\rev{C}$ be the largest coordinate \rev{in absolute value} of any element in $V_r\cap\{0,1\}^k$ when written with respect to the base $\one,\om^1,\dots,\om^d$. We then set
\begin{equation}\label{nj-def} 
N_j^{(i)} := \delta^{-1}  \cdot \rev{C} \cdot D^{(1 - \kappa/i) c_j}
\qquad\text{and}\qquad
N^{(i)} := \prod_{t = 1}^d N_{j(t)}^{(i)} . 
\end{equation} }

\begin{lemma}\label{upper-l} Assume $\delta>0$ is small enough so that $ \rev{r} e^{-2/\delta}\le \delta$. 
Then, we have
\begin{equation}\label{ell-contain} \mathscr{L}_{i}(\A) \subset 
\bigoplus_{t = 1}^d \big[-N_{j(t)}^{(i)},N_{j(t)}^{(i)}\big] \om^t
\end{equation}
with probability at least $1 -\delta$ in the choice of $\A \cap I_i(D)$.
\end{lemma}

\begin{proof} 
This follows quickly from the fact that $\psi$ 
is compatible and by Lemma \ref{sum_a}, the latter implying that
\[ 
\sum_{a \in \A\cap[2, D^{(1 - \kappa/i) c_j}]}  a \leq \delta^{-1} D^{(1 - \kappa/i) c_j} \qquad (1\le j\le \rev{r})
\]
with probability $\ge 1- \rev{r} e^{-2/\delta} \ge 1-\delta$.
 \end{proof}

\begin{proposition}\label{small-box} Let  $\delta$ and $\alpha$ be as in Proposition \ref{ellpS-new} and in Lemma \ref{upper-l}. 
With probability at least $1 - \rev{2 \delta}$ in the choice of $\A \cap I_i(D)$, $\mathscr{L}_i(\A)$ is a subset of the box 
$\bigoplus_{t = 1}^d [-N_{j(t)}^{(i)}, N_{j(t)}^{(i)}] \om^t$  of size $\gg \delta^{d+\alpha} N^{(i)}$.
\end{proposition}
\begin{proof}
This follows immediately upon combining Proposition \ref{ellpS-new} and Lemma \ref{upper-l} .
\end{proof}

\subsection{Zero sums with positive probability}\label{zero-positive}

\begin{lemma}\label{lemma16} \rev{Let $\delta$ and $\alpha$ be as in Proposition \ref{ellpS-new} and Lemma \ref{upper-l}, and let $D$ be large enough in terms of $\delta$ and $(\sV,\cc,\bmu)$.} 
Let $i\in\Z\cap[1,(\log D)^{1/3}]$. In addition, let $S\subset \bigoplus_{t = 1}^d [-N_{j(t)}^{(i)}, N_{j(t)}^{(i)}] \om^t$ with $|S| \gg \delta^{\rev{d+\alpha}} N^{(i)}$ and with $S \subset \Lambda$. Then
\[ 
\P\big(0 \in \mathscr{L}_{i+1}(\A)	\,\big|\, \mathscr{L}_i(\A) = S\big)  \gg \delta^{2d\rev{(d+\alpha)}} .
\]
\end{lemma}
\begin{proof} We condition on a fixed choice of $\A \cap I_i(D)$ for which $\mathscr{L}_i(\A) = S$. \rev{Note that 
	\begin{equation}\label{difference of consecutive I's}
	I_{i+1}(D) \setminus I_i(D) = \bigcup_{j = 1}^r (D^{(1 -  \kappa/i) c_j}, D^{(1 - \kappa/(i+1)) c_j}] \supset \bigcup_{j=1}^r [N^{(i)}_j, 100d N^{(i)}_j].
	\end{equation}
Then it is enough to show that with probability $\gg \rev{\delta^{2d(d+\alpha)}}$ , the set $\A$ contains $2d$ distinct elements $a_t$ and $a_t'$, $1\le t\le d$, such that
\be\label{sumatat'}
\sum_t (a_t'-a_t)\om^t   \in S
\qquad\text{and}\qquad
a_t,a_t' \in[N^{(i)}_j, 100d N^{(i)}_j] \quad\text{for}\ \ t=1,\dots,d.
\ee
}To see why this is sufficient, let $s=\sum_t(a_t'-a_t)\om^t$, which we know belongs to $S=\sL_i(\A)$. In particular, there is an compatible map $\psi$ supported on $I_i(D)$  such that $\sum_{a \in \A} a \overline{\psi}(a) = s$. Now, consider the function $\psi' : \A \cap I_{i+1}(D) \rightarrow \rev{\{0,1\}^k}$ with ${\psi'}(a) = \psi(a)$ for $a\in \A \cap I_i(D)$, $\psi'(a_t')=\one-\om^t$ and $\psi'(a_t)=\om^t$ for $1\le t\le d$, \rev{and $\psi'(a)=\mathbf{0}$ for all other values of $a\in \A\cap I_{i+1}(D)$}. \rev{Notice that $\psi'$ is compatible according to Definition \ref{compatible} by the second part of \eqref{sumatat'}.}  It is now clear that $0\in \sL_{i+1}(\A)$. Hence, if the conditional probability that \eqref{sumatat'} holds is $\gg \delta^{2\beta d}$, so is the probability that $0\in\sL_{i+1}(\A)$.
 
To find $a_t$ and $a_t'$ satisfying \eqref{sumatat'}, let
\[
n:= \cl{ d \rev{3^{d+1}} N^{(i)}/|S|} \ll \rev{\delta^{-(d+\alpha)}}.
\]
The number of elements $\sum_t s_t \om^t \in S$ with
$n|s_t$ for some $t$ is
\[
\leq \sum_{t=1}^d \big( 2N_{j(t)}^{(i)}/n+1 \big) \prod_{t'\ne t} \big(2N^{(i)}_j(t')+1\big) \le 
 \rev{d3^{d-1}\Big( \frac{2N^{(i)}}{n}+\frac{N^{(i)}}{\min_j N_j^{(i)}} \Big)} \le |S|/2
\]
\rev{as long as $D$ is large enough in terms of $\delta$ and $(\sV,\cc,\bmu)$.} Thus, there is a subset $S'\subset S$ of size
at least $|S|/2$ and with $n\nmid s_t$ for all $t$.
We will choose the sets $\{a_t: 1\le t\le d\}$ 
and $\{a_t': 1\le t\le d\}$ independently, by selecting
$a_t \equiv 0\pmod{n}$ and $a_t'\not\equiv 0\pmod{n}$. 

Note that
\[ 
I_{i+1}(D) \setminus I_i(D) = \bigcup_{j = 1}^r (D^{(1 -  \kappa/i) c_j}, D^{(1 - \kappa/(i+1)) c_j}] \supset \bigcup_{j=1}^r [N^{(i)}_j, 100d N^{(i)}_j]
\] 
provided that
$i\le (\log D)^{1/3}$. For each given $t$, $i$ and $j$, the probability that the interval $[4t N^{(i)}_{j}, (4t+2) N^{(i)}_{j}]$ contains
no element $a_t \equiv 0\pmod{n}$ of $\A$ equals
\[
\prod_{\substack{4t N^{(i)}_{j}\le a \le  (4t+2) N^{(i)}_{j} \\ a \equiv 0\pmod n}}(1-1/a) \le 1-\gamma/n
\]
for some small positive constant $\gamma=\gamma(d)$. Thus, the probability that, for each $t=1,2,\dots,d$, the set $\A$ contains some $a_t\equiv0\pmod n$ in the interval $[4t N^{(i)}_{\rev{j(t)}}, (4t+2) N^{(i)}_{\rev{j(t)}}]$ is $\gg 1/n^d\gg \rev{\delta^{d(d+\alpha)}}$.

Fix a choice of $a_1,\dots, a_d$ as described above, and set
\begin{equation}\label{x-def} 
X := \{ (a_1+s_1,\ldots,a_d+s_d) : s_1\om^1 + \cdots + s_d\om^d \in S'\}.
\end{equation}
By construction, every coordinate of $x\in X$ is
 $\not\equiv 0\pmod{n}$.  Also,
\begin{equation}\label{x-contain} 
X \subset \prod_{t = 1}^d \big[(4t-1)N_{j(t)}^{(i)}, (4t+3)N_{j(t)}^{(i)}\big].
\end{equation}
Now the intervals on the right-hand side above are disjoint,
and
\[
|X| \ge \frac{|S|}{2} \gg \delta^\beta
\prod_{t=1}^d N_{j(t)}^{(i)}. 
\]
Thus, by Lemma \ref{interval-product}, with probability $\gg (\rev{\delta^{d+\alpha}})^d$, there are $a_1',\ldots,a_d' \in \A$ such that $(a_1',\ldots,a_t')\in X$.   The relation \eqref{sumatat'} follows for such $a_t,a_t'$, 
which exist with probability $\gg \rev{\delta^{d(d+\alpha)}} \cdot \rev{\delta^{d(d+\alpha)}}$.
\end{proof}

\subsection{An iterative argument}

To complete the proof of Proposition \ref{main-b-restate}, we apply Lemma \ref{lemma16} iteratively. Let $\sS$ be the set of sets $S$  satisfying the assumptions of Lemma \ref{lemma16}. We say that $\mathscr{L}_i(\A)$ is \emph{large} if it satisfies the conclusions of Proposition \ref{small-box}, or equivalently if $\mathscr{L}_i(\A) = S$ with $S\in\sS$. Thus Lemma \ref{lemma16} implies that 
\begin{align*}
\P\big( 0 \in \mathscr{L}_{i+1}(\A)\setminus\sL_i(\A),\ \mathscr{L}_{i}(\A) \; \mbox{large} \big)
&=\sum_{\substack{S\ \text{large} \\ 0\notin S}} 
\P(\sL_i(\A)=S)\cdot  \P\big( 0 \in \mathscr{L}_{i+1}(\A) \,\big|\, \mathscr{L}_{i}(\A)=S\big)     \\
&\gg \delta^{2d\alpha}	\P\big(\sL_i(\A)\ \text{large},\ 0\notin\sL_i(\A)\big).
\end{align*}
We conclude there is some $\eps = \delta^{O(1)}$ such that
\begin{equation}\label{lem16-new} 
\P\big( 0 \in \mathscr{L}_{i+1}(\A) \,\big|\, \mathscr{L}_{i}(\A) \; \mbox{large}, 0 \notin \mathscr{L}_i(\A) \big) \geq\eps .
\end{equation} 

For brevity, write $E_i$ for the event that $0 \notin \mathscr{L}_i(\A)$, and $F_i$ for the event that $\mathscr{L}_{i}(\A)$ is large. \rev{In this notation, \eqref{lem16-new} becomes
\begin{equation}\label{lem16-new-again} 
	\P\big( E_{i+1}^c | E_i\cap F_i) \geq\eps .
\end{equation}
Moreover, Proposition \ref{small-box} implies that
\begin{equation}\label{15-new} 
	\P(F_i) \geq 1 - 2\delta.
\end{equation}
Lastly, note that $E_1 \supset E_2 \supset \cdots$ because $\mathscr{L}_1(\A) \subset \mathscr{L}_2(\A) \subset \cdots$}

We claim that \rev{$\P(E_i)< 4\delta$ for some $i\le I:=\lfloor (\log D)^{1/3}\rfloor$. Indeed, for each $i\le I$, we have}
\begin{align*}
\P(E_{i+1}) & = \P(E_{i+1} | E_i \cap F_i) \P(E_i \cap F_i) + \P(E_{i+1} | E_i \cap F^c_i) \P(E_i \cap F^c_i) \\ & \leq (1 - \eps) \P(E_i \cap F_i) + \P(E_i \cap F^c_i)\qquad \mbox{by \eqref{lem16-new-again}} \\ & = \P(E_i) - \eps \P(E_i \cap F_i) \\ & \leq \P(E_i) - \eps( \P(E_i) - \rev{2\delta}) \qquad \mbox{by \eqref{15-new}}.
\end{align*}
Thus, if $\P(E_i) \geq \rev{4 \delta}$, then $\P(E_{i+1}) \leq (1 - \eps/2) \P(E_i)$. If this holds for all $i\le I$, then $\P(E_I)\le (1-\eps/2)^{I-1}<4\delta$, a contradiction.  Therefore,
$\P(\E_{i^*}) \rev{< 4\delta}$ for some $i^*\le I$, \rev{as long as $D$ is large enough in terms of $\delta$ and the (fixed) system $(\sV,\cc,\bmu)$.
This completes the proof of Proposition \ref{main-b-restate}.}

\clearpage
\thispagestyle{fancy}
\fancyhf{} % sets both header and footer to nothing
\renewcommand{\headrulewidth}{0cm}
\lhead[{\scriptsize \thepage}]{}
\rhead[]{{\scriptsize\thepage}}
\part{The optimisation problem}\label{part:optimization}

\section{The optimisation problem -- basic features} \label{optimisation-sec}

In this section we consider Problem \ref{opt-problem}, the optimisation problem on the cube, 
 which is a key feature of our paper. We will give some kind of a solution to this for a fixed nondegenerate flag $\sV$, leaving aside the question of how to choose $\sV$ optimally.

Let us refresh ourselves on the main elements of the setup of Problem \ref{opt-problem}. We have a \rev{nondegenerate, $r$-step} flag
\[ \sV: \langle \mathbf{1} \rangle = V_0 \leq  V_1 \leq  V_2 \leq  \cdots \leq  V_r \leq  \Q^k \] of distinct vector spaces.  \rev{In light of Lemma \ref{gammak:reduction}, we may  restrict our attention to flags such that
\[
\dim(V_1/V_0)=1,
\]
 which we henceforth assume.
With the flag $\sV$ fixed, we wish to find $\gamma_k(\sV)$, the supremum of numbers $c\ge 0$ such that
there are  thresholds $1 = c_1 \ge c_2 \ge \cdots \ge  c_{r+1} = c$ (we may assume
that $c_1=1$ by arguing as in Lemmas \ref{gammak-tilde:reduction} and \ref{gammak:reduction}) and probability measures $\mu_1,\dots, \mu_r$ on $\{0,1\}^k$ satisfying $\Supp(\mu_j) \subset V_j$ for each $j$, and such that the entropy condition \eqref{entropy-cond} holds, that is to say
\begin{equation}\label{ent-condition-repeat}
\er(\sV',\cc,\bmu) \geq \er(\sV,\cc,\bmu)
\end{equation} 
for all subflags $\sV' \leq \sV$.  \rev{We recall that}
\[ 
\er(\sV',\cc,\bmu) 
 := \sum_{j = 1}^r (c_j - c_{j+1}) \HH_{\mu_j}(V'_j) + \sum_{j = 1}^r c_j \dim (V'_j/V'_{j-1}).
\]
}

\rev{\noindent\textit{Remarks.}  
(a) It is easy to see that $\gamma_k(\sV)$ always exists by considering the following example with $c=0$.
Take $c_1=1$ and $c_2=\cdots=c_{r+1}=0$ and recall that $\dim(V_1/V_0)=1$.   Suppose that
$V_1=\Span(\one,\omega)$ with $\omega\in \{0,1\}^k$.  Thus,
$\er(\sV,\cc,\bmu)=1$ for any choice of $\bmu$.  If $V_1'=V_1$ then likewise we have
$\er(\sV',\cc,\bmu)=1$, and if $V_1'=V_0$ then $\er(\sV',\cc,\bmu)=\HH_{\mu_1}(V_0)$.
Now $V_0+\one$, $V_0+\omega$ and $V_0+(\one-\omega)$ are three different cosets.
Taking $\mu_1(\one)=\mu_1(\omega)=\mu_1(\one-\omega)=1/3$ we have  $\er(\sV',\cc,\bmu)=\log 3$.  Thus, \eqref{entropy-cond} holds.  As we shall see in this section,
this choice of $\mu_1$ is the optimal choice for a very general class of flags,
including those of interest to us.

(b) A simple compactness argument shows that the supremum is realised, that is,
there is a choice of $\cc$ and $\bmu$ satisfying the entropy condition
\ref{entropy-cond} and with $c_{r+1}=\gamma_k(\sV)$.

(c) As long as we can show that $\gamma_k>0$ (which will be taken care of in Part \ref{part:binary-systems}), we can always find an optimal system $(\sV,\cc,\bmu)$ that also has $c_j>c_{j+1}$
for each $j$ (cf.~Lemma \ref{gammak:reduction}(a)).
}

\subsection{A restricted optimisation problem}

It turns out to be very useful to consider a restricted variant of the problem in which the entropy condition \eqref{ent-condition-repeat} is only required to be satisfied for certain ``basic'' subflags $\sV'$, rather than all of them.

\begin{definition}[Basic subflag]
Given a flag $\sV : \langle\mathbf{1}\rangle = V_0 \leq V_1 \leq \cdots \leq V_r$, the basic subflags $\sV'_{\basic(m)}$ are the ones in which $V'_i = V_{\min(m,i)}$, for $m = 0,1,\dots, r-1$ (note that when $m = r$ we recover $\sV$ itself).
\end{definition}

Here is the restricted version of Problem \ref{opt-problem}.
Recall that a flag is non-degenerate if the top space $V_r$ 
is not contained in any of the subspaces 
$\{x\in \R^k: x_i=x_j  \}$.   The restriction to nondegenerate  flags
ensures that the subsets $A_1,\ldots,A_k$ in our main problem are distinct.

\begin{problem}\label{restricted-opt}
Let $\sV$ be a nondegenerate flag of distinct spaces in $\Q^k$. Define $\gamma_k^{\res}(\sV)$ to be the supremum of all constants \rev{$c\ge0$} for which \rev{there are measures $\mu_1,\dots, \mu_r$ such that} $\Supp(\mu_i) \subset V_i$, \rev{and parameters $1 = c_1 \ge \cdots \ge  c_{r+1} = c$} such that the restricted entropy condition
\begin{equation}\label{rest-ent} 
\er(\sV'_{\basic(m)};\cc,\bmu) \geq \er(\sV; \cc,\bmu) 
\end{equation} 
holds for all $m = 0,1,\dots, r-1$. 
\end{problem}

It is clear that

\begin{equation}\label{gamk-gamk-basic}   \gamma_k^{\res}(\sV) \geq \gamma_k(\sV) .\end{equation}
In general there is absolutely no reason to suppose that the two quantities are equal, since after all the restricted entropy condition \eqref{rest-ent} apparently only captures a small portion of the full condition \eqref{ent-condition-repeat}.

Our reason for studying the restricted problem is that we do strongly believe that 
\[ \sup_{\sV \nondegenerate} \gamma_k^{\res}(\sV) = \sup_{\sV \nondegenerate} \gamma_k(\sV) = \gamma_k. \] One might think of this unproven assertion, on an intuitive level, in two (roughly equivalent) ways:

\begin{itemize} \item for those flags optimal for Problem \ref{opt-problem}, the critical cases of \eqref{ent-condition-repeat} are those for which $\sV'$ is basic;
\item for those flags optimal for Problem \ref{opt-problem}, and for the critical choice of the $c_i, \mu_i$, the restricted condition \eqref{rest-ent} in fact implies the more general condition \eqref{ent-condition-repeat}.
\end{itemize}

\subsection{The $\rho$-equations, optimal measures and optimal parameters}\label{rho-equations-sec}

The definitions and constructions of this section will appear unmotivated at first sight. They are forced upon us by the analysis of subsection \ref{sec65} below.

Let the flag $\sV$ be fixed. 

It is convenient to call the intersection of a coset $x + V_i$ with the cube $\{0,1\}^k$ a \emph{cell at level $i$}, and to denote the cells at various levels by the letter $C$. (The terminology comes from the fact it can be useful to think of $V_i$ defining a $\sigma$-algebra (partition) on $\{0,1\}^k$, the equivalence relation being given by $\om \sim \om'$ iff $\om - \om' \in V_i$: however, we will not generally use the language of $\sigma$-algebras in what follows.)

If $C$ is a cell at level $i$, then it will be a union of cells $C'$ at level $i-1$. These cells we call the children of $C$, and we write $C \rightarrow C'$. 

Let ${\bm \rho} = (\rho_1,\dots, \rho_{r-1})$ be  real parameters in $(0,1)$, and for each cell $C$  define functions $f^C({\bm \rho})$ by the following recursive recipe:
\begin{itemize}
\item If $C$ has level $0$, then $f^C({\bm \rho}) = 1$;
\item If $C$ has level $i$, then
\begin{equation}\label{fc-eqs} f^C({\bm \rho}) = \sum_{C \rightarrow C'} f^{C'}({\bm \rho})^{\rho_{i-1}},\end{equation}
\end{itemize}
with the convention that $\rho_0 = 0$.

Write
\[
\Gamma_i = V_i \cap \{0,1\}^k
\]for the cell at level $i$ which contains $\mathbf{0}$. Note that 
\[ 
\rev{\{\mathbf{0},\mathbf{1}\}}=\Gamma_0 \subset \Gamma_1 \subset \cdots \subset \Gamma_r.
\]

\begin{figure}[tbp]
\begin{sideways}

\scriptsize{
\begin{tikzpicture}
  \tikzstyle{level 1}=[level distance=50mm,sibling distance=2.4cm]
  \tikzstyle{level 2}=[level distance=40mm,sibling distance=1.2cm]
  \tikzstyle{every label}=[red]
  
  \node[draw,label={[xshift=0cm, yshift=0.1cm,red]$3^{\rho_1} + 4 \cdot 2^{\rho_1} + 4$},  label={[xshift=-1.1cm, yshift=-0.6cm,blue]$\Gamma_2$}]{$\{0,1\}^4$}
    child {node[draw, left=1.5cm,label={[xshift=0cm, yshift=0.1cm]$3$},label={[xshift=-1.2cm, yshift=-0.6cm,blue]$\Gamma_1$}]{$\begin{smallmatrix} 0000 & 0011 \\ 1100 & 1111   \end{smallmatrix}$ }
    child {node[draw,label={[xshift=0cm, yshift=0.1cm]$1$},label={[xshift=-1cm, yshift=-0.6cm,blue]$\Gamma_0$},label={[xshift=10cm,yshift=-4cm]\textcolor{black}{Figure 7.1: the tree structure corresponding to the binary flag $\langle \mathbf{1} \rangle = V_0 \leq V_1 \leq V_2 \leq \mathbb{Q}^4$. Values of $\textcolor{red}{f^C(\rho)}$ are given in red.}}]{$\begin{smallmatrix} 0000 \\ 1111 \end{smallmatrix}$}} 
    child{node[draw,label={[xshift=0.5cm, yshift=0.1cm]$1$}]{$0011$}} 
    child{node[draw,label={[xshift=0.5cm, yshift=0.1cm]$1$}]{$1100$}}   
    }
    child {node[draw,label={[xshift=0cm, yshift=0.1cm,red]$2$}]{$\begin{smallmatrix} 0001 \\ 1101 \end{smallmatrix}$}
       child{node[draw,label={[xshift=0.5cm, yshift=0.1cm]$1$}]{$0001$}}
       child {node[draw,label={[xshift=0.5cm, yshift=0.1cm]$1$}] {$1101$}}
    }
    child {node[draw,label={[xshift=0cm, yshift=0.1cm]$2$}]{$\begin{smallmatrix} 0010 \\ 1110 \end{smallmatrix}$}
       child{node[draw,label={[xshift=0.5cm, yshift=0.1cm]$1$}]{$0010$}}
       child{node[draw,label={[xshift=0.5cm, yshift=0.1cm]$1$}] {$1110$}}
    }
    child {node[draw,label={[xshift=0.6cm, yshift=0.1cm]$2$}]{$\begin{smallmatrix} 0100 \\ 0111 \end{smallmatrix}$}
       child {node[draw,label={[xshift=0.5cm, yshift=0.1cm]$1$}] {$0100$}}
       child {node[draw,label={[xshift=0.5cm, yshift=0.1cm]$1$}] {$0111$}}
       }
       child {node[draw,label={[xshift=0.5cm, yshift=0.1cm]$2$}]{$\begin{smallmatrix} 1000 \\ 1011 \end{smallmatrix}$}
       child {node[draw,label={[xshift=0.5cm, yshift=0.1cm]$1$}] {$1000$}}
       child{node[draw,label={[xshift=0.5cm, yshift=0.1cm]$1$}] {$1011$}}
       }
       child {node[draw, right =1cm,label={[xshift=0.5cm, yshift=0.1cm]$1$}] {$0101$}
       child {node[draw,label={[xshift=0.5cm, yshift=0.1cm]$1$}] {$0101$}}
       }
       child {node[draw, right =0.5cm,label={[xshift=0.5cm, yshift=0.1cm]$1$}] {$0110$}
       child{node[draw,label={[xshift=0.5cm, yshift=0.1cm]$1$}] {$0110$}}
       }
       child {node[draw,right =0cm,label={[xshift=0.5cm, yshift=0.1cm]$1$}] {$1001$}
       child {node[draw,label={[xshift=0.5cm, yshift=0.1cm]$1$}] {$1001$}}
       }
       child {node[draw,right =-0.5cm,label={[xshift=0.5cm, yshift=0.1cm]$1$}] {$1010$}
       child {node[draw,label={[xshift=0.5cm, yshift=0.1cm]$1$}] {$1010$}}
       };      
\end{tikzpicture}
}
\end{sideways}
\end{figure}

\begin{definition}[$\rho$-equations]
The $\rho$-equations are the system of equations
\begin{equation}\label{rho-equations}
f^{\Gamma_{j+1}}({\bm \rho}) = (f^{\Gamma_j}({\bm \rho}))^{\rho_j} e^{\dim (V_{j+1}/V_j)}, \qquad j = 1,2,\dots, r-1. \end{equation}
We say that they have a solution if they are satisfied with $\rho_1,\ldots,\rho_{r-1}
\in (0,1)$.
\end{definition}

\begin{example*}
\rev{Figure 7.1 illustrates these definitions for the so-called \emph{binary flag} in $\Q^4$, which will be a key object of study from Section \ref{binary-system} onwards. Here $V_1 = \{ (x_1,x_2,x_3,x_4) \in \Q^4 : x_1 = x_2, x_3 = x_4\}$ and $V_2 = \Q^4$.
The $\rho$-equations consist of the single equation $\textcolor{red}{f^{\Gamma_2}(\rho)} = (\textcolor{red}{f^{\Gamma_1}(\rho)})^{\rho_1} e^2$, that is to say $3^{\rho_1} + 4 \cdot 2^{\rho_1} +4 = 3^{\rho_1} e^2$. This has the unique solution $\rho_1 \approx 0.306481$.}
\medskip
\end{example*}

In general the $\rho$-equations may or may not have a solution, but for flags $\sV$ of interest to us, it turns out that they have a unique such solution. In this case, we make the following definition.

\begin{definition}[Optimal measures]\label{opt-meas-def}
Suppose that $\sV$ is a flag for which the $\rho$-equations have a solution. Then the corresponding \emph{optimal measure on $\mu^*$ on $\{0,1\}^k$ with respect to $\sV$} is defined as follows: we set $\mu^*(\Gamma_r) = 1$, and 
\begin{equation}\label{eq46} 
\frac{\mu^*(C')}{\mu^*(C)} = \frac{f^{C'}({\bm \rho})^{\rho_{i-1}}}{f^C({\bm \rho})}\end{equation} for any cell $C$ at level $i \geq 1$ and any child $C \rightarrow C'$. We also set \rev{$\mu^*(\mathbf{0})=\mu^*(\mathbf{1}) = \mu^*(\Gamma_0)/2$}. Lastly, we define the restrictions $\mu^*_j(\om) := \mu^*(\Gamma_j)^{-1}\mu^*(\om)1_{\om \in \Gamma_j}$ for $j = 1,2,\dots, r$ (thus $\mu^*_r = \mu^*$). We call these\footnote{Note that we have not said that the $\rho_i$ are unique. However, in cases of interest to us this will turn out to be the case.} \emph{optimal measures} (on $\{0,1\}^k$, with respect to $\sV$). Finally, we write $\bmu^* = (\mu_1^*,\mu_2^*,\dots, \mu_r^*)$.
\end{definition}

\begin{remark}\label{rem:optimal measure}
(a) By taking telescoping products of \eqref{eq46} for $i = r, r-1,\cdots, 0$, we see that $\mu^*$ is uniquely defined on all cells at level $0$, and these are  the cell $\{ \mathbf{0}, \mathbf{1}\}$ and singletons $\{\om\}$ for all $\om \in\rev{ \{0,1\}^k\setminus \{\mathbf{0},\mathbf{1}\}}$. Since we also specified \rev{$\mu^*(\mathbf{0})=\mu^*(\mathbf{1}) = \mu^*(\Gamma_0)/2$}, we see that $\mu^*(\om)$ is completely and uniquely determined by these rules, for all $\om$.  In particular,
 the $\rho$-equations \eqref{rho-equations} are equivalent to
 \rev{
 \[
\frac{ \mu^*(\Gamma_j)}{\mu^*(\Gamma_{j+1})} = e^{-\dim(V_{j+1}/V_j)} \quad\text{for}\ j=1,\dots,r-1,
 \]
 and thus}
 \be\label{mu-gammaj}
 \rev{\mu_j^*}(\Gamma_m) = e^{-\dim(V_j/V_m)} \qquad (j\ge m\ge \rev{1}).
 \ee
 \rev{In addition, we have
 	\begin{equation}\label{mu-gamma0}
 	\mu^*(\Gamma_0) = \mu^*(\Gamma_1) \cdot \frac{1}{f^{\Gamma_1}(\bs \rho)} =  \frac{e^{-\dim(V_1/V_r)}}{|\Gamma_1|-1} .
 \end{equation}}
 
 \smallskip
 
\rev{(b) By construction, the measures $\mu_j^*$ satisfy statements (d) and (e) of Lemma \ref{gammak-tilde:reduction} for all $j$:
\begin{equation}
	\label{mu* basic properties}
		\Supp(\mu_j)=\Gamma_j\quad\text{and}\quad \mu_j(\om)=\mu_j(\mathbf{1}-\om)\quad\forall \omega. 
\end{equation}

\smallskip}

\rev{(c)} At the moment, the term ``optimal measure'' is just a name. We will establish the sense in which (in situations of interest) the measures $\mu^*_j$ are optimal in Proposition \ref{main-optim} below.

\medskip

\rev{(d)} Note that $\bmu^*$ and $\mu^*$ are two different (but closely related) objects. The former is an $r$-tuple of measures $\mu_j^*$, all of which are induced from the single measure $\mu^*$.
\end{remark}

\begin{definition}[Optimal parameters]\label{opt-param-def}
Suppose that $\sV$ is a flag for which the $\rho$-equations have a solution. Let $\mu^*$ be the corresponding optimal measure on $\{0,1\}^k$ with respect to $\sV$. Suppose additionally that 
\begin{equation}\label{non-degen} \HH_{\mu^*_{m+1}}(V_m) \neq \dim(V_{m+1}/V_m)\end{equation} for $m = 0,1,\dots, r-1$.
 Then the corresponding \emph{optimal parameters} with respect to $\sV$ and the solution $\bs\rho$ are the unique choice of $\cc^* : 1 = c^*_1 > c^*_2 >\cdots > c^*_{r+1} > 0$, if it exists, such that 
\begin{equation}\label{79-pre}
 \rev{\er(\sV'_{\basic(m)}, \bmu^*, \cc^*) = \er(\sV,\bmu^*, \cc^*)}\qquad \mbox{for $m = 0,1,\dots, r - 1$}.\end{equation} 
\end{definition}

\rev{
The equations \eqref{79-pre}, written out in full, are
\begin{equation}\label{basic-m-e} \sum_{j = m+1}^r (c^*_j - c^*_{j+1}) \HH_{\mu^*_j}(V_m) = \sum_{j = m+1}^r c^*_j \dim (V_j/V_{j-1})\qquad m = 0,1,\dots, r-1.\end{equation} 
By \eqref{non-degen}, this uniquely determines $c^*_{m+1} \in \R$ in terms of $c_{m+2}^*,\ldots,c^*_{r+1}$. Hence, we recursively determine $c_1,\cdots,c_r$ in terms of $c_{r+1}$. Since we must further have $c_1=1$, this implicitly determines $c_{r+1}$ as well, and thus the entire vector $\cc^*$.
}

\textbf{Remark.}
\rev{By Lemma \ref{gammak-tilde:reduction} (ii)}, a stronger form of the condition \eqref{non-degen} is required in order for the entropy gap condition to hold, and so in practice this assumption is not at all restrictive.

We conclude this subsection with a characterization of the
optimal measure  $\mu^*$ and parameters $\cc^*$.
Given an $r$-step flag $\sV$, there
is an associated rooted tree $\sT(\sV)$, which captures the
structure of the cells at different levels $0,\ldots,r-1$.
In particular,
this tree always has exactly $2^k-1$ leaves at level $0$, corresponding
to the cell $\Gamma_0=\{\mathbf{0},\mathbf{1}\}$ and the 
singletons $\{\omega\}$ for each
 $\omega \in \{0,1\}^k \setminus \{\mathbf{0},\mathbf{1}\}$.

\begin{lemma}\label{trees}
The optimal constant $\gamma_k^{res}(\sV)$, associated measures
$\mu^*_i(C)$ and optimal parameters $c_i^*$  depend only on the
tree $\sT(\sV)$ and the sequence of dimensions $\dim(V_j)$,  $0\le j\le r$.
\end{lemma}

\begin{proof}
Let $\sV$ and $\widetilde{\sV}$ be different flags with the same tree structure,
that is, $\sT(\sV)$ is isomorphic to $\sT(\widetilde{\sV})$, \rev{and with the same sequence of dimensions $\dim(V_j)$ and $\dim(V_j')$}.  By an easy induction on the level and the definition of $f^C(\boldsymbol{\rho})$, 
if $C\in \sT(\sV)$ and $\tilde{C}\in \sT(\widetilde{\sV})$ correspond, we find that 
$f^C(\boldsymbol{\rho}) = f^{\widetilde{C}} (\boldsymbol{\rho})$.
The statements now follow from Definitions \ref{opt-meas-def}
and \ref{opt-param-def}.
\end{proof}

\subsection{Solution of the optimisation problem: statement}

Here is the main result of this section, which explains the introduction of the various concepts above, as well as their names. 

\begin{proposition}\label{main-optim}
Suppose that $\sV : \mathbf{1} = V_0 \leq V_1 \leq \cdots \leq V_r \leq \Q^k$ is a \rev{nondegenerate flag} such that \rev{$\dim(V_1/V_0)=1$ and} the $\rho$-equations have a solution. Let \rev{$\bmu^*$ be the corresponding optimal measures}, and suppose that the corresponding optimal parameters $\cc^*$ exist. Then
\begin{equation}\label{res-opt} 
\gamma_k^{\res}(\sV) =  (\log 3 - 1)\Big/\bigg( \log 3 + \sum_{i = 1}^{r-1} \frac{\dim (V_{i+1}/V_i)}{\rho_1 \cdots \rho_i} \bigg) .
\end{equation}
Moreover, the \rev{optimal measures $\bmu^*$} and optimal parameters $\cc^*$ provide the solution to Problem \ref{restricted-opt}; in particular, $c^*_{r+1}$ is precisely the right-hand side of \eqref{res-opt}.
\end{proposition}

For this result to be of any use, we need methods for establishing, for flags $\sV$ of interest, that the $\rho$-equations have a solution, and also that the optimal parameters exist. The former is a very delicate matter, highly dependent on the specific structure of the flags of interest. Once this is sorted out, the latter problem is less serious, at least in situations relevant to us.

\bigskip

\subsection{Linear forms in entropies}

In the next section we will prove Proposition \ref{main-optim}. In this section we isolate some lemmas from the proof. 

Let $\sV : \langle \mathbf{1} \rangle = V_0 \leq \cdots \leq V_r \leq \Q^k$ be a flag. We use the terminology of cells $C$ at level $i$, introduced at the beginning of subsection \ref{rho-equations-sec}.

\begin{lemma}\label{hc-lem}
Let $\yy = (y_0,\cdots, y_{r-1})$ be real numbers with the property that all the partial sums $y_{< i} := y_0 + \dots + y_{i-1}$ are positive. If $C$ is a cell (at some level $i$), then we write 
\begin{equation}\label{hc-def} 
h^C(\yy) := \sup_{\Supp(\mu_C) \subset C} 
	\Big(\sum_{0\le m<r} y_m \HH_{\mu_C}(V_m)\Big) ,
\end{equation} 
where the supremum is over all probability measures $\mu_C$ supported on $C$. 

\begin{enumerate}
	\item The quantities $h^C(\yy)$ are completely determined by the following rules:
		\begin{itemize}
		\item If $C$ has level $0$, then $h^C(\yy) = 0$;
		\item If $C$ has level $i$, then
		\begin{equation}\label{hc-eqs} 
		h^C(\yy) = y_{< i}\log\Big(\sum_{C':\, C \rightarrow C'} e^{h^{C'}(\yy)/y_{< i}}\Big).
		\end{equation}
		\end{itemize}
	\item For any $C$, the maximum in \eqref{hc-def} occurs for a unique
	measure $\mu^*_{C,\yy}$.  Furthermore, all of the $\mu^*_{C,\yy}$ are restrictions of the ``top'' measure $\mu^*_{\yy} := \mu^*_{\Gamma_r,\yy}$, that is to say $\mu^*_{C,\yy}(x) = \mu^*_{\yy}(x)/\mu^*_{\yy}(C)$ for all $x \in C$, and
	\begin{equation}\label{compat} \frac{\mu^*_{\yy}(C')}{\mu^*_{\yy}(C)} = \frac{e^{h^{C'}(\yy/y_{<i})}}{e^{h^{C}(\yy/y_{<i})}}.  \end{equation}
\end{enumerate}
\end{lemma}

\begin{remark*} As will be apparent from the proof, we do not use the linear structure of the cells $C$ (that is, the fact that they come from cosets). We leave it to the reader to formulate a completely general version of this lemma in which the cells at level $i$ are the atoms in a $\sigma$-algebra $\mathscr{F}_i$, with $\mathscr{F}_{i}$ being a refinement of $\mathscr{F}_{i+1}$ for all $i$.
\end{remark*}

\begin{proof}
\rev{We prove both parts simultaneously.} 
Let us temporarily write $\tilde h^C(\yy)$ for the function defined by \eqref{hc-eqs}, thus the aim is to prove that $h^C(\yy) = \tilde h^C(\yy)$, \rev{where $h^C(\yy)$ is
defined in \eqref{hc-def}.} We do this by induction on $i$,
the $i=0$ case being trivial since, in this case, all the entropies $\HH_{\mu_C}(V_m)$ are zero \rev{because each cell of level 0 lies in some coset mod $V_0$, and thus in the same coset mod $V_m$ for $m=0,1,\dots,r-1$.}

Suppose now that we know the result for cells of level $i -1$. Note that both $h^C$ and $\tilde h^C$ satisfy a homogeneity property
\[ \tilde h^C(t\yy) = t \tilde h^C(\yy), \qquad h^C(t\yy) = t h^C(\yy).\] This is obvious for $h^C$, and can be proven very easily for $\tilde h^C$ by induction. Therefore we may assume that $y_{< i} = 1$. This does not affect the measure $\mu^*_{\yy}$, which does not depend on the scaling of the parameters $y_m$.

Suppose that $C$ is a cell at level $i$. A probability measure $\mu_C$ on $C$ is completely determined by probability measures $\mu_{C'}$ on the children $C'$ of $C$ (at level $i - 1$) together with the probabilities $\mu_C(C')$, which must sum to $1$, with the relation being that $\mu_{C'}(x) = \mu_C(x)/\mu_C(C')$ for $x\in C'$.

Suppose that $0\le m < i$. Let the random variables $X,Y$ be random cosets of $V_m, V_{i-1}$ respectively, sampled according to the measure $\mu_C$. Then $X$ determines $Y$ and so, by Lemma \ref{det-ent}, $\HH(X,Y) = \HH(X)$. The chain rule for entropy, Lemma \ref{ent-chain-rule}, then yields
\[ \HH(X) = \HH(Y) + \sum_{y} \P(Y = y) \HH(X | Y = y).\] Translated back to the language we are using, this implies that
 \[ 
\HH_{\mu_C}(V_{m}) = \HH_{\mu_C}(V_{i-1}) + \sum_{C'} \mu_C(C') \HH_{\mu_{C'}}(V_m).
\] 
Therefore
\[ 
\sum_{0\le m < i} y_m \HH_{\mu_C}(V_{m}) = \HH_{\mu_C}(V_{i-1}) + \sum_{C'} \mu_C(C') \sum_{0\le m < i} y_m \HH_{\mu_{C'}}(V_m).
\] 
(Here we used our assumption that $y_{< i} = 1$.)
Since $\HH_{\mu_C}(V_m) = 0$ for $m \geq i$, and $\HH_{\mu_{C'}}(V_m) = 0$ for $m \geq i - 1$, we may extend the sums over all $m\in\{0,1,\dots,r-1\}$ 
thereby obtaining
\[ 
\sum_{0\le m<r} 
y_m \HH_{\mu_C}(V_{m}) 
	= \HH_{\mu_C}(V_{i-1}) 
	+ \sum_{C'} \mu_C(C') \sum_{0\le m<r} y_m \HH_{\mu_{C'}}(V_m).
\]
Since the $\mu_{C'}$ can be arbitrary probability measures, and $\HH_{\mu_C}(V_{i-1})$ depends only on the value of $\mu_C(C')$, 
it follows from the inductive hypothesis that
\begin{align}
 h^C(\yy) 
 	& = \sup_{\mu_C}\Big( \sum_{0\le m<r} y_m \HH_{\mu_C}(V_{m}) \Big)  \label{eq:h-sup1}\\ 
 	& = \sup_{\mu_C(C'), \mu_{C'}}\Big( \HH_{\mu_C}(V_{i-1}) + \sum_{C'} \mu_C(C') \sum_{0\le m<r} y_m \HH_{\mu_{C'}}(V_m)\Big) \label{eq:h-sup2}\\ 
 	& = \sup_{\mu_C(C')}\Big( \HH_{\mu_C}(V_{i-1}) + \sum_{C'} \mu_C(C') \tilde h^{C'}(\yy)\Big),\label{eq:h-sup3}
 \end{align} 
with equality when going from \eqref{eq:h-sup2} to \eqref{eq:h-sup3} when $\mu_{C'} = \mu^*_{C',\yy}$ for all $C'$. 
Applying Lemma \ref{real-var-ent} with the $p_j$ being the $\mu_C(C')$ and the $a_j$ being the $\tilde h^{C'}(\yy)$, and noting that $\HH_{\mu_C}(V_{i-1}) = \HH(\mathbf{p})$ (where $\mathbf{p} = (p_1,p_2,\dots)$), 
it follows that
\begin{equation}
\label{eq:h-optimal}
\sup_{\mu_C(C')}\Big( \HH_{\mu_C}(V_{i-1}) + \sum_{C'} \mu_C(C') \tilde h^{C'}(\yy)\Big)
	= \log \Big( \sum_{C' : \, C \rightarrow C'} e^{\tilde h^{C'}(\yy)} \Big) 
	= \tilde h^C(\yy).
\end{equation}
In addition, Lemma \ref{real-var-ent} implies that equality occurs in \eqref{eq:h-optimal} precisely when $p_j = e^{a_j}/\sum_i e^{a_i}$, that is to say when
\[
\mu_C(C') =  \frac{e^{h^C(\yy)}}{\sum_{C':\, C \rightarrow C'} e^{h^C(\yy)}} = \frac{\mu^*_{\yy}(C')}{\mu^*_{\yy}(C)} .
\]
(Here we used again that $y_{<i}=1$.) Recalling that $\mu_{C'} = \mu^*_{C',\yy}$ for all $C'$, we see that 
the measure $\mu_C$ for which equality occurs in \eqref{eq:h-sup1} is the restriction of $\mu^*_{\yy} = \mu^*_{\Gamma_r,\yy}$ to $C$. This completes the inductive step. 
\end{proof}

\subsection{Solution of the optimisation problem: proof}
\label{sec65}
This section is devoted to the proof of Proposition \ref{main-optim}. Strictly speaking, for our main theorems we only need a lower bound on $\gamma_k^{\res}(\sV)$, and for this it suffices to show that $c_{r+1}^*$ is given by the right-hand side of \eqref{res-opt}. This could, in principle, be phrased as a calculation, but it would look complicated and unmotivated. Instead, we present it in the way we discovered it, by showing that the RHS of \eqref{res-opt} is an \emph{upper} bound on $\gamma_k^{\res}(\sV)$, and then observing that equality does occur when $\mu = \mu^*$ is the optimal measure (Definition \ref{opt-meas-def}) and $\cc = \cc^*$ the optimal parameters (Definition \ref{opt-param-def}). We establish this upper bound using the duality argument from linear programming and Lemma \ref{hc-lem}.

To ease the notation, we use the shorthand $d_i := \dim (V_i)$ throughout this subsection.
%It makes the argument look a tiny bit prettier if we use the shorthand $d_i := \dim (V_i)$, which we do throughout this subsection.
Let us, then, consider the restricted optimisation problem, namely Problem \ref{restricted-opt}. The condition \eqref{rest-ent} may be rewritten as
\begin{equation}\label{o-eq-1} 
\sum_{j = m+1}^r (c_j - c_{j+1})(\HH_{\mu_j}(V_m) + d_m - d_j) \geq c_{r+1} (d_r - d_m) \qquad (0\le m\le r-1).
\end{equation} This holds for $m = 0,1,\dots, r-1$. Therefore for any choice of ``dual variables'' $\yy = (y_0,y_1,\dots,$ $y_{r-1})$, $y_0,\cdots, y_{r-1} \geq 0$, we have
\begin{equation}\label{o-eq-2} \sum_{m = 0}^{r-1} y_m \sum_{j = m+1}^r (c_j - c_{j+1})(\HH_{\mu_j}(V_m) + d_m - d_j) \geq c_{r+1} \sum_{m = 0}^{r-1} y_m (d_r - d_m),\end{equation} which, rearranging, gives
\begin{equation}\label{o-eq-3} \sum_{j = 1}^r (c_j - c_{j+1}) E_j(\yy) + c_{r+1}E_{r+1}(\yy)  \geq c_{r+1}.\end{equation}
where 
\[ 
E_j(\yy) := \sum_{m = 0}^{j-1} y_m (\HH_{\mu_j}(V_m) + d_m - d_j)
\] 
for $j = 1,\dots, r$, and
\[ 
E_{r+1}(\yy) := 1 - \sum_{m = 0}^{r-1} y_m (d_r - d_m).
\]
Since the $c_j - c_{j+1}$, $j = 1,\dots, r$, and $c_{r+1}$ are nonnegative and sum to 1, this implies that 
\begin{equation}\label{o-eq-4} 
c_{r+1} \leq  \min_{y_i\ge 0\; \forall i} \max \{ E_1(\yy),\cdots, E_r(\yy), E_{r+1}(\yy) \}.
\end{equation}
By Lemma \ref{hc-lem}, this implies that
\begin{equation}\label{o-eq-5} 
c_{r+1} \leq \min_{y_i\ge 0\; \forall i} \max \{ E'_1(\yy), \cdots E'_r(\yy), E_{r+1}(\yy)\}, 
\end{equation} 
where
\begin{equation}\label{ej-def} 
E'_j(\yy) 
		:=  h^{\Gamma_j}(\yy) +  \sum_{m = 0}^{j-1} y_m ( d_m - d_j) 	
		= \sum_{m = 0}^{j-1} y_m (\HH_{\mu^*_{\Gamma_j,\yy}}(V_m) + d_m - d_j),
\end{equation} 
for $j = 1,\dots, r$, and $\mu^*_{\Gamma_j,\yy}$ is the measure $\nu$ supported on $\Gamma_j = V_j \cap \{0,1\}^k$ for which the sum $\sum_m y_m \HH_{\nu}(V_m)$ is maximal, as defined in Lemma \ref{hc-lem}.

Now we specify a choice of $\yy$. To do this, we make a change of variables, defining $\rho_i = y_{< i}/y_{< i+1}$. Note that for fixed $y_0 > 0$, choices of $y_1,\cdots, y_{r-1}> 0$ are in one-to-one correspondence with choices of $\rho_1,\cdots, \rho_{r-1}$ with $0 < \rho_i < 1$. 
We must then have that
\begin{equation}\label{fh-rel} 
\log f^C({\bm \rho}) = h^C(\yy/y_{<i})
 = \frac{1}{y_{< i}} h^C(\yy) = \frac{\rho_1 \cdots \rho_{i-1}}{y_0} h^C(\yy)
\end{equation} 
for the cells $C$ at level $i$, which may easily be proven by induction on the level $i$, using the defining equations for the $h^C$ and $f^C$ (see \eqref{hc-eqs}, \eqref{fc-eqs} respectively).

Now choose the $\rho_i$ to satisfy the $\rho$-equations \eqref{rho-equations}. In virtue of \eqref{fh-rel}, the $j$-th $\rho$-equation 
\[ 
f^{\Gamma_{j+1}}({\bm \rho}) = (f^{\Gamma_{j}}({\bm \rho}))^{\rho_j} e^{d_{j+1}  - d_j}
\] 
with $j\in\{1,2,\dots,r-1\}$ is equivalent to 
\begin{equation}\label{ej-ej1} 
E'_j(\yy) = E'_{j+1}(\yy),
\end{equation} 
with $E'_j(\yy)$ defined as in \eqref{ej-def} above.

Recall that $d_1-d_0=\dim(V_1/V_0)=1$. Thus, if we choose
\[ 
y_0 := 1\Big/\Big( \log 3 + \sum_{i = 1}^{r-1} \frac{d_{i+1} - d_i}{\rho_1 \cdots \rho_i} \Big),
\] 
a short calculation confirms that
\begin{equation}\label{e01} 
E_{r+1}(\yy) = E'_1(\yy) = y_0 (\log 3-1).
\end{equation}
With this choice of $\yy$ we therefore have, from \eqref{ej-ej1} with $j = 1,\dots, r-1$, \eqref{e01} and \eqref{o-eq-5}, 
\begin{equation}\label{o-eq-6} 
c_{r+1} \leq 
E'_1(\yy) = (\log 3-1) \big/\Big( \log 3 + \sum_{i = 1}^{r-1} \frac{d_{i+1} - d_i}{\rho_1 \cdots \rho_i} \Big).
\end{equation}

In the above analysis, the $\mu_i$ and the $c_i$ were arbitrary subject to the conditions of Problem \ref{restricted-opt}, thus $\Supp(\mu_i) \subset V_i$ and $1 = c_1 > c_2 > \cdots > c_{r+1}$. Therefore, recalling the definition of $\gamma_k^{\res}(\sV)$ (see Problem \ref{restricted-opt}), we have proven that 
\[ 
\gamma_k(\sV) \le \gamma_k^{\res}(\sV) 
	\leq (\log 3-1) \big/\Big( \log 3 + \sum_{i = 1}^{r-1} \frac{d_{i+1} - d_i}{\rho_1 \cdots \rho_i} \Big).
\] 
Proposition \ref{main-optim} asserts that equality occurs in this bound when $c_j = c^*_j$ and $\mu_j = \mu^*_j$, where $\cc^* = (c_1^*,\cdots, c^*_{r+1})$ are the optimal parameters defined in Definition \ref{opt-param-def}, and $\mu^*$ and its restrictions $\mu^*_j$ are the optimal measures defined in Definition \ref{opt-meas-def}. To establish this, we must go back through the argument showing that equality occurs at every stage with these choices.

First note that \eqref{o-eq-1} is equivalent (as we stated at the time) to $\er(\sV'_{\basic(m)},\cc,\bmu) \geq \er(\sV,\cc,\bmu)$. The fact that equality occurs here when $\cc = \cc^*$ and $\bmu = \bmu^*$ is essentially the definition of the optimal parameters $\cc^*$ (Definition \ref{opt-param-def}). That equality occurs in \eqref{o-eq-2} and \eqref{o-eq-3} is then automatic.

Working from the other end of the proof, the choice of $\yy$ was made so that $E'_1(\yy) = \cdots = E'_r(\yy) = E_{r+1}(\yy)$. We claim that, with this choice of $\yy$, 
\begin{equation}\label{mu-muy} 
\mu^* = \mu^*_{\yy}.
\end{equation}
By \eqref{compat}, it suffices to check that 
\[ 
\frac{\mu^*(C')}{\mu^*(C)} = \frac{e^{h^{C'}(\yy/y_{<i})}}{e^{h^{C}(\yy/y_{<i})}}.  
\]
This follows immediately from \eqref{eq46} and \eqref{fh-rel}.

Since $\mu^*_j$ is defined to be the restriction of $\mu^*$ to $\Gamma_j$, it follows from \eqref{mu-muy} that $\mu^*_j = \mu^*_{\Gamma_j,\yy}$, and hence that $E_j(\yy) = E'_j(\yy)$ for $j = 1,\dots, r$.

Thus all $2r + 1$ of the quantities $E'_j(\yy)$ ($j = 1,\dots, r$) and $E_j(\yy)$ ($j = 1,\dots, r+1$) are equal. It follows from this and the fact that equality occurs in \eqref{o-eq-3} that equality occurs in \eqref{o-eq-4}, \eqref{o-eq-5} and \eqref{o-eq-6} as well. This concludes the proof of Proposition \ref{main-optim}.\qed

%
% local macros, this section only
% 
\renewcommand{\semibasic}{\operatorname{semi}}

\section{The strict entropy condition}\label{entropy-binary-alt}

\subsection{Introduction}

Fix an $r$-step, \rev{nondegenerate flag $\sV$}. In the previous section, we studied a restricted optimization problem (Problem \ref{restricted-opt}) asking for the supremum of $c_{r+1}$ when ranging over all systems $(\sV,\cc,\bmu)$ satisfying the ``restricted entropy condition''
\begin{equation}\label{rest-ent-again}
\er(\sV'_{\basic(m)},\cc,\bmu) 
	\geq \er(\sV,\cc,\bmu) \qquad ( m = 0,1,\dots, r-1). 
\end{equation}
The aim of the present section is two-fold: we wish to establish, under general conditions, that an ``optimal system'' with respect to  \eqref{rest-ent-again} satisfies the more general entropy condition
\begin{equation}\label{gen-ent-again}
\er(\sV', \cc,\bmu) \geq \er(\sV, \cc,\bmu)
\qquad (\text{all } \sV'\le \sV).
\end{equation}
In addition, we want to show that if we slightly perturb such a system, we may guarantee 
\rev{the strict entropy condition \eqref{strict-entropy-cond}, which is}
a version of \eqref{gen-ent-again} with strict inequalities for
 all proper subflags $\sV'$ of $\sV$. 

Before stating our result, we need to define the notion of the automorphism group of a flag.  

\newcommand\Aut{\operatorname{Aut}}
\begin{definition}[Automorphism group]
For a permutation $\sigma \in S_k$ and $\om=(\om_1,\ldots,\om_k)\in \Q^k$, denote by $\sigma \om$ the usual coordinate permutation action $\sigma \om=
(\omega_{\sigma(1)},\ldots,\omega_{\sigma(k)})$.
The \emph{automorphism group} $\Aut(\sV)$ is the group of all $\sigma$  that satisfy $\sigma V_i=V_i$ for all $i$.
\end{definition}

\begin{proposition}\label{entrop-gap-general}
Let $\sV$ be an $r$-step, nondegenerate flag of distinct spaces. 
Assume that the $\rho$-equations \eqref{rho-equations} have a solution, and define the optimal measures $\bmu^*$ on $\{0,1\}^k$ as in Definition \ref{opt-meas-def}. 
Furthermore, assume that:
\begin{enumerate}
\item no intermediate subspace is fixed by $\Aut(\sV)$, that is to say there is no space $W$ that is invariant under the action of $\Aut(\sV)$ 
and such that $V_{i-1} < W <V_i$ (the inclusions being strict);
\item the optimal parameters $\cc^*$ exist \rev{and they are distinct and positive}, that is to say the system
of equations \eqref{basic-m-e} has a unique solution $\cc^*$ satisfying
$1=c_1^* > c_2^* > \cdots > c_{r+1}^*\rev{>0}$;
\item the following ``positivity inequalities'' hold:
\begin{enumerate}
	\item[(i)] $\HH_{\mu^*_{m+1}}(V_m)>\dim(V_{m+1}/V_m)$ for $0\le m\le r-1$;
	\item[(ii)] $\HH_{\mu^*_i}(V_{m-1})-\HH_{\mu^*_i}(V_m) < \dim(V_m/V_{m-1})$ for $1\le m<i\le r$.
\end{enumerate}
\end{enumerate}
Then, \rev{for every $\eps>0$}, there exists a perturbation $\tilde{\cc}$ of $\cc^*$  such that $1 = \tilde c_1  > \tilde c_2 >\cdots >  \tilde c_{r+1}\rev{\ge c_{r+1}-\eps}$ and such that we have the strict entropy condition 
\be\label{strict-perturbed}
\er(\sV',\tilde{\cc},\bmu^*) > \er(\sV,\tilde{\cc},\bmu^*) \qquad \mbox{for all proper subflags $\sV' \le \sV$.}
\ee
\end{proposition}

We assume throughout the rest of the section that (a), (b) and (c) of Proposition \ref{entrop-gap-general} are satisfied, and we now fix the system $(\sV,\cc^*,\bmu^*)$. For notational brevity in what follows, we write
\[ 
\er(\sV') := \er(\sV',\mathbf{c}^*,\bmu^*).
\] 

Our strategy is as follows. First, we show the weaker ``unperturbed'' statement that 
\be\label{unperturbed}
\er(\sV') \geq \er(\sV) \qquad \mbox{for all subflags $\sV' \le \sV$,}
\ee noting that we have strict inequality for certain subflags $\sV'$ along the way.
Then, in subsection \ref{perturb-sec}, we show how to perturb $\cc^*$ to $\tilde{\cc}$ so that the strict inequality \eqref{strict-perturbed} is satisfied. We also sketch a second way of effecting the perturbation which is in a sense more robust, but which in essence requires a perturbation of the whole proof of \eqref{unperturbed}.

\subsection{Analysis of non-basic flags}

We turn now to the task of proving \eqref{unperturbed}. We will prove it for progressively wider sets of subflags $\sV'$, each time using the previous statement. In order, we will prove it for subflags $\sV'$ which we call:
\begin{enumerate}
\item \emph{semi-basic}: \rev{flags $\sV':V_0\le V_1 \leq V_2 \leq \cdots \leq V_{m-1} \leq \cdots \leq V_{m-1} \leq V_m \leq \cdots \leq V_m$} with $m\ge 1$ (that is, $\sV'$ is like a basic flag, but there can be more than one copy of $V_{m-1}$);
\item \emph{standard}: each $V'_i$ is one of the spaces $V_j$;
\item  \emph{invariant}: this means that $\sigma V'_i = V'_i$ for all automorphisms $\sigma \in \Aut(\sV)$ and all $i$;
\item general subflags, i.e. we assume no restriction on the $V'_i$ other than that $V'_i \leq V_i$.
\end{enumerate}

Note that a semi-basic flag is standard, a standard flag is invariant, and of course an invariant flag is general. 

\medskip

We introduce some notation for standard flags. Let $J \subset \N_0^r$ be the set of all $r$-tuples $\bs j = (j_1,\cdots, j_r)$ such that $j_1 \leq \cdots \leq j_r$ and $j_i \leq i$ for all $i$. Then we define the flag $\sV'_{\bs j} = \sV'_{(j_1,\dots, j_r)}$ to be the one with $V'_i = V_{j_i}$. This is a standard flag, and conversely every standard flag is of this form. 
If we define 
\[ 
\basic(m) := (1,2,\dots, m-1, m,\dots, m)  
\] 
then $\basic(m) \in J$, and $\sV'_{\basic(m)}$ agrees with our previous notation.

\subsection{Semi-basic subflags}
\label{sec83}
In this subsection we prove the following result, establishing that \eqref{unperturbed} holds for semi-basic subflags, and with strict inequality for those which are not basic. 

\begin{lemma}\label{semi-basic-ineq} (Assuming that (a), (b) and (c) of Proposition \ref{entrop-gap-general} hold) we have
$\er(\sV') > \er(\sV)$ for all non-basic, semi-basic flags $\sV'$.
\end{lemma}

We begin by setting a small amount of notation for semi-basic flags. We note that the idea of a semi-basic flag, which looks rather \emph{ad hoc}, will only be used here and in subsection \ref{standard-sec}.

\begin{definition}[Semi-basic flags that are not basic]
Suppose that $1 \leq m \leq r - 1$ and that $m \leq s \leq r-1$. Then we define the element $\semibasic(m,s) \in J$ to be $\bs j = (1,2,\dots, m-1,m-1,\dots, m,\dots, m)$ such that $j_i = i$ for $i \leq m-1$, $j_i = m-1$ for $m \leq i \leq s$ and $j_i = m$ for $i > s$.
\end{definition}
It is convenient and natural to extend the notation to $s = m-1$ and $s = r$, by defining
\be\label{extend}
 \semibasic(m,r) = \basic(m-1), \qquad  \semibasic(m, m-1) = \basic(m).
 \ee
One can think of the semi-basic flags as interpolating between the basic flags.

\begin{example*}
When $r = 3$ there are three semi-basic flags $\sV_{\bs j}$ that are not basic, corresponding to  
\[ 
\bs j = \semibasic(1,1)  = (0,1,1),
\]
\[ 
\bs j = \semibasic(1,2) = (0,0,1),\]\[  \bs j = \semibasic(2,2) = (1,1,2).
\]
\end{example*}

\begin{proof}[Proof of Lemma \ref{semi-basic-ineq}]
Assume that $\sV'$ is semi-basic but not basic.
We will show that 
\begin{equation}\label{chain} \er(\sV'_{\semibasic(m,s)}) > \er(\sV'_{\semibasic(m,s+1)})\end{equation} for $m \leq  s \leq  r-1$. 
Since $\sV'_{\semibasic(m ,r)} = \sV'_{\basic(m-1)}$ is basic, this establishes Lemma \ref{semi-basic-ineq}.

To prove \eqref{chain}, we simply compute that
\begin{align*}
 \er(\sV'_{\semibasic(m,s)}) - \er(\sV'_{\semibasic(m,s+1)}) 
  = (c^*_{s+1} - c^*_{s+2}) 
  	\big[ \HH_{\mu_{s+1}}(V_m) - \HH_{\mu_{s+1}}(V_{m-1})   + \dim(V_m/V_{m-1}) \big]
\end{align*}
 when $m \leq  s \leq  r-2$, and
\begin{align*}
 \er(\sV'_{\semibasic(m,r-1)}) - \er(\sV'_{\semibasic(m,r)}) 
& = (c^*_r - c^*_{r+1}) 
	\big[ \HH_{\mu_r}(V_m)  - \HH_{\mu_r}(V_{m-1})  +  \dim(V_m/V_{m-1}) \big] \\
&\qquad\qquad+ \dim(V_m/V_{m-1}) c^*_{r+1}.
\end{align*}
In both cases, the result follows from \rev{part (ii) of condition(c)} of Proposition \ref{entrop-gap-general}; \rev{in the second case, we also need to use our assumption that $c_{r+1}^*\ge0$.}
\end{proof}

\subsection{Submodularity inequalities}\label{sec:submodularity}

To proceed further, we make heavy use of a submodularity property of the expressions $\er()$. 

Suppose that $\sV', \tilde{\sV}'$ are two subflags of $\sV$. We can define the \emph{sum} $\sV' + \tilde{\sV}'$ and \emph{intersection} $\sV' \cap \tilde{\sV}'$ by
\[ (\sV' + \tilde{\sV}')_i := V'_i + \tilde V'_i\] and
\[ (\sV' \cap \tilde{\sV}')_i := V'_i \cap \tilde V'_i.\]
Both of these are indeed subflags of $\sV$. 

\begin{lemma}\label{sub-mod-e}
We have
\[ \er(\sV') + \er(\tilde{\sV}') \geq \er(\sV' + \tilde{\sV}') + \er(\sV' \cap \tilde{\sV}').\]
\end{lemma}

\begin{proof} We first note that the entropies $\HH_{\mu}(W)$ satisfy a submodularity inequality. Namely, if $W_1, W_2$ are subspaces of $\Q^k$ and $\mu$ is a probability
 measure then
\begin{equation}\label{ent-submodular} \HH_{\mu}(W_1) + \HH_{\mu}(W_2) \geq \HH_
{\mu}(W_1 \cap W_2) + \HH_{\mu}(W_1 + W_2).\end{equation}
To prove this, consider the following three random variables:
\begin{itemize}
\item $X$ is a random coset of $W_1 + W_2$, sampled according to the measure $\mu$;
\item $Y$ is a random coset of $W_1$, sampled according to the measure $\mu$;
\item $Z$ is a random coset of $W_2$, sampled according to the measure $\mu$.
\end{itemize}
Then, more-or-less by definition, 
\[ \HH(X) = \HH_{\mu}(W_1 + W_2),\quad \HH(Y) = \HH_{\mu}(W_1),\quad \HH(Z) = \HH_{\mu}(W_2).\]
Note also that $Y$ determines $X$ and so $\HH(Y) = \HH(X,Y)$, and similarly $\HH(Z) = \HH(X,Z)$. Finally, $(Y,Z)$ uniquely defines a random coset of $W_1 \cap W_2$, and so
\[ \HH_{\mu}(W_1 \cap W_2)  = \HH(Y,Z) = \HH(X,Y,Z).\]
The inequality to be proven, \eqref{ent-submodular} is therefore equivalent to
\[ \HH(X,Y) + \HH(X,Z) \geq \HH(X,Y,Z) + \HH(X),\] which is a standard entropy inequality (Lemma \ref{submodular-entropy}; usually known as ``submodularity of entropy'' or ``Shannon's inequality'' in the literature).

Lemma \ref{sub-mod-e} is essentially an immediate consequence of \eqref{ent-submodular} and the formula
\[ \dim (W_1) + \dim (W_2) = \dim(W_1 \cap W_2) + \dim (W_1 + W_2).\] (It is very important that this formula holds with \emph{equality}, as compared to \eqref{ent-submodular}, which holds only with an inequality.)
\end{proof}

This has the following immediate corollary when applied to standard subflags. Here, the $\max$ and $\min$ are taken coordinatewise.

\begin{corollary}\label{standard-submodular}
Suppose that $\bs j_1, \bs j_2 \in J$. Then
\[ \er(\sV'_{\bs j_1}) + \er(\sV'_{\bs j_2}) \geq  \er(\sV'_{\max(\bs j_1, \bs j_2)})  + \er(\sV'_{\min(\bs j_1, \bs j_2)}) \]
\end{corollary}

\subsection{Standard subflags}\label{standard-sec}

Now we extend the result of the subsection \ref{sec83} to all standard subflags.

\begin{lemma}\label{standard-ineq}
(Assuming that (a), (b) and (c) of Proposition \ref{entrop-gap-general} hold) we have $\er(\sV') > \er(\sV)$ for all standard, non-basic subflags $\sV'\le \sV$.
\end{lemma}
\begin{proof}
Let $\bs j \in J$ with $\bs j$ non-basic, and let
$\sV'=\sV'_{\bs j}$. Then $r\ge 3$, since when $r\le 2$ all
standard flags are basic. 
We proceed by induction on $\| \bs j\|_\infty$, the
case $\| \bs j\|_\infty=1$ being trivial, since then
$\sV$ is semibasic and we may invoke Lemma \ref{semi-basic-ineq}.
Now suppose we have proved $\er(\sV') > \er(\sV)$ 
for all non-basic standard flags $\sV'=\sV'_{\bs j}$
with $\|\bs j\|_\infty < m$, and let $\bs j \in J$ with
$\|\bs j\|_\infty=m$.
 We apply Corollary \ref{standard-submodular} with $\bs j_1 = \bs j$ and $\bs j_2 = \basic(j_r - 1)$. Noting that $\max(\bs j, \basic(j_r - 1)) = \semibasic(j_r, s)$, where $s$ is the largest index in $\bs j$ such that $j_s < j_r$, we see that
\begin{equation}\label{stand-sub}
	 \er(\sV'_{\bs j}) + \er(\sV'_{\basic(j_r - 1)} )\geq \er(\sV'_{\bs j_*}) + \er(\sV'_{\semibasic(j_r, s)}),
 \end{equation} 
where
\[ 
\bs j_* := \min(\bs j, \basic(j_r - 1)).
\]
 Suppose that both of the flags on the right of \eqref{stand-sub} are basic. If $\semibasic(j_r, s)$ is basic then it must be $\basic(j_r)$, which means that $s = j_r - 1$. But then $\bs j_* = (j_1,\dots, j_s, j_r - 1,\cdots j_r - 1)$ which, if it is basic, must be $\basic(j_r - 1)$; this then implies that $j_i = i$ for $1 \leq i \leq s$, and hence that $\bs j = \basic(j_r)$, a contradiction.
 Thus, at least one of the two flags $\bs j_*, \semibasic(j_r,s)$ on the right of \eqref{stand-sub} is not basic. Since $\Vert \bs j_* \Vert_{\infty} < \Vert \bs j \Vert_{\infty}=m$, the induction
 hypothesis together with Lemma \ref{semi-basic-ineq} implies that
 $\er(\sV') > \er(\sV)$, as desired.
\end{proof}

\subsection{Invariant subflags}

\label{special-sec}

Now we extend our results to all invariant flags, but now without the strict inequality.

\begin{lemma}\label{invariant-ineq-pre}
(Assuming that (a), (b) and (c) of Proposition \ref{entrop-gap-general} hold) we have $\er(\sV') \geq  \er(\sV)$ for all 
invariant subflags $\sV'\le \sV$. 
\end{lemma}

\begin{proof}
 We associate a pair $(i,\ell)$, $i\ge \ell$, of positive integers to $\sV'$, which we call the \emph{signature}, in the following manner. 
If $\sV'$ is standard, then set $(i,\ell)=(-1,-1)$. Otherwise, let $i$ be maximal so that $V'_i$ is not a standard space $V_t$, and then let $\ell$ be minimal such that $V'_i \leq V_{\ell}$. The fact that $\ell \leq i$ is immediate from the definition of a subflag.
We put a \rev{partial} ordering on signatures as follows: $(i', \ell') \preceq (i,\ell)$ iff $i' < i$, or if $i' = i$ and $\ell' \leq \ell$.
We proceed by induction on the pair $(i,\ell)$ with respect to this ordering, the case $(i,\ell)=(-1,-1)$ handled by 
Lemma \ref{standard-ineq}.

For the inductive step, suppose $\sV'$ is nonstandard 
with signature $(i,\ell)$.  
By submodularity, 
\begin{equation}\label{smm} \er(\sV') + \er(\sV'_{\basic(\ell-1)}) \geq \er(\sV_1)+\er(\sV_2),
\end{equation}
where 
\[
\sV_1 = \sV' \cap \sV'_{\basic(\ell-1)}, \qquad 
\sV_2 = \sV' + \sV'_{\basic(\ell-1)}.
\]
Suppose that $\sV_1, \sV_2$ 
have signatures $(i_1,\ell_1),(i_2,\ell_2)$, respectively.
We show that
\be\label{V1V2-sig}
(i_1,\ell_1)\precneqq (i,\ell) \quad \text{ and} \quad
(i_2,\ell_2)\precneqq (i,\ell).
\ee
Both $\sV_1$ and $\sV_2$ are invariant flags. Thus, if \eqref{V1V2-sig} holds, then both flags on the right-hand side of \eqref{smm} have strictly smaller signature than $\sV'$, and
the lemma follows by induction.
%Moreover, when $V_r'$ is a standard space $V_h$, we claim that
% $\sV_1$ is non-basic throughout the 
%induction, from which we deduce the strict inequality $\er(\sV')>\er(\sV)$
%from that for standard flags.  To prove the claim, suppose that
%$V_r'=V_h$ and $\sV_1 = \sV'_{\basic(m)}$.  Clearly $m\le \ell-1$
%and also $i<r$.
%In fact, since $V_i'$ is invariant, $V_i'$ does not contain 
%$V_{\ell-1}$ and hence $V_m = V_i' \cap V_{\ell-1}$ implies
%that $m\le \ell-2$. 
%But  $V_{i+1}'$ is a standard space,
%and $(\sV_1)_{i+1} = V_{i+1}'\cap V_{\ell-1}=V_{m}$,
%and hence $V_i' \le V_{i+1}' \le V_{\ell-2}$, a contradiction.
%

Finally, we prove \eqref{V1V2-sig}.
Note that if $j>i$, then $V_j'$ is a standard space $V_m$ and thus so are $(\sV_1)_j$ and $(\sV_2)_j$.
In particular, $i_1\le i$ and $i_2\le i$. 
We have that $(\sV_2)_i$ contains $V_{\ell-1}$, is not equal
to $V_{\ell-1}$, 
and is contained in $V_\ell$.  But $(\sV_2)_i$ is invariant,
 and hence by our assumption that (a) of Proposition \ref{entrop-gap-general} holds, $(\sV_2)_i=V_\ell$.  Consequently, $i_2<i$ if $\sV_2$ is nonstandard.
In the case that $\sV_1$ is nonstandard,
we also have that $\ell_1<\ell$ because every space in the 
flag $\sV_1$ is contained in $V_{\ell - 1}$. This proves 
\eqref{V1V2-sig}.
\end{proof}

\subsection{General subflags}\label{general-flag-sec}

In this section we establish \eqref{unperturbed}, that is to say the inequality $\er(\sV')\ge \er(\sV)$
 for \emph{all} subflags $\sV'$, of course subject to our standing assumption that (a), (b) and (c) of Proposition \ref{entrop-gap-general} hold. We need a simple lemma about the action of the automorphism group $\Aut(\sV)$ on subflags. 

\begin{lemma}\label{sigma-flag}
Let $\sigma \in \Aut(\sV)$ and let $\sV'$ be a subflag of $\sV$. Then one may define a new subflag $\sigma(\sV')$, setting $\sigma(\sV')_i := \sigma(V'_i)$. Moreover, $\er(\sigma(\sV')) = \er(\sV')$. 
\end{lemma}

\begin{proof}
Since $\sV'$ is a subflag, $V'_i \leq V_i$. Applying $\sigma$, and recalling that $V_i$ is invariant under $\sigma$, we see that $\sigma(V'_i) \leq V_i$. Therefore $\sigma(\sV')$ is also a subflag.  To see that $\er(\sigma(\sV')) = \er(\sV')$, recall Lemma
\ref{trees}, which implies that $\mu_i$ is invariant under $\sigma$, since 
 the trees $\sT(\sV')$
and $\sT(\sigma(\sV'))$ are isomorphic \rev{and we have $\dim(V_j')=\dim(\sigma(V_j'))$ for all $j$}. 
 It follows that, for any subspace $W \leq \Q^k$, 
\begin{align*}
\HH_{\mu_i}(\sigma(W)) 
	& = -\sum_x \mu_i(x) \log \mu_i(\sigma(W) + x) \\ 
	& = -\sum_y \mu_i(\sigma(y)) \log \mu_i(\sigma(W + y)) \\ 
	& = -\sum_y \mu_i(y) \log \mu_i(W + y) \\ 
	& = \HH_{\mu_i}(W).
\end{align*}
This completes the proof of the lemma.
\end{proof}

\begin{proof}[Proof of \eqref{unperturbed}]
Let $m$ be the minimum of $\er(\sV')$ over all 
subflags $\sV' \le \sV$, and among the flags
with $\er(\sV')=m$, take the one with $\sum_i \dim V'_i$ minimal.
Let $\sigma \in \Aut(\sV)$ be an arbitrary automorphism. By Lemma \ref{sigma-flag},  $\er(\sV') = \er(\sigma(\sV'))$, and hence submodularity implies that
\begin{equation}\label{sym-e} 
2\er(\sV') \geq \er(\sV' + \sigma(\sV')) + \er(\sV' \cap \sigma(\sV')). 
\end{equation} 
In particular, we have $\er(\sV \cap \sigma(\sV')) = m$ (and also $\er(\sV' + \sigma(\sV')) = \er(\sV)$, but we will not need this). Moreover, by the minimality of $\sum_i \dim V'_i$, 
\[ \sum_i \dim (V'_i \cap \sigma(V'_i)) = \sum_i \dim V'_i,\]
which means that $\sV'$ is invariant.
Invoking Lemma \ref{invariant-ineq-pre}, we conclude that $m=\er(\sV')\ge \er(\sV)$.
\end{proof}

\subsection{The strict entropy condition}\label{perturb-sec}
In this section we complete the proof of Proposition \ref{entrop-gap-general} by showing how to perturb \eqref{unperturbed} to the desired strict inequality \eqref{strict-perturbed}.

\noindent \emph{First argument.} Consider first the collection $\cU$ of all subflags $\sV'$ which satisfy, for some $1\le j\le r-1$,
the relations
\[
V_i' = V_i \;\; (i\ne j), \quad V_{j-1} \le V_{j'} < V_j.
\]
These are flags which differ from $\sV$ in exactly one space. Our first task will be to establish the \emph{strict} inequality
\be\label{entropy-intermediate}
\er(\sV') > \er(\sV)
\ee
for all $\sV' \in \cU$, by elaborating upon the 
argument of the previous subsection. We already know that $\er(\sV') \geq \er(\sV)$, so suppose as a hypothesis for contradiction that $\er(\sV') = \er(\sV)$ for some $\sV' \in \cU$. Amongst all such flags, take one with minimal
$\sum \dim(V_i')$.  By submodularity, we have \eqref{sym-e}
 and hence
$\er(\sV' \cap \sigma(\sV')) = \er(\sV)$
for any automorphism $\sigma\in \Aut(\sV)$.
But 
\[
\sV' \cap \sigma(\sV') = (V_1,\ldots,V_{j-1},V_{j}'\cap \sigma(V_j'),V_{j+1},\ldots,V_r)
\]
is evidently in $\cU$ as well,
and by our minimality assumption it follows that
$\dim (V_j' \cap \sigma(V_j'))=\dim(V_{j}')$.
Thus, $\sV'$ is invariant, and by assumption (a) of
Proposition \ref{entrop-gap-general}, it follows that
$V_j'=V_{j-1}$.  Thus, $\sV'$ is a standard flag, which
is not basic since $j\le r-1$.  Hence,
$\er(\sV')>\er(\sV)$ by Lemma \ref{standard-ineq}.
This contradition establishes \eqref{entropy-intermediate}.

Let $1\le j\le r-1$ and let $V$ be a space satisfying $V_{j-1}\le V<V_j$. Let $\sV'$ be the subflag  $\langle\mathbf{1}\rangle = V_0\le \ldots V_{j-1} \le V \le V_{j+1}\le\cdots \le V_r$. Then one easily computes that 
\[ \er(\sV') - \er(\sV) = (c_{j}-c_{j+1}) \big( \HH_{\mu_j}(V)-\dim(V_j/V) \big),\] and so \eqref{entropy-intermediate} implies that
\be\label{strict-intermediate}
 \HH_{\mu_j}(V) > \dim(V_j/V).
\ee
Now let $\eps>0$ be sufficiently small and consider the pertubation $\tilde{\cc}$ given by
\[
\tilde{c}_1=1, \quad \tilde{c}_j = c_j^*-\frac12\sum_{\ell=1}^{j-1} \eps^\ell \quad (2\le j\le r+1).
\]
Evidently, $1=\tilde{c}_1 > \tilde{c}_2 >\cdots > \tilde{c}_{r+1} \ge c^*_{r+1}-\eps$, \rev{as needed}. 
For any proper subflag  $\sV' \le \sV$,
\begin{multline*}
\er(\sV',\tilde{\cc},\bmu^*)  -\er(\sV,\tilde{\cc},\bmu^*) \\  = 
 \er(\sV')-\er(\sV) + \frac12\sum_{j=1}^r
\eps^j \big( \HH_{\mu_j}(V_j')-\dim(V_j/V_j')  \big)  + \frac12(\eps+\eps^2+\cdots+\eps^{\rev{r}})\dim(V_r/V_r').
\end{multline*}
Let $J=\min\{j : V_j' \ne V_j\}$. 
If $J=r$, then $\dim(V_r/V_r')\ge 1$ and the right side above is
at least $\eps/2 + O(\eps^r)$, which is positive for small enough 
$\eps$.  If $J\le r-1$,  then $V_{J-1} \le V_{J}' < V_J$
and 
 we see that the right side above is at least
\[
\er(\sV')-\er(\sV) + \eps^J \big( \HH_{\mu_J}(V_J')-\dim(V_J/V_J')  \big) + O(\eps^{J+1}),
\]
which is also positive for sufficiently small $\eps$ by \eqref{unperturbed} and \eqref{entropy-intermediate}.

\noindent \emph{Second argument.} We now sketch a second approach to the proof of Proposition \ref{entrop-gap-general}. The idea is to introduce a small perturbation of our fundamental quantity $\er()$, namely
\[ 
\er_{\lambda}(\sV', \cc, \bmu) 
	:= \lambda\sum_{j = 1}^r (c_{j+1} - c_j) \HH_{\mu_j}(V'_j) + \sum_{j = 1}^r c_j \dim (V'_j/V'_{j-1}),
\] where $\lambda \approx 1$.
Note that $\er_1(\sV',\cc,\bmu) = \er(\sV',\cc,\bmu)$, and also that $\er_{\lambda}(\sV,\cc,\bmu)$ does not depend on $\lambda$, since all the entropies $\HH_{\mu_j}(V_j)$ vanish. Define the $\lambda$-perturbed optimal parameters \rev{$\cc^*(\lambda)$ to be the unique solution to the $\lambda$-perturbed version of \eqref{79-pre}, that is to say the equations
$\er_{\lambda}(\sV'_{\basic(m)},\cc^*(\lambda),\bmu) 
= \er_{\lambda}(\sV,\mathbf{c}^*(\lambda),\bmu)$, $m = 0,1,\dots, r-1$. By a continuity argument, these exist for $\lambda$ sufficiently close to $1$ and they satisfy $\lim_{\lambda \rightarrow 1} \mathbf{c}^*(\lambda) = \mathbf{c}^*(1) = \mathbf{c}^*$.}

\rev{Now, assume that $\lambda$ is close enough to 1 so that $1 = c_1^*(\lambda) > c_2^*(\lambda) > \cdots > c_{r+1}^*(\lambda)>0$ and we have the following ``positivity inequalities'':
\begin{enumerate}
	\item[(i)] $\lambda\HH_{\mu^*_{m+1}}(V_m)>\dim(V_{m+1}/V_m)$ for $0\le m\le r-1$;
	\item[(ii)] $\lambda\cdot\big(\HH_{\mu^*_i}(V_{m-1})-\HH_{\mu^*_i}(V_m) \big)< \dim(V_m/V_{m-1})$ for $1\le m<i\le r$.
\end{enumerate}
These conditions can be clearly guaranteed by a continuity argument and our assumption that they hold when $\lambda=1$. For a parameter $\lambda$ satisfying (i) and (ii) above}, the proof of \eqref{unperturbed} holds verbatim for the $\lambda$-perturbed quantities $\er_\lambda$, allowing one to conclude that 
	\[
	\er_{\lambda}(\sV', \cc^*(\lambda),\bmu) \geq \er_{\lambda}(\sV,\cc^*(\lambda),\bmu)
	\] 
for all subflags $\sV'$ of $\sV$.

Now suppose that $\lambda < 1$. Then we have
\[ 
\er(\sV', \cc,  \bmu^*) \geq \er_{\lambda}(\sV',\cc, \bmu^*),
\] 
with equality if and only if $\sV' = \sV$ \rev{because $\Supp(\mu_j^*)=V_j\cap\{0,1\}^k$ for all $j$}. Therefore if $\sV'$ is a proper subflag of $\sV$ we have
\[
\er(\sV',\cc^*(\lambda),\bmu^*)  > \er_{\lambda}(\sV', \cc^*(\lambda),\bmu^*)
	 \geq \er_{\lambda}(\sV,\cc^*(\lambda),\bmu^*)  = \er(\sV, \cc^*(\lambda),\bmu^*).
\]
Taking $\tilde{\cc} = \cc^*(\lambda)$ for $\lambda$ sufficiently close to 1, Proposition \ref{entrop-gap-general} follows.

\newpage
\thispagestyle{fancy}
\fancyhf{} % sets both header and footer to nothing
\renewcommand{\headrulewidth}{0cm}
\lhead[{\scriptsize \thepage}]{}
\rhead[]{{\scriptsize\thepage}}
\part{Binary systems}\label{part:binary-systems}

\section{Binary systems and a lower bound for $\beta_k$}\label{binary-system}

In this section we define certain special flags $\sV$ on $\Q^k$, $k = 2^r$, which we call the \emph{binary systems of order $r$}. It is these systems which lead to the lower bound on $\beta_k$ given in Theorem \ref{beta-rho}, which is one of the main results of the paper. 

In this section we will define these flags (which is easy) and state their basic properties. The proofs of these properties, some of which are quite lengthy, are deferred to Section \ref{binary-calcs}. 

We are then in a position to prove part of one of our main theorems, Theorem \ref{beta-rho} (a), which we do in subsection \ref{theta-r}. 

\rev{For the convenience of the reader, recall us here the three parts of Theorem \ref{beta-rho}, as stated at the end of subsection \ref{sec:equal sums}:
\begin{enumerate}
	\item[(a)] Showing that for every $r\ge 1$, $\beta_{2^r} \ge \theta_r$ for a certain explicitly defined constant $\theta_r$; 
	\item[(b)] Showing that $\lim_{r\to\infty} \theta_r^{1/r}$ exists;
	\rev{\item[(c)] Showing that \eqref{rho-eq-main} has a unique solution
		$\rho\in [0,1/3]$ and that $\rho=2\lim_{r\to\infty} \theta_r^{1/r}$.}
\end{enumerate}}

\subsection{Binary flags and systems: definitions and properties}\label{bin-props-sec}

\begin{definition}[Binary flag of order $r$]\label{binary-flag}
Let $k = 2^r$ be a power of two. Identify $\Q^k$ with $\Q^{\mathcal{P}[r]}$ (where $\mathcal{P}[r]$ means the power set of $[r] = \{1,\dots, r\}$) and define a flag $\sV$, $\langle \mathbf{1} \rangle = V_0 \leq V_1 \leq \cdots \leq V_r = \Q^{\mathcal{P}[r]}$, as follows: $V_i$ is the subspace of all $(x_S)_{S \subset [r]}$ for which $x_S = x_{S \cap [i]}$ for all $S \subset [r]$. 
\end{definition}

\begin{remark*} We have $\dim (V_i) = 2^i$, and $V_r = \Q^{\mathcal{P}[r]}$, so the system is trivially nondegenerate. Note that we have been using the letter $r$ to denote the number of $V_i$ in the flag $\sV$, throughout the paper. It just so happens that, in this example, this is the same $r$ as in the definition of $k = 2^r$.
\end{remark*}

One major task is to show that optimal measures and optimal parameters, as described in Section \ref{optimisation-sec}, may be defined on the binary flags. Since we will be seeing them so often, let us write down the $\rho$-equations \eqref{rho-equations} for the binary flags explicitly:

\begin{equation}
	\label{binary-rho-eqs} 
f^{\Gamma_{j+1}}(\brho) = f^{\Gamma_j}(\brho)^{\rho_j} e^{2^j}  , \quad \mbox{$j = 1,2,\dots$.} 
\end{equation}

\begin{proposition}\label{rho-binary-prop}
Let $\sV$ be the binary flag of order $r$. Then 
\begin{enumerate}
\item the $\rho$-equations \eqref{binary-rho-eqs} have a solution with $0 < \rho_i < 1$ for $i \geq 1$, and consequently we may define the optimal measures $\bmu^*$ on $\{0,1\}^k$ as in Definition \ref{opt-meas-def};
\item  the optimal parameters $\mathbf{c}^*$ (in the sense of Definition \ref{opt-param-def}) exist. 
\end{enumerate} 
\end{proposition}

We call the binary flag $\sV$ (of order $r$) together with the additional data of the optimal measures $\mu = \mu^*$ and optimal parameters $\mathbf{c} = \mathbf{c}^*$, the binary \emph{system} (of order $r$). 
We caution that for fixed $i$ (such as $i = 2$) the parameters $c_i$ do depend on $r$, although not very much.

\medskip
The second major task is to show that the binary systems satisfy the entropy condition \eqref{entropy-cond}, or more accurately that arbitrarily small perturbations of them satisfy the strict entropy condition \eqref{strict-entropy-cond}. In the last section we provided a tool for doing this in somewhat general conditions, namely Proposition \ref{entrop-gap-general}. That proposition has four conditions, (a), (b), (c)(i) and (c)(ii) which must be satisfied. Of these, (b) (the existence of the optimal parameters $\mathbf{c}^*$) has already been established, \rev{assuming the validity of Proposition \ref{rho-binary-prop}}. We state the other three conditions separately as lemmas.

\begin{lemma}\label{inv-max}
Suppose that $V_{i-1} \leq W \leq V_i$ and that $W$ is invariant under $\Aut(\sV)$. Then $W$ is either $V_{i-1}$ or $V_i$. Thus, the binary flags satisfy Proposition \ref{entrop-gap-general} (a).
\end{lemma}

\begin{lemma}\label{binary-entropy-spec-1}
We have $\HH_{\mu_{m+1}^*}(V_m) > 2^m$ for $0 \leq m \leq r - 1$. Thus, the binary flags satisfy Proposition \ref{entrop-gap-general} (c)(i).
\end{lemma}

\begin{lemma}\label{binary-entropy-spec} We have $\HH_{\mu_i^*}(V_{m-1})-\HH_{\mu^*_i}(V_m) < 2^{m-1}$ for $1\le m< i\le r$. Thus, the binary flags satisfy Proposition \ref{entrop-gap-general} (c)(ii).
\end{lemma}

The proofs of these various facts are given in Section \ref{binary-calcs}.

\subsection{Proof of Theorems \ref{beta-rho} (a) and \ref{thm:beta-gamma}}\label{theta-r}

We are now in a position to complete the proof of Theorem \ref{beta-rho} (a), modulo the results stated above. First, we define the constants $\theta_r$.

\begin{definition}\label{theta-r-def}
Let $\rho_1,\rho_2,\dots$ be the solution to the $\rho$-equations \eqref{binary-rho-eqs} for the binary flag. Then we define
\[ \theta_r := (\log 3-1)\Big/\Big(\log 3 + \sum_{i = 1}^{r-1} \frac{2^i}{\rho_1 \cdots \rho_i}\Big).
\]
\end{definition}

\begin{proof}[Proof of Theorem \ref{beta-rho} (a)]
By Proposition \ref{main-optim}, $\theta_r$ is equal to $c^*_{r+1}$, where $\mathbf{c}^*$ are the optimal parameters on the binary flag $\sV$ of order $r$, the existence of which is Proposition \ref{rho-binary-prop} (b) above.

\rev{Fix $\delta\in(0,\theta_r/2]$. By Proposition \ref{entrop-gap-general} (the hypotheses of which are satisfied by Lemma \ref{inv-max}, Proposition \ref{rho-binary-prop} (b) and Lemmas \ref{binary-entropy-spec-1} and \ref{binary-entropy-spec}), there exists a perturbation $\tilde{\cc}$ of $\cc^*$ such that $1=\tilde{c}_1>\tilde{c}_2>\cdots>\tilde{c}_{r+1}\ge c_{r+1}^*-\delta=\theta_r-\delta>0$ and $(\sV,\tilde{\cc},{\bmu}^*)$ satisfies the strict entropy condition \eqref{strict-entropy-cond}. By Lemma \ref{lem:entropy-gap}, there exists some $\eps>0$ such that the ``entropy gap'' condition \eqref{entropy-gap} holds. Finally, by Remark \ref{rem:optimal measure} (b), we have that $\Supp(\mu_j^*)=\Gamma_j$ for all $j$. Hence, Proposition \ref{bgamk} implies that $\beta_{2^r} \ge \tilde{c}_{r+1}=\theta_r-\delta$. Since $\delta$ is arbitrary, this proves Theorem \ref{beta-rho} (a).}
\end{proof}

\rev{\begin{proof}[Proof of Theorem \ref{thm:beta-gamma}] The upper bound $\beta_k\le\gamma_k$ is established in Section \ref{sec:upper}. The lower bound $\beta_k\ge\tilde{\gamma}_k$ follows by Lemma \ref{gammak-tilde:reduction}, Proposition \ref{bgamk} and the fact that there exists at least one system satisfying the strict entropy condition \eqref{strict-entropy-cond}, as per the proof of Theorem \ref{beta-rho} (a) above.
\end{proof}}

\subsection{Remarks on Theorem \ref{beta-rho} (b)}

Theorem \ref{beta-rho} (b) is a problem of a combinatorial and analytic nature which can be considered more-or-less completely independently of the first three parts of the paper. 

To get a feel for it, and a sense of why it is difficult, let us write down the first two $\rho$-equations \eqref{binary-rho-eqs} for the binary flags. The 
equation with $j = 1$ is
\begin{equation}\label{first-rho} f^{\Gamma_2}(\rho) = f^{\Gamma_1}(\rho)^{\rho_1} e^2.\end{equation} This has the numerical solution $\rho_1 \approx 0.306481$.

To write down the $\rho$-equation for $j = 2$, one must compute $f^{\Gamma_3}(\rho)$, and without any additional theory the only means we have to do this is to draw the full tree structure for the binary flag $\sV$ of order $3$ (on $\Q^8$). This is a tractable exercise and one may confirm that
\[ f^{\Gamma_3}(\rho) = (3^{\rho_1} + 4 \cdot 2^{\rho_1} + 4)^{\rho_2} + 8(2 \cdot 2^{\rho_1} + 4)^{\rho_2} + 16 \cdot 4^{\rho_2} + 8 \cdot (2^{\rho_1} + 2)^{\rho_2} + 32 \cdot 2^{\rho_2} + 16.\]
The $\rho$-equation with $j = 2$ is then 
\[ f^{\Gamma_3}(\rho) = f^{\Gamma_2}(\rho)^{\rho_2} e^4,\]where \rev{(recall from Figure 7.1)} $f^{\Gamma_2}(\rho) = 3^{\rho_1} + 4 \cdot 2^{\rho_1} + 4$. This may be solved numerically, with the value $\rho_2 \approx 0.2796104\dots$, using Mathematica. 

Such a numerical procedure, however, is already quite an unappetising prospect if one wishes to compute $\rho_3$.

Consequently, we must develop more theory to understand the $\rho_i$ and to prove Theorem \ref{beta-rho} (b). This is the task of the last two sections of the paper.

\section{Binary systems: proofs of the basic properties}\label{binary-calcs}

In this section, we prove the various statements in subsection \ref{bin-props-sec}.

We begin, in subsection \ref{sec:binary-auto}, by proving Lemma \ref{inv-max}. This is a relatively simple and self-contained piece of combinatorics.

In subsection \ref{cell-enum} we introduce the concept of \emph{genotype}, which allows us to describe the tree structure induced on $\{0,1\}^k$ by the binary flag $\sV$. In subsection \ref{fc-genotype} we show how to compute the quantities $f^C(\bs\rho)$ in terms of the genotype.

We are then, in subsection \ref{rhoi-exist}, in a position to prove Proposition \ref{rho-binary-prop} (a), guaranteeing that the $\rho_i$ exist and allowing us to define the optimal measures $\bmu^*$.

In subsection \ref{sec:binary-entropy} we establish the two entropy inequalities, Lemmas \ref{binary-entropy-spec-1} and \ref{binary-entropy-spec}.

Finally, in subsection \ref{c-params-sec} we prove Proposition \ref{rho-binary-prop} (b), which confirms the existence of the optimal parameters $\cc^*$.

\subsection{Basic terminology}

Throughout the section, $\sV$ will denote the binary flag or order $r$, as defined in Definition \ref{binary-flag}. That is, we take $k = 2^r$, identify $\Q^k$ with $\Q^{\mathcal{P}[r]}$, and take $V_i$ to be the subspace of all $(x_S)_{S \subset \mathcal{P}[r]}$ for which $x_S = x_{S \cap [i]}$ for all $S \subset [r]$.

In addition, we will write $\mathbf{0}_j, \mathbf{1}_j$ for the vectors in $\{0,1\}^{\mathcal{P}[j]}$ consisting of all $0$s (respectively all $1$s). We call these (or any multiples of them) \emph{constant} vectors. 

Finally, we introduce the notion of a {\it block} of a vector $x = (x_S)_{S \subset [r]} \in \Q^{\mathcal{P}[r]}$. For each $A \subset [i]$ we consider the $2^{r-i}$-tuple
\[ 
x(A,i) := (x_{A \cup A'})_{A' \subset \{i+1,\cdots, r\}}.
\] 
We call these the $i$-\emph{blocks} of $x$. 

\begin{remark}\label{rem-blocks} 
(a) One should note carefully that the $i$-blocks are strings of length $2^{r-i}$. In this language, $V_i$ is the space of vectors $x$, all of whose $i$-blocks are constant.

\medskip

(b) If we put together the coordinates of the $i$-blocks $x(A,i)$ and $x(A\triangle\{i\},i)$, then we obtain the $(i-1)$-block $x(A\cap[i-1],i-1)$.
\end{remark}

In order to visualize the structure of the flag $\sV$ and of the partition of $\{0,1\}^{\cP[r]}$ by the cosets of $V_j$, it will be often useful to write elements of $\{0,1\}^{\mathcal{P}[r]}$ as strings of $0$s and $1$s of length $2^r$. When we do this we use the \emph{reverse binary order}, which is the one induced from $\N$ via the map $f(S) = \sum_{s \in S} 2^{r - s}$.

\begin{example}\label{exm-blocks} 
For concreteness, let us consider the case $r = 3$. In this case, the ordering of the coordinates of $x$ is
\begin{equation}\label{8-ordering} (x_\emptyset,x_{\{3\}},x_{\{2\}},x_{\{2,3\}},x_{\{1\}},x_{\{1,3\}},x_{\{1,2\}},x_{[3]}).
\end{equation} 
If $x = 01001110$ then its $2$-blocks are $01, 00, 11, 10$, and its $1$-blocks are $0100, 1110$. 
\end{example}

\subsection{Automorphisms of the binary system}
\label{sec:binary-auto}

\begin{proof}[Proof of Lemma \ref{inv-max}]
We begin by defining some permutations of $\mathcal{P}[r]$ for which, we claim,  the corresponding coordinate permutations give elements of $\Aut(\sV)$.
Suppose that $1 \leq  j \leq  r$ and that $A \subset [j-1]$. Then we may consider the permutation $\pi(A,j)$ defined by
\[ 
\pi(A,j)(S) = \begin{cases} 	
				S \triangle \{j\} & \mbox{if $S \cap [j-1] = A$}, \\  
				S & \mbox{otherwise}. \end{cases}.
\]

To visualize the action of this permutation on the coordinates of a vector $x$, it is useful to order its coordinates as we explained above. The action of $\pi(A,j)$ is then to permute the two adjacent $j$-blocks $x(A,j)$ and $x(A\sqcup\{j\},j)$, which together form the $(j-1)$-block $x(A,j-1)$, as per Remark \ref{rem-blocks}(b). More concretely, below are some examples of the action of the permutations $\pi(A,j)$ in the setting of Example \ref{exm-blocks}:

\begin{center}
\begin{tikzpicture}[<->,>=stealth',auto,node distance=0.6cm,
  thick,main node/.style={circle,draw,font=\sffamily\Large\bfseries}]

% nodes
\node (A) {$x_{\emptyset}$};
\node [right=of A] (B) {$x_{\{3\}}$ };
\node [draw, right=of B] (C) {$x_{\{2\}}$ };
\node [draw, right=of C] (D) {$x_{\{2,3\}}$};
\node [right=of D] (E) {$x_{\{1\}}$};
\node [right=of E] (F) {$x_{\{1,3\}}$};
\node [right=of F] (G) {$x_{\{1,2\}}$ };
\node [right=of G] (H) {$x_{[3]}$};

% arrows
\path[every node/.style={font=\sffamily\tiny}]
    
    (C) edge[bend left = 50] node{$\pi (\{2\}, 3)$} (D);
    \end{tikzpicture}
\end{center}

\begin{center}
\begin{tikzpicture}[<->,>=stealth',auto,node distance=0.6cm,
  thick,main node/.style={circle,draw,font=\sffamily\Large\bfseries}]

% nodes
\node (A) {$x_{\emptyset}$};
\node [right=of A] (B) {$x_{\{3\}}$ };
\node [right=of B] (C) {$x_{\{2\}}$ };
\node [right=of C] (D) {$x_{\{2,3\}}$};
\node [right=of D] (E) {$x_{\{1\}}$};
\node [right=of E] (F) {$x_{\{1,3\}}$};
\node [right=of F] (G) {$x_{\{1,2\}}$ };
\node [right=of G] (H) {$x_{[3]}$};

\node[below left = -0.2cm and -0.2cm of A] (J) {};
\node[above right = -0.2cm and -0.2cm of B] (K) {};
\node[above right = -0.2cm and 0.3cm of A](L){};
\draw[] (J) rectangle (K);

\node[below left = -0.2cm and -0.2cm of C] (J2) {};
\node[above right = -0.2cm and -0.2cm of D] (K2) {};
\node[above right = -0.2cm and 0.3cm of C](L2){};
\draw[] (J2) rectangle (K2);

% arrows
\path[every node/.style={font=\sffamily\tiny}]
    
    (L) edge[bend left = 25] node{$\pi(\emptyset, 2)$} (L2);
    \end{tikzpicture}
\end{center}

\begin{center}
\begin{tikzpicture}[<->,>=stealth',auto,node distance=0.6cm,
  thick,main node/.style={circle,draw,font=\sffamily\Tiny\bfseries}]

% nodes
\node (A) {$x_{\emptyset}$};
\node [right=of A] (B) {$x_{\{3\}}$ };
\node [right=of B] (C) {$x_{\{2\}}$ };
\node [right=of C] (D) {$x_{\{2,3\}}$};
\node [right=of D] (E) {$x_{\{1\}}$};
\node [right=of E] (F) {$x_{\{1,3\}}$};
\node [right=of F] (G) {$x_{\{1,2\}}$ };
\node [right=of G] (H) {$x_{[3]}$};

\node[below left = -0.2cm and -0.2cm of E] (J) {};
\node[above right = -0.2cm and -0.2cm of F] (K) {};
\node[above right = -0.2cm and 0.3cm of E](L){};
\draw[] (J) rectangle (K);

\node[below left = -0.2cm and -0.2cm of G] (J2) {};
\node[above right = -0.2cm and -0.2cm of H] (K2) {};
\node[above right = -0.2cm and 0.3cm of G](L2){};
\draw[] (J2) rectangle (K2);

% arrows
\path[every node/.style={font=\sffamily\tiny}]
    
    (L) edge[bend left = 25] node{$\pi(\{1\}, 2)$} (L2);
    \end{tikzpicture}
\end{center}

\begin{center}

\begin{tikzpicture}[<->,>=stealth',auto,node distance=0.6cm,
  thick,main node/.style={circle,draw,font=\sffamily\Large\bfseries}]

% nodes
\node (A) {$x_{\emptyset}$};
\node [right=of A] (B) {$x_{\{3\}}$ };
\node [right=of B] (C) {$x_{\{2\}}$ };
\node [right=of C] (D) {$x_{\{2,3\}}$};
\node [right=of D] (E) {$x_{\{1\}}$};
\node [right=of E] (F) {$x_{\{1,3\}}$};
\node [right=of F] (G) {$x_{\{1,2\}}$ };
\node [right=of G] (H) {$x_{[3]}$};

\node[below left = -0.2cm and -0.2cm of A] (J) {};
\node[above right = -0.2cm and -0.2cm of D] (K) {};
\node[above right = -0.2cm and 1.8cm of A](L){};
\draw[] (J) rectangle (K);

\node[below left = -0.2cm and -0.2cm of E] (J2) {};
\node[above right = -0.2cm and -0.2cm of H] (K2) {};
\node[above right = -0.2cm and 2.3cm of E](L2){};
\draw[] (J2) rectangle (K2);

% arrows
\path[every node/.style={font=\sffamily\tiny}]
    
    (L) edge[bend left = 10] node{$\pi(\emptyset, 1)$} (L2);
    \end{tikzpicture}
\end{center}

If the readers wish, they may translate the arguments below in the above more visual language.

\medskip

\noindent 
\emph{Claim.} $\pi(A,j)$ preserves $V_i$ for all $i$, and therefore $\pi(A,j) \in \Aut(\sV)$.

\begin{proof}
Suppose that $x= (x_S)_{S \subset [r]} \in V_i$ and let us write for simplicity $\pi$ instead of $\pi(A,j)$.

Suppose first that $j > i$. Then $\pi(S) \cap [i] = S \cap [i]$ for all $S$, and so 
\[x_{\pi(S)}  = x_{\pi(S) \cap [i]} = x_{S \cap [i]} = x_S.\]
where the first and last steps used the fact that $\mathbf{x} \in V_i$. Thus the claim follows in this case.

Suppose now that $j \leq  i$. Let $t > i$. Then the conditions $(S \triangle \{t\}) \cap [j-1] = A$ and $S \cap [j-1] = A$ are equivalent. Hence, if $S \cap [j-1] = A$, then we find that
\[ 
x_{\pi(S \triangle \{t \})} = x_{S \triangle \{t\} \triangle \{j\}} = x_{S \triangle \{j\}} = x_{\pi(S)},
\]  
where we used that $x\in V_i$ and that $t>i$ at the second step. 
Similarly, if $S \cap [j-1] \neq A$, then 
\[ 
x_{\pi(S \triangle \{t\})}= x_{S\triangle\{t\}} = x_S = x_{\pi(S)}.
\]
In all cases, we have found that $x_{\pi(S \triangle \{t\})}  = x_{\pi(S)}$. Since this is true for all $t > i$, $\pi(x)$ indeed lies in $V_i$. This completes the proof of the claim.
\end{proof}

Suppose now that $W$ is an invariant subspace of $\sV$ satisfying the inclusions $V_{i-1}<W \leq  V_i$. We want to conclude that $W=V_i$. To accomplish this, we introduce some auxiliary notation. 

For each $A\subset[i-1]$, we consider the vector $y^A = (y^A_S)_{S \subset [r]}\in V_i$ that is uniquely determined by the relations $y^A_{A} = 1$, $y^A_{A \cup \{i\}} = -1$ and $y^A_S = 0$ for all other $S \subset [i]$. There are $2^{i-1}$ such vectors $y^A$. They are mutually orthogonal, hence linearly independent. In addition, together with $V_{i-1}$, they generate all of $V_i$. \rev{Since $V_{i-1}<W\leq V_i$, there must exist $A\subset[i-1]$ such that $y^A\in W$.} 

Now, it is easy to check that for any $j<i$ and any $A\subset [i-1]$, we have 
\[
\pi(A\cap [j-1],j) y^A =
y^{A \triangle \{j\}}.
\]
From the above relation and the invariance of $W$ under $\Aut(\sV)$, it is clear that if $W$ contains at least one vector $y^A$ with $A\subset[i-1]$, then it contains all such vectors. Since we also know that $V_{i-1}\leq W\leq V_i$, we must have that $W=V_i$, which \rev{completes the proof of Lemma \ref{inv-max}}.
%It remains to exhibit a vector $y^A$ lying in $W$. Since we have assumed that $W \neq V_{i-1}$, 
%there is some $x \in W$ and some $A\subset [i-1]$ such that $a = x_A \neq x_{A \cup \{i\}} = b$. It is then easy to check that $(a - b)^{-1} \big(x- \pi(A, i)(x)\big) = y^{A}$.  The vector of the left hand side is in $W$ by our assumptions that $x\in W$ and that $W$ is invariant. Thus, $y^A \in W$ as well. This completes the proof of Lemma \ref{inv-max}.
\end{proof}

\begin{remark*}  A minor elaboration of the above argument in fact allows one to show that the subspaces of $\Q^{\mathcal{P}[r]}$ invariant under $\Aut(\sV)$ are the $V_i$, the orthogonal complements of $V_{i-1}$ in $V_i$, and all direct sums of these spaces. However, we will not need the classification in this explicit form.
\end{remark*}

\subsection{Cell structure and genotype}\label{cell-enum} 

The cosets of $V_i$ partition $\{0,1\}^{\mathcal{P}[r]}$ into sets which we call the \emph{cells at level $i$}. Our first task is to describe these explicitly.

Consider $\omega, \omega' \in \{0,1\}^{\mathcal{P}[r]}$. It is easy to see that $\omega - \omega' \in V_i$ (and so $\omega, \omega'$ lie in the same cell at level $i$) if and only if for every $A \subset [i]$ one of the following is true:
\begin{enumerate}
\item Both $\omega(A,i)$ and $\omega'(A,i)$ are constant blocks (that is, they both lie in $\{ \mathbf{0}_{r-i}, \mathbf{1}_{r-i}\})$.
\item $\omega(A, i) = \omega'(A,i)$, and neither of these blocks is constant (that is, neither is $\mathbf{0}_{r-i}$ nor $\mathbf{1}_{r-i}$).
\end{enumerate}
Thus a cell at level $i$ is completely specified by the \emph{positions} $A$ of its constant $i$-blocks, and the \emph{values} $\omega(A,i)$ (for an arbitrary $\omega \in C$) of its non-constant $i$-blocks.

\begin{example*}
With $r=3$ and $\omega = 01001110$, the level $2$ cell that contains $\omega$ is the set 
\[\{ \omega, 01111110, 01000010, 01000010\}.\]
Its constant $2$-blocks are at $A = \{2\}$ and $A = \{1\}$. Its non-constant $2$-blocks are at $A = \emptyset$ (taking the value $\omega(A,2) = 01$) and at $A = \{1,2\}$ (taking the value $\omega(A, 2) = 10$).
The level $1$ cell containing $\omega$ is just $\{\omega\}$.
\end{example*}

The positions of the constant $i$-blocks play an important role, and we introduce the name \emph{genotype} to describe these\footnote{The term genotype is appropriate, as each component in 
$g$ acts like recessive gene with respect to child cells.}.

\begin{definition}[Genotype]\label{gen-def} If $C$ is a cell at level $i$, its \emph{genotype} $g(C) \subset \mathcal{P}[i]$ is defined to be the collection of $A \subset [i]$ for which \rev{$\omega(A,i)\in\{\mathbf{0}_{r-i},\mathbf{1}_{r-i}\}$ for all $\omega\in C$}. We refer to any subset of $\mathcal{P}[i]$ as an \emph{$i$-genotype}. If $g, g'$ are two $i$-genotypes, then we write $g \leq g'$ to mean the same as $g \subseteq g'$. We write $|g|$ for the cardinality of $g$.
\end{definition}

\begin{example*} If $C$ is the cell at level $2$ containing $\omega = 01001110$, the genotype $g(C)$ is equal to \rev{$\big\{\{2\}, \{1\}\big\}$}. (We have listed these sets in the reverse binary ordering once again.)
\end{example*}

\begin{definition}[Consolidations]\label{consol-def}
If $g$ is an $i$-genotype, then its \emph{consolidation} is the $(i-1)$-genotype $g^*$ defined by $g^* := \{A' \subset [i-1] : A'\in g, A' \cup \{i\} \in g\}$ (cf.~Remark \ref{rem-blocks} (b)).
\end{definition}

Let us pause to note the easy inequality
\begin{equation}\label{ggstar} \frac{1}{2}|g| \geq |g^*| \geq |g| - 2^{i-1},
\end{equation} valid for all $i$-genotypes.

The genotype is intimately connected to the cell structure on $\{0,1\}^k$ induced by $\sV$, as the following lemma shows.

\begin{lemma}\label{genotype-lemma}
We have the following statements.
\begin{enumerate}
\item If $C$ is a cell, we have $|C| = 2^{| g(C)|}$.
\item Suppose that $g$ is an $i$-genotype. There are
$(2^{2^{r-i}}-2)^{2^i-| g |}$
 cells (at level $i$) with $g(C) = g$.
\item If $g(C) = g$, and if $C'$ is a child of $C$, then
$g(C') \leq g^*$. In particular, $|g(C')| \leq \frac{1}{2}|g(C)|$.
\item Suppose that $g(C) = g$. Suppose that $g'$ is an $(i-1)$-genotype and that $g' \leq g^*$. Then number of children $C'$ of $C$ with $g(C') = g'$ is $2^{| g | - | g^* | - | g' |}$.
\item Suppose that $C$ is a cell at level $i$ with $g(C) = g$. Then the number of children of $C$ (at level $i-1$) is $2^{| g | - 2 |g^* |} 3^{|g^* |}$.
\end{enumerate}
\end{lemma}

\begin{proof}
(a) This is almost immediate: for each of the $A\subset g(C)$ of constant blocks, the are two choices ($\mathbf{0}_{r - i}$ or $\mathbf{1}_{r - i}$) for $\omega(A,i)$.

\medskip

(b) To determine $C$ completely (given $g$), one must specify the value of each of $2^i - | g |$ non-constant $i$-blocks. For each such block, there are $2^{2^{r-i}}-2$ possible non-constant values.

\medskip

(c) A set $A' \subset [i-1]$ can only possibly be the position of a constant block in some child cell of $C$ if both $A'$ and $A' \cup \{i\}$ are the positions of constant blocks in $C$, or in other words $A', A' \cup \{i\} \in g$, which is precisely what it means for $A'$ to lie in $g^*$.

Note that the child cell $C'$ containing $\omega$ only \emph{does} have a constant $(i-1)$-block at position $A'$ if $\omega(A', i) = \omega(A' \cup \{i\}, i)$, which may or may not happen.

The second statement is an immediate consequence of the first and \eqref{ggstar}.

\medskip

(d) Let $A \in g$. We say that $A$ is \emph{productive} if $A' :=  A \cap [i-1] \in g^*$, or equivalently if $A'$ and $A' \cup \{i\}$ both lie in $g$ (or, more succinctly, $A \triangle \{i\} \in g$). These are the positions which can give rise to constant $(i-1)$-blocks in children of $C$. There are $2|g^*|$ such positions, coming in $|g^*|$ pairs. To create a child $C'$ with genotype $g'$, we have a binary choice at $|g^*| - |g'|$ of these pairs: at each of them either $\omega(A', i) = \mathbf{0}_{r - i}$ and $\omega(A' \cup \{i\}, i) = \mathbf{1}_{r - i}$, or the other way around. There are $|g| - 2|g^*|$ non-productive positions $A \in g$, and for each of these there is also a binary choice, either $\omega(A,i) = \mathbf{0}_{r - i}$ or $\omega(A,i) = \mathbf{1}_{r - i}$. The total number of choices is therefore $2^{|g^*| - |g'|} \times 2^{|g| - 2|g^*|}$, which is exactly as claimed.

\medskip

(e) This is immediate from part (d), upon summing over $g' \subseteq g^*$.
\end{proof}

\subsection{The $f^C(\rho)$ and genotype}\label{fc-genotype}

We begin by recalling from \eqref{fc-eqs} the definition of the functions $f^C({\bm \rho})$. Here ${\bm \rho} = (\rho_1,\cdots, \rho_{r-1})$ is a sequence of parameters, and we define $\rho_0 = 0$. If $C$ has level $0$, we set $f^C({\bm \rho}) = 1$, whilst for $C$ at level $i \geq 1$ we apply the recursion
\[ f^C({\bm \rho}) = \sum_{C \rightarrow C'} f^{C'}({\bm \rho})^{\rho_{i-1}}.\]

\begin{proposition}\label{gen-fc} The quantities $f^C$ depend only on the genotype of $C$, and thus for any $i$-genotype $g$ we may define $F(g) := f^{C}({\bm \rho})$, where $C$ is any cell with $g(C) = g$. We have the recursion
\begin{equation}\label{geno-recurs} F(g) = \sum_{g' \leq g^*} 2^{|g| - |g^*| - |g'|} F(g')^{\rho_{i-1}}.\end{equation} 
\end{proposition}

\begin{remark*} The $F(g)$ depend on ${\bm \rho}$, as well as on $i$ (where $g$ is an $i$-genotype) but we suppress explicit mention of this. For example, it should be clear from context that $g$ on the left is an $i$-genotype, but the sum on the right is over $(i-1)$-genotypes, since $g^*$ is an $(i-1)$-genotype by definition.
\end{remark*}

\begin{proof} This is a simple induction on the level $i$ using the definition of the $f^C({\bm \rho})$, and parts (c) and (d) of Lemma \ref{genotype-lemma}.\end{proof}

Let us pause to record two corollaries which we will need later.

\begin{corollary}\label{g1-g2}
Suppose that $g_1, g_2$ are two $i$-genotypes with $g_1 \leq g_2$. Then $F(g_1) \leq F(g_2)$.
\end{corollary}
\begin{proof}
Note that $g_1^* \leq g_2^*$, and also that $|g_1| - |g_1^*| \leq |g_2| - |g_2^*|$, since 
\[
\rev{
	|g| - |g^*| = |g^*|+ \# \{ A \subset \mathcal{P}[i-1] :\#(\{A,A\cup\{i\} \}\cap g)=1\}.
 }
\] 
Hence, by two applications of Proposition \ref{gen-fc}, 
\[ F(g_1) = 2^{|g_1| - |g_1^*|} \sum_{g' \leq g_1^*} 2^{-|g'|} F(g')^{\rho_{i-1}} \leq 2^{|g_2| - |g^*_2|} \sum_{g' \leq g^*_2} 2^{-|g'|} F(g')^{\rho_{i-1}} = F(g_2).
\qedhere
\]
\end{proof}

Recall that $\Gamma_i$ is the cell at level $i$ containing $\mathbf{0}$. Note that $g(\Gamma_i) =\mathcal{P}[i]$.

\begin{corollary}\label{Gamma-j-largest}
If $C \ne \Gamma_i$ is a cell of level $i$, then $
f^C({\bm \rho}) < f^{\Gamma_i}({\bm \rho})$.
\end{corollary}\begin{proof} This is simply the special case $g_2 = \mathcal{P}[i]$ of the preceding corollary. The inequality is strict because if $g < \mathcal{P}[i]$, then $g^* < \mathcal{P}[i-1]$.\end{proof}

\subsection{Existence of the $\rho_i$}\label{rhoi-exist}

In this section we prove Proposition \ref{rho-binary-prop} (a), which asserts that for the binary flags there is a unique solution ${\bm \rho} = (\rho_1,\rho_2,\dots)$ to the $\rho$-equations \eqref{binary-rho-eqs}. In fact, we will prove the following more general fact which treats the $j$th $\rho$-equation in isolation, irrespective of whether the earlier ones have already been solved. 

\begin{proposition}\label{jth-rho-prop}
Let \rev{$j\in\N$ and let}  $\rho_1,\dots, \rho_{j-1} \in (0,1)$. Then there is a unique $\rho_j \in (0,1)$ such that the $j$th $\rho$-equation for the binary flag, $f^{\Gamma_{j+1}}(\brho) = e^{2^j} f^{\Gamma_j}(\brho)^{\rho_{j}}$, is satisfied. 
\end{proposition}

\begin{remark*} 
	We will prove in the next section (Lemma \ref{rho-j-rho-1}) that for the solution $\rho_1,\rho_2,\dots$ to the full set of $\rho$-equations we have $\rho_j \le \rho_1 = 0.30648\dots$ for all $j$. For a table of numerical values of the $\rho_j$, see Table \ref{rhoj} in Section \ref{sec:analytic}.
\end{remark*}

Before beginning the proof of Proposition \ref{jth-rho-prop}, we isolate a lemma.

\begin{lemma}\label{small-calc-gen}
Fix a $(j-1)$-genotype $g'$. Then
\[ \sum_{g :\, g^*\geq g'} 2^{-|g^*|} = 2^{-2^{j-1}}
7^{2^{j-1}-|g'|},\] where the sum is over all $j$-genotypes $g$.
\end{lemma}

\begin{proof} 
\rev{
	In order to determine $g$, we must determine for each $A\subset[j-1]$ whether $A$ and/or $A\cup\{j\}$ lie in $g$. Since we are only summing over $g$ whose consolidation $g^*$ contains $g'$, we must have that $A$ and $A\cup\{j\}$ belong to $g$ for all $A\in g'$, so the membership of $A$ and $A\cup\{j\}$ to $g$ is fully determined for all $A\in g'$. For any $A\subset[j-1]$ with $A\notin g'$, we have four choices, according to whether $A\in g$ and whether $A\cup\{j\}$. If both of these conditions hold, then we further have $A\in g^*$; in the other three cases, we have $A\notin g^*$. We conclude that \[
	\sum_{g :\, g^*\ge g'} 2^{-|g^*|} =\sprod{A\in g'} 2^{-1} \sprod{A\not\in g'}
	(\rev{1}\cdot 2^{-1}+3\cdot 2^{\rev{-0}})= 2^{-2^{j-1}} 7^{2^{j-1}-|g'|}. \]
	This completes the proof.
}	
%
%Assume that $g^*\ge g'$ and fix $A\subset [j-1]$.  
%If $A\in g'$, then $A\in g^*$ and hence both
%$A\in g$ and $A\cup \{j\}\in g$.   If $A\not\in g'$, then
%either $A\in g^*$ (whence $A\in g$ and $A\cup \{j\}\in g$) or
%$A\not\in g^*$ (whence at most one of $A$ and $A\cup \{j\}$
%is in $g$). Therefore,
%\[
% \sum_{g :\, g^*\ge g'} 2^{-|g^*|} =\sprod{A\in g'} 2^{-1} \sprod{A\not\in g'}
%(\rev{1}\cdot 2^{-1}+3\cdot 2^{\rev{-0}})= 2^{-2^{j-1}} 7^{2^{j-1}-|g'|}. \]
%This completes the proof.
\end{proof}

\begin{proof}[Proof of Proposition \ref{jth-rho-prop}]
For $j=1$, the equation to be satisfied is $3^{\rho_1} + 4 \cdot 2^{\rho_1} + 4 = e^2 3^{\rho_1}$.  It may easily be checked numerically that this has a unique solution $\rho_1 \approx 0.306481\dots$ in $(0,1)$. One may also proceed analytically as follows. Define
\[ 
G(x) = G_1(x) := e^2 3^x - (3^x + 4 \cdot 2^x + 4) = 3^x \big(e^2 - (1 + 4 \cdot (2/3)^x + 4/3^x)\big),
\] 
In particular, the roots of $G$ are in correspondence with the roots of $H(x)=e^2 - (1 + 4 \cdot (2/3)^x + 4/3^x)$. This is clearly a continuous and strictly increasing function. In addition, $H(0)=e^2-9<0$ and $H(1)=e^2-5>0$. Thus, $H$ has a unique root $\rho_1\in (0,1)$, and so does $G$.

Now assume $j\ge 2$. It turns out that much the same argument works, although the details are more elaborate.
Assume that $0<\rho_i<1$ for $1\le i<j$. Define
\[
G(x) := G_j(x) 
	 = e^{2^j} (f^{\Gamma_j}(\brho))^x - f^{\Gamma_{j+1}}(\rho_1,\ldots,\rho_{j-1},x) .
\]
Proposition \ref{gen-fc} implies that
\begin{align}
G(x) 
	& = e^{2^j} (F(\mathcal{P}[j]))^x - \sum_{g} 2^{2^j-|g|} F(g)^x  \label{g-exp} \\ 
	& = F(\mathcal{P}[j])^x \cdot H(x)\nonumber ,
\end{align}
where
\[
H(x) = e^{2^j} - 2^{2^j} \sum_{g} 2^{-|g|} \big(F(g) / F(\mathcal{P}[j])\big)^x
\]
and the sums over $g$ run over all genotypes $g \subset \mathcal{P}[j]$ at level $j$. Since (by an easy induction) $F(\mathcal{P}[j])>0$, it follows that $G$ and $H$ have the same roots. The latter is a continuous and strictly increasing function because \rev{Corollary \ref{Gamma-j-largest} implies that $F(g)/F(\mathcal{P}[j]) \leq 1$, with equality only when $g=\mathcal{P}[j]$}. Moreover, $H(0) = e^{2^j} - 3^{2^j} < 0$. Therefore to complete the proof it suffices to show that $H(1) > 0$.

To show this, we use \eqref{g-exp}. First note that
\begin{equation}\label{power-i-eq}
F(\mathcal{P}[j]) = (\sqrt{2})^{2^j} \sum_{g'} 2^{-|g'|} F(g')^{\rho_{j-1}},
\end{equation}
where the sum is over all genotypes $g'$ of level $(j-1)$.

Next, by Proposition \ref{gen-fc} and Lemma \ref{small-calc-gen} we have

\[ 
\sum_{g \subset \mathcal{P}[j]} 2^{-|g|} F(g) = \sum_{g} 2^{-|g^*|} \sum_{g' \le g^*}
2^{-|g'|} F(g')^{\rho_{j-1}} = \] \begin{equation}
= \sum_{g' \subset \mathcal{P}[j-1]} 2^{-|g'|}   F(g')^{\rho_{j-1}} \sum_{g:\, g^*\ge g'}2^{-|g^*|} = (7/2)^{2^{j-1}} \sum_{g'} 14^{-|g'|} F(g')^{\rho_{j-1}} .\label{second-g}\end{equation}
Putting \eqref{g-exp}, \eqref{power-i-eq} and \eqref{second-g} together we obtain
\[
H(1)\cdot F(\mathcal{P}[j])  = (e\sqrt{2})^{2^j} \sum_{g'} 2^{-|g'|} F(g')^{\rho_{j-1}}
- (\sqrt{14})^{2^j} \sum_{g'} 14^{-|g'|} F(g')^{\rho_{j-1}} .
\]
Since $e^2>7$, we have $\sqrt{14} < e\sqrt{2}$, and thus $H(1)>0$. This completes the proof.
\end{proof}

\subsection{Entropy inequalities for the binary systems}
\label{sec:binary-entropy}
We begin with a lemma which will be used a few times in what follows.

\begin{lemma}\label{largest-cell}
Let $C'$ be one of the children of $\Gamma_i$, thus $C'$ is a cell at level $(i-1)$. Then \[ \mu_i(C') \leq \mu_i(\Gamma_{i-1}) = e^{-2^{i-1}},\] and equality occurs only when $C' = \Gamma_{i-1}$.
\end{lemma}
\begin{proof}
We showed in Corollary \ref{Gamma-j-largest} that $f^{C'}({\bm \rho}) < f^{\Gamma_{i-1}}({\bm \rho})$, for any choice of ${\bm \rho} = (\rho_1,\dots, \rho_{r-1})$, and for any child $C'$ of $\Gamma_i$ with $C' \neq \Gamma_{i-1}$.
Now that we know that the $\rho$-equations have a solution, it follows immediately from the definition of the optimal measures $\bmu^*$ in \eqref{eq46}, applied with $C = \Gamma_i$, that $\mu_i(C') < \mu_i(\Gamma_{i-1})$, again for any child $C'$ of $\Gamma_i$ with $C' \neq \Gamma_{i-1}$. Finally, observe that $\mu_i(\Gamma_{i-1}) = e^{-2^{i-1}}$ by \eqref{mu-gammaj}.
\end{proof}

\begin{proof}[Proof of Lemma \ref{binary-entropy-spec-1}] This follows almost immediately from Lemma \ref{largest-cell} with $i = m+1$. Indeed since $\mu_{m+1}(C) \leq e^{-2^m}$ for all cells $C$ at level $m$, with equality only for $C = \Gamma_m$, we have
\[
\HH_{\mu_{m+1}}(V_m) = \sum_C \mu_{m+1}(C) \log \frac{1}{\mu_{m+1}(C)} >
2^m \sum_C \mu_{m+1}(C) = 2^m.
\] This concludes the proof.
\end{proof}

\begin{proof}[Proof of Lemma \ref{binary-entropy-spec}] Let $\mu = \mu_i$ with $m<i\le r$. \rev{We must show that
		\begin{equation}\label{binary-entropy-spec: goal}
		\HH_\mu(V_{m-1})-\HH_\mu(V_m) <2^{m-1}. 
		\end{equation} 
	
	Let $C$ denote a cell at level $m$ and $C'$ a child of $C$ at level $(m-1)$. In addition, let the notations $g(C)$ and $g(C)^*$ refer to the genotype of $C$ and its consolidation, as defined in Definitions \ref{gen-def} and \ref{consol-def}.} By the definition of entropy, Lemma \ref{genotype-lemma} (e), and the concavity of $L(x)=-x\log x$ we find that
\begin{align}\nonumber
\HH_{\mu}(V_{m-1}) -  \HH_{\mu}(V_m) &= \sum_C \mu(C) \sum_{C'} L\pfrac{\mu(C')}{\mu(C)} \\ \nonumber
&\leq \sum_C \mu(C) \log (\# C') \\
&\rev{=}\sum_C \mu(C) \log \Big[ 2^{|g(C)|} (3/4)^{|g(C)^*|}\Big].\label{ten-5}
\end{align}

Now by \eqref{ggstar} we have $|g(C)^*| \geq |g(C)| - 2^{m-1}$, whence
\begin{equation}\label{children-crude-bound}
2^{|g(C)|} (3/4)^{|g(C)^*|} \leq 2^{|g(C)|}(3/4)^{|g(C)| - 2^{m-1}} = (3/2)^{|g(C)|}(4/3)^{2^{m-1}}.
\end{equation}
Since we also have that $|g(C)|\le2^m$, we infer that 
\begin{equation}\label{children-crude-bound-2}
2^{|g(C)|} (3/4)^{|g(C)^*|}\leq 3^{2^{m-1}}.
\end{equation} 
This and \eqref{ten-5} already imply the bound
\[ \HH_{\mu}(V_{m-1}) - \HH_{\mu}(V_m) \leq  2^{m-1} \log 3,\] which is only very slightly weaker than Lemma \ref{binary-entropy-spec}.

To make the crucial extra saving, write $S$ for the union of all cells $C$ at level $m$ with $|g(C)| > \frac{3}{4} 2^m$. We claim that \be\label{S-meas}
\mu(S) < \frac12.
\ee
We postpone the proof of this inequality momentarily and show how to use it to complete the proof of Lemma \ref{binary-entropy-spec}. 

Observe that if $C$ is not one of the cells making up $S$, that is to say if $|g(C)| \leq \frac{3}{4} 2^m$, then
\begin{align*}
\log \Big[ 2^{|g(C)|} (3/4)^{|g(C)^*|} \Big] 
	&\leq  \log \Big[ (3/2)^{|g(C)|}(4/3)^{ 2^{m-1}} \Big] \\
	&\leq 
		\( \frac{3}{2} \log (3/2) + \log(4/3) \) 2^{m-1} \\
	& \leq 0.9 \cdot 2^{m-1},
\end{align*}
where we used \eqref{children-crude-bound} to obtain the first inequality. Assuming the claim \eqref{S-meas}, it follows from this, \eqref{ten-5} and \eqref{children-crude-bound-2} that 
\[
\HH_{\mu}(V_{m-1}) - \HH_{\mu}(V_m) \leq 2^{m-1}(\log 3) \mu(S) +
0.9\cdot 2^{m-1} (1-\mu(S)) < 2^{m-1},
\]
which is the statement of Lemma \ref{binary-entropy-spec}.

It remains to prove \eqref{S-meas}. Recall that $1\le m<i\le r$.

\rev{When $1\le m\le 2$, the only integer in $(\frac{3}{4}2^m,2^m]$ is $2^m$. Hence, if a cell $C$ at level $m$ satisfies the inequality $|g(C)|>\frac{3}{4}2^m$, we must have $|g(C)|=2^m$. The only cell with this property is $\Gamma_m$.} Since we have
$\mu(\Gamma_m)=e^{2^m-2^i} \leq e^{-1}$ by \eqref{mu-gammaj}, our claim \eqref{S-meas} follows in this case.

Assume now that $m\geq3$. Let $\tilde{S}$ be the union of all children $\tilde C$ of $\Gamma_i$ (thus these are cells at level $i-1 \geq m$) which contain a cell $C$ in $S$. By repeated applications of Lemma \ref{genotype-lemma} (c) we have $|g(\tilde C)| > 2^{i-1-m} (\frac34 2^{\rev{m}})=\frac34 2^{i-1}$ for any such $\tilde C$. Lemma \ref{genotype-lemma} (d), applied with $C = \Gamma_i$, implies that the
number of such cells $\tilde{C}$ is at most 
\dalign{
\sum_{h>(3/4)2^{i-1}} \binom{2^{i-1}}{h} 2^{2^{i-1}-h}
\leq 2^{\frac14 2^{i-1}} 2^{2^{i-1}} = 2^{(5/4)2^{i-1}}.
}
By Lemma \ref{largest-cell} and our assumption that $i-1\ge m\ge 3$, it follows that
\[
\mu(S) \leq \mu(\tilde{S}) \leq (2^{5/4}/e)^{2^{i-1}} < 0.35 .
\] 
This completes the proof of the claim \eqref{S-meas} and hence of Lemma \ref{binary-entropy-spec}.
\end{proof}

\subsection{Existence of the optimal parameters $\cc^*$}\label{c-params-sec}

\begin{proof}[Proof of Proposition \ref{rho-binary-prop} (b)] 
\rev{We have $\Supp(\mu_j^*) =\Gamma_j$ by Remark \ref{rem:optimal measure} (b), and hence $|\Supp(\mu_j^*)| =2^{2^j}$ by Lemma \ref{cube-subspace}.} By Lemma \ref{ent-trivial}, when $j\ge m+2$ we deduce the inequality
\begin{equation}\label{auxillary-ent}
\HH_{\mu_j^*}(V_m) \leq \log |\Supp(\mu_j^*)| \leq 2^j \log 2 < 
2^j - 2^m.
\end{equation}
Now recall (Definition \ref{opt-param-def}) that the optimal parameters should satisfy the conditions \eqref{basic-m-e} (which are the fully written out version of \eqref{79-pre}). We wish to show that there is a solution with $1 = c_1^* > c^*_2 > \cdots > c^*_{r+1} > 0$.
Rearranging  \eqref{basic-m-e} and recalling $\dim(V_j)=2^j$, we
find that
\begin{align*} 
(c^*_{m+1} & - c^*_{m+2})\big( \HH_{\mu_{m+1}^*}(V_m) - 2^m\big) \\ & = \sum_{j = m+2}^r \big( 2^j-2^m - \HH_{\mu_j^*}(V_m)\big)(c^*_j - c^*_{j+1}) +(2^r-2^m) c^*_{r+1} 
\end{align*}
for $0\le m\le r-1$.    By Lemma \ref{binary-entropy-spec-1} and \eqref{auxillary-ent},
we may apply a downwards induction on $m = r-1, r-2,\cdots$
to solve these equations with $0 < c^*_{r+1} < c^*_r < \cdots < c^*_1$. Rescaling,
 we may additionally ensure that $c^*_1 = 1$.\end{proof}

\section{The limit of the $\rho_i$} \label{lim-rho-sec}

In the last section we showed that there is a unique solution $\brho = (\rho_1,\rho_2,\ldots)$ to the $\brho$-equations \eqref{binary-rho-eqs} for the binary system with $0 < \rho_j < 1$ for all $j$. In this section, we show that the limit $\lim_{j \rightarrow \infty} \rho_j $ exists.

\begin{proposition}\label{mainthm-limit}
$\rho = \lim_{j \rightarrow \infty} \rho_j$ exists.
\end{proposition}

\subsection{$\rho_1$ is the largest $\rho_j$}

The estimates required in the proof of Proposition \ref{mainthm-limit} are rather delicate, and to make them usable for our purposes we need the following \emph{a priori} bound on the $\rho_j$.

\begin{lemma}\label{rho-j-rho-1}
For all \rev{$j\ge1$, we have} $\rho_j \le \rho_1 = 0.30648\dots$
\end{lemma}

The reader should recall the notion of genotype $g$ (Definition \ref{gen-def}) and of the function $F(g)$ (Proposition \ref{gen-fc}). 

The next lemma is a stronger version of Corollary \ref{g1-g2}, whose proof uses that result as an ingredient.

\begin{lemma}\label{strong-g1-g2}
  For any \rev{$j\ge1$} and $g_1 \le g_2$ at level $j$, we have
  \[
\frac{F(g_1)}{F(g_2)} \le \pfrac12^{|g_2|-|g_1|} \pfrac43^{|g_2^*|-|g_1^*|}.
  \]
\end{lemma}
\begin{proof} \rev{We have
%Applying Proposition \ref{gen-fc} to $g_2$, followed by Corollary \ref{g1-g2}, followed by an application of the binomial theorem, followed by an application of Proposition \ref{gen-fc} to $g_1$, we obtain
\begin{align*}
F(g_2) & =   2^{|g_2|-|g_2^*|} \sum_{g\le g_1^*} 2^{-|g|} \sum_{g'\le g_2^*\setminus g_1^*}
2^{-|g'|} F(g \cup g')^{\rho_{j-1}} 
&\text{(by Proposition \ref{gen-fc})}\\ 
& \geq  2^{|g_2|-|g_2^*|} \sum_{g\le g_1^*} 2^{-|g|} \sum_{g'\le g_2^*\setminus g_1^*}
2^{-|g'|} F(g)^{\rho_{j-1}} &\text{(by Corollary \ref{g1-g2})}\\
&= 2^{|g_2|-|g_2^*|} \sum_{g\le g_1^*} 2^{-|g|}  F(g)^{\rho_{j-1}} (3/2)^{|g_2^*|-|g_1^*|}&\text{(by the binomial theorem)}\\
&=F(g_1) 2^{|g_2|-|g_1|} (3/4)^{|g_2^*|-|g_1^*|}&\text{(by Proposition \ref{gen-fc}).}
\end{align*}
This concludes the proof.}
\end{proof}

\begin{proof}[Proof of Lemma \ref{rho-j-rho-1}]
We begin by observing that 
\begin{equation}\label{binom-double} \sum_{g \leq \mathcal{P}[j]} c_1^{|g|} c_2^{|g^*|} =
\prod_{A\subset [j-1]} \Big( \sum_{a,b\in \{0,1\}} c_1^{a+b}c_2^{ab} \Big) = 
 (1 + 2c_1 + c_1^2 c_2)^{2^{j-1}}.
 \end{equation}
The $\rho$-equations \eqref{binary-rho-eqs}, translated into the language of genotypes, are $F(\mathcal{P}[j+1]) = e^{2^j} F(P[j])^{\rho_j}$. Therefore, by Proposition \ref{gen-fc} (with $g = \mathcal{P}[j+1]$) followed by Lemma \ref{strong-g1-g2} (with $g_2 = \mathcal{P}[j]$), we have  
  \begin{align*}
    e^{2^j} F(\mathcal{P}[j])^{\rho_j} &= F(\mathcal{P}[j+1]) = 2^{2^j} \sum_{g\leq\mathcal{P}[j]} 2^{-|g|} F(g)^{\rho_j} \\
    &\le  2^{2^j} \sum_{g\leq \mathcal{P}[j]} 2^{-|g|} F(\mathcal{P}[j])^{\rho_j}
    \Big[ (1/2)^{2^j-|g|}(4/3)^{2^{j-1}-|g^*|}   \Big]^{\rho_j}\\
    &\rev{=2^{2^j} (1/3)^{2^{j-1}\rho_j} F(\mathcal{P}[j])^{\rho_j} \sum_{g\leq P[j]} 2^{(\rho_j-1)|g|} (3/4)^{\rho_j |g^*|}}.
  \end{align*}
  Dividing through by $F(\mathcal{P}[j])^{\rho_j}$, and applying \eqref{binom-double} with $c_1 = 2^{\rho_j - 1}$ and $c_2 = (3/4)^{\rho_j}$, we find that
  \begin{align*}
    e^{2^j}  &\le (4/3^{\rho_j})^{2^{j-1}} \big( 1 + 2^{\rho_j} + 2^{2\rho_j-2} (3/4)^{\rho_j}\big)^{2^{j-1}}\\
    &=\big(\rev{4/3^{\rho_j} + 4(2/3)^{\rho_j} + 1} \big)^{2^{j-1}}.
  \end{align*}
  Therefore
  \[
3^{\rho_j} e^2 \le \rev{4+ 4 \cdot 2^{\rho_j} + 3^{\rho_j}}.
\]
However, the first $\rho$-equation \eqref{first-rho} is precisely that
\[ 3^{\rho_1} e^2 = \rev{4+ 4 \cdot 2^{\rho_1} + 3^{\rho_1}}.\]
The result follows immediately (using the monotonicity of the function $1 + 4(2/3)^t + 4(1/3)^t$ - \rev{see the proof of Proposition \ref{jth-rho-prop}}).
\end{proof}

\subsection{Preamble to the proof}

In this section, we set up some notation and structure necessary for the proof of Proposition \ref{mainthm-limit}. \rev{Since we wish to let $r\to\infty$, it is convenient to embed all binary $r$-step systems into a universal infinite binary system. To this end, and with a slight abuse of notation, we let
	\[
	V_j= \big\{ (x_A)_{A\subset\cP(\N)}: x_A\in\Q \ \text{and}\ x_{A}=x_{A\cap[j]} \ \text{ for all } A\subset\cP(\N) \big\}\quad\text{for}\ j=0,1,\dots. 
	\]
Clearly, $V_j\simeq \Q^{2^j}$ for all $j$, and the flag $\sV^r:V_0\le V_1\le \cdots\le V_r$ is isomorphic to the flag of the $r$-step binary system. 

In this notation, we have
\[
\Gamma_j = \big\{ \omega\in\Omega: \omega\equiv\mathbf{0}\ \md{V_j} \big\}\quad\text{for}\ j=0,1,\dots,
\]
where
\[
\Omega = \big\{\omega=(\omega_A)_{A\subset\cP(\N)}:
	\mbox{$\omega_A\in \{0,1\}$ for all $A\subset\cP(\N)$} \big\} 
\]
is the discrete unit cube. We further set
\[
\Gamma_\infty  = \bigcup_{j=0}^\infty \Gamma_j . 
\]

Lastly, for each $j\ge0$, we say that $C$ is a cell at level $j$ if $C\subset \Gamma_\infty$ and there exists some $x=(x_A)_{A\subset\cP(\N)}$ such that $x_A\in\Q$ for all $A$ and $C=\Omega\cap(x+V_j)$.  We may easily check that the collection of cells lying in $\Gamma_r$ forms the tree corresponding to the $r$-step binary system. 

We may now define the functions $f^C$ for our infinite binary flag.} It is convenient to reverse the indices in $f^C$.
Specifically, let $\x = (x_1,x_2,\ldots)\rev{\in[0,1]^\N}$. If $C$ is a cell at level \rev{$j\ge0$}, then we define
\[ 
\psi^C(\x) :=\log f^C(x_{j-1},\dots, x_1).
\]
\rev{In particular, $\psi^C(\x)=0$ when $j=0$, and $\psi^C(\x)=\log|C\setminus\{\mathbf{0}\}|$ when $j=1$}.

In the special case $C = \Gamma_j$ we define also
\[ 
\phi_j(\x) = 2^{-j} \psi^{\Gamma_j}(\x) = 2^{-j} \log f^{\Gamma_j} (x_{j-1},\cdots, x_1).
\]
Thus $\phi_1(\x) = \frac{1}{2} \log 3$ and $\phi_2(\x) = \frac{1}{4} \log (3^{x_1} + 4 \cdot 2^{x_1} + 4)$. 

Note that $\psi^C, \phi_j$ are increasing in each variable. Moreover we have the following simple bounds.

\begin{lemma}[Simple bounds]\label{simple-bounds}
We have $\frac{1}{2}\log 3 \leq \phi_j(\x) < \log 2$.
\end{lemma}

\begin{proof}
For the upper bound, note that $f^{\Gamma_j}(\x) \leq f^{\Gamma_j}(\mathbf{1})$. By the definition of $f^C$ (see \eqref{fc-eqs}), we have that $f^{\Gamma_j}(\mathbf{1})$ is equal to  the number of children of $\Gamma_j$ at level 0, which, in turn, is equal to $2^{2^j} - 1$. This proves the claimed upper bound on $\phi_j(\x)$. 

For the lower bound, observe that $f^{\Gamma_j}(\x) \geq f^{\Gamma_j}(\mathbf{0})$. Using again the definition of $f^C$, we find that $f^{\Gamma_j}(\mathbf{0})$ equals the number of children of $\Gamma_j$ at level $j-1$. Thus $f^{\Gamma_j}(\mathbf{0}) = 3^{2^{j-1}}$ by Lemma \ref{genotype-lemma}. This proves the claimed lower bound of $\phi_j(\x)$, thus completing the proof of the lemma.
\end{proof}

The $\rho$-equations \eqref{binary-rho-eqs} may be expressed in terms of the $\phi_j$ in the following simple form:

\begin{equation}\label{rho-eq-log}
\phi_{j+1}(\rho_j,\rho_{j-1},\ldots) = \frac{1}{2}\big(\rho_j \phi_j(\rho_{j-1},\rho_{j-2},\ldots) + 1\big).
\end{equation}

\subsection{Product structure of cells and self-similarity of the functions $\phi_j$}

There is a natural bijection $\rev{\pi : \Q^{\mathcal{P}(\N)} \times \Q^{\mathcal{P}(\N)} \rightarrow \Q^{\mathcal{P}(\N)}}$ defined by $\pi((x, x')) = y$, where $y_{A} = x_{A-1}$ and $y_{\{1\} \cup A} = x'_{A-1}$, for all $A \subset\{2,3,\dots\}$. Here, 
we write $A-1$ for the set $\{a-1:a\in A\}$. 
\rev{There is a finite version of this map that can be visualized as a concatenation map. For each $r$, let $\pi_r : \Q^{\mathcal{P}[r-1]} \times \Q^{\mathcal{P}[r-1]} \rightarrow \Q^{\mathcal{P}[r]}$ defined by $\pi((x, x')) = y$, where $y_{A} = x_{A-1}$ and $y_{\{1\} \cup A} = x'_{A-1}$, for all $A \subset\{2,3,\dots,r\}$. If we place the coordinates of $x$ and $x'$ in reverse binary order, as per the map $\{2,\dots,r\}\supset A\to \sum_{a\in A}2^{r-a}\in\{0,1,\dots,2^{r-1}-1\}$, then $\pi_r$ is the concatenation map that generates $y$ by placing first all coordinates of $x$, followed by all coordinates of $x'$.}

Now one may easily check that $\rev{\pi(V_{j-1} \times V_{j-1}) = V_j}$ for all $j=1,2,\dots$ Therefore if $C_1, C_2$ are two cells at level $(j-1)$ in the infinite binary system, then $\pi(C_1 \times C_2)$ is a cell at level $j$, and conversely every cell of level $j$  is of this form. The children $C'$ of $C$ are precisely $\pi(C'_1 \times C'_2)$ where $C_1 \rightarrow C'_1$, $C_2 \rightarrow C'_2$.

\medskip

The product structure established above manifests itself in a self-similarity property $\phi_j\approx \phi_{j-1}$. In this section, we will establish the following precise version of this.

\begin{proposition}\label{self-similar-prop}
Let $\alpha\in(0,1]$ and consider a vector $\rev{\mathbf{x}=(x_1,x_2,\dots)\in [0,\alpha]^\N}$. In addition, let $C = \pi(C_1 \times C_2)$ be a cell of level $j \geq 2$. Then we have
\begin{equation}\label{psis}\psi^{C_1}(\x) + \psi^{C_2}(\x) \leq \psi^C(\x) \leq  \psi^{C_1}(\x) + \psi^{C_2}(\x) + \alpha^{j-1} \log 2.\end{equation}
In particular, taking $C = \Gamma_j = \pi(\Gamma_{j-1} \times \Gamma_{j-1})$, we have 
\begin{equation}\label{basic-ss} \phi_{j-1}(\x) \leq \phi_j(\x) \leq \phi_{j-1}(\x) +  (\alpha/2)^j \frac{\log 2}{\alpha}  .\end{equation}
\end{proposition}
\begin{proof}
%We give the proof of the upper bound in \eqref{psis}, the lower bound being very similar. 
We proceed by induction on $j$. When $j = 2$, we proceed by hand. \rev{Notice that at level 1, there are three different types of cells, having 4, 2 and 1 elements, respectively. There is only one cell with 4 elements, the cell $\Gamma_1$; it splits into three cells at level 0: one with two elements, and two unicells (singletons). All other cells at level 1 split into unicells at level 0. Hence, at level 2, there are six different types of cells $C = \pi(C_1 \times C_2)$ corresponding to the six possibilities for the unordered pair $\{|C_1|, |C_2|\}$. Their subcells are in 1-1 correspondence with the cells $\pi(C_1'\times C_2')$, where $C_1'$ is a subcell of $C_1$ (at level 0) and $C_2'$ is a subcell of $C_2$ (also at level 0).}

\rev{The three cases with $\max(|C_1|,|C_2|\}\le 2$ are trivial, because we then have that all the cells at level 1 are unicells, and thus we readily find that $f^C=f^{C_1}f^{C_2}=|C_1|\cdot |C_2|$.

The two other cases with $|C_1|\le2$ and $|C_2|=4$ (so that $C_2=\Gamma_1$) are only slightly harder: if $|C_1|=2$, then $f^C(\x) = 2 \cdot 2^{x_1} + 4$, $f^{C_1} = 2$, $f^{C_2} = 3$ and so the desired inequalities are $\log6\le \log (2 \cdot 2^{x_1} + 4) \leq \log 6 + x_1 \log 2$, which are immediately seen to be true for all $x_1 \geq 0$. Similarly, if $|C_1|=1$, then $f^C(\x) = 2^{x_1} + 2$, $f^{C_1} = 1$, $f^{C_2} = 3$, and so the desired inequalities are $\log3\le \log (2^{x_1} + 2) \leq \log 3 + x_1 \log 2$, which are again true for all $x_1\ge0$.}

A little trickier is the case $|C_1| = |C_2| = 3$, corresponding to $C = \Gamma_2 = \pi(\Gamma_1 \times \Gamma_1)$. In this case $f^C(\x) = 3^{x_1} + 4 \cdot 2^{x_1} + 4$, $f^{C_1} = f^{C_2} =  3$, so the desired inequalities are $2\log3\le \log(3^x + 4 \cdot 2^x + 4) \leq 2 \log 3 + x \log 2$.  \rev{The lower bound is evident. For the upper bound, we must equivalently show that} $g(x) := 5 \cdot 2^x - 3^x - 4\ge0$ for $x\in[0,1]$. Since $g(0) = 0$ and $g'(x)  = 5\log 2 \cdot 2^x - \log 3 \cdot 3^x > 0$ for $x\le1$, the desired inequality follows.

Now suppose that $j \geq 3$, and assume the result is true for cells at level $(j-1)$. 
By the recursive definition of $f^C$, if $C$ is a cell at level $j$, we have the recurrence
\begin{equation}\label{basic-recur} e^{\psi^C(\x)} = \sum_{C \rightarrow C'} e^{x_{1}\psi^{C'}(T\x)},\end{equation} where $T\x$ denotes the {\it shift operator}
\[ T\x = (x_2,x_3,\ldots).\]
\rev{For the upper bound, note that}
\[
e^{\psi^C(\x)}  = \sum_{C \rightarrow C'} e^{x_{1}\psi^{C'}(T\x)} 
		 \leq \sum_{\substack{C_1 \rightarrow C'_1 \\ C_2 \rightarrow C'_2}} 
			e^{x_{1}(\psi^{C'_1}(T\x) + \psi^{C'_2}(T\x) + \alpha^{j-2} \log 2  )} .
\]
Recalling that $x_1\le\alpha$, we conclude that
\[
e^{\psi^C(\x)}  \leq 2^{\alpha^{j-1}} \Big(\sum_{C_1 \rightarrow C'_1} 
			e^{x_{1} \psi^{C'_1}(T\x)}\Big)
				\Big(\sum_{C_2 \rightarrow C'_2} e^{x_{1} \psi^{C'_2}(T\x)}\Big) 
		 = 2^{\alpha^{j-1}} e^{\psi^{C_1}(\x)} e^{\psi^{C_2}(\x)} .
\]
\rev{The lower bound is proven similarly.} The result thus follows.
\end{proof}

\subsection{Derivatives and the limit of the $\rho_i$.}

\rev{Because of the implicit definition of the parameters $\rho_i$,} the self-similarity property \eqref{basic-ss} is not enough for us by itself. We will also require the following (rather ad hoc) derivative bounds.

Here, and in what follows, \rev{$\partial_m F(y_1,\ldots) := \frac{\partial F}{\partial y_m}(y_1,\ldots)$}, that is to say the derivative of the function $F$ with respect to its $m$th variable. Thus, for instance,
\begin{equation}\label{shift-der-basic} 
\partial_m \psi^C (T\x)  = 
\frac{\partial}{\partial x_{m+1}}\big[ \psi^C(T\x) \big].
\end{equation}

\begin{proposition}\label{der-bounds}
Set $\Delta_m := \sup_{j \geq 2} \sup_{\rev{\mathbf{x}\in[0,0.31]^\N}}| \partial_m \phi_j(\x)|$.
Then $\Delta_1 < 0.17$, $\Delta_2 < 0.05$, $\sum_{m \geq 3} \Delta_m < 0.01$ and $\Delta_m \ll 0.155^m$.
\end{proposition}
The proof of this proposition is given in subsection \ref{derivative-sec}. Let us now show how this proposition, together with \eqref{basic-ss}, implies Proposition \ref{mainthm-limit}. 

\begin{proof}[Proof of Proposition \ref{mainthm-limit}]
Write $\eps_i := \rho_{i+1}- \rho_{i}$, $i = 1,2,3,\dots$ The $\rho$-equation at level $(j+1)$ is
\[ 
\phi_{j+2}(\rho_{j+1},\rho_j,\ldots) = \frac{1}{2}\big(\rho_{j+1} \phi_{j+1}(\rho_j,\rho_{j-1},\ldots) + 1\big)\]
by \eqref{rho-eq-log}. \rev{Recall that that $\rho_j\le \rho_1\le 0.31$ for all $j$, by Lemma \ref{rho-j-rho-1}. Hence, two applications of \eqref{basic-ss} (with $\alpha = 0.31$) yield the asymptotic formula}
\[ 
\phi_{j+1}(\rho_{j+1},\rho_j,\ldots) = \frac{1}{2}\big(\rho_{j+1} \phi_j(\rho_j,\rho_{j-1},\ldots) + 1\big) + O(0.155^j).
\] 
Subtracting \eqref{rho-eq-log}, the $\rho$-equation at level $j$, from this gives
\begin{align}\nonumber
\phi_{j+1} & (\rho_{j+1},\rho_j,\ldots) - \phi_{j+1}(\rho_j,\rho_{j-1},\ldots) \\ 
					& = \frac{\rho_{j+1}}{2}
						\big(\phi_j(\rho_j,\rho_{j-1},\ldots)
					- \phi_j(\rho_{j-1},\rho_{j-2},\ldots)\big) 
					+ \frac{\eps_j}{2} \phi_j(\rho_j,\rho_{j-1},\ldots) + O(0.155^j).\label{shift-minus}
\end{align}
Now by the mean value theorem,
\begin{equation}\label{mvt-1}
 |\phi_{j+1}  (\rho_{j+1},\rho_j,\ldots) - \phi_{j+1}(\rho_j,\rho_{j-1},\ldots)|
 	 \leq \Delta_1 |\eps_j| + \cdots + \Delta_j |\eps_1| 
 \end{equation} 
and
\begin{equation}\label{mvt-2} 
|\phi_j(\rho_j,\rho_{j-1},\ldots) - \phi_j(\rho_{j-1},\rho_{j-2},\ldots)|
	 \leq \Delta_1 | \eps_{j-1}| + \cdots + \Delta_{j-1} |\eps_1|.
\end{equation} 
Therefore, from \eqref{shift-minus}, the triangle inequality and the fact that $\frac{\rho_{j+1}}{2} \leq \frac{\rho_1}{2}\leq 0.155$, we have
\begin{equation}\label{eps-j-bd} 
\begin{split}
|\eps_j| \Big(\frac{1}{2}\phi_j(\rho_j,\rho_{j-1},\ldots) - \Delta_1\Big) 
		&\leq  (\Delta_2 + 0.155 \Delta_1) |\eps_{j-1}| + (\Delta_3 + 0.155 \Delta_2) |\eps_{j-2}| + \cdots \\
		&\quad +O(0.155^j).
\end{split}
\end{equation} 
Now by Lemma \ref{simple-bounds} and Proposition \ref{der-bounds}, 
\[ \frac{1}{2}\phi_j(\rho_j,\rho_{j-1},\ldots) - \Delta_1  > \frac{1}{4}\log 3 - 0.17  > 0.104.\] 
Also, by Proposition \ref{der-bounds} we have
\[ (\Delta_2 + 0.155\Delta_1) + (\Delta_3 + 0.155\Delta_2) + \cdots  < 0.096.  \] 
Assuming that $j\geq j_0$ with $j_0$ large enough, \eqref{eps-j-bd} implies a bound
\begin{equation}\label{eps-j-bd-new} 
|\eps_j| \leq c_1 |\eps_{j-1}| + c_2 |\eps_{j-2}| + \cdots + c_{j-1}|\eps_1| + 2^{-j},
\end{equation} 
where $c_1,c_2,\ldots$ are fixed nonnegative constants with $\sum_{i} c_i < \frac{0.096}{0.104} < 0.93$ and, by Proposition \ref{der-bounds}, $c_i \leq 2^{-i}$ for all $i \geq i_0$ for some $i_0$. It is convenient to assume that $i_0, j_0 \geq 10$, which we clearly may.

We claim that \eqref{eps-j-bd-new} implies exponential decay of the $\eps_j$, which of course immediately implies Theorem \ref{mainthm-limit}. To see this, take $\delta \in (0, \frac{1}{4})$ so small that $0.94 (1 - \delta)^{-i_0} < 0.99$, and then take $A \geq 100$ large enough that $|\eps_j| \leq A(1- \delta)^j$ for all $j \leq j_0$. We claim that the same bound holds for all $j$, which follows immediately by induction using \eqref{eps-j-bd-new} provided one can show that 
\begin{equation}\label{to-show} 
\sum_{i \geq 1} c_i (1 - \delta)^{-i} + \frac{1}{A}\Big(\frac{1}{2(1 - \delta)}\Big)^j < 1 
\end{equation}
for $j \geq j_0$. Since $\delta < \frac{1}{2}$ and $A \geq 100$, it is enough to show that $\sum_{i \geq 1} c_i(1 - \delta)^{-i} < 0.99$. The contribution to this sum from $i \leq i_0$ is at most $0.93 (1 - \delta)^{-i_0}$, whereas the contribution from $i > i_0$ is (by summing the geometric series) at most $\sum_{i > i_0} 2^{-i}(1 - \delta)^{-i} < 2 \cdot 2^{-i_0}(1 - \delta)^{-i_0} < 0.01 (1 - \delta)^{-i_0}$. Therefore the desired bound follows from our choice of $\delta$.
\end{proof}

\subsection{Self-similarity for derivatives}\label{derivative-sec}

Our remaining task is to prove Proposition \ref{der-bounds}. Once again we use self-similarity of the $\phi_j$, but now for their derivatives, the key point being that $\partial_m \phi_j \approx \partial_m \phi_{j-1}$. Here is a precise statement.

\begin{proposition}\label{next-ders}
Suppose that $C = \pi(C_1 \times C_2)$ is cell at level \rev{$j\ge1$}.  Let $\alpha\in[0,1)$ and $m \geq 1$, and suppose that \rev{$\mathbf{x}\in[0,\alpha]^\N$}. Then we have
\[ 
0\le \partial_m \psi^C(\x) 
	\leq 2^{\sum_{i = 1}^m \alpha^{j-i}} \big( \partial_m \psi^{C_1}(\x) 
	+ \partial_m\psi^{C_2}(\x) + \alpha^{j-2} \log 2 \big).
\]
In particular, taking $C = \Gamma_j = \pi(\Gamma_{j-1} \times \Gamma_{j-1})$, we have
\begin{equation}\label{ss-der}
0\le \partial_m \phi_j(\x) \leq 2^{\sum_{i = 1}^m \alpha^{j-i}} \Big(  \partial_m \phi_{j-1}(\x) +\pfrac{\alpha}{2}^j \frac{\log 2}{\alpha^2} \Big). 
\end{equation}
\end{proposition}
\begin{proof} The lower bound follows by noticing that $\psi^C$ is increasing in each variable.  For the upper bound, we may assume that $m\le j-1$, for when
$m\ge j$, $\partial_m \phi_j(\x)$ is identically zero. We proceed by induction on $m$, first establishing the case $m = 1$.
Differentiating \eqref{basic-recur} gives
\[  e^{\psi^C(\x)} \partial_1 \psi^C(\x)= \sum_{C \rightarrow C'} \psi^{C'}(T\x) e^{x_1 \psi^{C'}(T\x)}.
\]
 By two applications of the upper bound in Proposition \ref{self-similar-prop} (to $C' = \pi(C'_1 \times C'_2)$), we obtain
\begin{equation}\label{eq2A} e^{\psi^C(\x)}\partial_1 \psi^C(\x)  \leq 2^{\alpha^{j-1}} \sum_{\substack{C_1 \rightarrow C'_1 \\ C_2 \rightarrow C'_2}} \big( \psi^{C'_1}(T\x) + \psi^{C'_2}(T\x) + \alpha^{j-2}\log 2 \big) e^{x_1(\psi^{C'_1}(T\x) + \psi^{C'_2}(T\x))}.
\end{equation}
On the other hand, for $i = 1,2$ we get by differentiating the recurrence 
\begin{equation}\label{eq3pre} e^{\psi^{C_i}(\x)} = \sum_{C_i \rightarrow C'_i} e^{x_1 \psi^{C'_i}(T\x)}\end{equation} with respect to $x_1$ that
\begin{equation}\label{eq3A}
e^{\psi^{C_i}(\x)} \partial_1 \psi^{C_i}(\x)  = \sum_{C_i \rightarrow C'_i} \psi^{C'_i}(T\x) e^{x_1 \psi^{C'_i}(T\x)}.
\end{equation} 
Substituting \eqref{eq3pre} and \eqref{eq3A} into \eqref{eq2A} gives
\begin{equation*} e^{\psi^C(\x)}  \partial_1 \psi^C(\x)  \leq 2^{\alpha^{j-1}}\big(  \partial_1 \psi^{C_1}(\x) + \partial_1 \psi^{C_2}(\x) + \alpha^{j-2} \log 2  \big) e^{\psi^{C_1}(\x) + \psi^{C_2}(\x)}.
\end{equation*}
Finally, Proposition \ref{self-similar-prop}
implies that $e^{\psi^{C_1}(\x) + \psi^{C_2}(\x)}\le e^{\psi^C(\x)}$.  Dividing both sides by $e^{\psi^C(\x)}$ gives the result when $m = 1$.

Now suppose that $m \geq 2$. Differentiating \eqref{basic-recur} with respect to $x_{m}$ and applying \eqref{shift-der-basic} gives
\begin{equation}\label{eq6} e^{\psi^C(\x)} \partial_m \psi^C(\x) = \sum_{C \rightarrow C'} x_{1} e^{x_1 \psi^{C'}(T\x)} \partial_{m-1} \psi^{C'}(T\x) .
\end{equation} 
 By the inductive hypothesis, if $C' = \pi(C'_1 \times C'_2)$ we have 
\begin{equation}\label{eq7}
 \partial_{m-1} \psi^{C'}(T\x) \leq 2^{\sum_{i=2}^m \alpha^{j - i}} \Big( \partial_{m-1} \psi^{C'_1}(T\x) + \partial_{m-1} \psi^{C'_2}(T\x)  + \alpha^{j-3} \log 2 \Big).
 \end{equation}
Also, by the upper bound in Proposition \ref{self-similar-prop}, we have
\begin{equation}\label{eq8} \psi^{C'}(T\x) \leq \psi^{C'_1}(T\x) + \psi^{C'_2}(T\x) + \alpha^{j-2} \log 2.\end{equation}
Substituting \eqref{eq7} and \eqref{eq8} into \eqref{eq6} and using the assumption that $0\le x_{1} \leq \alpha$ gives
\begin{align}
\nonumber e^{\psi^C(\x)} \partial_m \psi^C(\x) \nonumber &  \leq 2^{\sum_{i=1}^m \alpha^{j-i}} \times \\
&\sum_{\substack{C_1 \rightarrow C'_1 \\ C_2 \rightarrow C'_2}} x_1\Big[ \partial_{m-1} \psi^{C'_1}(T\x) + \partial_{m-1} \psi^{C'_2}(T\x)+ \alpha^{j-3}\log 2  \Big]  e^{x_1(\psi^{C'_1}(T\x)+ \psi^{C'_2}(T\x))}.
\label{eq9}\end{align}
Now, differentiating the recurrence \eqref{eq3pre} with respect to $x_{m}$ (using \eqref{shift-der-basic}) gives, for $i = 1, 2$,
\begin{equation}\label{eq10}
  e^{\psi^{C_i}(\x)}\partial_{m} \psi^{C_i}(\x)    = \sum_{C_i \rightarrow C'_i} x_{1}  e^{x_1 \psi^{C'_i}(T\x)}\partial_{m-1} \psi^{C'_i}(T\x).\end{equation} 
Substituting \eqref{eq3pre} and \eqref{eq10} into \eqref{eq9}, and using once again that $x_1\le\alpha$, gives
\[
 e^{\psi^C(\x)}\partial_{m} \psi^C(\x)  e^{\psi^C(\x)}
  \leq 2^{\sum_{i=1}^m \alpha^{j-i}} 
  \Big( \partial_{m} \psi^{C_1}(\x) + \partial_{m} \psi^{C_2}(\x)  +\alpha^{j-2}\log 2  \Big) e^{\psi^{C_1}(\x) + \psi^{C_2}(\x)  }. 
  \label{eq11}
  \]
Again, Proposition \ref{self-similar-prop} 
implies that $ e^{\psi^{C_1}(\x) + \psi^{C_2}(\x)  } \le e^{\psi^C(\x)}$, and so by
 dividing both sides by $e^{\psi^C(\x)}$, we obtain the stated result.
\end{proof}

Before proving Proposition \ref{der-bounds}, we isolate a lemma.

\begin{lemma}\label{phi2-lem} \rev{For $0 \leq x_1\leq 0.31$ we have $0 \leq 4\partial_1\phi_2(\mathbf{x}) \leq 0.481$.}
\end{lemma}
\begin{proof} \rev{We have $e^{4\phi(\mathbf{x})}= 3^{x_1} + 4\cdot 2^{x_1} + 4$, and thus}
\[ 
\rev{4\partial_1\phi_2(\mathbf{x})}  = \frac{\log 3 \cdot 3^{x_1} + \log 2 \cdot 4\cdot 2^{x_1}}{3^{x_1} + 4\cdot 2^{x_1} + 4}.
\] 
The lemma is therefore equivalent to $\frac{1}{4}(\log 3 - 0.481) 3^{x_1} + (\log 2 - 0.481)2^{x_1} \leq 0.481$. The left-hand side here is increasing in $x_1$ and, when $x_1 = 0.31$, it is equal to $0.480052\cdots$.
\end{proof}

\begin{proof}[Proof of Proposition \ref{der-bounds}]
Henceforth, set $\alpha := 0.31$ \rev{and fix two integers $m\ge1$ and $j\ge2$. Our goal is to bound $\partial_m\phi(\mathbf{x})$ uniformly for $\mathbf{x}\in[0,\alpha]^\N$. We} may assume that $j\ge m+1$, as $\partial_m \phi_j(\x)=0$ when $j\le m$. 

\rev{Now, let us define 
	\[ A_m := 2^{1 + \alpha + \cdots + \alpha^{m-1}}
	\quad\text{and}\quad
	B_m := 2^{\frac{1 + \alpha + \cdots + \alpha^{m-1}}{1 - \alpha}}.\]
Then,} if we apply \eqref{ss-der} $\ell$ times, we obtain
\begin{align}\nonumber 
0\le \partial_{m}  \phi_j(\x) 
	& \leq A_m^{\alpha^{j-m} + \cdots + \alpha^{j - m - (\ell - 1)}}\partial_{m}\phi_{j - \ell}( \x) + \frac{\log 2}{\alpha^2} \sum_{k = 0}^{\ell-1} A_m^{\alpha^{j-m} + \cdots + \alpha^{j - m - k}} \pfrac{\alpha}{2}^{j - k} \\ 
	& \leq B_m^{\alpha^{j - m - (\ell - 1)}}\partial_{m} \phi_{j - \ell}(\x) + \frac{\log 2}{\alpha^2} \sum_{k = 0}^{\ell - 1} B_m^{\alpha^{j - m - k}} \pfrac{\alpha}{2}^{j - k} \nonumber\\
	&\leq B_m^{\alpha^{j-m-\ell+1}} \Big( \partial_{m} \phi_{j - \ell}(\x) +  \frac{\log 2}{\alpha^2} \pfrac{\alpha}{2}^{j-\ell+1} \frac{1}{1 - \alpha/2} \Big) .
	\label{repeated}
\end{align}
\rev{Here, we observed that all the $B_m^{\alpha^t}$ terms in \eqref{repeated} have $t \geq s + 1 - m$; bounding them all above by $B_m^{\alpha^{s + 1 - m}}$ then allowed us to sum a geometric series.}

\rev{Let us fix some $s\in\{1,2,\dots,m+1\}$ independent of $j$. Then the number $j-s$ lies in $\{0,1,\dots,j-1\}$. Hence, applying \eqref{repeated} with $\ell = j - s$, and then taking the supremum over all $j\ge m+1$ and all $\mathbf{x}\in[0,\alpha]^\N$, we find that}
\begin{equation}\label{to-use-bd} 
\Delta_m
	 \leq B_m^{\alpha^{s + 1 - m}} 
	\Big( \sup_{\rev{\mathbf{x}\in[0,\alpha]^\N}}
		 |\partial_{m}\phi_s(\x)| + \frac{\log 2}{\alpha^2} \pfrac{\alpha}{2}^{s+1} \frac{1}{1 - \alpha/2} \Big).
\end{equation} 
\rev{When} $m = 1$, we take $s = 2$. Then Lemma \ref{phi2-lem} and relation \eqref{to-use-bd} give
\[ 
\Delta_1 \leq 2^{\alpha^2/(1 - \alpha)}\(\frac{0.481}{4} + \frac{\alpha \log 2}{8(1 - \alpha/2)}\) < 0.17,
\] 
as required. \rev{When} $m \geq 2$, we take $s = m$. Then $\partial_{m}\phi_{s} \equiv 0$ and so \eqref{to-use-bd} degenerates to 
\begin{equation}\label{use-new} 
\Delta_m \leq B_m^{\alpha} \frac{\log 2}{\alpha^2} \pfrac{\alpha}{2}^{m+1} \frac{1}{1 - \alpha/2}.
\end{equation}
This gives $\Delta_2 < 0.05$, and also confirms that $\Delta_m \ll 0.155^m$. To bound $\sum_{m \geq 3} \Delta_m$ we use \eqref{use-new} and the uniform bound $B_m \leq 2^{1/(1 - \alpha)^2}$, obtaining
\[ 
\sum_{m \geq 3} \Delta_m \leq \frac{\alpha^2 \log 2}{16(1 - \alpha/2)^2} 2^{\alpha/(1 - \alpha)^2} < 0.01.
\] 
This completes the proof of Proposition \ref{der-bounds}.
\end{proof}

\section{Calculating the $\rho_i$ and $\rho$}\label{sec:analytic}

In this section we conclude our analysis of the parameters \rev{$\rho_1,\rho_2,\ldots$} for the binary flags. The situation so far is that we have shown that these parameters exist, are unique and lie in $(0, 0.31)$. Moreover, their limit $\rho = \lim_{i \rightarrow \infty} \rho_i$ exists (Proposition \ref{mainthm-limit}).

None of this helps with actually computing the limit numerically or giving any kind of closed form for it, and the objective of this section is to provide tools for doing that. \rev{We prove two main results, Propositions \ref{rho-i-comp} and \ref{rho-comp} below. Recall the convention that $\rho_0 = 0$.}

\begin{proposition}\label{rho-i-comp}
\rev{Recall the convention that $\rho_0 = 0$.} Define a sequence \rev{$(a_{i,j})_{i\ge1,\,1\le j\le i+1}$} by the relations $a_{i,1}=2$, $a_{i,2}=2+2^{\rho_{i-1}}$ and \rev{\begin{equation}\label{aij-recur} a_{i,j} = a_{i,j-1}^2+a_{i-1,j-1}^{\rho_{i-1}}
	- a_{i-1,j-2}^{2\rho_{i-1}} \;\; (3\le j\le i+1).\end{equation}}
	Then 
	\begin{equation}\label{rho-eqs-a-form} a_{i,i+1} = a_{i-1,i}^{\rho_{i-1}} e^{2^{i-1}}\quad\rev{\text{for}\ i=2,3,\dots}
	\end{equation}
\end{proposition}
In practice, these relations are enough to calculate the $\rho_j$ to high precision. Indeed, a short computer program produced the data in Table \ref{rhoj}. (We suppress any discussion of the numerical precision of our routines.)

\begin{table}[h]
\begin{center}
\begin{tabular}{|c|c| c| c| }
\hline
$j$ & $\rho_j$ & $j$ & $\rho_j$ \rule[-8pt]{0pt}{22pt}\\
\hline\rule[-8pt]{0pt}{22pt}
1 & 0.3064810093305& 7&0.2812113502101\\
2 & 0.2796104150767 &8&0.2812113496729\\ 
3 & 0.2813005404710 & 9&0.2812113496974\\
4& 0.2812067224539& 10&0.2812113496963\\
5&0.2812115789381& 11&0.2812113496964\\
6&0.2812113387071&  12&0.2812113496964 \\
% 13&0.2812113496964\\
\hline
\end{tabular}
\medskip

\caption{Table of $\rho_j$}\label{rhoj}
\end{center}
\end{table}

Using Proposition \ref{rho-i-comp} we may obtain the following reasonably satisfactory description of $\rho$, \rev{which is equivalent to the statement of Theorem \ref{beta-rho} (c). }
\begin{proposition}\label{rho-comp}
\rev{For each $t \in (0,1)$, define a sequence $a_j(t)$ by 
\begin{equation}\label{a-recur} a_1(t) = 2, \; a_2(t) = 2 + 2^{t}, \; a_j(t) = a_{j-1}(t)^2 + a_{j-1}(t)^{t} - a_{j-2}(t)^{2t} \; \; (j \geq 3).\end{equation}
Then the limit $\rho = \lim_{i \rightarrow \infty} \rho_i$ is a solution (in the variable $t$) to the equation 
\be\label{rho-limit-eq-2}
\frac{1}{1 - t/2} = \lim_{j\to\infty} \frac{\log a_j(t)}{2^{j-2}}.
\ee
Furthermore, $\rho$ is the unique solution to \eqref{rho-limit-eq-2} in the interval $0\le t\le 1/3$.}
\end{proposition}
\rev{\emph{Remark.} This is easily seen to be equivalent to Theorem \ref{beta-rho} (c), but we have introduced $t$ as a dummy variable since $\rho$ now has the specific meaning $\rho = \lim_{i \rightarrow \infty} \rho_i$, and this will avoid confusion in the proof.

Before starting the proofs of Propositions \ref{rho-i-comp} and \ref{rho-comp}, let us pause to observe a simple link between the sequences $a_{i,j}$ and $a_j(t)$ defined in \eqref{aij-recur} and \eqref{a-recur} respectively.

\begin{lemma}\label{lemmax} For each fixed $j\ge1$, the limit $\lim_{i \rightarrow \infty} a_{i,j}$ exists and equals $a_j(\rho)$.  \end{lemma}
\begin{proof}
The existence of the limit follows by induction on $j$, using Proposition \ref{mainthm-limit}, noting that the result is trivial for $j = 1$ and immediate from Proposition \ref{mainthm-limit} when $j = 2$. The fact that the limit equals $a_j(\rho)$ then follows immediately by letting $i \rightarrow \infty$ in \eqref{aij-recur} and comparing with \eqref{a-recur}.
\end{proof}
}

\subsection{Product formula for $f^C(\rho)$ and a double recursion for the $\rho_i$}

Proposition \ref{rho-i-comp} is a short deduction from a product formula for $F(g)$, or equivalently for $f^C(\brho)$, given in Proposition \ref{recursiveformulas} below. Whilst is would be a stretch to say that this formula is of independent interest, it is certainly a natural result to prove in the context of our work. 

Before we state the formula, the reader should recall the notion of genotype $g$ (Definition \ref{gen-def}) and of the function $F(g)$ (Proposition \ref{gen-fc}). We require the following further small definition.

\begin{definition}[Defects]
\rev{Let $i,m\in\Z_{\ge0}$ and let} $g$ be an $i$-genotype. 

\rev{(a) If $m\le i$, then we define} the $m$th consolidation
\[ g^{(m)} := \{ A' \subset [i - m]: A' \cup X \in g \; \mbox{for all $X \subset \{i - m+1,\rev{\dots}, i\}$}\}.\] 
Otherwise, if $m \geq i+1$, then by convention we define $g^{(m)}$ to be empty.

\rev{(b) For $m\ge1$, we set}
\[ \Delta^m(g) := |g^{(m-1)}| - 2 |g^{(m)}|. \] 
\end{definition}

\begin{remark*} 
	Note that $g^{(0)} = g$, $g^{(1)} = g^*$ and $g^{(m)} = (g^{(m-1)})^*$. It is easy to see that $\Delta^m(g)$ is always a nonnegative integer. Observe that $\Delta^{i+1}(g) = 0$ unless $g = \mathcal{P}[i]$, in which case $\Delta^{i+1}(g) = 1$, and that $\Delta^m(g) = 0$ whenever $m > i+1$.
\end{remark*}

\begin{proposition}\label{recursiveformulas}
\rev{Let $i\in\N$ and} suppose that $g$ is an $i$-genotype. Then
\[\
 F(g) = \rev{\prod_{m=1}^{i+1}} a_{i,m}^{\Delta^m(g)},
\] with the $a_{i,m}$ defined as in Proposition \ref{rho-i-comp} above. 
	\end{proposition}
	
\begin{proof}[Proof of Proposition \ref{rho-i-comp}, given Proposition \ref{recursiveformulas}] 
\rev{Note that $\Delta^m(\mathcal{P}[i]) = 1_{m=i+1}$ for $1\le m\le i+1$. Together with Proposition \ref{recursiveformulas}, this implies that $F(\mathcal{P}[i]) = a_{i, i+1}$.} Thus $f^{\Gamma_i}({\bm \rho}) = F(\mathcal{P}[i]) = a_{i, i+1}$. The equation \eqref{rho-eqs-a-form} is then an immediate consequence of the $\rho$-equations \eqref{binary-rho-eqs}.
\end{proof}

Before turning to the proof of Proposition \ref{recursiveformulas}, we isolate a couple of lemmas from the proof.

\begin{lemma}\label{first}
Let $\alpha \in \R$ \rev{and $i\in\N$}. Let $g$ be an $i$-genotype, and suppose that $k$ is an $(i-1)$-genotype with $k \leq g^*$. Then
\[ \sum_{\substack{g' \leq g \\ (g')^* = k}} \alpha^{|g'|} = (1 + \alpha)^{\Delta^1(g)}  (1 + 2\alpha)^{|g^*| - |k|}\alpha^{2 |k|}.\]
\end{lemma}

\begin{proof}
	\rev{We have $g=\{A\subset[i-1]:A\in g\}\cup\{A\subset[i-1]:A\cup\{i\}\in g\}$. Hence, if we let $X=\{A\subset[i-1]:A\in g,\ A\cup\{i\}\notin g\}$ and $Y=\{A\subset[i-1]:A\notin g,\ A\cup\{i\}\in g\}$, then we have $|g|=2|g^*|+|X|+|Y|$, and thus $\Delta^1(g)=|X|+|Y|$. 
	
	Now, in order to choose $g'\le g$ with $(g')^*=k$, we must decide independently for each $A\subset[i-1]$ whether $A\in g'$ and/or $A\cup\{i\}\in g'$. The condition that $g'\le g$ means that if $A\notin g$ (resp. if $A\cup\{i\}\notin g$), then we are forced to have $A\notin g'$ (resp. $A\cup\{i\}\notin g'$). Let us now examine all admissible options for the conditions ``$A\in g'$'' and ``$A\cup\{i\}\in g'$'':
	\begin{itemize}
		\item $A\in\ k$: since $(g')^*=k$, we are forced to have $A,A\cup\{i\}\in g'$. 
		\item $A\in g^*\setminus k$: we know in this case that $A,A\cup\{i\}\in g$, so the condition $g'\le g$ imposes no further restrictions on the membership of $A$ and of $A\cup\{i\}$ in $g'$. On the other hand, we know that $A\notin k=(g')^*$, and thus at most one out of $A$ and of $A\cup\{i\}$ may belong to $g'$. 
		\item $A\in X$: the condition $g'\le g$ implies the restriction that $A\cup\{i\}\notin g'$, and we may then choose freely among the two options of having $A\in g'$ or $A\notin g'$.
		\item $A\in Y$: the condition $g'\le g$ implies the restriction that $A\notin g'$, and we may then choose freely among the two options of having $A\cup\{i\}\in g'$ or $A\cup\{i\}\notin g'$.
	\end{itemize}
	By the above discussion, we have 
\begin{align*}
\sum_{\substack{g' \leq g \\ (g')^* = k}} \alpha^{|g'|}
	 = \alpha^{2|k|} \prod_{A\in g^*\setminus k} (1+\alpha+\alpha) \prod_{A\in X} (1+\alpha) \prod_{A\in Y}(1+\alpha).
\end{align*}	 
Since $|X|+|Y|=\Delta^1(g)$, the proof is complete.}
%For brevity, for any genotype $g$ at level $i$ and $A\subset [i]$ we write $g_A = \mathbf{1}_{A\in g}$, the indicator function of $A\in g$. For $A\subset [i-1]$, write $a=g'_A$ and $b=g'_{A\cup \{i\}}$.
% In this notation, the sum is
%\begin{equation*}\label{sum-a-coord}
% \prod_{A \subseteq [i-1]} T_A, \qquad T_A := \sum_{\substack{0\le a\leq g_A \\ 0\le b \le g_{A\cup \{i\}} \\ ab = k_A }} \alpha^{a+b}.
%\end{equation*}
%Table \ref{table-genotype} then contains all of the information needed to complete the calculation.
%\begin{table}[h]
%\begin{tabular}{|l|l|l|l|l|l|}
%\hline
%              $(g_A, g_{A\cup \{i\}})$    & $g^*_A$ & $k_A$ & 
%              $(a,b)$ &   $T_A $ &  $\# A $                \\ \hline
%             $(0,0)$    & $0$ & $0$ & $(0,0)$  &                  $1$ &                   \\ \hline
%              $(0,1)$    & $0$ & $0$ & $(0,0)$ or $(0,1)$ & \multirow{2}{*}{} & \multirow{2}{*}{} \\ \cline{1-4}
%                $(1,0)$  & $0$ & $0$  & $(0,0)$ or $(1,0)$  &          $1 + \alpha$        &   $|g| - 2|g^*|=\Delta^1(g)$             \\ \hline
%\multirow{2}{*}{} & $1$ & $1$ & $(1,1)$ &                  $\alpha^2$ &  $|k|$\multirow{2}{*}{} \\ \cline{2-6}
%                $(1,1)$  & $1$ & $0$ & $(0,0)$, $(1,0)$ or $(0,1)$ &     $1 + 2 \alpha$                        & $|g^*| - |k|$ \\ \hline
%\end{tabular}
%\medskip
%\caption{Combinatorics of genotypes}\label{table-genotype}
%\end{table}
\end{proof}

For $\mathbf{a} = (a_1,a_2,\dots)$, and for some ($i$-)genotype $g$, write 
\begin{equation}\label{p-def} P_{\mathbf{a}}(g) := \rev{\prod_{m=1}^{i+1}} a_m^{\Delta^m(g)}.\end{equation}
\rev{ (Note that the $a_m$ here are just parameters, not related to the recursion \eqref{a-recur}, which does not feature in this subsection.)}
 If $\theta \in \R_{> 0}$, define \begin{equation}\label{phi-def} \Phi_{\theta,\mathbf{a}}(g) := \sum_{g' \leq g} \theta^{|g| - |g'|} P_{\mathbf{a}}(g').\end{equation}

\begin{lemma}\label{phi-functional}
We have the functional equation
\[
\Phi_{\theta,\mathbf{a}}(g) = (\theta + a_1)^{\Delta^1(g)} \Phi_{\theta^2 + 2a_1 \theta, T\mathbf{a}}(g^*).
\] 
As before, $T\mathbf{a}$ denotes the shift operator $T\mathbf{a} = (a_2,a_3,\cdots)$.
\end{lemma}
\begin{proof}
Using the relation $P_{\mathbf{a}}(g') = a_1^{\Delta^1(g')} P_{T\mathbf{a}}((g')^*)$, we have
\begin{align*}
\Phi_{\theta,\mathbf{a}}(g) 
	& = \theta^{|g|} \sum_{g' \leq g} \Big(\frac{a_1}{\theta}\Big)^{|g'|} 
		\Big(\frac{1}{a_1^2}\Big)^{|(g')^*|} P_{T\mathbf{a}}((g')^*) \\ 
	& = \theta^{|g|}\sum_{k \leq g^*} \Big(\frac{1}{a_1^2}\Big)^{|k|} 
	P_{T\mathbf{a}}(k) 
	\sum_{\substack{g' \leq g \\ (g')^* = k}} \Big(\frac{a_1}{\theta}\Big)^{|g'|}.
\end{align*}
The result now follows from Lemma \ref{first} and a routine short calculation.
\end{proof}

We are now in a position to prove Proposition \ref{recursiveformulas}.

\begin{proof}[Proof of Proposition \ref{recursiveformulas}]
Let $a_{i,m}$ be as in the statement of Proposition \ref{recursiveformulas}, and write $\mathbf{a}_i = (a_{i,1},a_{i,2},\dots)$. In the notation introduced above (cf.~\eqref{p-def}) the claim of Proposition \ref{recursiveformulas} is then that
\begin{equation}\label{to-prove-recurs} F(g) = P_{\mathbf{a}_i}(g).\end{equation}
We proceed by induction on $i$. \rev{Let us first consider the base case when $i=1$. 
	\begin{itemize}
		\item If $g=\mathcal{P}[1]$, we have $F(g)=f^{\Gamma_1}(\bs\rho)=3$. On the other hand, $P_{\mathbf{a}_1}(\mathcal{P}[1])=a_{1,2}=3$ in this case by the convention that $\rho_0=0$.
		\item If $g\subsetneqq\mathcal{P}[1]$, then $g^*=\emptyset$ and thus $\Delta^1(g)=|g|$ and $\Delta^2(g)=0$. So we conclude that $P_{\mathbf{a}_1}(g)=2^{|g|}$. On the other hand, for all such genotypes, the corresponding cell contains $2^{|g|}$ elements that all split into unicells at level 0. Consequently, $F(g)=2^{|g|}=P_{\mathbf{a}_1}(g)$ in this case too.
		\end{itemize}}	
	
Next, suppose that we have the result for $(i-1)$-genotypes \rev{for some $i\ge2$}, and let $g$ be an $i$-genotype. \rev{We know from \eqref{geno-recurs} that
\[
F(g) = \sum_{g' \leq g^*} 2^{|g| - |g^*| - |g'|} F(g')^{\rho_{i-1}}.
\]
By the induction hypothesis, we have $F(g')^{\rho_{i-1}} = P_{\mathbf{a}_{i-1}^{\rho_{i-1}}}(g')$ for all $g'\le g^*$, where $\mathbf{a}_{i-1}^{\rho_{i-1}}$ is shorthand for $(a_{i-1,1}^{\rho_{i-1}}, a_{i-1,2}^{\rho_{i-1}},\dots)$. Hence, it follows immediately that}
\begin{equation}\label{fphi} 
	F(g) = 2^{\Delta^1(g)} \Phi_{2,\mathbf{a}_{i-1}^{\rho_{i-1}}}(g^*).\end{equation} 
with $\Phi$ defined in \eqref{phi-def}. 
The fact that the right-hand side of \eqref{fphi} is a product $P_{\ast}(g)$ is now clear by an iterated application of Lemma \ref{phi-functional}. To get a handle on exactly which product, suppose that the result of applying Lemma \ref{phi-functional} $j-1$ times is that 
\begin{equation}\label{fphi-iterated} 
F(g) = \Big(\prod_{m = 1}^{j} \rev{b}_{i,m}^{\Delta^m(g)}\Big)
	 \Phi_{\rev{\theta}_{i, j}, T^{j-1} (\mathbf{a}_{i-1}^{\rho_{i-1}})} (g^{(j)}).
\end{equation}
Thus $\rev{b}_{i,1} = \rev{\theta}_{i,1} = 2$, and we have the relations 
\begin{equation}\label{eq-ab-1} 
b_{i,j+1} = \theta_{i,j} + a_{i-1, j}^{\rho_{i-1}}
\end{equation} 
and
\begin{equation}\label{eq-ab-2} 
\theta_{i,j+1} = \theta_{i,j}^2 + 2 a_{i-1,j}^{\rho_{i-1}} \theta_{i,j}
\end{equation}
\rev{for $j\in\{1,\dots,i\}$. We claim that $b_{i,j}=a_{i,j}$ for all $j\le i+1$. This will complete the proof of Proposition \ref{recursiveformulas}, because we may then apply \eqref{fphi-iterated} with $j=i+1$ to show that
	\[
	F(g) = \Big(\prod_{m = 1}^{i+1} a_{i,m}^{\Delta^m(g)}\Big)
	\Phi_{\rev{\theta}_{i, i+1}, T^{i+1} (\mathbf{a}_{i-1}^{\rho_{i-1}})} (g^{(i+1)})
	= \prod_{m = 1}^{i+1} a_{i,m}^{\Delta^m(g)}
	\]
because $g^{(i+1)}=\emptyset$ for all $i$-genotypes $g$. 

Let us now prove our claim that $b_{i,j}=a_{i,j}$ for all $j\le i+1$. We shall use induction on $j$. We have that $b_{i,1}=2=a_{i,1}$. In addition, $b_{i,2}=2+2^{\rho_{i-1}}=a_{i,2}$ by \eqref{eq-ab-1} with $j=1$ and by the fact that $\theta_{i,1}=2$. Now, assume that we have proven that $b_{i,j}=a_{i,j}$ for some $j\in\{2,\dots,i\}$. Relation \eqref{eq-ab-2} applied with $j-1$ in place of $j$ implies that 
	\[
	\theta_{i,j}+a_{i-1,j-1}^{2\rho_{i-1}} 
		= \big(\theta_{i,j-1}+ a_{i-1,j-1}^{\rho_{i-1}}\big)^2.
	\]
The right-hand side equals $b_{i,j}^2=a_{i,j}^2$ by applying \eqref{eq-ab-1} followed by the induction hypothesis. Thus, $\theta_{i,j}=a_{i,j}^2-a_{i-1,j-1}^{2\rho_{i-1}}$. Inserting this relation into \eqref{eq-ab-1} and using the recursive formula \eqref{aij-recur} shows that $b_{i,j+1}=a_{i,j+1}$. This completes the inductive step and thus the proof of Proposition \ref{recursiveformulas}.}
\end{proof}

\subsection{A single recurrence for $\rho$}\label{rho-i-to-rho}

In this section we deduce Proposition \ref{rho-i-to-rho}
from Proposition \ref{rho-i-comp} by a limiting argument.

\rev{To carry this out, we will need the following fairly crude estimates for the $a_{i,j}$ and the $a_j(t)$, defined in \eqref{aij-recur} and \eqref{a-recur} respectively.

\begin{lemma}\label{ai-crude}
We have
\begin{equation}\label{a-double} 
	a_{i,j+1} \leq a_{i,j}^2\quad\text{for}\ 1\le j\le i
\end{equation}
and
\begin{equation}\label{basic-aij-ineqs} 
	3^{2^{j-2}} \leq a_{i,j} \leq a_{i,2}^{2^{j-2}}\le 4^{2^{j-2}}
	\quad\text{for}\ 2\le j\le i+1.
\end{equation}
\end{lemma}
\begin{proof}
Since $\rho_{i-1}<1$ for all $i\ge1$ (cf.~Lemma \ref{rho-j-rho-1}), we have $a_{i,2}<4=a_{i,1}^2$. Hence, the inequality \eqref{a-double} follows from a simple induction using \eqref{aij-recur}. 

Using another simple induction, we readily confirm the inequality $a_{i,j}\le a_{i,2}^{2^{j-2}}$ in \eqref{basic-aij-ineqs}.

For the lower bound in \eqref{basic-aij-ineqs}, we know from \eqref{eq-ab-1} and \eqref{eq-ab-2} and from the fact that $b_{i,j}=a_{i,j}$ for all $j\le i+1$ that
\begin{equation}\label{eq-ab-1-again} 
	a_{i,j+1} = \theta_{i,j} + a_{i-1, j}^{\rho_{i-1}}
\end{equation} 
and that
\begin{equation}\label{eq-ab-2-again} 
	\theta_{i,j+1} = \theta_{i,j}^2 + 2 a_{i-1,j}^{\rho_{i-1}} \theta_{i,j}
\end{equation}
for $j\in\{1,\dots,i\}$. By a simple induction, these formulas imply that $a_{i,j}>1$ and $\theta_{i,j}>0$ for all $j\le i+1$, and thus $\theta_{i,j+1} + 1 \geq (\theta_{i,j} + 1)^2$ for $j=1,2,\dots,i$. By yet another induction, we find $\theta_{i,j} \geq 3^{2^{j - 1}} - 1$. Finally, the lower bound on the $a_{i,j}$ in \eqref{basic-aij-ineqs} follows from this and \eqref{eq-ab-1-again}.
\end{proof}

\begin{lemma}\label{a-crude}
	Let $t \in (0,1)$. We have 
		\begin{equation}\label{a-sings} 
			a_{j+1}(t) \leq a_{j}(t)^2\quad\text{for}\ j\ge1
		\end{equation} 
and
		\begin{equation}\label{basic-a-ineqs} 
			3^{2^{j-2}} \leq a_{j}(t) \leq a_{2}(t)^{2^{j-2}}\le 4^{2^{j-2}}
			\quad\text{for}\ j\ge 2.
		\end{equation}
\end{lemma}

\begin{proof}
The inequality \eqref{a-sings} follows from a simple induction using \eqref{a-recur}, and the upper bound in \eqref{basic-a-ineqs} follows with a further induction. 

For the lower bound, we first set up relations analogous to \eqref{eq-ab-1-again} and \eqref{eq-ab-2-again}, defining $\theta_j(t)$ for $j\ge1$ via the relation
\begin{equation} \label{eq-ab-single-1} 
	a_{j+1}(t) = \theta_j(t) + a_j(t)^t.
\end{equation} 
We then note that we also have
\begin{equation}\label{eq-ab-single-2} 
	\theta_{j+1}(t) = \theta_j(t)^2 + 2 a_j(t)^t \theta_j(t).\end{equation} 
Indeed, on the one hand, we have
\begin{align*}
	\theta_{j+1}(t)&= a_{j+2}(t)-a_{j+1}(t)^t = a_{j+1}(t)^2-a_j(t)^{2t}
\end{align*}
by \eqref{a-recur}. On the other hand,
\[
 \theta_j(t)^2 + 2 a_j(t)^t \theta_j(t) = \big(\theta_j(t) +a_j(t)\big)^2 - a_j(t)^{2t} = a_{j+1}(t)^2-a_j(t)^{2t}
\]
by \eqref{eq-ab-single-1}. 

Having proven \eqref{eq-ab-single-2}, we now proceed analogously to the proof of Lemma \ref{ai-crude}. We have $a_j(t)>1$ and $\theta_j(t)>0$ for all $j\ge1$, by a simple induction using \eqref{eq-ab-single-1} and \eqref{eq-ab-single-2}. Therefore, from \eqref{eq-ab-single-2}, we have that $\theta_{j+1}(t) +1  \geq (\theta_j(t) + 1)^2$. By induction, this implies that $\theta_{j}(t) \geq 3^{2^{j - 1}} - 1$. Finally, the lower bound on the $a_{j}(t)$ in \eqref{basic-a-ineqs} follows from this and \eqref{eq-ab-single-1}.
\end{proof}

We are now in a position to prove that the relation
\begin{equation}\label{rho-limit-eq-2-again}
\frac{1}{1 - t/2} = \lim_{j\to\infty} \frac{\log a_j(t)}{2^{j-2}}
\end{equation}
holds with $t=\rho$, which is one of the main statements of Proposition \ref{rho-comp}. Iterating \eqref{rho-eqs-a-form} gives
\begin{align*}
a_{i, i+1}  = \exp(2^{i-1}) a_{i-1, i}^{\rho_{i-1}}  
	&= \exp(2^{i-1} + \rho_{i-1} 2^{i-2}) a_{i-2,i-1}^{\rho_{i-2}\rho_{i-1}}
	= \cdots \\ 
	& = \exp \Big(2^{i-1} + \sum_{j=1}^{i-2}(\rho_{i-j} \cdots \rho_{i-1}) 2^{i - j - 1}   \Big) 
	a_{1,2}^{\rho_1 \cdots \rho_{i-1}}.
\end{align*}
By Proposition \ref{mainthm-limit} , we have $\rho_i\to\rho$. In addition, by Lemma \ref{rho-j-rho-1}, we have $0\le \rho_i\le \rho_1<0.31$ for all $i$. Thus, taking limits as $i \rightarrow \infty$ gives
\be\label{limaii1}
\lim_{i \rightarrow \infty} \frac{\log a_{i, i+1}}{2^{i-1}} 
		= 1 + \frac{\rho}{2} 
		+ \Big(\frac{\rho}{2}\Big)^2 + \ldots
		 = \frac{1}{1 - \rho/2}.
\ee

We now derive another expression for the left-hand side of \eqref{limaii1}.
A telescoping argument gives
\be\label{aii-telescope}
\frac{\log a_{i, i+1}}{2^{i-1}} = \log 4 + \sum_{j=1}^i \frac{1}{2^{j-1}} \log
\pfrac{a_{i,j+1}}{a_{i,j}^2}.
\ee
The terms on the right-hand side of \eqref{aii-telescope} are rapidly decreasing. Indeed, by \eqref{a-double} we have $1\ge a_{i,j+1}/a_{i,j}^2$ for all $j\ge1$. On the other hand, by \eqref{aij-recur} (with $j$ replaced by $j+1$ there) and by \eqref{basic-aij-ineqs}, we have
\[
\frac{a_{i,j+1}}{a_{i,j}^2} \ge 1 - \frac{a_{i-1,j-1}^{2\rho_1}}{a_{i,j}^2}
=1 + O\Big(\Big(\frac{2^{\rho_{i-1}}}{3}\Big)^{2^{j-1}}\Big).
\]
for all $j\in\{2,\dots,i\}$. Since $\rho_{i-1}\le \rho_1\le 0.31$, we have $2^{\rho_{i-1}}/3<1/2$. In conclusion,
\begin{equation}\label{1212a}
\log \pfrac{a_{i,j+1}}{a_{i,j}^2} = O(2^{-2^{j-1}})
\end{equation}
for all $j\in\{1,\dots,i\}$. By a simple limiting argument using relation \eqref{aii-telescope}  and Lemma \ref{lemmax}, we thus find that
\[
\lim_{i \rightarrow \infty} \frac{\log a_{i, i+1}}{2^{i-1}} = \log 4 +
\sum_{j=1}^\infty \frac{1}{2^{j-1}} \log\pfrac{a_{j+1}(\rho)}{a_j(\rho)^2} = 
\lim_{j\to\infty} \frac{\log a_j(\rho)}{2^{j-2}}.
\]
Here, we used \eqref{1212a} to bound the terms with $j$ large. 
Comparing this with \eqref{limaii1} confirms that indeed \eqref{rho-limit-eq-2-again}
is satisfied with $t=\rho$.
\vspace*{10pt}

We turn now to the final statement in Proposition \ref{rho-comp}, the statement that \eqref{rho-limit-eq-2-again} has a unique solution in $t \in [0,\frac{1}{3}]$ (which must, by the above discussion, be $\rho$). This is a purely analytic problem. 
Write
\[ W_j(t) := \frac{1}{1 - t/2} - \frac{\log a_j(t)}{2^{j-2}}, \quad W(t) := \lim_{j \rightarrow \infty} W_j(t).\]
We must show that there is only one solution to $W(t) = 0$. We already know $W(\rho)=0$, so it would suffice to show that $W$ is strictly increasing in $[0,1/3]$. This would certainly follow if we could show that 
\[ 
W_j(t') - W_j(t) \geq \frac{1}{6} (t' - t)
\] 
for all $j\ge2$ and all $0\le t\le t'\le 1/3$. Since the derivative of $\frac{1}{1 - t/2}$ is bounded below by $\frac{1}{2}$ on $[0,\frac{1}{3}]$, it is enough to establish the derivative bound
\[ 
\frac{\dee}{\dee t} \pfrac{\log a_j(t)}{2^{j-2}} \le \frac13 
\] 
for all $j \geq 2$ and all $t \in (0,\frac{1}{3})$. The remainder of the section is devoted to proving this bound, which it is convenient to write in the form
\begin{equation}\label{logaj-deriv} 
	\ell_j(t) \leq \frac{1}{3} \cdot 2^{j - 2},
\end{equation} 
where $\ell_j(t) := a'_j(t)/a_j(t)$. 

We begin by observing that, since $t \in (0,\frac{1}{3})$, we have $a_2(t) \leq 2 + 2^{1/3}$ and so we may upgrade the upper bound in \eqref{basic-a-ineqs} to
\begin{equation} \label{stronger-ajt} a_j(t) \leq (2 + 2^{1/3})^{2^{j - 2}}\end{equation} for $j \geq 2$. Note also that, by induction using \eqref{eq-ab-single-1} and \eqref{eq-ab-single-2}, both $a_j(t)$ and $\theta_j(t)$ are increasing functions of $t$. In particular, $a_j(t)$ is an increasing function of $t$ so the derivative $a'_j(t)$ is positive.

Differentiating \eqref{a-recur} gives
\begin{equation}\label{first-ajs}
a_{j+1}' = 2a_j a_j' + \big( a_j^t \log a_j - 2 a_{j-1}^{2t}\log a_{j-1} \big)
+ t a_j^t \frac{a_j'}{a_j} - 2t a_{j-1}^{2t} \frac{a_{j-1}'}{a_{j-1}},
\end{equation} where here and in the next few lines we have omitted the argument $(t)$ from the functions for brevity.  The term in parentheses is non-positive by \eqref{a-sings}, and the final term $- 2t a_{j-1}^{2t} \frac{a_{j-1}'}{a_{j-1}}$ is negative since the derivative $a'_{j-1}$ is positive. It follows from \eqref{first-ajs} that \[
a_{j+1}'  < 2a_j a_j' 
+ t a_j^t \frac{a_j'}{a_j} .
\]
A little computation using \eqref{a-recur} shows that this may equivalently be written as
\begin{equation}\label{thets}
\ell_{j+1} <  2\ell_j
\bigg( \frac{1}{1+a_j^{t-2}-a_{j-1}^{2t}a_j^{-2}} + \frac{t a_j^t}{2a_{j+1}}\bigg),
\end{equation}
where we used our notation $\ell_j = a'_j/a_j$.

Denote
\begin{equation}\label{xi-def} 
\xi_j := \sup_{t \in [0,\frac{1}{3}]} \bigg(\frac{1}{1+a_j(t)^{t-2}-a_{j-1}(t)^{2t}a_j(t)^{-2}} + \frac{t a_j(t)^t}{2a_{j+1}(t)}\bigg).
\end{equation} 
Then \eqref{thets} implies that $\ell_{j+1}(t) < 2 \ell_j(t) \xi_j$ for all $t\in[0,1/3]$ and all $j\ge2$. Telescoping this inequality gives
\[ 
\ell_j(t) \le (\ell_2(t) \xi_2 \xi_3 \cdots \xi_{j-1}) \cdot 2^{j - 2}.
\]
We have 
\[ 
\ell_2(t) = \frac{2^t\log 2}{2+2^t} \leq  \frac{\log 2}{1+2^{2/3}} <  0.268
\]
for all $t\in[0,1/3]$. Hence, in order to obtain the desired bound \eqref{logaj-deriv}, it is enough to show
\begin{equation}\label{prod-xis} 
\xi_2 \xi_3 \cdots \xi_{j-1} < 1.2.
\end{equation}
The $\xi_i$ tend to $1$ exceptionally rapidly, and crude bounds (together with a little computation) turn out to suffice, as follows. 

First, by \eqref{basic-a-ineqs} and the fact that $a_2(t)^{2-t} = (2 + 2^t)^{2 - t} \leq 9$ for $t \in [0,1]$ (a calculus exercise), we have
\begin{equation}\label{1-xi} 
a_j(t)^{t-2}  \ge (a_2(t)^{t-2})^{2^{j-2}} \ge 9^{-2^{j-2}}
\quad\text{for}\ j\ge 2.
\end{equation}
Second, by the lower bound in \eqref{basic-a-ineqs} and by \eqref{stronger-ajt} we have
\[
	a_{j-1}(t)^{2t}a_j(t)^{-2} \le \big((2 + 2^{1/3})^{2^{j-3}}\big)^{2/3} (3^{2^{j-2}} )^{-2}< 6^{-2^{j-2}}\quad \text{for}\ j\ge 3.
\]
We may also check by hand that $a_1(t)^{2t}/a_2(t)^2=(2^{1-t}+1)^{-2}<1/6$ for all $t\in[0,1/3]$. Hence,
\begin{equation}\label{2-xi}  
	a_{j-1}(t)^{2t}a_j(t)^{-2}< 6^{-2^{j-2}}\quad \text{for}\ j\ge 2.
\end{equation}
Third, again by the lower bound in \eqref{basic-a-ineqs} and by \eqref{stronger-ajt}, we have
\begin{equation}\label{3-xi} 
\frac{a_j(t)^t}{a_{j+1}(t)} \leq \frac{\big( (2 + 2^{1/3})^{2^{j-2}}\big)^{1/3} }{3^{2^{j-1}}} \le \pfrac{1}{6}^{2^{j-2}}\quad \text{for}\ j\ge 2.
\end{equation}
Substituting \eqref{1-xi}, \eqref{2-xi} and \eqref{3-xi} into the definition \eqref{xi-def} gives 
\[
\xi_j \le \frac{1}{1+(\frac{1}{9})^{2^{j-2}}-(\frac{1}{6})^{2^{j-2}}} + \pfrac{1}{6}^{1+2^{j-2}} \quad \text{for}\ j\ge 2.
\]
Using this bound, one may check the bound $\prod_{j = 2}^{\infty} \xi_j \leq 10/9$, which is stronger than the desired bound \eqref{prod-xis}, on a pocket calculator or even by hand.  For example, we have $\xi_2\xi_3 \le \frac{46751495}{42169248}$ and can use a very crude bounds for the higher terms.  Since $\frac{1}{1 - x} + \frac{x}{6} \leq e^{2x}$ for $0\le x\le 0.1$, taking $x = 6^{-2^{j-2}}$ gives
\[ 
\xi_j \leq \exp \big( 2 \cdot 6^{-2^{j - 2}}\big)
\] for $j \geq 4$. Therefore
\[ 
\prod_{j = 4}^{\infty} \xi_j < \exp \bigg( 2 \sum_{i = 4}^{\infty} \frac{1}{6^i} \bigg) = e^{2/(5\cdot 6^3)} < 1.002.
\]

This concludes the proof of the final statement in Proposition \ref{rho-comp}.
} %% \rev for this whole subsection

\subsection{Proof of \rev{parts (b) and (c) of Theorem \ref{beta-rho}}}
To conclude this paper, we complete the proof of \rev{parts (b) and (c) of Theorem \ref{beta-rho}, as defined in the end of subsection \ref{sec:equal sums}}. In fact, all of the ingredients have already been assembled and we must simply remark on how they fit together. 

First, recall from Definition \ref{theta-r-def} that
\[ 
\theta_r = (\log 3-1)\Big/\Big(\log 3 + \sum_{i = 1}^{r-1} \frac{2^i}{\rho_1 \cdots \rho_i}\Big).
\]
Now, it is an easy exercise to see that if $x_1,x_2,\dots$ is a sequence of positive real numbers for which $x = \lim_{i \rightarrow \infty} x_i$ exists and is positive, then
\[ 
\lim_{r \rightarrow \infty}  \Big( \sum_{i = 1}^r x_1 \cdots x_i \Big)^{1/r} = \max (x,1).
\] 
Applying this with $x_i = 2/\rho_i$ gives, by Proposition \ref{mainthm-limit}, that 
\[ \lim_{r \rightarrow \infty} \theta_r^{1/r} = \frac{\rho}{2}.\] This, together with Proposition \ref{rho-comp}, completes the proof of Theorem \ref{beta-rho}.

%\addcontentsline{toc}{part}{Appendix}
\clearpage
\thispagestyle{fancy}
\fancyhf{} % sets both header and footer to nothing
\renewcommand{\headrulewidth}{0cm}
\lhead[{\scriptsize \thepage}]{}
\rhead[]{{\scriptsize\thepage}}
\part*{Appendix}
\appendix

\section{Some probabilistic lemmas}

Throughout this section, $\A \subset \N$ will be a random set, with $\P(i \in \A) = 1/i$ and these choices being independent for different values of $i$.

\begin{lemma}\label{AB-Poisson}
For any finite subset $\rev{B \subset \Z_{\ge4}}$ and any $k\in\Z_{\ge0}$, we have
\[
\(1 - \frac{2k^2 (\sum_{m\in B} 1/(m-1))^{-2}}{\min B}  \) M \le
\P( \#(\A \cap B)=k ) \le M,
\]
where
\[
M = \frac{1}{k!} \bigg( \sum_{m\in B} \frac{1}{m-1} \bigg)^{k}\prod_{m\in B} \(1-\frac{1}{m}\) .
\]
\end{lemma}

\begin{proof}
The result follows by a standard inclusion-exclusion argument. \rev{We have}
\dalign{
\P( \#(\A \cap B) = k) &= \ssum{a_1,\ldots,a_k\in B \\a_1<\cdots<a_k} \frac1{a_1\cdots a_k} \sprod{m\in B \\ m\not \in \{a_1,\ldots,a_k\}} \(1-\frac{1}{m}\)\\
&= \prod_{m\in B} \(1-\frac{1}{m}\) \ssum{a_1,\ldots,a_k\in B \\a_1<\cdots<a_k} \frac1{(a_1-1)\cdots (a_k-1)} 
\le M.
}
For the lower bound, we note that
\rev{\begin{align*}
 &\frac{1}{k!} \bigg( \sum_{m\in B} \frac{1}{m-1} \bigg)^{k} - 
 \ssum{a_1,\ldots,a_k\in B \\a_1<\cdots<a_k} \frac{1}{(a_1-1)\cdots(a_k-1)} \\
&\qquad = \frac{1}{k!}\ssum{a_1,\ldots,a_k\in B \\\exists i<j\ \text{with}\ a_i=a_j}  \frac{1}{(a_1-1)\cdots (a_k-1)} \\
	&\qquad\le\frac{1}{k!} \binom{k}{2} \bigg(\sum_{a\in B} \frac{1}{(a-1)^2} \bigg)\bigg(\sum_{a\in B}\frac{1}{(a-1)}\bigg)^{k-2}.
\end{align*}
Since $\sum_{a\in B}1/(a-1)^2 < 1/(\min B-2)^2\le 4/(\min B)^2$, the proof is complete.}
\end{proof}

\begin{lemma}\label{A-normal} 
Uniformly for $B\subset\N$ with $\lambda:=\sum_{m\in B}1/m\ge1$ \rev{and $0\le \eps\le 1$}, we have
\[
\P \Big( \big|\#(\A\cap B) - \lambda \big| > \eps \lambda \Big)
\ll \exp(-\eps^2\lambda/3).
\]
\end{lemma}

\rev{\begin{proof} 
	This follows by \rev{the upper bound in Lemma \ref{AB-Poisson}} with standard bounds on the tails of the Poisson distribution, e.g. Norton's bounds \cite[Theorem 09]{HT-Divisors}.
\end{proof}}

\begin{lemma}\label{A-moments}
For any $x>0$ and finite set $B\subset \N$, 
\[
\E x^{\# (\A \cap B)} \le \exp \Big( (x-1)\sum_{j\in B} \frac{1}{j} \Big). 
\]
\end{lemma}

\begin{proof}
The random variable $\# (\A \cap B)$ is the sum of independent
Bernouilli random variables and thus
\[
\E x^{\# (\A \cap B)}  = \prod_{j\in B} \bigg(1+\frac{x-1}{j}\bigg).
\]
\rev{Note that all factors are positive because $x>0$.} The lemma now follows from the inequality $1+y\le e^y$, valid
for all real $y$.
\end{proof}

\begin{lemma}\label{A-cond}
\rev{Let $k\in\N$, and let $B$ and $G$ be finite sets such that $B\subset G\subset\Z_{\ge4}$ and 
	\[
	|B|=k \le \frac{\sqrt{\min(G)}}{2}\sum_{m\in G} \frac{1}{m}.
	\]
Then} 
\[
\P\big( \A\cap G = B \,\big|\, \#(\A\cap G) = k  \big) =\frac{k!(1+O(\frac{k^2(\sum_{m\in G}1/m)^{-2}}{\min(G)}))}{(\sum_{m\in G} 1/(m-1))^k}\, \prod_{b\in B} \frac{1}{b} \prod_{m\in G} \(1-\frac{1}{m}\).
\] 
\end{lemma}

\begin{proof} Since $|B|=k$, we have 
\[
	\P\big( \A\cap G = B \,\big| \, \#(\A\cap G) = k  \big) =
	\frac{\P(\A\cap G=B)}{\P(\#(\A\cap G)=k)} .
\]
The denominator is estimated using Lemma \ref{AB-Poisson}, whereas for the numerator we simply note that
\begin{align*}
\P(\A\cap G=B)
	= \prod_{b\in B}\frac{1}{b}\prod_{m\in G\setminus B}\bigg(1-\frac{1}{m}\bigg) 
	&= \prod_{b\in B}\frac{1}{b-1}
	\prod_{m\in G}\bigg(1-\frac{1}{m}\bigg).
\end{align*}
This completes the proof of the lemma.
\end{proof}

\begin{lemma}\label{normalA}
  Given $0<c<1$ and $D \ge e^{100/c}$, the probability that
   $\A\subset (D^c,D]$ satisfies
   \begin{equation}\label{typical A}
     \Big| \#\big(\A\cap(D^{\alpha},D^{\beta}]\big)- (\beta-\alpha)\log D\Big| \le (\log D)^{3/4}
       \quad (c\le \alpha \le \beta \le 1)
   \end{equation}
   is $\ge 1 - O(e^{-(1/4) (\log D)^{1/2}})$.
  \end{lemma}

\begin{proof}
It suffices to bound the probability that
\be\label{typical2}
\Big| \#\A\cap(D^{\alpha},D^{\beta}]- (\beta-\alpha)\log D\Big| \ge (\log D)^{3/4}-2
\ee
whenever $\alpha\log D, \beta\log D \in \N$.   The random variable
$\rev{N=N(\alpha,\beta):= \# (\A\cap (D^{\alpha},D^{\beta}])}$ is the sum of Bernoulli random variables and has \rev{expectation
 $\E N = M+O(1)$}, where
\[
\rev{M= (\beta-\alpha)\log D.}
\]
By Lemma \ref{A-moments}, $\E \lambda^{ N}\le e^{(\lambda-1) \E N}$.
Thus, \rev{for $y = (\log D)^{3/4}$} and
$\lambda_j = 1+ (-1)^j \frac{y}{\log D}$ we have
\dalign{
\P (N \ge M + y) &\le \E \lambda_2^{N-M-y} \ll \lambda_2^{-M-y} e^{(\lambda_2-1)M}
\ll e^{-(1/3) (\log D)^{1/2}}, \\
\P (N \le M - y) &\le \E \lambda_1^{N-M+y}
     \ll \rev{\lambda_1^{-M+y}} e^{(\lambda_1-1)M}
\ll e^{-(1/3) (\log D)^{1/2}}.
}
Summing over all possible $\alpha,\beta$ completes the proof.
\end{proof}

 \begin{lemma}\label{sum_a}
 Uniformly for $X \geq 2$ \rev{and $K\ge 2$} we have
 \[
\rev{ \sum_{a\in \A \cap [2,X] } a \leq K X }
 \]
 with \rev{probability $\geq 1-e^{2 - K}$.}
 \end{lemma}
 \begin{proof}
 We use \rev{Chernoff}'s inequality, often called Rankin's trick in this context:
\begin{align*}
\rev{\P \Big( \sum_{a\in \A\cap[2,X]} a > KX \Big)} & \rev{\leq e^{-K}} \sum_{A' \subset [2,X]} \P\big( \A \cap [2,X]=A' \big) e^{\frac{1}{X}\sum_{a\in A'} a } \\
&=\rev{e^{-K}} \sum_{A' \subset [2,X]} \prod_{\substack{2\leq a\leq X \\ a\not\in A'} }
\bigg(1-\frac{1}{a}\bigg) \prod_{a\in A'} \frac{e^{a/X}}{a} \\
&= \rev{e^{-K}} \prod_{2\leq a\leq X} \bigg(1-\frac{1}{a}\bigg)\bigg (1+ \frac{e^{a/X}}{\rev{a-1}}\bigg) \\ 
&= \rev{e^{-K} \prod_{2\leq a\leq X} \bigg(1+\frac{e^{a/X}-1}{a}\bigg)} \\
& \leq \rev{e^{-K} (1 + 2/X)^X\leq e^{2 - K}}
\end{align*}
because $e^t\le 1+2t$ for all $t\in[0,1]$. This concludes the proof.
 \end{proof}

\begin{lemma}\label{interval-product}
Let \rev{$\eta\in[0,1]$ and let} $J_1,\dots, J_d \subset \N$ be mutually disjoint intervals. 
%be intervals, any pair
%of which are either equal or disjoint.
Suppose that $X \subset J_1 \times \cdots \times J_d$ is a set of size $\eta \prod_i \max J_i$. If $\min_i |J_i|$ is sufficiently large in terms of $\eta$ and $d$, then with probability $\ge (\eta/4)^d$, there are distinct elements $a_i \in \A$ with $(a_1,\dots, a_d) \in X$.%, where $C_d$ is a large constant, depending only on $d$.
\end{lemma}

\begin{proof}
%We first handle the case in which the $J_i$ are all disjoint.
Let $M_i = \max J_i$ for each $i$. We will prove the lemma by induction on $d$. 

The case $d = 1$ follows by direct calculation: Suppose that $X \subset J_1$ has size $\ge \eta M_1$.  Then 
\[ 
\P(\A \cap X = \emptyset) 
= \prod_{n \in X} (1 - 1/n) \leq (1-1/M_1)^{\eta M_1} \leq e^{-\eta}
\le 1-\eta/2.
\]
\rev{Let us now assume we have proven the lemma for $d-1$ intervals, and let us prove it for $d$ intervals $J_1,\dots,J_d$. For each $j_1\in J_1$, we set
\[ 
X_{j_1} := \{(j_2,\dots, j_d) \in J_2 \times \cdots \times J_d : (j_1,j_2,\dots, j_d) \in X\}.
\]
Let $Y=\{j_1\in J_1:|X_{j_1}|\ge (\eta/2)M_1\}$. Then $|Y|\ge (\eta/2)M_1$, because otherwise we would have $|X|<\eta\prod_i M_i$, a contradiction to our hypotheses.} By the case $d = 1$ (just described), $\A \cap Y$ is nonempty with probability $\ge \eta/4$. Fix some $a_1 \in \A \cap Y$. Then, by the inductive hypothesis and the fact that the $J_i$ are disjoint, with probability $\ge (\eta/4)^{d-1}$, independent of the choice of $a_1$, there are elements $a_i \in \A \cap J_i$, $i = 2,\dots, d$ with $(a_2,\dots, a_d) \in X_{a_1}$, and therefore $(a_1,\dots, a_d) \in X$. The disjointness of the $J_i$ of course guarantees that the $a_i$ are all distinct.  This completes the proof.
% when the $J_i$ are disjoint.
%
%Now, suppose the list $J_1,\dots,J_d$ consists of $m_t$ copies of $I_t$, $t = 1,2,\dots, d'$, where the $I_t$ are disjoint.
%Let $K$ be sufficiently large in terms of $d$ and $\eta$. 
%Each $I_t$ contains $K$ subintervals $I_{t,u}$, each of size $|I_t|/K+O(1)$.  Assume that $|I_t| \ge K^3$ for all $t$, say.
%Then $J_1 \times \cdots \times J_d$ may be expressed as a union of $K^d$ products of $d$ intervals $I_{t,u}$, each having size $\ge (K+1)^{-d} \prod |J_i|$. Of these $K^d$ products, the number which do not consist of \emph{disjoint} intervals is $\le d^2K^{d-1}$, and so the union of all of these has size at most $\frac{d^2}{K-1} \prod_i |J_i|$. This is $< \frac{\eta}{2}\prod M_i$ if $K
%= \fl{d^2 2^{d+2}/\eta}$, say.   It follows that 
% some product of disjoint sets $I_{j,u}$ (as a subset of
%$J_1\times \cdots \times J_d$) has at least $\eta/(2K^d)M_1\cdots M_d$ elements of $X$.
%The result follows from the disjoint case already established.
\end{proof}
 
 \rev{
\begin{lemma}\label{lem:dTV}
If $X_j,Y_j$ live on the same discrete probability space for $1\le j\le k$, and furthermore
$X_1,\ldots,X_k$ are independent, and $Y_1,\ldots,Y_k$ are also independent, then
\[
d_{\TV}((X_1,\ldots,X_k),(Y_1,\ldots,Y_k)) \le \sum_{j=1}^k d_{\TV}(X_j,Y_j),
\]
\end{lemma} 
 
\begin{proof}
 We begin with the following identity
\[
a_1 \cdots a_m - b_1\cdots b_m = \sum_{j=1}^m  (a_j-b_j)\prod_{i<j}
a_i \prod_{i>j} b_i.
\]
Denoting $\Omega$ the domain of $(X_1,\ldots,X_m)$, and writing
$a_i=\P(X_i=\omega_i)$, $b_i=\P(Y_i=\omega_i)$,
 we then have
\begin{align*}
d_{TV}  ( (X_1,\ldots,X_m),(Y_1,\ldots,Y_m) ) &= \frac12 \sum_{
  (\omega_1,\ldots,\omega_m)\in \Omega} \big| \P(X_j=\omega_j, 1\le
j\le m) - \P(Y_j=\omega_j, 1\le j\le m) \big| \\
 &= \frac12 \sum_{  (\omega_1,\ldots,\omega_m)\in \Omega} |a_1\cdots
a_m - b_1\cdots b_m|\\
 &\le \frac12 \sum_{j=1}^m \sum_{\omega_j} |a_j-b_j| \sum_{\omega_i \;
  (i\ne j)}  \prod_{i<j} a_i \prod_{i>j} b_i \\
&= \frac12 \sum_{j=1}^m \sum_{\omega_j} |a_j-b_j| \\
&=  \sum_{j=1}^m d_{TV} (X_j,Y_j).
\end{align*}
\end{proof}
}

\section{Basic properties of entropy}\label{entropy-appendix}

The notion of entropy plays a key role in our paper. In this appendix we record the key facts about it that we need. Proofs may be found in many places. One convenient resource is \cite{alon-spencer}.

If $X$ is a random variable taking values in a finite set then we define
\[ \HH(X) := - \sum_x \P(X = x) \log (\P(X = x)),\] where the log is to base $e$ \rev{and the summation runs over the range of $X$}.

If $\mathbf{p} = (p_1,\dots, p_n)$ is a vector of probabilities (that is, if $p_1,\dots, p_n \geq 0$ and $p_1 + \dots + p_n = 1$), then we write
\[ \HH(\mathbf{p}) := - \sum_{i = 1}^n p_i \log p_i.\]

There should be no danger of confusing the two slightly different usages. 

Our first lemma gives a simple upper bound for multinomial coefficients in terms of entropies.

\begin{lemma}\label{entropy-multinomial}
Let $n, n_1,\dots, n_k$ be non-negative integers with $\sum n_i = n$. Then
\[ \frac{n!}{n_1! \cdots n_k!} \leq e^{\HH(\mathbf{p}) n},\] where $\mathbf{p} = (p_1,\dots, p_k)$ with $p_i := n_i/n$.
\end{lemma}
\begin{proof}
The right-hand side is $(n/n_1)^{n_1} \cdots (n/n_k)^{n_k}$. Now simply observe that 
\[
\frac{n!}{(n_1)! \cdots (n_k)!}  (n_1/n)^{n_1} \cdots (n_k/n)^{n_k} \leq  \sum_{k_1+\cdots+k_m=n}\frac{n!}{k_1!\cdots k_m!} (n_1/n)^{k_1}\cdots (n_k/n)^{k_m} = 1.
\qedhere\]
\end{proof}

Our next lemma is a simple and well-known upper bound for the entropy.

\begin{lemma}\label{ent-trivial}
Let $X$ be a random variable taking values in a set of size $N$. Then $\HH(X) \leq \log N$.
\end{lemma}
\begin{proof}
Follows immediately from the convexity of the function $L(x) = -x \log x$ and Jensen's inequality. See \cite[Lemma 14.6.1 (i)]{alon-spencer}.
\end{proof}

The next lemma is simple and has no doubt appeared elsewhere, but we do not know an explicit reference. In its statement, we use the notation $\langle \mathbf{a},\mathbf{p}\rangle = \sum_{i = 1}^n a_i p_i$.

\begin{lemma}\label{real-var-ent}
Let $\mathbf{p} = (p_1,\dots, p_n)$ be a vector of probabilities, and let $\mathbf{a} = (a_1,\dots, a_n)$ be a vector of real numbers. Then
\[ \HH(\mathbf{p}) + \langle \mathbf{a}, \mathbf{p}\rangle \leq \log \Big(\sum_{j=1}^n e^{a_j}\Big), 
\] and equality occurs \rev{if and only if} $p_j = e^{a_j} / \sum_{i=1}^n e^{a_i}$ \rev{for all $j$}. 
\end{lemma}
\begin{proof} \rev{Let us begin by recalling that if $t_1,\dots,t_n>0$ are such that $t_1+\cdots+t_n=1$, then the concavity of the logarithm implies that 
\begin{equation}\label{log-concavity}
t_1\log x_1+\cdots+t_n\log x_n\le \log(t_1x_1+\cdots+t_nx_n)
\end{equation}
for all $x_1,\cdots,x_n>0$. In addition, equality occurs in \eqref{log-concavity} if and only if $x_1=\cdots=x_n$. One may also prove this fact by induction on $n$, and by noticing that the case $n=2$ is equivalent to having $u^t\le tu+ 1-t$ for all $u>0$ and all $t\in(0,1)$, with equality occurring if and only if $u=1$.}

\rev{Let us now proved the lemma. If $p_j=1$ for some $j$, then $\HH(\mathbf{p})+\langle \mathbf{a}, \mathbf{p}\rangle= a_j$. If $n=1$, then this is equal to $\log(\sum_{i=1}^ne^{a_i})$, whereas if $n\ge2$, then we have $a_j<\log(\sum_{i=1}^n e^{a_i})$, so that the lemma holds in both cases. Assume now that $p_j\in(0,1)$ for all $j$. We then have
\[
\HH(\mathbf{p})+\langle \mathbf{a}, \mathbf{p}\rangle =  \sum_{j=1}^n p_j \log (e^{a_j}/p_j).
\]
We may then use \eqref{log-concavity} with $t_j=p_j$ and $x_j=e^{a_j}/p_j$ to complete the proof of the lemma.}
\end{proof}
          
The next lemma, known as the \rev{{\it chain rule for entropy}}, is nothing more than a short computation.
\begin{lemma}\label{ent-chain-rule}
Let $X, Y$ be random variables taking values in finite sets. Then 
\[ \HH(X,Y) = \HH(Y) + \sum_y \P(Y = y) \HH(X | Y = y).\]
\end{lemma}

\begin{remark*} The sum over $y$ is usually written $\HH(X | Y)$ and called the \rev{{\it conditional entropy}}.
\end{remark*}

We will apply the preceding result together with the following observation.

\begin{lemma}\label{det-ent} Suppose that $X, Y$ are random variables with finite ranges and that $Y$ is a deterministic function of $X$. Then $\HH(X,Y) = \HH(X)$.
\end{lemma}
\begin{proof}
This follows from Lemma \ref{ent-chain-rule} with the role of $X$ and $Y$ reversed, since all the entropies $\HH(Y | X = x)$ are zero.
\end{proof}

The next result, known as the submodularity property of entropy, is a crucial ingredient in our paper.

\begin{lemma}\label{submodular-entropy}. Let $X, Y, Z$ be any random variables taking values in finite sets. Then
\[ \HH(X,Y) + \HH(X,Z) \geq \HH(X,Y,Z) + \HH(X).\] 
\end{lemma}
\begin{proof} This is \cite[Lemma 14.6.1 (iv)]{alon-spencer}.\end{proof}

\section{Maier-Tenenbaum flags}\label{previous-app}

The purpose of this appendix is to say a little more about the bound \eqref{gamma-2r}, which corresponds in the language of this paper to \cite[Theorem 1.4]{MT09}. Numerically, this bound is $\tilde\gamma_{2^r} \gg (0.12885796477\ldots)^r$, which is a little weaker than the bound leading to Theorem \ref{beta-rho}, which is $\tilde\gamma_{2^r} \gg (0.140605674848\ldots)^r$. What is interesting, however, is that the flags $\sV$ which lead to \eqref{gamma-2r} are completely different to the binary flags which have been the main focus of our paper. The fact that these very different flags -- the ``Maier--Tenenbaum flags'' -- lead to a result which appears to be within 10 \% of optimal suggests that they will have a key role to play in any future upper bound arguments for these questions.

\begin{definition}[Maier--Tenenbaum flag of order $r$]
Let $k = 2^r$ be a power of two. Identify $\Q^k$ with $\Q^{\mathcal{P}[r]}$ and define a flag $\sV$, $\langle \mathbf{1}\rangle = V_0 \leq V_1 \leq \cdots \leq V_r  \leq \Q^{\mathcal{P}[r]}$, as follows: $V_i = \Span(\mathbf{1}, \omega^1,\dots, \omega^{\rev{i}})$, where $\omega^i_S = 1_{i \in S}$ for $S \subset [r]$.
\end{definition}

\begin{remark*} We have $\dim(V_i) = i+1$ and in particular $V_r$ is much smaller than $\Q^k$, in contrast to the situation for binary systems. We leave it to the reader to check that $\sV$ is nondegenerate.
\end{remark*}

Recall that $\sV$ gives rise to a tree structure, with the cells at level $i$ being the intersections of cosets $x + V_i$ with the cube $\{0,1\}^k$ (cf.~subsection \ref{rho-equations-sec}). It is easy to check that this tree structure has a very simple form, with the cell $\Gamma_i = V_i \cap \{0,1\}^k$ being $\{\mathbf{0},\mathbf{1}, \omega^1, \mathbf{1} - \omega^1,\dots, \omega^i, \mathbf{1} - \omega^i\}$, this dividing into three children at level $i-1$; the cell $\Gamma_{i-1}$ together with two singletons, $\{ \omega^i\}$ and $\{\mathbf{1} - \omega^i\}$.

The recursive definition of the quantities $f^C({\bm \rho})$ (see \eqref{fc-eqs}) therefore becomes  $f^{\Gamma_1}({\bm \rho}) = 3$, 
\begin{equation}\label{MT-recursion}
f^{\Gamma_{j+1}}({\bm \rho}) = f^{\Gamma_j}({\bm \rho})^{\rho_j} + 2.
\end{equation}
\rev{In addition}, the $\rho$-equations \eqref{rho-equations} become  
\begin{equation}\label{MT-rho}
f^{\Gamma_{j+1}}({\bm \rho}) = e (f^{\Gamma_j}({\bm \rho}))^{\rho_j}.
\end{equation}
\rev{On the one hand, iterating \eqref{MT-rho} yields that $\log f^{\Gamma_j}(\bs\rho) = \rho_1\cdots\rho_{j-1}\log 3 + \sum_{i=0}^{j-2}\rho_{j-1}\cdots\rho_{j-i}$ for all $j\ge1$. On the other hand, combining \eqref{MT-recursion} and \eqref{MT-rho}, we find that $\rho_j\log f^{\Gamma_j}(\bs\rho)=\log2-\log(e-1)$, and thus $\rho_1\cdots\rho_j\log 3 + \sum_{i=0}^{j-2}\rho_j\rho_{j-1}\cdots\rho_{j-i}=\log2-\log(e-1)$ for all $j\ge1$. Hence, we obtain the formulas}
\[ 
\rho_1 = \frac{\log 2 - \log(e-1)}{\log 3},\qquad 
\rho_2 = \rho_3 = \cdots = \frac{\log 2 - \log (e-1)}{\log 2 + 1 - \log (e - 1)} =: \kappa.
\] 
\rev{Let us also note that the above discussion implies that
\begin{equation}\label{MT-f}
	\log f^{\Gamma_j}({\bm \rho}) = \frac{\log2-\log(e-1)}{\rho_j} 
		= \begin{cases} 
				\log 3&\text{if}\ j=1,\\ 
				\log2-\log(e-1)+1&\text{if}\ j\ge2.
			\end{cases}
\end{equation}}

Now, assuming that the conditions of Proposition \ref{main-optim} hold, we therefore have
\[
\gamma_k^{\res}(\sV) 
	= (\log 3 - 1)\Big/\Big(\log 3 + \frac{1}{\rho_1}\Big(1 + \frac{1}{\kappa}+\cdots + \frac{1}{\kappa^{r-2}}\Big) \Big)
	= \Big(1 - \frac{1}{\log 3}\Big) \kappa^{r-1}.
\]
Now it can be shown by explicit calculation that the conditions of Proposition \ref{main-optim} do hold. The optimal measures $\mu_i^*$ are all induced from the measure $\mu^*$ in which 
\[ 
\rev{ \mu^*(\omega^j) = \mu^*(\mathbf{1} - \omega^j) = \mu^*(\Gamma_j) \cdot \frac{1}{f^{\Gamma_j}(\bs\rho)}}  
	=  \begin{cases} 
		 \frac{1}{3}e^{1-r} &\text{if}\ j=1,\\ 
		\frac{e-1}{2e}  e^{j - r} &\text{if}\ j\ge2.
	\end{cases}
\] 
In addition, we have
\[
\rev{ \mu^*(\mathbf{0}) =  \mu^*(\mathbf{1})  = \frac{\mu^*(\Gamma_0)}{2} = \frac{1}{6} e^{1-r} .}
\]
\rev{We may then prove by a slightly lengthy computation whose details we leave to the reader that the optimal parameters $\cc^*$ are given by}
\[ 
c^*_1 = 1,
	\quad  c^*_j = \frac{1}{\kappa^2}\Big(\frac{e - \kappa}{e - 1}\Big) 
		\Big(1 - \frac{1}{\log 3}\Big)  \kappa^j,		
		\quad c^*_{r+1} 
	 	= \Big(1 - \frac{1}{\log 3}\Big) \kappa^{r-1}.
\] 
It can also be shown that $\gamma^{\res}_k(\sV) = \gamma_k(\sV)$, by showing that the full entropy condition \eqref{ent-condition} follows from the restricted conditions \eqref{79-pre}. This is a little involved, but a fairly direct inductive argument can be made to work and this is certainly less subtle than the arguments of Section \ref{entropy-binary-alt}. In this way one may establish the bound
\begin{equation}\label{gamm-sec} 
\gamma_{2^r} 
		\geq \bigg(1 - \frac{1}{\log 3}\bigg) 
		\bigg(\frac{\log 2 - \log (e-1)}{\log 2 + 1 - \log (e - 1)}\bigg)^{r-1} 
		\gg ( 0.131810543\dots)^r.
\end{equation} 
Finally, a relatively routine perturbative argument yields the same bound for $\tilde\gamma_{2^r}$.

It will be noted that \eqref{gamm-sec} is strictly stronger than \eqref{gamma-2r}, the bound obtained in \cite{MT09}. This is because, in essence, Maier and Tenenbaum chose slightly suboptimal measures and parameters on the system $\sV$, roughly corresponding \rev{to $\mu(\omega^j) \sim 3^{j - r-1}$,} which then leads to $c_j \sim \big( \frac{1 - 1/\log 3}{1 - 1/\log 27} \big)^j$.

%%%%%%%%%%%  Refs  %%%%%%%%%%%%

\bibliography{delta}
\bibliographystyle{siam}

%%%%%%%%%%%%%%%%%%%%%%%%%%%%%%%%

\end{document}